\documentclass[reqno]{amsart}
\usepackage[margin=1in]{geometry}
\usepackage{bbm}
\usepackage{bm}
\usepackage{enumitem}
\usepackage{amsfonts}
\usepackage{latexsym, amssymb, amsmath, amscd, amsthm, amsxtra}
\usepackage{mathtools}
\usepackage[all]{xy}
\usepackage{mathrsfs}
\usepackage{fancyhdr}
\usepackage{listings}
\usepackage{hyperref}
\usepackage{cleveref}
\usepackage{soul}
\usepackage{xcolor}
\usepackage{dsfont}
\usepackage{stmaryrd}
\usepackage{comment}

\usepackage{tikz}
\usetikzlibrary{arrows}
\usetikzlibrary{arrows.meta}
\usetikzlibrary{calc}
\allowdisplaybreaks

\newcommand{\vertiii}[1]{{\left\vert\kern-0.25ex\left\vert\kern-0.25ex\left\vert #1 
    \right\vert\kern-0.25ex\right\vert\kern-0.25ex\right\vert}}
\DeclareMathOperator{\re}{{\mathrm{Re}}}
\DeclareMathOperator{\im}{{\mathrm{Im}}}

\hypersetup{
     colorlinks=true
}
\def\P{\mathbb{P}}

\def\E{\mathbb{E}}

\def\Z{\mathbb{Z}}
\def\R{\mathbb{R}}
\def\N{\mathbb{N}}

\def\Q{\mathbb{Q}}
\def\11{{\mathbf{1}}}
\def\Im{\im }

\newcommand{\HS}{\mathrm{HS}}

\newcommand{\BA}{\mathrm{BA}}
\newcommand{\WO}{\mathrm{WO}}

\newcommand{\bp}{\mathbf p}

\newcommand{\T}{\mathbb T}

\newcommand{\bE}{\mathbf{E}}
\newcommand{\fd}{{\mathfrak d}}

\newcommand{\bx}{{\bf{x}}}

\newcommand{\bu}{{\bf{u}}}
\newcommand{\bv}{{\bf{v}}}
\newcommand{\bw}{{\bf{w}}}

\newcommand{\C}{\mathbb C}
\newcommand{\cB}{\mathcal B}

\newcommand{\cK}{\mathcal K}
\newcommand{\cL}{\mathcal L}
\newcommand{\cM}{\mathcal M}
\newcommand{\cW}{\mathcal W}
\newcommand{\OK}{\mathcal O_{\mathcal K}}
\newcommand{\lenk}{l_{\mathcal K}}

\renewcommand{\bar}{\overline}
\newcommand{\wt}{\widetilde}
\newcommand{\wh}{\widehat}

\newcommand{\al}{\alpha}

\newcommand{\qq}[1]{\langle{#1}\rangle}
\newcommand{\qqq}[1]{\llbracket{#1}\rrbracket}

\newcommand{\LK}{{L\to n}}

\newcommand{\Zn}{\widetilde{\mathbb Z}_n^d}
\newcommand{\ZL}{{\mathbb Z}_L^d}
\newcommand{\Gc}{{\mathring G}}

\newcommand{\fa}{{\mathfrak a}}

\newcommand{\fc}{{\mathfrak c}}

\newcommand{\Err}{{\mathcal Err}}

\DeclareMathOperator{\OO}{O}
\DeclareMathOperator{\oo}{o}

\newcommand{\sig}{\sigma}
\newcommand{\bsig}{\boldsymbol{\sigma}}
\newcommand{\bfa}{\boldsymbol{a}}
\newcommand{\cut}{\mathrm{Cut}}
\newcommand{\cutL}{(\mathrm{Cut}_L)}
\newcommand{\cutR}{(\mathrm{Cut}_R)}

\newcommand{\qll}[1]{[\![{#1}]\!]}
\newcommand{\rep}[1]{({#1})}

\newcommand{\erre}{\mathfrak d}

\newcommand{\heta}{\mathfrak h_\lambda}

\DeclareMathOperator{\diag}{diag}
\DeclareMathOperator{\tr}{Tr}
\DeclareMathOperator{\var}{Var}

\newcommand{\be}{\begin{equation}}
\newcommand{\ee}{\end{equation}}
\newcommand{\ii}{\mathrm{i}}
\newcommand{\dd}{\mathrm{d}}

\newcommand{\e}{{\varepsilon}}
\newcommand{\cal}{\mathcal}

\newcommand{\nc}{\normalcolor}

\newcommand{\fn}{{\mathfrak n}}
\newcommand{\fm}{{\mathfrak m}}

\newcommand{\ba}{{\mathbf{a}}}
\newcommand{\bfb}{{\mathbf{b}}}
\newcommand{\thn}{\vartheta}

\newcommand{\dthn}{\boldsymbol{\Theta}}

\makeatletter

\DeclarePairedDelimiter{\@p}{\lparen}{\rparen}
\DeclarePairedDelimiter{\@br}{\lbrack}{\rbrack}
\DeclarePairedDelimiter{\@avg}{\langle}{\rangle}
\DeclarePairedDelimiter{\@bbr}{\llbracket}{\rrbracket}
\DeclarePairedDelimiter{\@abs}{\lvert}{\rvert}
\DeclarePairedDelimiter{\@norm}{\lVert}{\rVert}
\DeclarePairedDelimiter{\@set}{\{}{\}}

\newcommand{\p}{\@ifstar{\@p}{\@p*}}
\newcommand{\br}{\@ifstar{\@br}{\@br*}}
\newcommand{\avg}{\@ifstar{\@avg}{\@avg*}}
\newcommand{\bbr}{\@ifstar{\@bbr}{\@bbr*}}
\newcommand*{\abs}{\@ifstar{\@abs}{\@abs*}}
\newcommand*{\norm}{\@ifstar{\@norm}{\@norm*}}
\newcommand{\set}{\@ifstar{\@set}{\@set*}}

\newcommand{\Thetagen}{\widehat{\Theta}}
\newcommand{\MMgen}{\widehat{M}}
\newcommand{\Mgen}{\widehat{\mathcal{M}}}
\newcommand{\Kgen}{\widehat{\mathcal{K}}}
\newcommand{\Sigmagen}{\widehat{\Sigma}}
\newcommand{\suchthat}{\:\middle\vert\:}

\DeclareMathOperator{\TSP}{TSP}
\DeclareMathOperator{\slice}{SLICE}

\newcommand{\numberthis}{\addtocounter{equation}{1}\tag{\theequation}}

\makeatother

\theoremstyle{plain} 
\newtheorem{theorem}{Theorem}[section]
\newtheorem*{theorem*}{Theorem}
\newtheorem{lemma}[theorem]{Lemma}
\newtheorem{assumption}[theorem]{Assumption}
\newtheorem*{lemma*}{Lemma}
\newtheorem{corollary}[theorem]{Corollary}
\newtheorem*{corollary*}{Corollary}

\newtheorem*{proposition*}{Proposition}
\newtheorem{claim}[theorem]{Claim}
\newtheorem*{claim*}{Claim}
\newtheorem{notation}[theorem]{Notation}

\newtheorem{definition}[theorem]{Definition}
\newtheorem*{definition*}{Definition}
\theoremstyle{remark}
\newtheorem{example}[theorem]{Example}
\newtheorem*{example*}{Example}
\newtheorem{remark}[theorem]{Remark}

\newtheorem*{remark*}{Remark}
\newtheorem*{remarks*}{Remarks}

\allowdisplaybreaks
\def\@setthanks{\vspace{-\baselineskip}\def\thanks##1{\@par##1\@addpunct.}\thankses}

\numberwithin{equation}{section}
\usepackage{environ}

\title{On the localization length of finite-volume random block Schr{\"o}dinger operators}

\author{Steven Khang Truong$^\dagger$}

\author{Fan Yang$^\star$}

\author{Jun Yin$^\ddagger$}

\thanks{$^\dagger$Department of Mathematics, University of California, Los Angeles, \href{mailto:steven@math.ucla.edu}{steven@math.ucla.edu}.}
 \thanks{$^\star$Yau Mathematical Sciences Center, Tsinghua University, and Beijing Institute of Mathematical Sciences and Applications, \href{mailto:fyangmath@mail.tsinghua.edu.cn}{fyangmath@mail.tsinghua.edu.cn}. 
 }
 \thanks{$^\ddagger$Department of Mathematics, University of California, Los Angeles, \href{mailto:jyin@math.ucla.edu}{jyin@math.ucla.edu}. 
}

\begin{document}

\begin{abstract}
We study a general class of \emph{random block Schr{\"o}dinger operators} (RBSOs) in dimensions 1 and 2, which naturally extend the Anderson model by replacing the random potential with a random block potential. Specifically, we focus on two RBSOs---the \emph{block Anderson} and \emph{Wegner orbital} models---defined on the $d$-dimensional torus $(\mathbb Z/L\mathbb Z)^d$. They take the form $H=V + \lambda \Psi$, where $V$ is a block potential with i.i.d.~$W^d\times W^d$ Gaussian diagonal blocks, $\Psi$ describes interactions between neighboring blocks, and $\lambda>0$ is a coupling parameter. We normalize the blocks of $\Psi$ so that each block has a Hilbert-Schmidt norm of the same order as the blocks of $V$.
Assuming $W\ge L^\delta$ for a small constant $\delta>0$ and $\lambda\gg W^{-d/2}$, we establish the following results. In dimension $d=2$, we establish delocalization and quantum unique ergodicity for bulk eigenvectors. 
Combined with the localization result from \cite{Wegner}, which holds under the condition $\lambda\ll W^{-d/2}$, this provides a rigorous proof of the Anderson localization–delocalization transition as $\lambda$ crosses the critical threshold $W^{-d/2}$. 
In dimension $d=1$, we show that the localization length of bulk eigenvectors is at least of order $(W\lambda )^2$, which is believed to be the correct scaling.

\end{abstract}

\maketitle

{
\hypersetup{linkcolor=black}
\tableofcontents
}



\section{Introduction}

Anderson localization-delocalization transition \cite{Anderson} is a fundamental physical phenomenon describing the metal-insulator transition in disordered quantum systems. 
Mathematically, this phenomenon can be studied using the following random Schr{\"o}dinger operator on the $d$-dimensional lattice $\Z^d$: 
\begin{align}\label{Anderson_orig}
H=-\lambda \Delta + V,
\end{align}
where $\Delta$ is the discrete Laplacian on $\Z^d$, $V$ is a random potential with i.i.d.~diagonal entries, and $\lambda>0$ is a coupling parameter that controls the disorder strength---smaller $\lambda$ corresponds to stronger disorder. A key quantity in the Anderson model is the \emph{localization length} $\ell$, which characterizes the typical length scale of the region in which most $L^2$-mass of an eigenvector is concentrated.
In the strong disorder regime (i.e., small $\lambda$), for all dimensions $d\ge 1$, the eigenvectors of $H$ are expected to have finite localization lengths and exhibit exponential decay beyond this scale. 
In the weak disorder regime (i.e., large $\lambda$), it is conjectured that for dimensions $d\ge 3$, mobility edges exist, separating localized and delocalized phases---near the spectral edge, eigenvectors remain localized, whereas within the bulk of the spectrum, they are expected to be delocalized, meaning their localization lengths become infinite.

The localization phenomenon in the one-dimensional (1D) Anderson model has been well understood for a long time; see, for example, \cite{GMP, KunzSou, Carmona1982_Duke, Damanik2002} among many other references. Studying the Anderson model in higher dimensions $d\ge 2$, however, presents significantly greater challenges. In particular, in the two-dimensional (2D) Anderson model, localization is expected to occur at all energies \cite{PRL_Anderson}, with the localization length conjectured to be of the order $\exp(\OO(\lambda^2))$ under weak disorder.
The first rigorous localization result supporting this conjecture was established by Fr{\"o}hlich and Spencer \cite{FroSpen_1983}
through multi-scale analysis. A simpler proof was later provided by Aizenman and Molchanov \cite{Aizenman1993} using a fractional moment method.
Numerous significant results regarding the localization of the Anderson model in dimensions $d\ge 2$ have been proven (see, for instance, \cite{FroSpen_1985, Carmona1987, SimonWolff, Aizenman1994,ASFH2001, Bourgain2005, Germinet2013, DingSmart2020, LiZhang2019}). 
However, the conjecture remains unresolved, as localization has only been confirmed for strong disorder or energies near the spectral edges. Furthermore, the existing literature does not provide any quantitative bounds (in terms of $\lambda$) on the localization length. 
For more comprehensive reviews and related references, we refer readers to \cite{CarLa1990,Kirsch2007, Stolz2011, Spencer_Anderson, Aizenman_book, Hundertmark_book}. 

To understand the Anderson localization-delocalization transition phenomenon, another well-known model has been investigated in the literature: the \emph{random band matrix} (RBM) ensemble \cite{ConJ-Ref1, ConJ-Ref2,fy}, a finite-volume model defined on $d$-dimensional lattices $\ZL:=\{1,2, \cdots, L\}^d$ of linear size $L$. 
The RBM is a Wigner-type random matrix $(H_{xy})_{x,y\in \Z_L^d}$ with non-negligible entries only when $|x-y|\le W$, where $W$ represents the band width. Heuristically, the RBM and the Anderson model are expected to exhibit similar qualitative properties when $\lambda \asymp  W$. 
More precisely, simulations \cite{ConJ-Ref1, ConJ-Ref2, ConJ-Ref4, ConJ-Ref6} and non-rigorous supersymmetric arguments \cite{fy} suggest that the localization length of the 1D RBM should be of order $W^2$, while the localization length of the 2D RBM is conjectured to grow exponentially as $\exp(\OO(W^{2}))$. In particular, when the localization lengths exceed $L$ as $W$ increases, the RBM exhibits a transition from the localized phase to the delocalized phase.
Significant progress has been made in establishing both the localized and delocalized phases of the RBM in all dimensions. For a brief review of relevant references, we direct readers to \cite{PB_review, PartI, CPSS1_4,BandI}. 
As of now, the best upper bound on the localization length for the 1D RBM is of the order $W^4$, as established in a series of works \cite{Sch2009,Wegner,CS1_4,CPSS1_4}. However, the localization result for 2D random band matrices is still missing from the literature, let alone any quantitative upper bounds on the localization length.
Regarding lower bounds, a recent breakthrough \cite{Band1D} establishes the delocalization of 1D RBM under the sharp condition $W\gg L^{1/2}$ on band width, and hence provides a sharp lower bound of order $W^2$ for the localization length. Later, a delocalization result for 2D RBM was proven in \cite{Band2D} when $W\ge L^\e$ for an arbitrary constant $\e>0$, which implies a lower bound $\Omega(W^C)$ on the localization length of the 2D RBM for any large constant $C>0$.

In this paper, we aim to gain a deeper understanding of the localization length of the Anderson model by exploring a natural class of interpolations between the Anderson model and the RBM, called \emph{random block Schr{\"o}dinger operators} (RBSOs), which are defined by replacing the i.i.d.~potential with an i.i.d.~block potential. 
Specifically, we define our finite-volume RBSO, denoted by $H$, on $\ZL$ as follows: decompose $\ZL$ into $n^d$ disjoint boxes, each of side length $W$ (with $L=nW$). Correspondingly, define $V$ as a diagonal block matrix: $V=\diag(V_1,\ldots, V_{n^d})$, where $V_i$'s are i.i.d.~$W^d\times W^d$ random matrices, with indices labeled by the vertices in these boxes. For simplicity, we assume that they are drawn from the Gaussian unitary ensemble (GUE). 
Then, we replace the interaction term $-\Delta$ in \eqref{Anderson_orig} with a block matrix $\Psi$ that introduces interactions between neighboring blocks, thereby defining the random Hamiltonian $H=\lambda \Psi + V$. 
The nearest-neighbor blocks of $\Psi$ can either be deterministic matrices (which can be almost arbitrary; see \Cref{def: BM} (i) below) or independent random matrices. In the former case, we refer to it as the \emph{block Anderson model}, while in the latter, it is known as the celebrated \emph{Wegner orbital model}.  
Our model is inspired by the work of Wegner \cite{Wegner1} and its subsequent developments in \cite{Wegner2, Wegner3}, which proposed using the Wegner orbital model to describe the motion of quantum particles with multiple internal degrees of freedom (referred to as orbits or spin) in a disordered medium. Physically, the particle moves according to the Anderson model on the lattice $\Z_n^d$ of blocks, with its spin either rotating deterministically or randomly as the particle hops between sites on $\Z_n^d$.

We normalize the blocks of $\Psi$ so that each block has a Hilbert-Schmidt (HS) norm of order $W^{d/2}$, which is the same order as the blocks of $V$. 
The localized regime of the block Anderson and Wegner orbital models has been analyzed in various settings under strong disorder \cite{Sch2009,Wegner,MaceraSodin:CMP,Shapiro_JMP,CPSS1_4,CS1_4}. In particular, it was shown in \cite{Wegner} that for all dimensions $d\ge 1$, the block Anderson and Wegner orbital models are localized with a localization length of order $\OO(W)$ when $\lambda\ll W^{-d/2}$, marking the strong disorder regime.

In this paper, we establish a counterpart to this result for the weak disorder regime when $\lambda\gg W^{-d/2}$. More precisely, suppose $W\ge L^\e$ for an arbitrarily small constant $\e>0$ and $\lambda\gg W^{-d/2}$.
We prove that the bulk eigenvectors of $H$ in 2D are delocalized and satisfy a quantum unique ergodicity (QUE) estimate, implying that the localization length of the finite-volume RBSO is equal to $L$. 
In other words, this result provides a lower bound of order $\Omega(W^C)$ on the localization length for any large constant $C>0$. 
For 1D RBSOs, we show that their localization lengths are of order $\Omega((W\lambda)^2 \wedge L)$. In particular, when $(W\lambda)^2\ll L$, our result offers a lower bound on the localization length of order $(W\lambda)^2$. As $\lambda\to 1$, this aligns with the $W^2$ scaling for the localization length of 1D RBM. 
On the other hand, when $W\to 1$ and $\lambda$ is large, we recover the expected $\lambda^2$ scaling for the localization length of the Anderson model. This suggests that $(W\lambda)^2$ should be the correct scaling for our RBSOs. 
To establish this rigorously, it remains to prove an upper bound of order $(W\lambda)^2$ for the localization length, which is currently missing in the literature. 
For Wegner orbital models with $\lambda=1$ (which reduce to the RBM), the best-known upper bound is $\OO(W^4)$, as previously noted. 

Our result for the block Anderson model in 1D builds upon and extends the findings of \cite{QC_YY}. In that work, a localization–delocalization transition for bulk eigenvectors was established under the simpler regime $W=\Omega(L)$, with the critical threshold occurring at $\lambda\asymp W^{-d/2}$. The results of the present paper are consistent with those of \cite{QC_YY} but significantly extend the analysis to the more challenging regime $W\ll L$. (However, the results in \cite{QC_YY} are more comprehensive in certain aspects: they allow for non-Gaussian entries in $V$, include results in the localized regime, and also address bulk eigenvalue statistics.) 
For the Wegner orbital model, we remark that while it can be viewed as an extension of the RBM studied in \cite{Band1D, Band2D}, our focus will be on the small parameter regime with $\lambda\ll 1$. 
This does not cover the settings discussed in \cite{Band1D, Band2D} where $\lambda\asymp 1$. 
In \cite{Band2D}, it was shown that as long as $W$ is slightly greater than 1 (i.e., $W\ge L^\delta$ for a small constant $\delta>0$), the 2D RBM is already delocalized. Our study of 2D RBSOs near the transition threshold when $\lambda =W^{-d/2+\e}$ reveals an even more surprising fact: provided that the interaction strength, measured by \smash{$\|\lambda\Psi\|_{\HS}^2/n^d$}, is slightly greater than \smash{$W^{-d}\cdot (\|V\|_{\HS}^2/n^d)\approx 1$}, 
the system will be delocalized in 2D. 
In particular, for the Wegner orbital model, our result indicates that as long as the variances of the entries in $\lambda\Psi$ are of order $W^{-2d+\e}$ (compared to $W^{-d}$ in RBM), delocalization will occur. 
We note that this phenomenon is fundamentally related---through the $\Theta$-propagator defined in \Cref{def_Theta}, which is the Green's function for a random walk (RW)---to the fact that the RW is only marginally recurrent in 2D.

Finally, it is worth mentioning the delocalization of the Anderson model in dimensions $d\ge 3$, which is one of the most significant challenges in the study of random Schr{\"o}dinger operators. Currently, our understanding of the delocalization conjecture is far more limited than that for localization---almost nonexistent. To the best of our knowledge, the existence of a delocalized regime has only been established for the Anderson model on the Bethe lattice \cite{Bethe_PRL, Bethe_JEMS,Bethe-Anderson}, and not for any finite-dimensional integer lattice $\Z^d$.
For random band matrices in high dimensions, the delocalization of RBM has been proven in a series of papers \cite{BandI, BandII, BandIII} under the conditions $d\ge 7$ and $W\ge L^\e$. This result was recently extended to RBSOs under the same assumptions \cite{RBSO}.
Compared to \cite{RBSO}, the current paper employs a method recently developed in \cite{Band1D}, which is based on a careful analysis of the so-called \emph{loop hierarchy} (as defined in \Cref{subsec:Gloop} below). In contrast, the proof in \cite{RBSO} relies on a complicated diagrammatic expansion method. It appears that the approach in this paper may be challenging to extend to higher dimensions, as is done in \cite{RBSO}, and vice versa. 
The results presented here can be readily extended to cases where the blocks of $V$ are drawn from i.i.d.~Gaussian Orthogonal Ensemble (GOE) or more general Gaussian divisible ensembles, but extending beyond this assumption using the method in this paper seems to be challenging. On the other hand, extending the method in \cite{RBSO} beyond the Gaussian divisible assumption is relatively straightforward.
Unlike the current paper, the results established in \cite{RBSO} are not sharp in several aspects: delocalization was shown when $\lambda\gg W^{-d/4}$, while the critical threshold is conjectured to be $W^{-d/2}$, and the delocalization estimate provided there is non-optimal. 
In this paper, we achieve an (almost) optimal delocalization estimate under a sharp condition on $\lambda$, which has two crucial implications: it offers the first rigorous proof of the Anderson localization-delocalization transition in the literature for a class of 2D random (block) Schr{\"o}dinger operators as the interaction strength varies, and it provides a sharp lower bound on the localization length in 1D.

\subsection{The model}

For definiteness, throughout this paper, we assume that $L=nW$ for some $n,W\in 2\N+1$. Then, we choose the center of the lattice as $0$. However, our results still hold for even $n$ or $W$, as long as we choose a different center for the lattice. Consider a one or two-dimensional lattice in \(\mathbb{Z}^d\), $d\in \{1,2\}$, with $N=L^d$ lattice points, i.e., \(\Z_L^d:=\qll{ -(L-1)/2 , (L-1)/2}^2 \).
Hereafter, for any $a,b\in \R$, we denote $\llbracket a, b\rrbracket: = [a,b]\cap \Z$. 
We will view $\Z_L^d$ as a torus and denote by $\rep{x-y}_L$ the representative of $x-y$ in $\ZL$, i.e.,  
\be\label{representativeL}\rep{x-y}_L:= \left((x-y)+L\Z^d\right)\cap \Z_L^d.\ee
Now, we impose a block structure on $\Z_L^d$ with blocks of side length $W$.

\begin{definition}[Block structure] \label{def: BM2}
For $d\in \{1,2\}$, suppose 
\be
L=nW,\quad N=L^d,
\ee
for some integers $n, W\in 2\N+1$. We divide $\mathbb Z_L^d$ into $n^d$ blocks of linear size $W$, such that the central one is $\qll{ -(W-1)/2, (W-1)/2}^d$. Given any $x\in \Z_L^d$, denote the block containing $x$ by $[x]$. Denote the lattice of blocks $[x]$ by $\Zn$. We will also view $\Zn$ as a torus and denote by $\rep{[x]-[y]}_n$ the representative of $[x]-[y]$ in $\Zn$. For convenience, we will regard $[x]$ both as a vertex of the lattice $\Zn$ and a subset of vertices on the lattice $\Z_L^d$. Denote by $\{x\}$ the representative of $x$ in the block $[0]$ containing $0$, i.e., 
\be\label{eq:xremainder}\{x\}:=(x+W\Z^d)\cap [0] = x - W[x].\ee
Any $x\in \ZL$ can be labeled as $([x],\{x\})$. Correspondingly, we define the tensor product of two vectors $\bu$ and $\bv$ with entries indexed by \smash{$\Zn$} and $[0]$, respectively, as 
\be\label{eq:tensor} \bu\otimes \bv (x):=\bu([x])\bv(\{x\}),\quad x\in \ZL.
\ee
Then, the tensor product of matrices $A_n$ and $A_W$ indexed by $\Zn$ and $[0]$, respectively, is defined through
\be\label{eq:tensor2} A_n\otimes A_W (\bu\otimes \bv) :=(A_n \bu)\otimes (A_W\bv).\ee
\end{definition}

We will use $x\sim y$ to mean that $x$ and $y$ are neighbors on $\ZL$, i.e., $|x-y|=1$. Similarly, $[x]\sim [y]$ means that $[x]$ and $[y]$ are neighbors on \smash{$\Zn$}. 
For definiteness, we use $L^1$-distance in this paper, i.e., $\|x-y\|_L:=\|\rep{x-y}_L\|_1$, which is the (periodic) graph distance on $\ZL$. 
Similarly, we also define the periodic $L^1$-distance $\|\cdot\|_n$ on {$\Zn$}. For simplicity of notations, we will abbreviate 
\begin{align}\label{Japanesebracket} |x-y|\equiv \|x-y\|_L,\ \  &\langle x-y \rangle \equiv \|x-y\|_L + W,\quad x,y \in \ZL, \\
\label{Japanesebracket2} |[x]-[y]|\equiv \|[x]-[y]\|_n,\ \ &\langle [x]-[y] \rangle \equiv \|[x]-[y]\|_n + 1, \quad    [x],[y] \in \Zn.
\end{align}


We are mainly interested in the localization lengths of RBSOs for $\lambda$ near the critical threshold $W^{-d/2}$. Moreover, the delocalization of our model is in principle ``weaker" as $\lambda$ decreases. Hence, for simplicity of presentation, we will assume throughout this paper that $\lambda=\oo(1)$. 
While our proofs could be extended to larger  $\lambda$ without significant difficulties, we do not pursue this direction to maintain clarity. We will now outline the precise assumptions for our model. 

\begin{definition}[Random block Schr{\"o}dinger operators] \label{def: BM}
Fix $d\in\{1,2\}$. Define an $N\times N$ complex Hermitian random block matrix $V$, whose entries are independent Gaussian random variables up to the Hermitian condition $V_{xy}=\overline V_{yx}$. Specifically, the off-diagonal entries of $V$ are complex Gaussian random variables:
\be\label{bandcw0}
V_{xy}\sim {\cal N }_{\C}(0, s_{xy})\quad \text{with}\quad  s_{xy}:=W^{-d} {\bf 1}\left( [x] =[y] \right), \quad \text{for} \ \  x\ne y,
\ee
while the diagonal entries of $V$ are real Gaussian random variables distributed as ${\cal N }_{\R}(0, W^{-d})$. In other words, $V$ is a diagonal block matrix with i.i.d.~GUE blocks. We will refer to $V$ as a ``block potential". Then, we define a general class of random block Schr{\"o}dinger operators of the form  
$\lambda \Psi + V,$ where $\lambda>0$ is the coupling parameter, and $\Psi$ is the interaction Hamiltonian that introduces hoppings between different blocks. For definiteness, we consider the following two models in this paper.  
\medskip

\noindent{(i)} {\bf Block Anderson ($\BA$) model}. Define 
\be\label{eq:H_blocka}
H\equiv H^{\BA} =\lambda \Psi^\BA +V,\ee 
where $\Psi\equiv \Psi^{\BA}$ takes the following form. Let $\Psi|_{[x][y]}$ denote the $([x],[y])$-th block of $\Psi$. If $|[x]-[y]|>1$, we let $\Psi|_{[x][y]}=0$. Otherwise, we set the diagonal blocks as $\Psi|_{[x][x]}=A_0$, and define the off-diagonal blocks as follows: in 1D, 
\be\label{eq:Psi1D}\Psi|_{[x][y]}=\begin{cases}
    A_1 , \ &\text{if}\  [y] = [x]+1\\
    A_1^* , \ &\text{if}\  [y] = [x]-1
\end{cases},\ee 
and in 2D,  
\be\label{eq:Psi2D}\Psi|_{[x][y]}=\begin{cases}
    A_1 , \ &\text{if}\  [y] = [x]+(1,0)\\
    A_1^* , \ &\text{if}\  [y] = [x]-(1,0)
\end{cases},\quad  \Psi|_{[x][y]}=\begin{cases}
    A_2 , \ &\text{if}\  [y] = [x]+(0,1)\\
    A_2^* , \ &\text{if}\  [y] = [x]-(0,1)
\end{cases}.\ee 
Here, $A_0$, $A_1$, and $A_2$ are deterministic $W^d\times W^d$ matrices satisfying $\|A_0\|_{\HS}^2\lesssim W^d$, $\|A_i\|_{\HS}^2\asymp W^{d}$ for $i\in\{1,2\}$, along with the following condition on their operator norms for some constant $\e_A>0$:
\be\label{eq:cond_A12}
\begin{split}
&\heta:=\lambda + \lambda\max_{i=0}^2\|A_i\| \le W^{-\e_A} \, .
\end{split}
\ee
In addition, in 2D, we assume the following symmetry condition for $\Psi$:
\be\label{sym_cond}
A_1,A_2 \text{ are Hermitian},\ \  \quad \text{or}\quad \ \ \Psi \text{ is symmetric}.
\ee
To the best of our knowledge, all random (block) Schr{\"o}dinger operators in the mathematical literature exhibit this parity symmetry shown in \eqref{sym_cond}, whose physical meaning we will explain around \eqref{eq:explainsym}. 

\medskip
\noindent{(ii)} {\bf Wegner orbital ($\WO$) model}. The neighboring blocks of $ \Psi\equiv \Psi^{\WO}$ consist of independent blocks of $W^d\times W^d$ complex Ginibre matrices up to the Hermitian symmetry $\Psi = \Psi^*$. In other words, the entries of $\Psi$ are independent complex Gaussian random variables up to the Hermitian symmetry:
\be\label{bandcw1}
\Psi_{xy}\sim {\cal N }_{\C}(0, s'_{xy}),\quad \text{with}\quad  s'_{xy}:=W^{-d} {\bf 1}\left( [x] \sim [y] \right).
\ee
For the Wegner orbital model, we define 
\be\label{eq:WO}
H\equiv H^{\WO}=(1+2d\lambda^2)^{-1/2}(\lambda\Psi^\WO + V),
\ee
so the total variance of the entries of $H$ within each row and column is equal to 1. 
\end{definition}

The block Anderson model defined above includes the important special cases where $A_0=0$ and $A_1=A_2=I_W$, with $I_W$ being the $W^d\times W^d$ identity matrix, as well as the case where $\Psi=\Delta_L$, with $\Delta_L$ being the discrete Laplacian on $\ZL$. These scenarios cover the block Anderson models considered in \cite{Wegner,RBSO}. Here, we have generalized these special models to permit almost arbitrary deterministic interactions $A_1$ and $A_2$, provided that their Hilbert–Schmidt norms are normalized to order $W^d$, along with the condition \eqref{eq:cond_A12} and the symmetry condition \eqref{sym_cond} in 2D.

The condition \eqref{eq:cond_A12} ensures that $H$ can be viewed as a perturbation of $V$ in terms of eigenvalues, while the symmetry condition \eqref{sym_cond} guarantees parity symmetry of the model at the block level. 
Under the first case in \eqref{sym_cond}, where $A_1$ and $A_2$ are Hermitian matrices, we have \smash{$H|_{[a][b]}=H|_{[b][a]}$} for any \smash{$[a],[b]\in \Zn$}. In this setting, the block Anderson model \eqref{eq:H_blocka} can be interpreted as an Anderson model on the lattice \smash{$\Zn$}, where the particle carries a fiber space of dimension $W^d$, representing internal degrees of freedom such as orbitals or spin. 
The Hermitianity of $A_1$ implies that the transition from the fiber at $[x]$ to that at $[x] + (1,0)$ is equal to the reverse transition, i.e., for any $\{\al\},\{\beta\}\in[0]$ (recall the notation in \eqref{eq:xremainder}), we have: 
\be\label{eq:explainsym}
\Psi|_{[x]([x]+(1,0))}(\{\al\},\{\beta\})=A_1\p{\{\al\},\{\beta\}}=A_1^*\p{\{\al\},\{\beta\}}=\Psi|_{([x]+(1,0))[x]}(\{\al\},\{\beta\}).
\ee 
This yields parity symmetry along the first coordinate direction. Similarly, the Hermitianity of $A_2$ implies parity symmetry along the second coordinate direction.
In contrast, under the second case in \eqref{sym_cond}, the transpose \smash{$H^\top$} has the same distribution as $H$. In this case, the model is interpreted as a spatial model, where each block \smash{$[a]\in \Zn$} corresponds to a real spatial block in $\ZL$. Hence, the transition from site $([x],\{\al\})$ to $([x]+(1,0),\{\beta\})$ is required to be identical to the reverse transition, which leads to another symmetry condition: 
\be\label{eq:explainsym2}
\Psi|_{[x]([x]+(1,0))}(\{\al\},\{\beta\})=A_1\p{\{\al\},\{\beta\}}=A_1^*\p{\{\beta\},\{\al\}}=\Psi|_{([x]+(1,0))[x]}(\{\beta\},\{\al\}).
\ee
This implies that $A_1$ must be a real matrix. Applying the same reasoning to $A_0$ and $A_2$, we see that $\Psi$ must be real symmetric. We refer the reader to \Cref{Fig:BAP} for an illustration of these two settings: the left panel depicts a model where each $[a]$ is associated with a fiber, while the right panel shows a spatial model where each $[a]$ represents a physical block in $\ZL$.




\begin{figure}[h]
\begin{center}
 \scalebox{0.8}{
\begin{tikzpicture}
  \draw (-4, -2) -- (2, -2) -- (4.5, 2) -- (-1.5, 2) -- cycle;
\draw (-2.25, 0) -- (2.75, 0);
\draw (-1, -1.6) -- (1, 1.6);
  \draw[red] (0, 0) -- (0, 3);
  \draw[red] (2.25, 0) -- (2.25, 3);
  \draw[red] (0.8, 1.28) -- (0.8, 4.28);
  \draw[red] (-1.75, 0) -- (-1.75, 3);
  \draw[red] (-0.85, -1.36) -- (-0.85, 1.64);
  \fill (0, 0) circle (1pt) node[below right]{$\br{0}$};
  \fill (2.25, 0) circle (1pt) node[below]{$\p{1, 0}$};
  \fill (0.8, 1.28) circle (1pt) node[right, yshift=-2pt]{$\p{0, 1}$};
  \fill (-1.75, 0) circle (1pt) node[below]{$\p{-1, 0}$};
  \fill (-0.85, -1.36) circle (1pt) node[right, yshift=-2pt]{$\p{0, -1}$};
  \draw[blue, ->, shorten <=2pt, shorten >=2pt] (0, 2.5) to[out=30, in=150]
    node[pos=0.5, above]{$A_1$} (2.25, 2.5);
  \draw[blue, ->, shorten <=2pt, shorten >=2pt] (0, 2.5) to[out=150, in=30]
    node[pos=0.5, above]{$A_1$} (-1.75, 2.5);
  \draw[blue, ->, shorten <=4pt, shorten >=2pt] (0, 2.5) to[out=70, in=220]
    node[pos=0.5, above left]{$A_2$} (0.8, 3.78);
  \draw[blue, ->, shorten <=2pt, shorten >=6pt] (0, 2.5) to[out=210, in=70]
    node[pos=0.5, above left]{$A_2$} (-0.85, 1.14);


  
    
  \draw ++(8, 0.4) +(-2.4, -2.4) -- +(2.4, -2.4) -- +(2.4, 2.4) -- +(-2.4, 2.4) -- cycle;
  \draw ++(8, 0.4) +(-2.4, 0.8) -- +(2.4, 0.8);
  \draw ++(8, 0.4) +(-2.4, -0.8) -- +(2.4, -0.8);
  \draw ++(8, 0.4) +(-0.8, -2.4) -- +(-0.8, 2.4);
  \draw ++(8, 0.4) +(0.8, -2.4) -- +(0.8, 2.4);
  
  \node at (8, 0.4) {$\br{0}$};
  \node at (8, 2.0) {$\p{0, 1}$};
  \node at (8, -1.2) {$\p{0, -1}$};
  \node at (6.4, 0.4) {$\p{-1, 0}$};
  \node at (9.6, 0.4) {$\p{1, 0}$};
  
  \draw[->, blue, shorten <=8pt, shorten >=16pt] (8, 0.4) -- (8, 2.0) node[pos=0.7, right]{$A_2$};
  \draw[->, blue, shorten <=8pt, shorten >=16pt] (8, 0.4) -- (8, -1.2) node[pos=0.7, right]{$A_2$};
  \draw[->, blue, shorten <=8pt, shorten >=16pt] (8, 0.4) -- (6.4, 0.4) node[pos=0.3, above]{$A_1$};
  \draw[->, blue, shorten <=8pt, shorten >=16pt] (8, 0.4) -- (9.6, 0.4) node[pos=0.3, above]{$A_1$};
\end{tikzpicture}
}
\caption{Illustration of the block Anderson model and the two parity symmetry conditions in \eqref{sym_cond}.}\label{Fig:BAP} 
\end{center}
\end{figure}

We note that the parity symmetry condition is only required to refine the technical bound from \eqref{sum_res_2} to the improved version \eqref{sum_res_2_sym} stated below. However, in 1D, these two bounds are equivalent, and thus the symmetry condition \eqref{sym_cond} is not necessary in that case. Aside from the symmetry condition \eqref{sym_cond}, our assumptions on $\Psi$ are nearly as general as possible. To illustrate this, suppose the assumption \eqref{eq:cond_A12} fails. Consider an extreme case in 1D, where $\lambda A_1$ contains only a small number of very large non-zero entries---significantly larger than $\|V\|$, which is of order 1 with high probability---a typical bulk eigenvector of $\Psi$ may become highly concentrated on those few large entries. In such a scenario, the ``perturbation" introduced by $V$ may not induce sufficient delocalization within each block, and as a result, our main conclusion \eqref{eq:delocalmax} can fail to hold. 


\begin{remark} \label{rmk:extension}
Due to the specific focus on random Schr{\"o}dinger operators with nearest-neighbor interactions in the literature, we have restricted our attention in \Cref{def: BM} to RBSOs involving only nearest-neighbor block interactions. However, our results can be readily extended to more general settings that include non-nearest-neighbor (either deterministic or random) block interactions, provided two key structural properties of $\Psi$ are preserved: translation invariance and parity symmetry at the block level.
A direct extension is to allow each block $[a]$ to interact with all blocks within a fixed distance $\le K\in \N$. In this case, all proofs in the current paper remain valid, provided that the blocks of $\Psi$ satisfy the following conditions: HS norm bounded by $\OO(W^{d/2})$, operator norms satisfying \eqref{eq:cond_A12}, and in each coordinate direction, at least one off-diagonal block has HS norm of order $\Omega(W^{d/2})$.
Furthermore, our framework can accommodate long-range interactions, where the HS and operator norms of the blocks $\Psi|_{[a][b]}$ decay faster than \smash{$|[a]-[b]|^{-C}$} for some large constant $C>d$. In such cases, our arguments still apply with minimal modifications, except possibly in the proof of certain deterministic estimates, such as those in \Cref{lem_propTH}.
Additionally, one can consider more intricate models involving non-flat variance profiles for both the block potential $V$ and the interaction $\Psi$ (in the Wegner orbital model). These generalizations can also be handled using the tools developed in this paper; see our later discussion in \Cref{sec:newidea}. 
For clarity and to keep assumptions concise, we do not pursue these extensions in the present work and instead leave them for future research.\end{remark}

\subsection{Overview of the main results}
Before discussing our main results, we first state the \emph{localization} of the block Anderson and Wegner orbital models established in the literature. We define the Green's function (or resolvent) of the Hamiltonian $H$ (for $H\in\{H^\BA,H^\WO\}$) as
\be\label{def_Green}
G(z):=(H-z)^{-1} = \left(\lambda \Psi + V-z\right)^{-1},  \quad z\in \C.
\ee
Since the entries of $H$ have continuous density, $G(z)$ is well-defined almost surely even when $z\in \R$. 
Utilizing the fractional moment method \cite{Aizenman1993}, it was proved in \cite{Wegner} that the fractional moments of the $G$ entries are exponentially localized for small enough $\lambda\ll W^{-d/2}$. More precisely, we consider the models $H$ in \Cref{def: BM}; for the block Anderson model, assume in addition that the interacting matrices are given by $A_1=A_2=I_W$. Fix any $s\in (0,1)$, there exists a constant $c_{s}>0$ such that if $\lambda \le c_{s}W^{-d/2},$ then we have a constant $C_{s}>0$ such that for all $L\ge W \ge 1$ and $z\in \R$,
\be\label{fraction-decay}
\mathbb E |G_{xy}(z)|^s\le  C_{s} W^{ds/2} \left( C_{s} \lambda W^{d/2}\right)^{s|[x]-[y]|}.
\ee
Hence, if we take $\lambda\le c W^{-d/2}$ for a constant $0<c<c_{s}\wedge C_{s}^{-1}$, then the estimate \eqref{fraction-decay} gives an exponential decay of $\mathbb E |G_{xy}(z)|^s$, from which we readily derive the localization of eigenvectors by \cite[Theorem A.1]{ASFH2001}.

In this paper, we give a counterpart of the above localization result. We define the \emph{localization length} of the eigenvectors $\{\bu_k\}$ of $H$ as
\be\label{eq:localization_length}
\ell(\bu_k):= \inf\Big\{0\le \ell\le L: \min_{x_0\in \ZL} \sum_{|x-x_0|\le \ell}|\bu(x)|^2 \ge 1/2 \Big\}.
\ee
We note that there are several reasonable ways to define localization length. For instance, the $1/2$ in \eqref{eq:localization_length} can be replaced with any fixed value $a \in (0,1)$. Another common approach, used in \cite{delocal,BandI}, defines $\ell(\bu_k)$ as the largest scale $\ell$ such that $\min_{x_0}\sum_x |\bu_k(x)|^2\exp\left[(|x-x_0|/\ell)^\gamma\right]=\OO(1)$ for some exponent $\gamma\in (0,1)$. 
For our results, all such definitions yield localization lengths of the same order. Therefore, for definiteness, we adopt the definition \eqref{eq:localization_length} in this paper. Assuming that \smash{$\lambda \gg W^{-d/2}$ and $W \ge L^{\delta}$} for some small constant $\delta > 0$, we establish the following results concerning the localization lengths of bulk eigenvectors:
\begin{itemize}

\item {\bf Localization length} (\Cref{thm:supu}). 
With very high probability, the localization lengths $\ell(\bu_k)$ corresponding to bulk eigenvalues have the following lower bounds: of order $ [(W\lambda)^2\wedge L]^{1-\e}$ in 1D, and of order $L^{1-\e}$ in 2D, for any constant $\e>0$. 

\item {\bf Quantum unique ergodicity} (\Cref{thm:QUE}). 
In 2D, we prove that with probability $1-\oo(1)$, every bulk eigenvector is nearly flat on all scales $\Omega(W)$. This implies that the localization lengths of bulk eigenvectors are indeed of order $L$. The same conclusion holds in 1D if $(W\lambda)^2\gg L$. 
\end{itemize}
The proofs of the above results are based on the following local law and quantum diffusion estimates. 
\begin{itemize}
\item {\bf Local law} (\Cref{thm_locallaw}). Within the bulk of the spectrum, we establish a sharp local law for the Green's function $G(z)$ for $\im z$ down to the scale $N^\e \eta_*$, where
\be\label{eq:defeta*}
\eta_*:=\frac{1}{(W\lambda)^2}\mathbf 1_{d=1}+\frac{1}{N}=\frac{1}{(W\lambda)^2}\mathbf 1_{d=1}+\frac{1}{L^d}.
\ee

\item {\bf Quantum diffusion} (\Cref{thm_diffu}). The evolution of the quantum particle follows a quantum diffusion on scales $\gg W$ and for times $t\gg 1$.

\end{itemize}
We refer readers to \Cref{sec:main} below for more precise statements of these results. Combined with the localization result \eqref{fraction-decay}, our findings establish the {\bf Anderson metal-insulator transition} for finite-volume RBSOs in 2D as $\lambda$ crosses the critical threshold $W^{-d/2}$. In 1D, when $(W\lambda)^2\ll L$, the established lower bound $(W\lambda)^2$ on the localization length is conjectured to be the correct scaling, as previously discussed. 

\subsection{New ideas}\label{sec:newidea}
Our proof of the main results extends the method recently developed in \cite{Band1D} for 1D RBM, which is based on a careful analysis of a system of infinitely many ``self-consistent equations" for the $G$-loops (defined in \Cref{Def:G_loop} below), referred to as the \emph{loop hierarchy} (given by \Cref{lem:SE_basic} below). The core of the proof involves identifying the deterministic limits, called \emph{primitive loops} (defined in \Cref{Def_Ktza} below), for the $G$-loops, which enables us to properly truncate the loop hierarchy and establish sharp estimates for the difference between the $G$-loops and the primitive loops. To extend the method to our setting, particularly for the block Anderson model, we need to develop new tools, which we briefly summarize here.

The core of our proof strategy for the block Anderson model is a new---and arguably more fundamental---tree representation formula for primitive loops. A key feature that distinguishes this new formula from the one developed in \cite{Band1D} for the 1D RBM is that it does not rely on any specific structure in the variance profile. Instead, we derive a pointwise representation of primitive loops, with indices ranging over \(\ZL\), in contrast to the earlier approach in \cite{Band1D}, where the formula depends crucially on the underlying block structure and uses indices in \smash{$\Zn$}. 
Our new formula encompasses both RBM and the RBSOs defined in \Cref{def: BM} as special cases and can be readily extended to more general models beyond our setting. For instance, in a forthcoming work \cite{Band_nonblock}, we extend the results of \cite{Band1D,Band2D} and this paper to RBM with translationally invariant variance profiles on $\ZL$ (which has no block structure), and RBSOs with (almost arbitrary) non-flat variance profiles within each block. 
The foundation of that extension is precisely the primitive loop formula introduced in the present paper. 
In addition, this pointwise representation of the primitive loops offers substantial flexibility. In principle, it could even be applied to the original Anderson model. However, it remains unclear how to effectively exploit the formula in that context, due to the lack of other essential tools.

In the setting of the RBM or the Wegner orbital model, the deterministic limit of the Green's function in \eqref{def_Green} is a scalar matrix $M(z)=m_{sc}(z)I_N$, where $m_{sc}$ denotes the Stieltjes transform of the Wigner semicircle law (see \eqref{self_mWO} below). In contrast, for the block Anderson model, the deterministic limit $M(z)$ is a non-diagonal matrix expressed in terms of the interaction matrix $\lambda\Psi$; see \eqref{def_G0}.
In \cite{Band1D}, the coefficients in the primitive loop representation formulas involve powers of $m_{sc}$ and $\overline m_{sc}$. When extending these formulas to the block Anderson model, these scalar quantities must be replaced by the matrices $M$ and $M^*$. This raises a fundamental challenge: how to incorporate these matrix-valued entries into the tree structure of the representation formula developed in \cite{Band1D}.
Fortunately, we identify a canonical solution by organizing the $M$-entries into properly structured loops, which we refer to as $M$-loops. Although the resulting representation may appear more intricate, it actually clarifies the structure of the primitive loops---even in the RBM and Wegner orbital settings. 
This insight leads to a new proof of Ward's identity for primitive loops, which is both more direct and simpler than the original argument in \cite[Section 3]{Band1D}. While the proof in \cite{Band1D} relied on analyzing the dynamics of primitive loops, our proof uses only the tree representation formula and a few deterministic matrix identities (see \eqref{ward-M}, \eqref{ward-Theta-1}, and \eqref{ward-Theta-2}), which can themselves be seen as certain versions of ``Ward's identity". 
This new approach supports the view that our representation formula reveals a more fundamental underlying structure of the primitive loops.

Finally, there is an additional challenge in establishing the upper bound on the primitive loops, as shown in \Cref{ML:Kbound}. This difficulty arises from the presence of the small coupling parameter $\lambda$, which does not appear in the RBM setting considered in \cite{Band1D,Band2D}.
Since $\lambda\ll 1$, the diffusion of the Green's function in our RBSO model—described by a deterministic propagator (see \Cref{def_Theta})---is significantly slower than in the RBM case. As a result, our propagator satisfies a weaker bound than that in \cite{Band1D,Band2D}, incurring an additional factor of $\lambda^{-2}$.
If one were to apply the methods from \cite{Band1D,Band2D} naively, the resulting estimates for the primitive loops would involve undesirable powers of $\lambda^{-2}$, rendering the bounds too weak for our purposes. In our proof, we show how to recover compensating factors of $\lambda^{2}$ in the application of certain sum-zero properties, which allow us to cancel these problematic $\lambda^{-2}$ terms.
We remark that this mechanism is closely related to the coupling renormalization phenomenon that has recently emerged in the study of higher-dimensional RBSOs \cite{RBSO}. 
The new tree representation formula for primitive loops, along with the proofs of several key properties---including Ward's identity, the sum-zero property, and sharp upper bounds---will be presented in detail in \Cref{sec:prim}.

Another challenge in extending the proof for RBM to the block Anderson model lies in establishing resolvent entry estimates---also known as local laws---for the entries of the Green's function $G$, based on $G$-loop bounds. In \cite{Band1D}, such estimates are derived using standard resolvent identities and large deviation bounds. However, this approach does not apply to the block Anderson model when the interaction matrix $\lambda\Psi$ is non-diagonal.
A natural alternative is the cumulant expansion method, as used in studies of deformed Wigner-type matrices (see, e.g., \cite{He2018,AEK_PTRF,EKS_Forum}). In the existing literature, this method typically employs Ward's identities to control resolvent expressions. However, in our non-mean-field setting where $W\ll L$, these tools are not sufficiently effective.
Instead, we establish (almost) optimal bounds on the \emph{$T$-variables} (defined in \eqref{eq:Tvariable}), which have played a central role in previous works on RBMs (e.g., \cite{delocal,PartI,PartII,Band1D_III,BandI,BandII,BandIII}). In our setting, these bounds serve as a substitute for Ward's identities.
Our goal is to obtain entrywise local laws for the resolvent defined in \eqref{def_Green}. However, under the general assumptions in \Cref{def: BM}, where the matrices $A_i$ (for $i=0,1,2$) are essentially arbitrary, we find it necessary to derive more general \emph{isotropic local laws} to achieve optimal entrywise control. This requires introducing a new class of \emph{generalized $T$-variables} (see \eqref{eq:gen_T}), which extend the original $T$-variables to the isotropic setting.
Our proof of the resolvent entry estimates then consists of two main steps: we bound the generalized $T$-variables using the $G$-loop bounds, and then we control the isotropic resolvent entries via the generalized $T$-variables.

Beyond controlling the resolvent in the $\max$-norm, it is also essential to establish good \emph{decay estimates} for the $G$-entries. In \Cref{appd:MDE}, we implement an iterative procedure based on Gaussian integration by parts to transfer decay from the $G$-loops to the generalized $T$-variables, ultimately yielding the desired entrywise decay of $G$. 
The introduction of generalized $T$-variables, the isotropic local laws, and our new procedure for deriving decay estimates represent new contributions to the study of RBMs and RBSOs. These ideas will be presented in detail in \Cref{appd:MDE}, and we expect them to be useful in future work on more general non-mean-field random matrices and operators, some of which have been discussed in \Cref{rmk:extension}.


The extension of the proof in \cite{Band1D} to 2D RBSOs presents additional challenges. A critical issue involves estimates related to the ``evolution kernels" (defined in \Cref{DefTHUST} below). Both this paper and \cite{Band1D,Band2D} employ a flow argument, considering a flow for $H$ and the spectral parameter $z$. 
Along this flow, we utilize It\^o's formula and Duhamel’s principle to derive an integrated loop hierarchy (see equation \eqref{int_K-LcalE} below), which allows us to transfer the $G$-loop estimates from smaller times to larger times via the evolution kernel. 
However, the $(L^\infty\to L^\infty)$-norm bound (specifically, equation \eqref{sum_res_2}) for the evolution kernel is insufficient in 2D. To address this, we need to explore a CLT-type cancellation mechanism, which is simultaneously demonstrated in \cite{Band2D}. 
Due to the similarity to the argument in \cite{Band2D}, and because this is not the focus of this paper, we choose to postpone the relevant proof to \Cref{sec:add_d=2}. 

With all the new tools mentioned above, our analysis of the loop hierarchy roughly follows the strategy outlined in \cite[Section 5]{Band1D}, albeit with many modifications to the relevant arguments. For reader convenience, we will briefly outline the proof in  \Cref{Sec:Stoflo}, while postponing (almost full) details to the appendix.  

\medskip
\noindent{\bf Organization of the remaining text.} 
In \Cref{sec:main}, we present the main results of this paper. \Cref{sec:tools} introduces several key tools that will be utilized in the proofs, including the flow framework, the $G$-loops, and the primitive loops. 
\Cref{sec:prim} is dedicated to establishing the tree representation for the primitive loops and proving some fundamental properties for them. 
The proof of the main results is presented in \Cref{sec:proof}, and is based on a key theorem, \Cref{lem:main_ind}, which extends the $G$-loop estimates along the flow to progressively larger times $t$. 
The proof of \Cref{lem:main_ind} is the primary focus of \Cref{Sec:Stoflo}, which conducts an analysis of the loop hierarchy for the $G$-loops. 
This analysis crucially relies on a technical lemma---\Cref{lem_GbEXP}---which is established separately in \Cref{appd:MDE}. 
Due to the similarity between our proof of \Cref{lem:main_ind} and the argument in \cite[Section 5]{Band1D}, most technical details in \Cref{Sec:Stoflo} are deferred to \Cref{sec:main_appd}. Lastly, \Cref{appd:deter} contains the proofs of several deterministic estimates in Lemmas \ref{lem:propM} and \ref{lem_propTH}.


To facilitate the presentation, we introduce some necessary notations that will be used in the proof. We will use the set of natural numbers $\N=\{1,2,3,\ldots\}$ and the upper half complex plane $\C_+:=\{z\in \C:\im z>0\}$.  
In this paper, we are interested in the asymptotic regime with $N\to \infty$. When we refer to a constant, it will not depend on $N$ or $\lambda$. Unless otherwise noted, we will use $C$, $D$ etc.~to denote large positive constants, whose values may change from line to line. Similarly, we will use $\e$, $\delta$, $\tau$, $c$, $\fc$, $\fd$ etc.~to denote small positive constants. 
For any two (possibly complex) sequences $a_N$ and $b_N$ depending on $N$, $a_N = \OO(b_N)$, $b_N=\Omega(a_N)$, or $a_N \lesssim b_N$ means that $|a_N| \le C|b_N|$ for some constant $C>0$, whereas $a_N=\oo(b_N)$ or $|a_N|\ll |b_N|$ means that $|a_N| /|b_N| \to 0$ as $N\to \infty$. 
We say that $a_N \asymp b_N$ if $a_N = \OO(b_N)$ and $b_N = \OO(a_N)$. For any $a,b\in\R$, we denote $\llbracket a, b\rrbracket: = [a,b]\cap \Z$, $\qqq{a}:=\qqq{1,a}$, $a\vee b:=\max\{a, b\}$, and $a\wedge b:=\min\{a, b\}$. For an event $\Xi$, we let $\mathbf 1_\Xi$ or $\mathbf 1(\Xi)$ denote its indicator function.  
Given a vector $\mathbf v$, $|\mathbf v|\equiv \|\mathbf v\|_2\equiv \|\mathbf v\|$ denotes the Euclidean norm and $\|\mathbf v\|_p$ denotes the $L^p$-norm. 
Given a matrix $\cal A = (\cal A_{ij})$, $\|\cal A\|$, $\|\cal A\|_{p\to p}$, and $\|\cal A\|_{\infty}\equiv \|\cal A\|_{\max}:=\max_{i,j}|\cal A_{ij}|$ denote the operator (i.e., $L^2\to L^2$) norm,  $L^p\to L^p$ norm (where we allow $p=\infty$), and maximum (i.e., $L^\infty$) norm, respectively. We will use $\cal A_{ij}$ and $ \cal A(i,j)$ interchangeably in this paper. Moreover, we introduce the following simplified notation for trace: $ \left\langle \cal A\right\rangle=\tr (\cal A) .$

\medskip
\noindent{\bf Acknowledgement.}   
Fan Yang is supported in part by the National Key R\&D Program of China (No. 2023YFA1010400). 
We would like to thank Horng-Tzer Yau for fruitful discussions.

\section{Main results} \label{sec:main}

For the RBSOs in \Cref{def: BM}, we denote the eigenvalues of $H$ by $\lambda_1\le \lambda_2\le \cdots\le \lambda_N$ and the corresponding eigenvectors by $\bu_1,\bu_2,\ldots, \bu_N$. Our first main result concerns the localization lengths (as defined in \eqref{eq:localization_length}) of the bulk eigenvectors.  

\begin{theorem}[Localization length]\label{thm:supu}
Consider the RBSOs in \Cref{def: BM} with $d\in\{1,2\}$. Let $\kappa,\erre,\delta\in (0,1)$ be arbitrarily small constants. Assume that $W\ge L^\delta$, and that $\lambda$ satisfies  
\be\label{eq:cond-lambda2}
\lambda \ge W^{-d/2+\fd}.
\ee
Then, the following estimate holds for any constants $\tau, D>0$ when $L$ is sufficiently large:
\begin{align}
	\P\bigg(\sup_{k: |\lambda_k | \leq 2 - \kappa} \|\bu_k\|_\infty^2 \leq \frac{W^\tau}{(W\lambda)^2}\mathbf 1_{d=1}+\frac{W^\tau}{N} \bigg) &\ge 1- L^{-D}. \label{eq:delocalmax}  
\end{align}
As a consequence, the following lower bounds hold with probability $\ge 1-L^{-D}$ for any constants $\tau,D>0$ and large enough $L$: 
\begin{align}
\inf_{k: |\lambda_k | \leq 2 - \kappa} \ell(\bu_k) \ge \br{(W\lambda)^2\wedge L}^{1-\tau} \ \ \text{in} \ \ d=1,\quad \text{and}\quad   \inf_{k: |\lambda_k | \leq 2 - \kappa} \ell(\bu_k) \ge  L^{1-\tau} \ \ \text{in} \ \ d=2.\label{eq:length}  
\end{align}
\end{theorem}

In the delocalized regime, where the localization length is of order $\Omega(L^{1-\oo(1)})$, we can further establish a stronger \emph{quantum unique ergodicity} (QUE) estimate for the bulk eigenvectors of $H$, albeit at the cost of a slightly weaker probability bound. 

\begin{theorem}[Quantum unique ergodicity]\label{thm:QUE}
In the setting of \Cref{thm:supu}, assume in addition that $W\lambda\ge W^\fc N^{1/2}$ for a constant $\fc>0$ when $d=1$. Given a constant $\e\in (0,\fd\wedge \fc)$ (where $\fd$ is the constant in \eqref{eq:cond-lambda2}), define that ${\cal I}_E\equiv {\cal I}_E(\e):=\left\{x: |x-E|\le W^{-\e}\eta_{0}\right\}$, where 
$$ \eta_{0}:=\begin{cases}
{W\lambda}/{N}, & \text{if} \ d=2\\
W\lambda/N^{3/2}, & \text{if} \ d=1\\
\end{cases}.$$ 
Then, for each $d\in\{1,2\}$, there exists a small constant $c$ depending on $\e$, $\fd$, $\fc$ such that the following estimate holds for large enough $L$:
\begin{equation}\label{Meq:QUE}
\sup_{E: |E|\le 2-\kappa}  
\max_{ [a]\in \Zn} \P\bigg(\max_{i,j:\lambda_i, \lambda_j \in {\cal I}_E} 
 \bigg| \sum_{x\in[a]}\overline \bu_i(x)\bu_j(x)-\frac{W^{d}}{N}\delta_{ij}  \bigg|^2 \ge \frac{W^{d-c}}{N} \bigg) \le  W^{-c} \, .
\end{equation}
More generally, for any subset $A\subset \Zn$, we have
\begin{equation}\label{Meq:QUE2}
\sup_{E: |E|<2-\kappa} \mathbb{P}\bigg(\max_{k: \lambda_k\in {\cal I}_E} 
\bigg|\sum_{[a]\in A}\sum_{x\in [a]}\left|\bu_k(x)\right |^2 -\frac{W^d}{N}|A|\bigg| \ge  \frac{W^{d-c}|A|}{N}  \bigg) \le W^{-c}.
\end{equation}  
\end{theorem}

The above QUE estimates \eqref{Meq:QUE} and \eqref{Meq:QUE2} indicate that all bulk eigenvectors in dimensions $1$ (when $(W\lambda )^2\gg L$) and $2$ are asymptotically uniformly distributed (in the sense of $L^2$-mass) across all scales larger than $W$. In particular, this implies that in the delocalized regime, the localization length of every bulk eigenvector is of order $\ell(\bu_k)=\Omega(L)$ with probability $1-\oo(1)$.

To establish the above results, we will derive a sharp local law for the Green's function of $H$, defined as in \eqref{def_Green}. In previous works (see e.g., \cite{LeeSchSteYau2015,Anisotropic,He2018,AEK_PTRF,EKS_Forum} 
for various settings of deformed Wigner-type matrices), it has been shown that if $W\to \infty$, $G(z)$ converges to a deterministic matrix limit $M(z)$ in the sense of local laws if $\im z\gg W^{-d}$ (although the convergence estimate is non-optimal in the literature).

\begin{definition}
\label{defn_Mm}
For the block Anderson model \eqref{eq:H_blocka}, define $m(z)\equiv m_N(z)$ as the unique solution to 
 \be\label{self_m}
\frac{1}{N}\tr   \frac{1}{\lambda \Psi -z- m(z)} = m(z)
\ee
 such that $\im m(z)>0$ for $z\in \C_+$. Then, we define the matrix $M(z)\equiv M_N(z)$ as 
\be\label{def_G0}
M(z):= \frac{1}{\lambda \Psi -z- m(z)}.
 \ee
For the Wegner orbital model \eqref{eq:WO}, we define $m(z)$ and $M(z)$ as
\be\label{self_mWO}
m(z)=m_{sc}(z),\quad M(z):=m_{sc}(z)I_N,\quad \text{with}\quad m_{sc}(z)=\frac{1}{2}\br{-z+\sqrt{z^2-4}},
\ee
where $m_{sc}$ denotes the Stieltjes transform of the Wigner semicircle law.
 \end{definition}

By the block translation symmetry of $\Psi$ for the block Anderson model, $M(z)$ in \eqref{def_G0} satisfies that 
\be\label{eq:averm}
W^{-d}\sum_{x\in[a]}M_{xx}=m, \quad \forall [a]\in \Zn.\ee 
It is known that $m(z)$ is the Stieltjes transform of a probability measure $\mu_{N}$, called the \emph{free convolution of the semicircle law and the empirical measure of $\lambda\Psi$}.  
Under \eqref{eq:cond_A12}, the support of $\mu_{N}$ consists of a single component $[a_\lambda, b_\lambda]$, with $a_\lambda$ and $b_\lambda$ denoting the left and right spectral edges, respectively. 
Since the semicircle law has edges $\pm 2$ and $\|\lambda\Psi\|$ is of order $\oo(1)$ by \eqref{eq:cond_A12}, we have that $|a_\lambda + 2| \vee |b_\lambda-2|=\oo(1)$. As a consequence, we can still take our bulk eigenvalue spectrum as $[-2+\kappa,2-\kappa]$ for a small constant $\kappa>0$, as we have done in the statements of \Cref{thm:supu,thm:QUE}. 
From \eqref{self_m}, we readily see that $m(z)=m_{sc}(z)+\oo(1)$. In fact, we can derive that 
\be\label{eq:expandm} 
m(z,\lambda)=m_{sc}(z) + \left( \frac{m_{sc}(z)^3}{1-m_{sc}(z)^2}\cdot \frac{1}{N}\tr (\Psi^2) \right)\lambda^2 +\oo(\lambda^2).
\ee 
Our next main result gives an (almost) sharp local law for the Green's function $G(z)$.

\begin{theorem}[Local law]\label{thm_locallaw}
In the setting of \Cref{thm:supu}, for any constants $\e, \tau,D>0$, the following events hold with probability $\ge 1-L^{-D}$ for large enough $L$ and  $z=E+\ii\eta$: 
\begin{align}\label{locallaw}
&\bigcap_{W^{\e}\eta_*\le \eta\le 1}\bigcap_{|E|\le 2- \kappa} \left\{ \|G(z) - M(z) \|_{\max}^2  \le \frac{W^\tau}{W^d[\ell(\eta)]^d \eta}\right\} \, , 
\\
\label{locallaw_aver}
&\bigcap_{W^{\e}\eta_*\le \eta\le 1}\bigcap_{|E|\le 2- \kappa}\bigg\{\max_{[a]} \Big|W^{-d}\sum_{x\in [a]}G_{xx}(z) - m(z) \Big| \le \frac{W^\tau}{W^d[\ell(\eta)]^d \eta}\bigg\}\, , 
\end{align}
where $G$, $M$, and $\eta_*$ are defined in \eqref{def_Green}, \eqref{def_G0}, and \eqref{eq:defeta*}, respectively, and $\ell(\eta)$ is defined by 
\be\label{eq:elleta}
\ell(\eta):= \min\big(\lambda \eta^{-1/2}+1, n\big) \, . 
\ee
\end{theorem}

Note \eqref{eq:delocalmax} is an immediate corollary of the local law \eqref{locallaw}. 
\begin{proof}[\bf Proof of \cref{thm:supu}]
For any $\eta>0$, we have the bound $|\bu_k(x)|^2 \le \eta\im G_{xx}(\lambda_k + \ii \eta)$. Then, taking $\eta=W^\tau\eta_*$ and using the local law \eqref{locallaw}, we conclude \eqref{eq:delocalmax}.
Combining \eqref{eq:delocalmax} with Markov's inequality (with respect to the counting measure on $\ZL$), we conclude \eqref{eq:length}. 
\end{proof}


Similar to RBM \cite{BandI,Band1D}, our RBSOs also satisfy the \emph{quantum diffusion conjecture}, which will be used to establish the QUE estimates in \Cref{thm:QUE}. To state it, we first define the relevant matrices $\Theta$ and $S^\pm$. 

\begin{definition}\label{def:projlift}
Define the variance matrix $S=(S_{xy})$ as
\(S_{xy}=\var(H_{xy})\) for $x,y\in \ZL.$ 
For the block Anderson model, we can express $S$ as 
\be\label{eq:SBA}
S\equiv S^{\BA}=I_n \otimes \bE,
\ee
where $\bE$ is the $W^d\times W^d$ matrix with $\bE_{ij}\equiv W^{-d}$; for the Wegner orbital model, we have 
\be\label{eq:SWO}
S\equiv S^{\WO}= \left(1+2d\lambda^2\right)^{-1}\left(I_n  + \lambda^2 \Lambda_n\right)\otimes \bE,
\ee
where $I_n$ and $\Lambda_n$ are respectively the identity and adjacency matrix matrices defined on \smash{$\Zn$}. 
Then, we define the following matrices adopting the notations in \cite{RBSO}:
\be \label{def:Theta}
\Theta: = \frac{1}{1-M^0 S}M^0,\quad S^+=(S^-)^*:= \frac{1}{1-M^+ S}M^+, 
\ee
where the matrices ${M}^0$ and $M^+$ are defined through ${M}^0_{xy}:= |M_{xy}|^2$ and $M^+_{xy}:=M_{xy}M_{yx}$ for $x,y\in \ZL$. 
\end{definition}

\begin{theorem}[Quantum diffusion]\label{thm_diffu}
In the setting of \Cref{thm:supu}, for any constants $\e,\tau, D>0$, the following two events hold with probability $\ge 1-L^{-D}$ for large enough $L$ and  $z=E+\ii\eta$:
\begin{align}
&\bigcap_{W^{\e}\eta_*\le \eta\le 1}\bigcap_{|E|\le 2- \kappa}\bigg\{\max_{[a],[b]} \bigg|\frac{1}{W^{2d}}\sum_{x\in[a],y\in[b]}\left(|G_{xy}|^2 -\Theta_{xy}\right)\bigg| \le \p{\frac{W^\tau}{W^d[\ell(\eta)]^d \eta}}^2\bigg\}\, ,
\label{eq:diffu1}\\
&\bigcap_{W^{\e}\eta_*\le \eta\le 1}\bigcap_{|E|\le 2- \kappa}\bigg\{\max_{[a],[b]} \bigg|\frac{1}{W^{2d}}\sum_{x\in[a],y\in[b]}\left(G_{xy}G_{yx} -S^+_{xy}\right)\bigg| \le \p{\frac{W^\tau}{W^d[\ell(\eta)]^d \eta}}^2\bigg\} \, .
\label{eq:diffu2}
\end{align}
Moreover, stronger bounds hold in the sense of expectation for each $z$ with $|E|\le 2- \kappa$ and $W^{\e}\eta_*\le \eta\le 1$: 
\begin{align}
&\max_{[a],[b]}\bigg|\frac{1}{W^{2d}}\sum_{x\in[a],y\in[b]}\E\left(|G_{xy}|^2 -\Theta_{xy}\right) \bigg| \le \p{\frac{W^\tau}{W^d[\ell(\eta)]^d \eta}}^3,\label{eq:diffuExp1}\\ 
&\max_{[a],[b]}\bigg|\frac{1}{W^{2d}}\sum_{x\in[a],y\in[b]}\E\left(G_{xy}G_{yx} -S^+_{xy}\right)\bigg|\le \p{\frac{W^\tau}{W^d[\ell(\eta)]^d \eta}}^3.\label{eq:diffuExp2}
\end{align}
\end{theorem}

Roughly speaking, the quantum diffusion phenomenon indicates that the long-time behavior of a particle's quantum evolution exhibits diffusive characteristics for $t\gg 1$. In the context of Green's function, this amounts to saying that $|G_{xy}|^2$ can be approximated by the Green's function of a classical random walk on $\ZL$ for $|x-y|\gg W$ and $\eta=t^{-1}\ll 1$. This Green's function of the random walk is represented by our matrix $\Theta$, so the above \Cref{thm_diffu} states that the quantum diffusion conjecture holds for our setting of RBSOs under the condition \eqref{eq:cond-lambda2}, provided we take a local average of $|G_{xy}|^2$ on the scale $W$. A similar quantum diffusion result was established in \cite{RBSO} for RBSOs in high dimensions with $d\ge 7$; however, a proper \emph{self-energy renormalization} must be introduced there to preserve the diffusive behavior.

In the delocalized regime assumed in \Cref{thm:QUE}, utilizing the local laws from \Cref{thm_locallaw}, the delocalization of eigenvectors from \Cref{thm:supu}, and the QUE estimates from \Cref{thm:QUE}, we can show that the eigenvalue statistics of $H$ asymptotically match those of GUE by employing the Green's function comparison argument developed in \cite{BandIII}. This argument is analogous to that used in the proof of \cite[Theorem 1.3]{BandIII}, and a similar approach has also been implemented in \cite{Band1D} for 1D RBM. Therefore, due to length constraints, we do not present it here.

The proof of the main results, \Cref{thm_locallaw} and \Cref{thm_diffu}, will be given in \Cref{sec:proof}, while the proof of \Cref{thm:QUE} will be postponed to \Cref{sec:propagator} after presenting the necessary estimates on $\Theta$ and $S^\pm$. 
For the remainder of this paper, we focus on the block Anderson model. The proof for the Wegner orbital model is simpler, and we will outline the necessary (minor) modifications in \Cref{rmk:ext_Wegner}.


\section{Main tools}\label{sec:tools}

In this section, we introduce several key tools for our proof. It is important to note that the stochastic and deterministic flows defined in \Cref{subsec:flow} are applicable exclusively to the block Anderson model, while the concepts of $G$-loops and primitive loops discussed in \Cref{subsec:Gloop}, along with the definition of the propagators in \Cref{sec:propagator}, are relevant to both models. The flow framework used for the Wegner orbital model is the same as that employed for RBM in \cite{Band1D}, which we will discuss in \Cref{rmk:ext_Wegner} below.

\subsection{Flows}\label{subsec:flow}

Consider the following matrix Brownian motion: 
\begin{align}\label{MBM}
\dd (H_{t})_{xy}=\sqrt{S_{xy}}\dd (B_{t})_{xy}, \quad H_{0}=\lambda\Psi, \quad \forall x,y\in \ZL.
\end{align}
Here, $(B_{t})_{xy}$ are independent complex Brownian motions up to the Hermitian symmetry \smash{$(B_{t})_{xy}=\overline {(B_{t})_{yx}}$},
i.e., $t^{-1/2}B_t$ is an $N\times N$ GUE whose entries have zero mean and unit variance;  
$S=(S_{xy})$ is the variance matrix defined in \Cref{def:projlift}, which is taken as \eqref{eq:SBA} in the following proof since we focus on the block Anderson model. Correspondingly, we define the deterministic flow as follows.  

\begin{definition}[Deterministic flow]\label{def_flow}
Given any $z\in \C_+$, define $m(z)\equiv m(z,\lambda_0)$ as the unique solution to (recall equation \eqref{self_m}) 
 \be\label{self_m0}
\frac{1}{N}\tr   \frac{1}{\lambda_0 \Psi -z- m(z)} = m(z)
\ee
such that $\im m(z)>0$ for $z\in \C_+$. Then, $M(z)\equiv M(z,\lambda_0)$ is defined as in \eqref{def_G0} with $\lambda=\lambda_0$.  
Next, for any parameter $\lambda_0>0$ and $t\in [0,1]$, define $m_t(z) \equiv m_{t}(z,\lambda_0)$ as the unique solution to 
 \be\label{self_mt}
\frac{1}{N}\tr   \frac{1}{\lambda_0 \Psi -z- tm_t(z,\lambda_0)} = m_t(z,\lambda_0)
\ee
such that $\im m_t(z,\lambda_0)>0$ for $z\in \C_+$. Then, we define the matrix flow $M_t(z)\equiv M_t(z,\lambda_0)$ as 
\be\label{def_G0t}
M_t(z,\lambda_0):= \frac{1}{\lambda_0 \Psi -z- tm_t(z,\lambda_0)}.
 \ee
For $E \in \mathbb R$, we denote $m{(E,\lambda_0)}:=m(E+\ii 0_+,\lambda_0)$. We then define the flow $z_t$ by 
\be\label{eq:zt}
  z_t(E,\lambda_0) = E + (1-t) m(E,\lambda_0),\quad t\in [0, 1].  
\ee
Finally, we will denote $z_t(E,\lambda_0)=E_t(E,\lambda_0) + \ii \eta_t(E,\lambda_0)$ with 
\begin{align}\label{eta}
E_t(E,\lambda_0)=E+(1-t)\re m(E,\lambda_0),\quad \eta_t(E,\lambda_0) = (1-t)  \im m(E,\lambda_0).
\end{align}

\end{definition}
Note that under the above flow, we have 
\be\label{eq:mt_stay} m_t(z_t(E,\lambda_0),\lambda_0)\equiv 
m(E,\lambda_0),\quad M_t(z_t(E,\lambda_0),\lambda_0)\equiv M(E,\lambda_0), \quad \forall t\in [0,1].
\ee
In the proof, we will use a slightly different $\lambda_0$ from $\lambda$; see \Cref{zztE} below. Next, we define the stochastic flow for the Green's function with $H_t$ given by \eqref{MBM} and $z_t$ given by \eqref{eq:zt}.

\begin{definition}[Stochastic flow]\label{Def:stoch_flow}
Consider the matrix dynamics $H_t\equiv H_t(\lambda)$  evolving according to \eqref{MBM}. We denote it as $ H_t(\lambda)=V_t+ \lambda\Psi$, where $V_t$ is the matrix Brownian motion with zero initial condition. We denote Green's function of $H_{t}$ as
$$G_t(z,\lambda):=(H_t(\lambda)-z)^{-1},\quad z\in \C_+,$$
and define the resolvent flow as 
\begin{align}
G_{t,E}\equiv G_{t,E}^{(\lambda)}\equiv G_t(z_t(E),\lambda):=(H_{t}(\lambda)-z_t{(E,\lambda)})^{-1}.
\end{align}
By It{\^o}'s formula, $G_{t,E}$ satisfies the following SDE:
\begin{align}\label{eq:SDE_Gt}
\dd G_{t,E}=-G_{t,E}(\dd H_{t})G_{t,E}+G_{t,E}\{\mathcal{S}[G_{t,E}]-m_t(z_t(E,\lambda),\lambda)\}G_{t,E}
\dd t,
\end{align}
where $\mathcal{S}:M_{N}(\C)\to M_{N}(\C)$ is a linear operator defined as
\begin{align}\label{eq:opS}
\mathcal{S}[X]_{ij}:=\delta_{ij}\sum_{k\in \qqq{N}}S_{ik}X_{kk},\quad \text{for}\quad X\in M_{N}(\C).
\end{align}
\end{definition}

For any spectral parameter $z$ in the bulk of the spectrum, we are interested in the resolvent  $G(z)=(H-z)^{-1}$. This can be achieved through the flow by carefully choosing the parameters $E$ and $\lambda$.

\begin{lemma}\label{zztE}
For the block Anderson model, fix any $z\in \mathbb C_+$ with $\im z\in (0, 1]$ and $|\re z|\le  2-\kappa$ for a constant $\kappa>0$. Choosing $t_0$, $E_0$, and $\lambda_0$ as follows: 
\be\label{eq:t0E0}t_0=\frac{\im m(z,\lambda)}{\im m(z,\lambda)+ \im z},\quad E_0= \frac{t_0 \re z -(1-t_0)\re m(z,\lambda)}{\sqrt{t_0}},\quad \lambda_0=\sqrt{t_0}\lambda,\ee
we have that 
 \begin{equation}\label{eq:zztE}
\sqrt{t_0}m(E_0,\lambda_0)=m(z,\lambda), \ \  z_{t_0}(E_0,\lambda_0) = \sqrt{t_0}z, \ \  \sqrt{t_0}M(E_0,\lambda_0)=M(z,\lambda), 
 \end{equation} 
and the following equality in distribution (denoted by ``$\stackrel{d}{=}$"):
   \begin{equation}\label{GtEGz}
    G(z,\lambda) \stackrel{d}{=} \sqrt{t_0} G_{t_0,E_0}^{(\lambda_0)} ,
   \end{equation}
where $G(z,\lambda):=(H-z)^{-1}=(V+\lambda\Psi-z)^{-1}$. Under the choice of parameters in \eqref{eq:t0E0}, we have  
\be\label{eq:E0bulk}
E_t(E_0,\lambda_0) = \frac{(1+t)\sqrt{t_0}}{1+t_0}\re z +\OO(\lambda^2),\quad \eta_t(E_0,\lambda_0) = \frac{1-t}{\sqrt{t_0}}\im m(z,\lambda) . 
\ee
These identities imply that $1-t\asymp \eta_t$ uniformly in $t\in [0,t_0]$ and $|E_t|\le (2-\kappa)\sqrt{t}+\oo(1)$ for $|1-t|=\oo(1)$. In other words, during the flow from $t=0$ to $t_0$, the imaginary part $\eta_t$ decreases from $\eta_0\asymp 1$ to $\eta_{t_0}\asymp 1-t_0 \asymp \im z $, while the spectral parameter $z_t$ stays within the bulk of the limiting spectrum of $H_t$, provided $t\ge 1 - \e$ for a small enough constant $\e>0$ depending on $\kappa$. 
\end{lemma}
\begin{proof}
First, the choice of $t_0$, $E_0$, and $\lambda_0$ in \eqref{eq:t0E0} gives that 
\be\label{eq:zE0}z=t_0^{-1/2}\left[ E_0 + (1-t_0)t_0^{-1/2}m(z,\lambda)\right].\ee
By \eqref{eq:zE0}, $m(E_0,\lambda_0)$ satisfies the self-consistent equation 
\begin{align*}
m(E_0,\lambda_0)&= \frac{1}{N}\tr \frac{1}{\lambda_0\Psi - E_0 - m(E_0,\lambda_0)} =  \frac{1}{N}\tr \frac{1}{\sqrt{t_0}\lambda\Psi - \sqrt{t_0} z + (1-t_0)t_0^{-1/2}m(z,\lambda) - m(E_0,\lambda_0)} . 
\end{align*}
Noticing that $m(E_0,\lambda_0)=t_0^{-1/2}m(z,\lambda)$ satisfies this equation, we obtain the first identity in \eqref{eq:zztE}. Plugging it into \eqref{eq:zE0} yields the second identity in \eqref{eq:zztE}. Using these two identities and $\lambda_0=\sqrt{t_0}\lambda$, we obtain the final identity in \eqref{eq:zztE} by applying \eqref{def_G0t} at $z=z_{t_0}(E_0,\lambda_0)$ with $t=t_0$ and utilizing the identities in \eqref{eq:mt_stay}. 

Next, using \eqref{eq:zztE} and the fact that \smash{$V_{t_0}\stackrel{d}{=}\sqrt{t_0}V$}, we get 
\begin{align*}
    G_{t_0,E_0}^{(\lambda_0)}\stackrel{d}{=}\frac{1}{\sqrt{t_0}V + \lambda_0\Psi - z_{t_0}(E_0,\lambda_0)}=\frac{t_0^{-1/2}}{ V + \lambda \Psi - z} = t_0^{-1/2} G(z,\lambda),
\end{align*}
which concludes \eqref{GtEGz}. The second identity of \eqref{eq:E0bulk} follows immediately from \eqref{eta} and \eqref{eq:zztE}. For the first identity, we use that (recall \eqref{eq:expandm})
$$m(E_0,\lambda_0)=m_{sc}(E_0)+\OO(\lambda_0^2) = \frac{-E_0+\sqrt{E_0^2-4}}{2}+\OO(\lambda^2).$$ 
Together with \eqref{eq:t0E0} and \eqref{eq:zztE}, it gives that 
\be\nonumber
E_0 = \sqrt{t_0}\re z - (1-t_0)\re m(E_0,\lambda_0) = \sqrt{t_0}\re z + \frac{1}{2}(1-t_0)E_0+\OO(\lambda^2) . 
\ee
From this equation, we derive that 
$$ E_0=\frac{2\sqrt{t_0}}{1+t_0}\re z +\OO(\lambda^2).$$
Plugging it into \eqref{eta} yields that 
\begin{align*}
   E_t(E_0,\lambda_0)&= E_0+(1-t)\re m(E_0,\lambda_0)=  E_0 - \frac{1}{2}(1-t)E_0+\OO(\lambda^2) =\frac{(1+t)\sqrt{t_0}}{1+t_0}\re z +\OO(\lambda^2), 
\end{align*}
which concludes the first identity in \eqref{eq:E0bulk}.
\end{proof}

In the proof, we will fix a target spectral parameter $z\in \C_+$ with $|\re z|\le 2-\kappa$ and $\im z \in [N^\fd\eta_*,1]$ for a small constant $\fd>0$. Accordingly, we choose $t_0$, $E_0$, and $\lambda_0$ as specified in \eqref{eq:t0E0}. For clarity, we will denote $\lambda_0$ and $E_0$ simply as $\lambda$ and $E$ throughout the remainder of the proof. Unless we want to emphasize their dependence on $\lambda$ or $E$, we will often omit these variables from various notations---particularly from $m$ and $M$ in \eqref{eq:mt_stay}---whose values remain fixed during the flow. 

\subsection{\texorpdfstring{$G$}{G}-loops and primitive loops}\label{subsec:Gloop}

Our focus will be on the dynamics of $G_{t}\equiv G_{t,E}^{(\lambda)}$ and the corresponding $G$-loops for $t\in [0,t_0]$ with $1-t_0\gtrsim \im z \ge N^\fd\eta_*$.

\begin{definition}[$G$-Loop]\label{Def:G_loop}
For $\sigma\in \{+,-\}$, we denote
 \begin{equation}\nonumber
  G_{t}(\sigma):=\begin{cases}
       (H_t-z_t)^{-1}, \ \ \text{if} \ \  \sigma=+,\\
        (H_t-\bar z_t)^{-1}, \ \ \text{if} \ \ \sigma=-.
   \end{cases}
 \end{equation}
In other words, we denote $G_{t}(+)\equiv G_{t}$ and $G_{t}(-)\equiv G_{t}^*$. Denote by $I_{[a]}$ and $E_{[a]}$, $[a]\in \Zn$, the block identity and rescaled block identity matrices, defined as 
\be\label{def:Ia} (I_{[a]})_{ij}= \delta_{ij}\cdot \mathbf 1_{i\in [a]} , \quad E_{[a]}=W^{-d}I_{[a]}.\ee
For any $\fn\in \N$, for $\bsig=(\sigma_1, \cdots \sigma_\fn)\in \{+,-\}^\fn$ and $\bfa=([a_1], \ldots, [a_\fn])\in (\Zn)^\fn$, we define the $\fn$-$G$ loop by 
\begin{equation}\label{Eq:defGLoop}
    {\cal L}^{(\fn)}_{t, \boldsymbol{\sigma}, \ba}=\left \langle \prod_{i=1}^\fn \left(G_{t}(\sigma_i) E_{[a_i]}\right)\right\rangle .
\end{equation}
Furthermore, we denote (recall \eqref{eq:mt_stay})
\begin{equation}\label{def_mtzk}
\begin{split}
& m (\sigma ):= \begin{cases}
    m(E) \equiv m_t(z_t(E)),  &\text{if} \ \ \sigma  =+ \\
    \bar m(E) \equiv m_t(\bar z_t(E)),  &\text{if} \ \ \sigma = -
\end{cases},\ \ M (\sigma ):= \begin{cases}
    M(E) \equiv M_t(z_t(E)),  &\text{if} \ \ \sigma  =+ \\
    M(E)^* \equiv M_t(\bar z_t(E)),  &\text{if} \ \ \sigma = -
\end{cases},    
\end{split}
\end{equation}
and introduce the ``centered resolvent" as
\begin{equation}\label{Eq:defwtG}
 \Gc_t(\sigma) := G_t(\sigma) -M (\sigma).
\end{equation}
\end{definition}

To derive the loop hierarchy for the $G$-loops, we introduce the following cut-and-glue operations. We refer readers to Figures 1--3 of \cite{Band1D} for illustrations of these operations.

\begin{definition}[Loop operations]\label{Def:oper_loop}
For the $G$-loop in \eqref{Eq:defGLoop}, we define the following operations on it.

\medskip
 
\noindent \emph{(1)} 
 For $k \in \qqq{\fn}$ and $[a]\in \Zn$, we define the first cut-and-glue operator ${\cut}^{[a]}_{k}$ as:
\be\label{eq:cut1}
{\cut}^{[a]}_{k} \circ {\cal L}^{(\fn)}_{t, \boldsymbol{\sigma}, \ba}:= \left \langle \prod_{i<k}  \left(G_{t}(\sigma_i) E_{[a_i]}\right)\left( G_{t}(\sigma_k) E_{[a]} G_{t}(\sigma_k)E_{[a_k]}\right)\prod_{i>k}  \left(G_{t}(\sigma_i) E_{[a_i]}\right)\right\rangle,
\ee
i.e., it is the $(\fn+1)$-$G$ loop obtained by replacing $G_t(\sigma_k)$ with $G_{t}(\sigma_k) E_{[a]} G_{t}(\sigma_k)$. Graphically, the operator \smash{${\cut}^{[a]}_{k}$} cuts the $k$-th edge $G_t(\sigma_k)$ and glues the two new ends with $E_{[a]}$. 
It can also be considered as an operator on the indices $(\boldsymbol{\sigma},\ba)$: 
\begin{align*}
{\cut}^{[a]}_{k} (\boldsymbol{\sigma}, \ba) = \big( &(\sigma_1,\ldots, \sigma_{k-1}, \sigma_k,\sigma_k ,\sigma_{k+1},\ldots, \sigma_\fn ),  ([a_1],\ldots, [a_{k-1}], [a],[a_k],[a_{k+1}],\ldots, [a_\fn] )\big).
\end{align*}
Hence, we will also express \eqref{eq:cut1} as 
    $${\cut}^{[a]}_{k} \circ {\cal L}^{(\fn)}_{t, \boldsymbol{\sigma}, \ba}\equiv 
     {\cal L}^{(\fn+1)}_{t, \;  {\cut}^{[a]}_{k} (\boldsymbol{\sigma}, \ba)}.
    $$

\medskip

 \noindent 
 \emph{(2)} For $k < l \in \qqq{\fn}$, we define the second cut-and-glue operator ${\cutL}^{[a]}_{k,l}$ from the left (``L") of $k$ as: 
 \begin{align}\label{eq:cutL}
 {\cutL}^{[a]}_{k,l} \circ {\cal L}^{(\fn)}_{t, \boldsymbol{\sigma}, \ba}:= \left \langle \prod_{i<k}  \left[G_{t}(\sigma_i) E_{[a_i]}\right]\left( G_{t}(\sigma_k) E_{[a]} G_{t}(\sigma_l)E_{[a_l]}\right)\prod_{i>l}  \left[G_{t}(\sigma_i) E_{[a_i]}\right]\right\rangle, 
 \end{align}
and define the third cut-and-glue operator ${\cutR}^{[a]}_{k,l}$ from the right (``R") of $k$ as: 
 \begin{align}\label{eq:cutR}
 {\cutL}^{[a]}_{k,l} \circ {\cal L}^{(\fn)}_{t, \boldsymbol{\sigma}, \ba}:= \left \langle \prod_{k\le i <l}  \left[G_{t}(\sigma_i) E_{[a_i]}\right]\cdot \left( G_{t}(\sigma_l) E_{[a]}\right)\right\rangle. 
 \end{align}
In other words, the above operators cut the $k$-th and $l$-th edges $G_t(\sigma_k)$ and $G_t(\sigma_l)$ and creates two chains: the left chain to the vertex $[a_k]$ is of length $(\fn+k-l+1)$ and contains the vertex $[a_\fn]$, while the right chain to the vertex $[a_k]$ is of length $(l-k+1)$ and does not contain the vertex $[a_\fn]$.  
Then, \smash{${\cutL}^{[a]}_{k,l}\circ {\cal L}^{(\fn)}_{t, \boldsymbol{\sigma}, \ba}$} (resp.~\smash{${\cutR}^{[a]}_{k,l}\circ {\cal L}^{(\fn)}_{t, \boldsymbol{\sigma}, \ba}$}) gives a $(\fn+k-l+1)$-$G$ loop (resp.~$(l-k+1)$-$G$ loop) obtained by gluing the left chain (resp.~right chain) at the new vertex $[a]$.
Again, we can also consider the two operators to be defined on the indices $(\boldsymbol{\sigma},\ba)$: 
\begin{align*}
    &{\cutL}^{[a]}_{k,l} (\boldsymbol{\sigma}, \ba) = \left((\sigma_1,\ldots, \sigma_k,\sigma_l ,\ldots, \sigma_\fn ),([a_1],\ldots, [a_{k-1}], [a],[a_l],\ldots, [a_\fn] )\right),\\
    &{\cutR}^{[a]}_{k,l} (\boldsymbol{\sigma}, \ba) = \left((\sigma_k,\ldots, \sigma_l),([a_k],\ldots, [a_{l-1}], [a])\right).
\end{align*}
Hence, we will also express \eqref{eq:cutL} and \eqref{eq:cutR} as 
    $${\cutL}^{[a]}_{k,l} \circ {\cal L}^{(\fn)}_{t, \boldsymbol{\sigma}, \ba}\equiv 
     {\cal L}^{(\fn+k-l+1)}_{t, \;  {\cutL}^{[a]}_{k,l} (\boldsymbol{\sigma}, \ba)},\quad {\cutR}^{[a]}_{k} \circ {\cal L}^{(\fn)}_{t, \boldsymbol{\sigma}, \ba}\equiv 
     {\cal L}^{(l-k+1)}_{t, \;  {\cutR}^{[a]}_{k,l} (\boldsymbol{\sigma}, \ba)}.
    $$

\end{definition}

Given a matrix $\cal A$ defined on $ \ZL$, we define its ``projection" $\cal A^{L\to n}$ to $\Zn$ as
\be\nonumber 
\cal A^{\LK}_{[x][y]}:=  W^{-d}\, \sum_{x'\in [x]}\sum_{y'\in [y]}\cal A_{x'y'}. 
\ee
Under this definition, the projections of \eqref{eq:SBA} and \eqref{eq:SWO} are given by
\be\label{eq:SLton}(S^\BA)^{\LK}=I_n,\quad (S^\WO)^{\LK}=(1+2d\lambda^2)^{-1}\left( I_n + \lambda^2\Lambda_n\right).\ee 
For $(x,y)\in (\ZL)^2$, we denote $\partial_{(x,y)}\equiv\partial_{xy}:=\partial_{(H_t)_{xy}}$. 
Then, by It\^o's formula and equation \eqref{eq:SDE_Gt}, it is easy to derive the following SDE satisfied by the $G$-loops. 

\begin{lemma}[Loop hierarchy] \label{lem:SE_basic}
An $\fn$-$G$ loop satisfies the following SDE, called the {\bf loop hierarchy}:   
\begin{align}\label{eq:mainStoflow}
    \dd \mathcal{L}^{(\fn)}_{t, \boldsymbol{\sigma}, \ba} &=\dd  
    \mathcal{B}^{(\fn)}_{t, \boldsymbol{\sigma}, \ba} 
    +\mathcal{W}^{(\fn)}_{t, \boldsymbol{\sigma}, \ba} \dd t 
    +  W^d\sum_{1 \le k < l \le \fn} \sum_{[a], [b]} \left( \cutL^{[a]}_{k, l} \circ \mathcal{L}^{(\fn)}_{t, \boldsymbol{\sigma}, \ba} \right) S^{\LK}_{[a][b]} \left( \cutR^{[b]}_{k, l} \circ \mathcal{L}^{(\fn)}_{t, \boldsymbol{\sigma}, \ba} \right) \dd t, 
\end{align}
 where the martingale term $\mathcal{B}^{(\fn)}_{t, \boldsymbol{\sigma}, \ba}$ and the term $\mathcal{W}^{(\fn)}_{t, \boldsymbol{\sigma}, \ba}$ are defined by 
 \begin{align} \label{def_Edif}
\dd\mathcal{B}^{(\fn)}_{t, \boldsymbol{\sigma}, \ba} :  = & \sum_{x,y\in \ZL} 
  \left( \partial_{xy}  {\cal L}^{(\fn)}_{t, \boldsymbol{\sigma}, \ba}  \right)
 \cdot \sqrt{S _{xy}}
  \left(\dd B_t\right)_{xy}, \\\label{def_EwtG}
\mathcal{W}^{(\fn)}_{t, \boldsymbol{\sigma}, \ba}: = &  {W}^d \sum_{k=1}^\fn \sum_{[a], [b]\in \Zn} \; 
 \left \langle \Gc_t(\sigma_k) E_{[a]} \right\rangle
   S^{\LK}_{[a][b]} 
  \left( {\cut}^{[b]}_{k} \circ {\cal L}^{(\fn)}_{t, \boldsymbol{\sigma}, \ba} \right) . 
\end{align}
\end{lemma}
  
We will see that this loop hierarchy is well-approximated by the primitive loops, defined as follows. 

\begin{definition}[Primitive loops]\label{Def_Ktza}
We define the primitive loop of length 1 as: 
$${\cal K}^{(1)}_{t,\sigma,[a]}:=m(\sigma),\quad \forall t\in [0,1],\ \sigma\in \{+,-\}.$$ 
For $\fn\ge 2$, we define the function ${\cal K}^{(\fn)}_{t, \boldsymbol{\sigma}, \ba}$ (of $t\in[0,1]$, $\bsig\in \{+,-\}^\fn$, and $\bfa\in (\Zn)^\fn$) to be the unique solution to the following system of differential equation: 
 \begin{align}\label{pro_dyncalK}
       \partial_t {\cal K}^{(\fn)}_{t, \boldsymbol{\sigma}, \ba} 
       =  
       W^d \sum_{1\le k < l \le \fn} \sum_{[a], [b]} \left( \cutL^{[a]}_{k, l} \circ \mathcal{K}^{(\fn)}_{t, \boldsymbol{\sigma}, \ba} \right) S^{\LK}_{[a][b]} \left( \cutR^{[b]}_{k, l} \circ \mathcal{K}^{(\fn)}_{t, \boldsymbol{\sigma}, \ba} \right) ,
    \end{align}
with the following initial condition at $t=0$: 
\be\label{eq:initial_K}
{\cal K}^{(k)}_{0, \boldsymbol{\sigma}, \ba} =  {\cal M}^{(k)}_{\boldsymbol{\sigma}, \ba} ,\quad \forall k\in \N, \ \bsig\in \{+,-\}^k,\ \bfa\in (\Zn)^k,\ee
where ${\cal M}^{(k)}_{\boldsymbol{\sigma}, \ba}$ is a $k$-$M$ loop defined as
\be\label{eq:KMloop} {\cal M}^{(k)}_{\boldsymbol{\sigma}, \ba}:=\left \langle \prod_{i=1}^k \left(M(\sigma_i) E_{[a_i]}\right)\right\rangle .\ee
In equation \eqref{pro_dyncalK}, the operators $\cutL$ and $\cutR$ act on ${\cal K}^{(\fn)}_{t, \boldsymbol{\sigma}, \ba}$ through the actions on indices, that is, 
  \be\label{calGonIND}
     {\cutL}^{[a]}_{k,l}  \circ {\cal K}^{(\fn)}_{t, \boldsymbol{\sigma}, \ba} := {\cal K}^{(\fn+k-l+1)}_{t,  {\cutL}^{[a]}_{k,l}  (\boldsymbol{\sigma}, \ba)} , 
     \quad  {\cutR}^{[b]}_{k,l}  \circ {\cal K}^{(\fn)}_{t, \boldsymbol{\sigma}, \ba} := {\cal K}^{(l-k+1)}_{t,  {\cutR}^{[b]}_{k,l}  (\boldsymbol{\sigma}, \ba)}.  
       \ee
    We will call ${\cal K}^{(\fn)}_{t, \boldsymbol{\sigma}, \ba}$ a primitive loop of length $\fn$ or a $\fn$-$\cK$ loop.  
\end{definition}

In equation \eqref{eq:mainStoflow}, the derivative $ \partial_{ij}  {\cal L}^{(\fn)}$ in $\mathcal{B}^{(\fn)}$ is a product of $(\fn+1)$-resolvent entries, while the term $\mathcal{W}^{(\fn)}$ includes an $(\fn+1)$-$G$ loop \smash{${\cut}^{[b]}_{k} \circ {\cal L}^{(\fn)}_{t, \boldsymbol{\sigma}, \ba}$}. Hence, equation \eqref{eq:mainStoflow} represents a ``hierarchy" rather than a ``self-consistent equation" for the $G$-loops. However, in equation \eqref{pro_dyncalK}, the primitive loops \smash{$\cutL^{[a]}_{k, l} \circ \mathcal{K}^{(\fn)}_{t, \boldsymbol{\sigma}, \ba} $} and \smash{$\cutR^{[a]}_{k, l} \circ \mathcal{K}^{(\fn)}_{t, \boldsymbol{\sigma}, \ba}$} have lengths of at most $\fn$. Therefore, \eqref{pro_dyncalK} indeed gives a self-consistent equation that can be solved inductively in $\fn$. In \Cref{sec:prim}, we will present an explicit representation of the primitive loops. 

For clarity of presentation, we will also call $G$-loops and primitive loops as $\cL$-loops and $\cK$-loops, respectively. Moreover, we will call \smash{$(\cL-{\cal K})^{(\fn)}_{t, \boldsymbol{\sigma}, \ba}\equiv \cL^{(\fn)}_{t, \boldsymbol{\sigma}, \ba}-\cK^{(\fn)}_{t, \boldsymbol{\sigma}, \ba}$} an $(\cL-\cK)$-loop. 

\subsection{Propagators} \label{sec:propagator}
The primitive loops will be expressed with the $\Theta$-propagators defined as follows.

\begin{definition}[$\Theta$-propagator]\label{def_Theta}
Given $\sigma_1,\sigma_2\in \{+,-\}$, define an $n^d\times n^d$ matrix $M^{(\sigma_1,\sigma_2)}$ as a rescaled 2-$M$ loop:
\be\label{eq:Msig} M^{(\sigma_1,\sigma_2)}_{[x][y]}:= W^d\left\langle M(\sigma_1) E_{[x]}M(\sigma_2) E_{[y]}\right\rangle,\quad \forall [x],[y]\in \Zn.\ee
Then, for $t\in[0,1]$ and $\sigma_1,\sigma_2\in \{+,-\}$, the $\Theta$-propagator $\Theta_t^{(\sigma_1,\sigma_2)}$ is an $n^d\times n^d$ matrix defined as:
\begin{equation}\label{def_Thxi}
    \Theta_{t}^{(\sigma_1,\sigma_2)} := \frac{1}{1 - t M^{(\sigma_1,\sigma_2)}S^{\LK}}, 
\end{equation}
where $S^{\LK}$ is given by \eqref{eq:SLton}. 
For the block Anderson model, it can be simplified by using $S^{\LK}=I_n$. For the Wegner orbital model, we have $ M^{(\sigma_1,\sigma_2)}=m(\sigma_1)m(\sigma_2)I_n$, allowing us to simplify \eqref{def_Thxi} as:
\be\label{eq:Theta_WO}\Theta_{t}^{(\sigma_1,\sigma_2)}= \frac{1}{1-t m(\sigma_1)m(\sigma_2)S^{\LK}}.\ee
\end{definition}

We now state some fundamental properties of $M^{(\sigma_1,\sigma_2)}$ and $\Theta_{t}^{(\sigma_1,\sigma_2)}$ in the following two lemmas, \Cref{lem:propM} and \Cref{lem_propTH}. 
Their proofs are basic and will be outlined in \Cref{appd:deter}. 

\begin{lemma} [Properties of $M$]\label{lem:propM}
For the block Anderson model, $M$ and $M^{(\sig_1,\sig_2)}$ satisfy the following properties for any $|E|\le 2-\kappa$ for a constant $\kappa>0$ (recall that $E$ is the spectral parameter in the flow \eqref{eq:zt}). 
\begin{enumerate}
    \item {\bf Transposition}: $M$ and $M^{(\sigma_1,\sigma_2)}$ are generally neither symmetric nor Hermitian, but we have that 
\be\label{eq:Msym}
(M^{(\sigma_1,\sigma_2)})^\top =M^{(\sigma_2,\sigma_1)}.
\ee

\item {\bf Symmetry}: The following block translation invariance holds: 
\be\label{symmetryM}
  \qquad  M^{(\sigma_1,\sigma_2)}([x]+[a],[y]+[a])= M^{(\sigma_1,\sigma_2)}([x],[y]),\ \ \forall [x],[y],[a]\in \Zn.
\ee
In addition, in 2D, we have the parity symmetry: 
\be\label{symmetryP} \qquad  M^{(\sigma_1,\sigma_2)}(0,[x])= M^{(\sigma_1,\sigma_2)}([x], 0)= M^{(\sigma_1,\sigma_2)}(0,-[x]),\ \ \forall [x] \in \Zn.\ee

\item {\bf Ward's identity}: For any $\sig_1,\sig_2\in \{+,-\}$, we have
\be\label{eq:WardM2}
\sum_{[b]}|M^{(\sigma_1,\sigma_2)}_{[a][b]}|\le 1,\quad \forall [a]\in \Zn.
\ee
Moreover, if $\sigma_1\ne \sigma_2$, then we have 
\be\label{eq:WardM1}
\sum_{[b]}M^{(\sigma_1,\sigma_2)}_{[a][b]} = 1,\quad \forall [a]\in \Zn.
\ee

\item {\bf Basic bounds}: There exists a constant $C>0$ such that the following estimates hold:
\begin{align}\label{eq:M-msc} \|M - m_{sc}(E)I_N\|\le &C\heta ,\quad \|M^+ - [m_{sc}(E)]^2 I_N\|_{\infty\to\infty}\le C\heta ,
\\ \label{Mbound_AO}
&\|M|_{[x][y]}\|\le (C\heta)^{|[x]-[y]|},
\end{align}
where we recall that $M^+$ was defined below \eqref{def:Theta}, $\heta$ was defined in \eqref{eq:cond_A12}, $|\cdot|$ denotes the (periodic) $L^1$-distance, and we have used \smash{$M|_{[x][y]}$} to denote the $([x],[y])$-th block of $M$, which is a \smash{$W^d\times W^d$} matrix. 
Furthermore, for any $\sig_1,\sig_2\in \{+,-\}$ and $x\ne y\in \ZL$, we have
\be\label{Mbound_AO2}
C^{-1}\lambda^2 \mathbf 1([x]\sim [y])\le |M_{[x][y]}^{(\sig_1,\sig_2)}|\le C\lambda^2(C\heta)^{2|[x]-[y]|-2}.
\ee

\end{enumerate}
 \end{lemma}


\begin{lemma}\label{lem_propTH}
Under the flow setup in \Cref{zztE}, define
\be\label{eq:ellt}
 \ell_t:= \min\left(\lambda|1-t|^{-1/2}+1,\; n\right). 
\ee
Note that since $\eta_t\asymp 1-t$ by \eqref{eq:E0bulk}, $\ell_t$ is of the same order as $\ell(\eta_t)$ defined in \eqref{eq:elleta}. 
The $\Theta$-propagators satisfy the following properties for any $t\in [0,1-N^{-1}]$ and $\sigma_1,\sigma_2\in \{+,-\}$:
\begin{enumerate}
    
\item {\bf Symmetry}: $\Theta_{t}^{(\sigma_1,\sigma_2)}$ also satisfies the properties \eqref{eq:Msym}--\eqref{symmetryP}.

\item {\bf Exponential decay}: For any large constant $D>0$, there exists a constant $c>0$ such that the following estimate holds when $\sigma_1\ne \sigma_2$: 
\begin{equation}\label{prop:ThfadC}     
\qquad\quad \Theta^{(\sigma_1,\sigma_2)}_{t}(0,[x])\prec \frac{e^{-c |[x]|/ {\ell}_t}}{|1-t|  {\ell}_t^d}+ W^{-D}.  
\end{equation}
For $\sigma_1=\sigma_2$, there exists a constant $C>0$ such that 
\begin{equation}\label{prop:ThfadC_short}     \Theta^{(\sigma_1,\sigma_2)}_{t}(0,[x])\le C\left(\mathbf 1_{[x]=0}+\lambda^2(C\heta)^{2(|[x]|-1)}\mathbf 1_{[x]\ne 0}\right)+W^{-D}.
\end{equation}

\item {\bf First-order finite difference}: The following estimate holds for any $[x],[y]\in \Zn$ and $\sigma_1\ne \sigma_2$: \begin{equation}\label{prop:BD1}     
\qquad   \left| \Theta^{(\sigma_1,\sigma_2)}_{t}(0, [x])-\Theta^{(\sigma_1,\sigma_2)}_{t}(0, [y])\right|\prec \frac{1}{ \lambda^{2}+|1-t|}\frac{|[x]-[y]|}{\langle [x]\rangle^{d-1}+\qq{[y]}^{d-1}}.
     \end{equation}

\item {\bf Second-order finite difference}: The following estimate holds for any \smash{$[x],[y]\in \Zn$} and $\sigma_1\ne \sigma_2$: 
\begin{equation}\label{prop:BD2}
\begin{split}
&\Theta^{(\sigma_1,\sigma_2)}_{t} (0,[x]+[y]) + \Theta^{(\sigma_1,\sigma_2)}_{t} (0,[x]-[y])-  2\Theta^{(\sigma_1,\sigma_2)}_{t} (0,[x]) \prec \frac{1}{\lambda^{2}+|1-t|}\frac{|[y]|^2}{\langle [x]\rangle^{d}} .    
\end{split}
 \end{equation}
\end{enumerate}
The properties listed above all apply to $\Theta_{t}^{(\sigma_1,\sigma_2)}$ in \eqref{eq:Theta_WO} defined for the Wegner orbital model (whose corresponding flow will be detailed in \Cref{rmk:ext_Wegner} below).
\end{lemma}

In the statement of \Cref{lem_propTH}, we have adopted the following notion of stochastic domination introduced in \cite{EKY_Average}. It will be used frequently throughout the remainder of this paper for clarity of presentation.

\begin{definition}[Stochastic domination and high probability event]\label{stoch_domination}
	{\rm{(i)}} Let
	\[\xi=\left(\xi^{(N)}(u):N\in\mathbb N, u\in U^{(N)}\right),\hskip 10pt \zeta=\left(\zeta^{(N)}(u):N\in\mathbb N, u\in U^{(N)}\right),\]
	be two families of non-negative random variables, where $U^{(N)}$ is a possibly $N$-dependent parameter set. We say $\xi$ is stochastically dominated by $\zeta$, uniformly in $u$, if for any fixed (small) $\tau>0$ and (large) $D>0$, 
	\[\mathbb P\bigg(\bigcup_{u\in U^{(N)}}\left\{\xi^{(N)}(u)>N^\tau\zeta^{(W)}(u)\right\}\bigg)\le N^{-D}\]
	for large enough $N\ge N_0(\tau, D)$, and we will use the notation $\xi\prec\zeta$. 
	If for some complex family $\xi$ we have $|\xi|\prec\zeta$, then we will also write $\xi \prec \zeta$ or $\xi=\OO_\prec(\zeta)$. 
	
	\vspace{5pt}
	\noindent {\rm{(ii)}} As a convention, for two \emph{deterministic} quantities $\xi$ and $\zeta$, we will write $\xi\prec\zeta$ if and only if $|\xi|\le N^\tau |\zeta|$ for any constant $\tau>0$. 
	
	\vspace{5pt}
	\noindent {\rm{(iii)}} We say an event $\Xi$ holds with high probability (w.h.p.) if for any constant $D>0$, $\mathbb P(\Xi)\ge 1- N^{-D}$ for large enough $N$. More generally, we say an event $\Xi'$ holds $w.h.p.$ in $\Xi$ if for any constant $D>0$,
	$\P( \Xi\setminus \Xi')\le N^{-D}$ for large enough $N$.
\end{definition}


We can solve the primitive equation \eqref{pro_dyncalK} explicitly in the special cases $\fn = 2,  3$.

\begin{example} 
We claim that the $2$-$\cK$ loop is given by  
\be\label{Kn2sol}\cK^{(2)}_{t,(\sig_1,\sig_2),([a_1],[a_2])}=W^{-d} \left(\Theta_t^{(\sigma_1,\sigma_2)}  M^{(\sig_1,\sig_2)}\right)_{[a_1][a_2]}.\ee
In fact, when $t=0$, we have $\Theta_t^{(\sigma_1,\sigma_2)}=I$, so $W^{-d}\cdot  \Theta_t^{(\sigma_1,\sigma_2)} M^{(\sig_1,\sig_2)}$ satisfies the initial condition \eqref{eq:initial_K}. Moreover, taking the derivative with respect to $t$, we get
\be\label{eq:diff_Theta}\partial_t \Theta_t^{(\sigma_1,\sigma_2)} = \Theta_t^{(\sigma_1,\sigma_2)} M^{(\sig_1,\sig_2)}S^{\LK}\Theta_t^{(\sigma_1,\sigma_2)} .\ee
With \eqref{eq:diff_Theta}, we see that \eqref{Kn2sol} satisfies the desired equation \eqref{pro_dyncalK} with $\fn=2$. 
\end{example}

\begin{remark}
In the setting of \Cref{thm_diffu}, given $z=E+\ii \eta$ with $|E|\le 2-\kappa$ and $W^\e\eta_* \le \eta \le 1$, we can choose the deterministic flow as described in \Cref{zztE} such that \eqref{eq:zztE} and \eqref{GtEGz} hold. Then, using the notations introduced earlier, we can express the quantities in \Cref{thm_diffu} as
\be\label{eq:avg_GTheta}
\begin{split}
W^{-2d}\sum_{x\in[a],y\in[b]}\left(|G_{xy}|^2-\Theta_{xy}\right)&= t_0 \left(\cL-\cK\right)^{(2)}_{t_0,(-,+),([a],[b])},\\ W^{-2d}\sum_{x\in[a],y\in[b]}\left(G_{xy}G_{yx}-S^{+}_{xy}\right)&= t_0\left(\cL-\cK\right)^{(2)}_{t_0,(+,+),([a],[b])}.
\end{split}\ee
For the Wegner orbital model, these identities follow directly from the definitions \eqref{Eq:defGLoop}, \eqref{def:Theta}, \eqref{eq:Theta_WO}, and \eqref{Kn2sol}. To understand why they hold for the block Anderson model, observe that by using the identity $S^2=S$ for the variance matrix in \eqref{eq:SBA} and the last identity in \eqref{eq:zztE}, we can express $\Theta$ in \eqref{def:Theta} as
$$S\Theta S= \frac{SM^0S}{1-S M^0 S} = \frac{t_0 M^{(-,+)}}{1-t_0 M^{(-,+)}}\otimes \mathbf E \implies
W^{-2d}\sum_{x\in[a],y\in[b]}\Theta_{xy}=t_0W^{-d}\cdot \left(\Theta_{t_0}^{(-,+)}M^{(-,+)}\right)_{[a][b]}. $$
A similar identity holds for $S^+$. By \eqref{eq:avg_GTheta}, \Cref{thm_diffu} essentially states that the $2$-$G$ loops are well-approximated by the 2-$\cK$ loops as $N\to \infty.$
\end{remark}

\begin{example}
For  $k=3$, let 
$\boldsymbol{\sigma}=(\sigma_1, \sigma_2,\sigma_3)$ and $
\ba=([a_1], [a_2],[a_3])$.
Then, equation \eqref{pro_dyncalK} becomes 
\begin{align}    
\partial_t \mathcal{K}^{(3)}_{t, \boldsymbol{\sigma}, \ba} &= W^d \sum_{[b_1], [c_1]} \mathcal{K}^{(2)}_{t, (\sigma_1,\sigma_2), ([a_1], [b_1])}  S^{\LK}_{[b_1][c_1]}  \mathcal{K}^{(3)}_{t, \boldsymbol{\sigma}, ([c_1], [a_2], [a_3])}\nonumber \\
&+ W^d \sum_{[b_2], [c_2]} \mathcal{K}^{(2)}_{t, (\sigma_2,\sigma_3), ([a_2], [b_2])}  S^{\LK}_{[b_2][c_2]}  \mathcal{K}^{(3)}_{t, \boldsymbol{\sigma}, ([a_1], [c_2], [a_3])}\nonumber \\
&+ W^d \sum_{[b_3], [c_3]} \mathcal{K}^{(2)}_{t, (\sigma_3,\sigma_1), ([a_3], [b_3])}  S^{\LK}_{[b_3][c_3]}  \mathcal{K}^{(3)}_{t, \boldsymbol{\sigma}, ([a_1], [a_2], [c_3])} . \nonumber 
\end{align}
Recalling the $M$-loop defined in \eqref{eq:KMloop} and using \eqref{Kn2sol}--\eqref{eq:diff_Theta}, one can verify that 
\be\label{Kn3sol} \mathcal{K}^{(3)}_{t, \boldsymbol{\sigma}, \ba} = \sum_{[b_1],[b_2],[b_3]}\Theta_{t}^{(\sig_1,\sig_2)}([a_1],[b_1])\Theta_{t}^{(\sig_2,\sig_3)}([a_2],[b_2])\Theta_{t}^{(\sig_3,\sig_1)}([a_3],[b_3]) \cM^{(3)}_{\bsig,\mathbf b}
\ee
satisfies the initial condition \eqref{eq:initial_K} and the above equation.
\end{example}

In \Cref{sec:prim}, we will give a tree representation formula for general $k$-$\cK$ loops with $k\ge 4$. Moreover, using this tree representation, the estimates in \Cref{lem_propTH}, and a key \emph{sum zero property}, we will establish the following upper bound \eqref{eq:bcal_k} for $\mathcal{K}$-loops. 

\begin{lemma}[Upper bounds on $\cK$-loops]\label{ML:Kbound}
The primitive loops satisfy the following upper bound for each $d\in \{1,2\}$, $\fn\in \N$, and $t\in [0,t_0]$:
\begin{equation}\label{eq:bcal_k}
\max_{\bsig\in\{+,-\}^\fn}\max_{\ba\in(\Zn)^\fn}  {\cal K}^{(\fn)}_{t, \boldsymbol{\sigma}, \ba}  \prec \left(W^d \ell_t^d (1-t) \right)^{-\fn+1}\asymp \left(W^d \ell_t^d \eta_t \right)^{-\fn+1}.
\end{equation}
\end{lemma}

The following classical Ward's identity, which follows from a simple algebraic calculation, will be used
tacitly throughout this paper. 
\begin{lemma}[Ward's identity]\label{lem-Ward}
Given any Hermitian matrix $\cal A$, define its resolvent as $R(z):=(\cal A-z)^{-1}$ for a $z= E+ \ii \eta\in \C_+$. Then, we have 
    \be\label{eq_Ward0}
    \begin{split}
\sum_x \overline {R_{xy'}}  R_{xy} = \frac{R_{y'y}-\overline{R_{yy'}}}{2\ii \eta},\quad
\sum_x \overline {R_{y'x}}  R_{yx} = \frac{R_{yy'}-\overline{R_{y'y}}}{2\ii \eta}.
\end{split}
\ee
As a special case, if $y=y'$, we have
\be\label{eq_Ward}
\sum_x |R_{xy}( z)|^2 =\sum_x |R_{yx}( z)|^2 = \frac{\im R_{yy}(z) }{ \eta}.
\ee
\end{lemma}

Applying \eqref{eq_Ward0} to $G$, we can show that the $G$-loops satisfy the following identity \eqref{WI_calL}, which we will also refer to as a ``Ward's identity". In \Cref{sec:prim}, we will show that a similar Ward's identity \eqref{WI_calK} holds for the $\cal K$-loops. 
\begin{lemma}[Ward's identity for $G$-loops and $\cK$-loops]\label{lem_WI_K} 
For an $\fn$-$G$ loop ${\cal L}^{(\fn)}_{t, \boldsymbol{\sigma}, \ba}$ with $\fn\ge 2$ and $\sigma_1=-\sig_{\fn}$, 
we have the following identities, which are called Ward's identities at the vertex $[a_\fn]$: 
\begin{align}\label{WI_calL}
&\sum_{a_\fn}{\cal L}^{(\fn)}_{t, \boldsymbol{\sigma}, \ba}=
\frac{1}{2\ii W^d\eta_t}\left( {\cal L}^{(\fn-1)}_{t,  \wh\bsig^{(+,\fn)}, \wh\ba^{(\fn)}}- {\cal L}^{(\fn-1)}_{t,  \wh\bsig^{(-,\fn)} , \wh\ba^{(\fn)}}\right),  \\
\label{WI_calK}
& \sum_{a_\fn}{\cal K}^{(\fn)}_{t, \boldsymbol{\sigma}, \ba}=
\frac{1}{2\ii W^d\eta_t}\left( {\cal K}^{(\fn-1)}_{t,  \wh\bsig^{(+,\fn)}, \wh\ba^{(\fn)}}- {\cal K}^{(\fn-1)}_{t,  \wh\bsig^{(-,\fn)}, \wh\ba^{(\fn)}}\right) , 
\end{align}
where $\wh\bsig^{(\pm,\fn)}$ is obtained by removing $\sigma_\fn$ from $\boldsymbol{\sigma}$ and replacing $\sigma_1$ with $\pm$, i.e., $\wh\bsig^{(\pm,\fn)}=(\pm, \sigma_2, \cdots \sigma_{\fn-1})$, and  $\wh\ba^{(\fn)}$ is obtained by removing $a_\fn$ from $\ba$, i.e., 
$\wh\ba^{(\fn)}=(a_1, a_2,\cdots, a_{\fn-1}).$ 
\end{lemma}
\begin{proof}
The identity \eqref{WI_calL} follows from \eqref{eq_Ward0}, while \eqref{WI_calK} is a consequence of \Cref{ward-identity} below (see also \Cref{lem:KgentoK}). 
\end{proof}

\begin{remark}\label{rmk:ext_Wegner}
For the Wegner orbital model, the stochastic flow is defined in the same way as in \eqref{MBM} and \Cref{Def:stoch_flow}, while the deterministic flow $z_t$ remains as in \eqref{eq:zt}. Then, as shown in Lemma 2.8 of \cite{Band1D}, for any $z\in \mathbb C_+$ with $\im z\in (0, 1]$ and $|\re z|\le  2-\kappa$, we can choose the parameters $t_0$ and $E_0$ as follows such that \eqref{eq:zztE} and \eqref{GtEGz} still hold: 
\be \nonumber
t_0=|m_{sc}(z)|^2=\frac{\im m_{sc}(z)}{\im m_{sc}(z)+ \im z},\quad E_0= -2\frac{\re m_{sc}(z)}{\sqrt{t_0}}.\ee
Under this flow, using the notations introduced in \Cref{subsec:Gloop} and the estimates established in \Cref{lem_propTH} for the $\Theta$-propagators, we can complete the proof of the main results, \Cref{thm_locallaw} and \Cref{thm_diffu}, following almost the same arguments as in Sections \ref{sec:proof} and \ref{Sec:Stoflo} below. 
The only major modification is that our entrywise estimates on $G_t$ in terms of 2-$G$ loops, as given in Lemmas \ref{lem_GbEXP} and \ref{lem G<T} below, do not directly apply to the Wegner orbital model. However, these results can be established using a similar (and much simpler) argument as in \Cref{appd:MDE}. 
Alternatively, one can employ Lemma 4.1 in \cite{Band1D}, which was proved using resolvent identities and can be easily adapted to the Wegner orbital model setting.

We also remark that the tree representation for the Wegner orbital model is identical to that of RBM presented in Section 3 of \cite{Band1D}. Additionally, the corresponding Ward’s identity for $G$ and $\cK$-loops has also been established in \cite{Band1D}.
However, the upper bound on 
$\cK$-loops, as stated in \Cref{ML:Kbound}, is not fully covered by the arguments in \cite{Band1D}, because our $\Theta$-propagators satisfy much weaker bounds than those for RBM due to the presence of the small parameter $\lambda$. 
In fact, the tree representation, the Ward’s identity, and the upper bounds on $\cK$-loops, which we will establish in \Cref{sec:prim}, also apply to the Wegner orbital model as a special case. In particular, we will explain how to obtain an additional 
$\lambda^2$ factor when applying the sum-zero property to derive the upper bound \eqref{eq:bcal_k}.  

\end{remark}

Before concluding this section, we present the proof of \Cref{thm:QUE} using \Cref{thm_diffu} and \Cref{lem_propTH}.

\begin{proof}[\bf Proof of \Cref{thm:QUE}]  
Let $z=E+\ii\eta$ with $|E|\le 2-\kappa$ and $\eta=W^{-\e}\eta_0$. With the spectral decomposition of $G(z)$ and the definition of $ {\cal I}_E$, we find that 
\begin{align}\label{ssfa2}
 \sum_{i,j\in \ZL: \lambda_i,\lambda_j\in {\cal I}_E} &\left|\bu_i^*\left(E_{[a]}-N^{-1}\right) \bu_j\right|^2 \lesssim  \eta^4 \sum_{i, j  } \frac{\left|\bu_i^*\left(E_{[a]}-N^{-1}\right) \bu_j\right|^2}{\left|\lambda_i-z\right|^2\left|\lambda_j-z\right|^2}\nonumber\\
& = 
\eta^2\tr \left[\operatorname{Im} G(z)\left(E_{[a]}-N^{-1}\right) \operatorname{Im} G(z)\left(E_{[a]}-N^{-1}\right)\right].
\end{align}
The expectation of the right-hand side (RHS) can be written as 
\begin{align*}
&~\frac{\eta^2}{n^{2d}}\sum_{[b],[b']\in \Zn}
\mathbb{E}  
 \tr \left[\operatorname{Im} G(z)\left(E_{[a]}-E_{[b]}\right) \operatorname{Im} G(z)\left(E_{[a]}-E_{[b']}\right)\right] \nonumber
\\ 
\lesssim &~ \eta^2\max_{[a],[b],[b']}
\Big|\mathbb{E}\tr \left(\operatorname{Im} G(z) E_{[a]}  \operatorname{Im} G(z)E_{[b]} \right)
-\mathbb{E}\tr \left(\operatorname{Im} G(z) E_{[a]}  \operatorname{Im} G(z)E_{[b']}\right) \Big|. 
\end{align*}
Writing $\im G=(G-G^*)/(2\ii)$, 
we can bound the RHS by $I_1+I_2$, where 
\begin{align*}
I_1:=&~\eta^2 \max_{[a],[b]}\max_{\sig_1,\sig_2} \left|\E (\cL-\cK)^{(2)}_{t_0,(\sig_1,\sig_2),([a],[b])}\right| \\
I_2:=&~\eta^2  \max_{[a],[b],[b']}\max_{\sig_1,\sig_2} \left|\cK^{(2)}_{t_0,(\sig_1,\sig_2),([a],[b])}-\cK^{(2)}_{t_0,(\sig_1,\sig_2),([a],[b'])}\right| . 
\end{align*}

Under the assumptions \eqref{eq:cond-lambda2} and that $W\lambda \ge W^\fc N^{1/2}$ in 1D, we have $W^{-\e}\eta_0 \ge W^{\fd\wedge \fc - \e}/N$. Thus, when $\e<\fd\wedge \fc$, the QUE estimates \eqref{eq:diffuExp1} and \eqref{eq:diffuExp2} apply (recall \eqref{eq:avg_GTheta}), yielding the following bound on $I_1$:  
\be\label{eq:I1} I_1 \le \frac{\eta^2 W^\tau}{(W^d[\ell(\eta)]^d \eta)^3}\le \frac{ W^\tau\eta^2}{(N\eta)^3}\mathbf 1_{\eta\le \frac{\lambda^2}{n^2}} + \frac{W^\tau {\eta}^{1/2} }{(W^d\lambda)^3}\mathbf 1_{\eta> \frac{\lambda^2}{n^2}}
\le \frac{W^\tau}{N^2} \left(\frac{1}{N\eta}+W^{-\e/2-5\fc/2}\right) ,
\ee
for any constant $\tau>0$. Here, in the second step we used the definition of $\ell(\eta)$ in \eqref{eq:elleta}; in the third step, we also used the fact that $W^{-\e}\eta_0 =\lambda W^{d/2-\e}/N \ll \lambda^2/n^2$ for $d=2$. To bound $I_2$, we apply the estimates \eqref{prop:ThfadC_short} and \eqref{prop:BD1} to the 2-$\cK$ loop in \eqref{Kn2sol}. More precisely, we can choose the flow as described in \Cref{zztE} (for the block Anderson model) or \Cref{rmk:ext_Wegner} (for the Wegner orbital model) and apply these estimates at time $t_0$ with $1-t_0\asymp \eta$ and $\ell_{t_0}\asymp \ell(\eta)$. By \eqref{prop:ThfadC_short}, we have that 
\be\label{eq:samesig}\eta^2 \max_{[a],[b],[b']}\max_{\sig_1=\sig_2} \left|\cK^{(2)}_{t_0,(\sig_1,\sig_2),([a],[b])}-\cK^{(2)}_{t_0,(\sig_1,\sig_2),([a],[b'])}\right| \lesssim W^{-d}\eta^2 \le \frac{W^{-2\e}}{N^2}.\ee
For the case $\sig_1\ne \sig_2$, with \eqref{prop:BD1}, we can get the following bounds: when $d=2$,
\be\label{eq:diffsig1}\eta^2 \max_{[a],[b],[b']}\max_{\sig_1\ne \sig_2} \left|\cK^{(2)}_{t_0,(\sig_1,\sig_2),([a],[b])}-\cK^{(2)}_{t_0,(\sig_1,\sig_2),([a],[b'])}\right| \prec \frac{\eta^2}{W^{d}\lambda^{2}}\le \frac{W^{-2\e}}{N^2};
\ee
when $d=1$, we have that 
\be\label{eq:diffsig2}
\eta^2 \max_{[a],[b],[b']}\max_{\sig_1\ne \sig_2} \left|\cK^{(2)}_{t_0,(\sig_1,\sig_2),([a],[b])}-\cK^{(2)}_{t_0,(\sig_1,\sig_2),([a],[b'])}\right| \prec \frac{\eta^2 n}{W\lambda^{2}}\lesssim \frac{W^{-2\e}}{N^2}.
\ee
Combining \eqref{eq:I1}--\eqref{eq:diffsig2}, we get that $ I_1+I_2 \prec (W^{-2\e}+ W^{-\fd\wedge \fc+\e})/N^2$ since $\tau$ can be arbitrarily small. Together with \eqref{ssfa2}, 
it gives that 
\begin{align*}
    \E \sum_{\lambda_i,\lambda_j\in {\cal I}_E} \bigg|\frac{N}{W^{d}}\sum_{x\in[a]}\overline \bu_i(x)\bu_j(x) -\delta_{ij} \bigg|^2 &= \E \sum_{\lambda_i,\lambda_j\in {\cal I}_E} \left|N\cdot \bu_i^*\left(E_{[a]}-N^{-1}\right) \bu_j\right|^2  \prec W^{-2\e}+ W^{-\fd\wedge \fc+\e} .
\end{align*} 
Applying Markov's inequality and taking $c$ sufficiently small depending on $\e$, $\fd$, $\fc$, we conclude \eqref{Meq:QUE}. The proof of \eqref{Meq:QUE2} follows similarly---we simply replace $E_{[a]}$ in \eqref{ssfa2} with $|A|^{-1}\sum_{[a]\in A} E_{[a]}$, after which all subsequent arguments remain valid. 
\end{proof}

\section{Properties of primitive loops}\label{sec:prim}

All the results in this section, except those in \Cref{subsec:estimates}, apply to a broader class of models beyond the RBSOs considered in this paper. In particular, we will not impose any block structure on this general model. This approach is chosen to ensure that the results developed here can be applied in future studies of other random matrices, particularly those without block structures. More precisely, we will analyze the primitive loops defined in terms of the primitive equation presented in \Cref{general-primitive-equation} below, with initial conditions expressed in terms of the $M$-loops of a matrix $M$ that satisfies the following assumption.

\begin{assumption}\label{assm:general}
We consider an $N\times N$ matrix $M\in \C^N$ satisfying the following $t$-dependent equation:
\be\label{eq:Mt}
M  = \frac{1}{\Lambda - z_t - t\cal S(M)},\quad \forall t\ge 0.
\ee
Here, $\Lambda$ is an arbitrary (deterministic) Hermitian matrix, $z_t\in \C_+$ is an arbitrary deterministic flow, and $\cal S(\cdot)$ is a linear operator defined as in \eqref{eq:opS}. Specifically, $\cal S(M)$ is a diagonal matrix with entries \smash{$\cal S(M)_{ii}=\sum_{j}S_{ij}M_{jj}$}, where $S$ is an arbitrary symmetric matrix (which may not necessarily represent any variance profile). In equation \eqref{eq:Mt}, we implicitly assume that $M$ and $z_t + t\cal S(M)$ do not depend on $t$. Additionally, we do not impose the constraint $t \le 1$.
\end{assumption}

The tree representation formula for primitive loops and Ward's identity, stated in Lemmas \ref{tree-representation} and \ref{ward-identity-gamma}, respectively, rely only on the assumption above.  
In contrast, the proofs of the sum-zero property and the upper bounds on $\cal{K}$-loops, given by Lemmas \ref{sum-zero} and \ref{ML:Kbound}, respectively, require the more specific assumptions for RBSOs. This is because they rely on the estimates in Lemmas \ref{lem:propM} and \ref{lem_propTH}. 
However, we emphasize that these proofs can be extended to more general models, provided that the corresponding estimates for the matrices appearing in the tree representation formula are established.

\subsection{Canonical partitions with $M$-loops}\label{subsec:partition}
The general case will build upon the concept of \emph{canonical partitions of polygons} introduced in \cite[Definition 3.1]{Band1D}. Essentially, canonical partitions of an oriented polygon $\mathcal{P}_{\ba}$ are partitions in which each edge of the polygon is in one-to-one correspondence with each region in the partition. 

\begin{definition}[Canonical partitions]\label{def:canpnical_part}
Fix $\fn\ge 3$ and let $\cal P_{\ba}$ be an oriented polygon with vertices $\ba=(a_1,a_2, \ldots ,a_\fn)$ arranged in counterclockwise order. We adopt the cyclic convention that $a_{i}=a_{j}$ if and only if $i=j\mod \fn$. Then, we use $(a_{k-1},a_k)$ to denote the $k$-th side of $\cal P_{\ba}$. A planar partition of the polygonal domain enclosed by $\cal P_{\ba}$ is called {\bf canonical} if the following properties hold:
\begin{itemize}
\item Every sub-region in the partition is also a polygonal domain. 
\item There is a one-to-one correspondence between the edges of the polygon and the sub-regions, such that each side $(a_{k-1},a_k)$ belongs to exactly one sub-region, and each sub-region contains exactly one side of $\cal P_{\ba}$. Denote the sub-region containing $(a_{k-1},a_k)$ by $R_k$. 

\item Every vertex $a_k$ of $\cal P_{\ba}$ belongs to exactly two regions, $R_k$ and $R_{k+1}$ (with the convention $R_{\fn+1}=R_1$). 
  
\end{itemize}
Given a canonical partition, removing the $\fn$ sides of the polygon ${\cal P}_{\ba}$ leaves a collection of interior edges that form a tree. In this tree, the leaves correspond to the vertices of ${\cal P}_{\ba}$, and each internal vertex has degree at least three. 
Following \cite{Band1D}, we define the equivalence classes of all such trees under graph isomorphism, and denote the collection of equivalence classes by $\TSP({\cal P}_{\ba})$. Subsequently, we will consider each element of $\TSP({\cal P}_{\ba})$ as an abstract tree structure,  and call it a \emph{canonical tree partition}.  
\end{definition}

We now introduce several notations. Given two vertices $b_1$ and $ b_2$ in a graph, we will use $\p{b_1, b_2}$ to represent an \emph{undirected edge} and $\p{b_1 \to b_2}$ to represent a \emph{directed edge}. In particular, we identify $\p{b_1, b_2}$ with $\p{b_2, b_1}$, but do not identify $\p{b_1 \to b_2}$ with $\p{b_2 \to b_1}$. We may also refer to a directed edge $\p{b_1 \to b_2}$ with the undirected notation $\p{b_1, b_2}$ when the direction is not relevant.
Typically, we will use the letter $a$ to refer to the vertices of the polygon $\mathcal{P}_{\ba}$, which we call \emph{external vertices}, while $b$ and $c$ generally denote the remaining interior vertices, referred to as \emph{internal vertices}. 
We will call an edge containing exactly one external vertex as an \emph{external edge}, and an edge between internal vertices as an \emph{internal edge}. 
In the definition above, we refer to both the external and internal edges as \emph{interior} edges, in contrast to the \emph{boundary edges} which denote the sides of $\mathcal{P}_{\ba}$. 
In the following proof, for clarity of presentation, it is useful to view $\TSP({\cal P}_{\ba})$ also as a specific partition of the polygon $\cal P_{\ba}$, which includes sub-regions as well as internal, external, and boundary edges. 
Moreover, we will draw boundary edges as \emph{dashed lines}, since they are irrelevant to the graph values as we will see.

Next, we define the edges (i.e., matrices) that will appear in the tree representation for primitive loops. 
Since we do not assume any block structure in this section, we will use slightly different notations to avoid confusion with the notations used elsewhere in the paper.

\begin{definition}[$\Theta$-propagator and charge symmetry]\label{def:propagator-entrywise}
For any $a\in\qqq{N}$, let $F_a$ be the matrix with only one non-zero entry: $\p{F_a}_{ij} = \mathbf 1(i = j = a)$. For $\sigma, \sigma' \in \set{+, -}$ and $a, b \in \bbr{N}$, we define
  \begin{equation}\label{M-definition}
    \MMgen^{(\sigma,\sigma')}_{ab}
      \coloneqq \avg{M\p{\sigma}F_a M\p{\sigma'}F_b} = M(\sig)_{ba}M(\sig')_{ab},
  \end{equation}
  which corresponds to \eqref{eq:Msig}. Then, similar to \eqref{def_Thxi}, we define the $\Theta$-propagator as
  \begin{equation}\label{Theta-general-propagator}
    \Thetagen_t^{(\sigma,\sigma')}
      \coloneqq \big(I - \MMgen^{(\sigma,\sigma')} S_t\big)^{-1},
  \end{equation}
  where we abbreviate $S_t\equiv tS$.
  Note that $\MMgen$ and $\Thetagen_t$ are generally neither symmetric nor Hermitian. However, they satisfy the following identities as in \eqref{eq:Msym} by definitions and the fact that \smash{$S_t$} is symmetric: 
  \begin{equation}\label{Theta-almost-symmetry}
(\MMgen^{(\sigma,\sigma')})^\top = \MMgen^{(\sigma',\sigma)}, \ \ \big(S_t\Thetagen_t^{(\sigma,\sigma')}\big)^\top = S_t\Thetagen_t^{(\sigma',\sigma)},\  \ \big(\Thetagen_t^{(\sigma,\sigma')}\MMgen^{\p{\sigma,\sigma'}})^\top = \Thetagen_t^{(\sigma',\sigma)} \MMgen^{(\sigma',\sigma)}.       
\end{equation}
  We will refer to these identities as \textbf{charge symmetry}.
\end{definition}


The concept of charge symmetry motivates the definition of vertex charges in our graphs (see \Cref{vertex-charge} below). Indeed, if we view \smash{$\MMgen$ and $S_t\Thetagen_t$} as matrices with indices $\p{a, \sigma} \in \bbr{N} \times \set{+, -}$, then they become true symmetric matrices.

We now introduce an extension of \Cref{def:canpnical_part} in the following definition, which incorporates loops consisting of edges labeled by $M$. As the name suggests, these edges will correspond to entries of $M$, whereas the remaining non-boundary edges in our graphs will be unlabeled and correspond to entries of $\Thetagen_t$ or $S_t\Thetagen_t$.

\begin{definition}[Canonical partitions with $M$-loops]\label{m-loop-tsp}

  Let $\Gamma_{\ba} \in \TSP\p{\mathcal{P}_{\ba}}$ be a canonical tree partition of the oriented polygon $\mathcal{P}_{\ba}$, and denote the internal vertices of $\Gamma_{\ba}$ by $\mathbf b = \p{b_1, \ldots, b_\fm}$. We define the graph {$\Gamma^{\p{M}}_{\ba}$} by replacing each $b_i$ with an $M$-loop in the following way.
\begin{enumerate}[label=\arabic*.]
\item 
 Consider the subgraph $(V^{\p{b_i}}, E^{\p{b_i}})$ of all vertices in $\Gamma_{\ba}$ connected to $b_i$. More precisely, we let
      \begin{equation}\label{b_i-subgraph}
        V^{\p{b_i}} \coloneqq \set{b_i,c_1, \ldots, c_{\delta_i}},
        \quad \text{and}\quad E^{\p{b_i}} = \set{\p{c_1, b_i}, \ldots, \p{c_{\delta_i}, b_i}}
      \end{equation}
      be the subsets of all vertices (including $b_i$) and edges connected to $b_i$, respectively. Reordering the $c_k$'s if necessary, we can ensure that $\p{c_1, \ldots, c_{\delta_i}}$ form a polygon without any crossing edges and with vertices arranged in counterclockwise order.
      
\item Construct a new graph $(\widetilde{V}^{\p{b_i}}, \widetilde{E}^{\p{b_i}})$ with vertices $\widetilde{V}^{\p{b_i}} = \set{b_{i,1}, \ldots, b_{i,\delta_i}, c_1, \ldots, c_{\delta_i}}$ and edges
      \[
        \widetilde{E}^{\p{b_i}}
          \coloneqq \set{\p{c_j, b_{i,j}} \mid 1 \leq j \leq \delta_i} \cup \set{\p{b_{i,1}, b_{i,2}; M}, \p{b_{i,2}, b_{i,3}; M}, \ldots, \p{b_{i,\delta_i}, b_{i,1}; M}}.
      \]
      Here, the notation $e = \p{e_i, e_f; M}$ refers to an edge labeled with $M$. We will call such an edge an $M$-edge. For convenience, we denote
      \begin{equation}\label{b_i-cycle}
        C^{\p{b_i}} \coloneqq \p{b_{i,1}, b_{i,2}, \ldots, b_{i,\delta_i}},
      \end{equation}
      which is a polygon obtained by ``splitting" $b_i$ into $\delta_i$ new vertices and connecting them by $M$-edges. 
      Relabeling the $b_{i,j}$'s if needed, we ensure that $\p{b_{i,1} \to b_{i,2} \to \cdots \to b_{i,\delta_i}}$ are also arranged in counterclockwise order. 
    
\item Lastly, we replace the subgraph $(V^{\p{b_i}}, E^{\p{b_i}})$ with $(\widetilde{V}^{\p{b_i}}, \widetilde{E}^{\p{i}})$.
 \end{enumerate}
  In \Cref{Fig:Mloop}, we illustrate the above procedure for an example with $\delta_i=5$. 
  \begin{figure}[h]\label{b_i-loop-replacement}
    \centering
    \begin{tikzpicture}
      \coordinate (bi) at (0, 0);
      
      \fill (bi) circle (1pt) node[left=4pt]{$b_i$};
      \foreach \i/\t in {1/72, 2/144, 3/216, 4/288, 5/360} {
        \fill (\t:2) circle (1pt);
        \draw (bi) -- (\t:2);
        \node at (\t+72:2.3) {$c_{\i}$};
      }
    \end{tikzpicture}
    $\quad$
    \begin{tikzpicture}
      \node at (-3, 0) {$\to$};
      \node at (0, 0) {$M$};
      \foreach \i/\t in {1/72, 2/144, 3/216, 4/288, 5/360} {
        \fill (\t:2) circle (1pt);
        \fill (\t:1) circle (1pt);
        \draw (\t:1) -- (\t:2);
        \draw (\t:1) -- (\t+72:1);
        \node at (\t+72:2.3) {$c_{\i}$};
        \node at (\t+92:1.3) {$b_{i,\i}$};
      }
    \end{tikzpicture}
    \caption{Replacing $b_i$ with an $M$-loop with 5 sides.}\label{Fig:Mloop}
  \end{figure}

Repeating these steps for each internal vertex $b_i$, $i\in\qqq{\fm}$, we get a graph \smash{$\Gamma^{\p{M}}_{\ba}$}, which we will refer to as the \emph{$M$-graph corresponding to $\Gamma_{\ba}$}. Note that the order in which we replace $b_i$'s by $M$-loops does not matter, 
and every boundary edge $\p{a_{k-1}, a_k}$ still belongs to exactly one polygonal region in the $M$-graph \smash{$\Gamma^{\p{M}}_{\ba}$}. 
With a slight abuse of notation, we will use $R_k$ to denote the sub-region containing $\p{a_{k-1}, a_k}$ in both $\Gamma_{\ba}$ and \smash{$\Gamma^{\p{M}}_{\ba}$}; see \Cref{Fig:Mloop2} for an example.
\begin{figure}[h]\label{m-graph-region}
    \centering
    \begin{tikzpicture}
      \fill (-1, 0) circle (1pt);
      \fill (1, 0) circle (1pt);
      \fill (0, 2) circle (1pt);
      
      \draw (-1, 0) node[below]{$a_{k-1}$}
        [dashed]-- (1, 0) node[below]{$a_k$}; 
    \draw (1, 0)
        -- (0, 2) node[above]{$b_i$}
        -- (-1, 0);
      \node at (0, 0.75) {$R_k$};
      \node at (0, -1) {$\Gamma_{\ba}$};
    \end{tikzpicture}
    \hspace{1in}
    \begin{tikzpicture}
      \fill (-1, 0) circle (1pt);
      \fill (1, 0) circle (1pt);
      \fill (0.5, 2) circle (1pt);
      \fill (-0.5, 2) circle (1pt);
      
      \draw (-1, 0) node[below]{$a_{k-1}$}
        [dashed]-- (1, 0) node[below]{$a_k$};
    \draw (1,0)
        -- (0.5, 2)  node[above right]{$b_{i,j}$}
        -- (-0.5, 2) node[above left]{$b_{i,j-1}$}
                     node[pos=0.5, above]{$M$}
        -- (-1,0);
      \node at (0, 0.75) {$R_k$};
      \node at (0, -1) {$\Gamma^{\p{M}}_{\ba}$};
    \end{tikzpicture}
    \caption{$R_k$ refers to the region containing $\protect\p{a_{k-1}, a_k}$ in both $\Gamma_{\ba}$ and its $M$-graph.} 
    \label{Fig:Mloop2}
  \end{figure}
\end{definition}

We refer readers to \Cref{example} for an example of a canonical tree partition $\Gamma_{\ba}$ and its $M$-graph $\Gamma_{\ba}^{(M)}$. By definition, it is easy to see the following properties for $M$-graphs.
  \begin{figure}[h]
    \centering
     \scalebox{0.9}{
    \begin{tikzpicture}
      \coordinate (b1) at (-1, 0);
      \coordinate (b2) at (1, 0);
      
      \fill (b1) circle (1pt);
      \fill (b2) circle (1pt);

      \foreach \i/\t in {1/60, 2/120, 3/180, 4/240, 5/300, 6/360} {
        \coordinate (a\i) at (\t+90:2.5);
        \fill (a\i) circle (1pt);
        \draw (a\i) [dashed]-- (\t+150:2.5);
        \node at (\t+90:2.8) {$a_{\i}$};
        \node at (\t+60:2.5) {$\sigma_{\i}$};
      }
      
      \draw (a6) -- (b1);
      \draw (a1) -- (b1);
      \draw (a2) -- (b1);
      \draw (a3) -- (b1);
      
      \draw (a4) -- (b2);
      \draw (a5) -- (b2);
      
      \draw (b1) -- (b2);

      \node at (0, -3.5) {$\Gamma_{\ba}$};
    \end{tikzpicture}
    \hspace{1in}
    \begin{tikzpicture}
      \foreach \i/\t in {1/60, 2/120, 3/180, 4/240, 5/300, 6/360} {
        \coordinate (a\i) at (\t+90:2.5);
        \fill (a\i) circle (1pt);
        \draw (a\i) [dashed]-- (\t+150:2.5);
        \node at (\t+90:2.8) {$a_{\i}$};
        \node at (\t+60:2.5) {$\sigma_{\i}$};
      }
      
      \coordinate (c11) at ($(b1)!0.45!(a6)$);
      \coordinate (c12) at ($(b1)!0.45!(a1)$);
      \coordinate (c13) at ($(b1)!0.45!(a2)$);
      \coordinate (c14) at ($(b1)!0.45!(a3)$);
      \coordinate (c15) at ($(b1)!0.45!(b2)$);
      
      \coordinate (c21) at ($(b2)!0.25!(b1)$);
      \coordinate (c22) at ($(b2)!0.45!(a4)$);
      \coordinate (c23) at ($(b2)!0.45!(a5)$);
      
      \foreach \i/\nc in {1/5, 2/3} {
        \foreach \j in {1, 2, ..., \nc} {
          \fill (c\i\j) circle (1pt);
        }
      }
      
      \foreach \i/\nc in {1/5, 2/3} {
        \draw (c\i1)
        \foreach \j in {1, ..., \nc} {
           -- (c\i\j)
        }
          -- cycle;
      }
      
      \draw (c11) -- (a6);
      \draw (c12) -- (a1);
      \draw (c13) -- (a2);
      \draw (c14) -- (a3);
      
      \draw (c21) -- (c15);
      
      \draw (c22) -- (a4);
      \draw (c23) -- (a5);
      
      \node at ($0.2*(c11) + 0.2*(c12) + 0.2*(c13) + 0.2*(c14) + 0.2*(c15)$) {$M$};
      \node at ($0.333*(c21) + 0.333*(c22) + 0.333*(c23)$) {$M$};
      
      \node at (0, -3.5) {$\Gamma^{\p{M}}_{\ba}$};
    \end{tikzpicture}}
    \caption{Example of $\Gamma_{\ba} \in \TSP\protect\p{\mathcal{P}_{\ba}}$ and its corresponding $M$-graph.}\label{example}
  \end{figure}
  
\begin{claim}[$3$-connectedness of $M$-graphs]\label{unique-interior-edge}
  Let $\Gamma_{\ba} \in \TSP\p{\mathcal{P}_{\ba}}$ and let $\Gamma^{\p{M}}_{\ba}$ be its corresponding $M$-graph as in Definition \ref{m-loop-tsp}. Then, every vertex $c$ is connected to exactly three edges in \smash{$\Gamma^{\p{M}}_{\ba}$}:
  \begin{enumerate}
    \item if $c$ is an internal vertex, then $c$ is connected to exactly $2$ $M$-edges; 
    \item if $c$ is an external vertex, then $c$ is connected to exactly $2$ boundary edges;
    \item every vertex is connected to exactly one unlabeled interior edge (i.e., an internal edge between $M$-loops or an external edge between an external vertex and an $M$-loop). 
  \end{enumerate}
\end{claim}

Each vertex will inherit a charge from one of the boundary edges. However, each vertex is adjacent to two regions. To have a consistent way of fixing charges for all vertices, we introduce the following notation. 

\begin{definition}[Canonical directed edge]\label{canonical-directed-edge}
Let $\Gamma_{\ba} \in \TSP\p{\mathcal{P}_{\ba}}$ and let $\Gamma^{\p{M}}_{\ba}$ be its corresponding $M$-graph as in Definition \ref{m-loop-tsp}. 
For each $k \in \qqq{\fn}$, the region $R_k$ is a polygon and hence has exactly two paths from $a_{k-1}$ to $a_k$; one is $\p{a_{k-1} \to a_k}$, and we label the other as $p_k \coloneqq \p{d_{k,1} \to d_{k,2} \to \cdots \to d_{k,\ell_k}}$ with $d_{k,1} = a_{k-1}$ and $d_{k,\ell_k} = a_k$. We call it a \emph{canonical directed path}.

Now, for each vertex $c \in \Gamma^{\p{M}}_{\ba}$, \Cref{unique-interior-edge} guarantees that there is a unique {\bf unlabeled interior edge} (i.e., a non-$M$ and non-boundary edge) $\p{c_0, c}$ connected to $c$. Then, there is a unique pair $k \in \qqq{\fn}$ and $2 \leq j \leq \ell_k$ such that $\p{c_0 \to c} = \p{d_{k,j-1} \to d_{k,j}}$. We call $\p{d_{k,j-1} \to d_{k,j}}$ a \emph{canonical directed edge}. 

We refer readers to \Cref{Fig:charge-example} for an example of the above definitions: the canonical directed path in $R_6$ is $(a_5\to b_{2,3}\to b_{2,1}\to b_{1,5}\to b_{1,1}\to a_6)$, which leads to the canonical directed edges $\p{a_5 \to b_{2,3}}$, $\p{b_{2,1}\to b_{1,5}}$, and $\p{b_{1,1} \to a_6}$ for the vertices $b_{2,3}$, $b_{1,5}$, and $a_6$, respectively. 
\end{definition}


For each vertex $c$, with its canonical directed edge selected as in \Cref{canonical-directed-edge}, we can define its principal region and its charge. 

\begin{definition}[Charges of vertices]\label{vertex-charge}
  For each vertex $c \in \Gamma^{\p{M}}_{\ba}$, let $\p{k, j}$ be the unique pair that labels its canonical directed edge \(\p{d_{k,j-1} \to d_{k,j}}\) from \Cref{canonical-directed-edge}. (For example, we have $(k,j)=(6,4)$ for the vertex $b_{1,5}$ in \Cref{Fig:charge-example}.) 
  Then, we call $R_k$ the \emph{principal region of $c$} and denote $\mathcal{P}R\p{c} := k$. 
  We then define the \textbf{charge} $q\p{c} \in \set{+, -}$ of $c$ by $q\p{c} := \sigma_{\mathcal{P}R\p{c}}$, i.e., the charge of its principal region.
\end{definition}

\begin{example}\label{charge-example}
  \begin{figure}[h]
    \centering
    \scalebox{0.8}{
    \begin{tikzpicture}
      \foreach \i/\t/\co/\charge in {1/60/red/+, 2/120/blue/-, 3/180/red/+, 4/240/blue/-, 5/300/red/+, 6/360/blue/-} {
        \coordinate (a\i) at (\t+90:3.5);
        \fill (a\i) circle (1pt);
        \draw[
    \co, arrows={->[scale=1.5]},
    shorten >= 2pt
        ] (a\i) [dashed]-- (\t+150:3.5);
        \node at (\t+90:3.8) {$a_{\i}$};
        \node at (\t+60:3.5) {$\charge$};
      }
      
      \coordinate (c11) at ($(b1)!0.45!(a6)$);
      \coordinate (c12) at ($(b1)!0.45!(a1)$);
      \coordinate (c13) at ($(b1)!0.45!(a2)$);
      \coordinate (c14) at ($(b1)!0.45!(a3)$);
      \coordinate (c15) at ($(b1)!0.45!(b2)$);
      
      \coordinate (c21) at ($(b2)!0.25!(b1)$);
      \coordinate (c22) at ($(b2)!0.45!(a4)$);
      \coordinate (c23) at ($(b2)!0.45!(a5)$);
      
      \foreach \i/\nc in {1/5, 2/3} {
        \foreach \j in {1, 2, ..., \nc} {
          \fill (c\i\j) circle (1pt);
        }
      }
      
      \foreach \i/\nc in {1/5, 2/3} {
        \draw (c\i1)
        \foreach \j in {1, ..., \nc} {
           -- (c\i\j)
        }
          -- cycle;
      }
      
      \draw (c11) node[xshift=14pt]{$b_{1,1}$} -- (a6);
      \draw (c12) node[above=6pt]{$b_{1,2}$} -- (a1);
      \draw (c13) node[xshift=2pt, yshift=-14pt]{$b_{1,3}$} -- (a2);
      \draw (c14) node[xshift=14pt]{$b_{1,4}$} -- (a3);
      
      \draw (c21) node[yshift=12pt]{$b_{2,1}$} -- (c15) node[xshift=8pt, yshift=-12pt]{$b_{1,5}$};
      
      \draw (c22) node[below left]{$b_{2,2}$} -- (a4);
      \draw (c23) node[xshift=12pt, yshift=-8pt]{$b_{2,3}$} -- (a5);
      
      \node at ($0.2*(c11) + 0.2*(c12) + 0.2*(c13) + 0.2*(c14) + 0.2*(c15)$) {$M$};
      \node at ($0.333*(c21) + 0.333*(c22) + 0.333*(c23)$) {$M$};
    \end{tikzpicture}
    \hspace{1in}
    \begin{tikzpicture}
      \fill (a5) circle (1pt);
      \fill (a6) circle (1pt);
      
      \node at (420:3.5) {$-$};
      \node at (390:3.8) {$a_5$};
      \node at (90:3.8) {$a_6$};
      
      \fill (c11) circle (1pt);
      \fill (c15) circle (1pt);
      \fill (c21) circle (1pt);
      \fill (c23) circle (1pt);
      
      \draw[red, arrows={->[scale=1.5]}, shorten >= 2pt] (a5) [dashed]-- (a6);
      \draw[red, arrows={->[scale=1.2]},
    shorten >= 2pt] (a5) [thick]-- (c23) node[below right, black]{$b_{2,3}$};
      \draw[arrows={->[scale=1.2]},
    shorten >= 2pt] (c23) [thick]-- (c21) node[pos=0.5, xshift=6pt, yshift=-8pt, black]{$M$};
      \draw[red, arrows={->[scale=1.2]},
    shorten >= 2pt] (c21)  node[xshift=6pt, yshift=-10pt, black]{$b_{2,1}$} [thick]-- (c15);
      \draw[arrows={->[scale=1.2]},
    shorten >= 2pt] (c15) node[xshift=-6pt, yshift=-10pt, black]{$b_{1,5}$} [thick]-- (c11) node[pos=0.5, xshift=-10pt, yshift=-5pt, black]{$M$};
      \draw[red, arrows={->[scale=1.2]},
    shorten >= 2pt] (c11) node[left, black]{$b_{1,1}$} [thick]-- (a6);
      
      \node at (65:1.75) {$R_6$};
    \end{tikzpicture}}
    \caption{Example of an $M$-graph with an alternating boundary. The right picture depicts the principal region for $b_{2,3}$, $b_{1,5}$, and $a_6$, which
    assigns the charges $q\protect\p{b_{2,3}} = q\protect\p{b_{1,5}} = q\protect\p{a_6} = -$. Here, the corresponding canonical directed edges are highlighted in red. 
    }\label{Fig:charge-example}
  \end{figure}
  In Figure \ref{Fig:charge-example}, the external vertices have charges
  \begin{gather*}
    q\p{a_1} = q\p{a_3} = q\p{a_5} = +, \quad
    q\p{a_2} = q\p{a_4} = q\p{a_6} = -,
  \end{gather*}
  and the internal vertices have charges
  \begin{gather*}
    q\p{b_{1,1}} = q\p{b_{1,3}} = q\p{b_{2,2}} = +, \quad
    q\p{b_{1,2}} = q\p{b_{1,4}} = q\p{b_{1,5}} = q\p{b_{2,1}} = q\p{b_{2,3}}  = -.
  \end{gather*}
  As illustrated in the right picture, the vertices $b_{2,3}$, $b_{1,5}$, and $a_6$ all have principal region $R_6$, and hence have charge $-$. To see two more examples of the charges: for the vertex $b_{1,1}$, the canonical directed edge is $\p{a_6 \to b_{1,1}}$ in $R_1$, which gives ${\cal P}R(b_{1,1})=R_1$ and the charge $q\p{b_{1,1}} = +$; for the vertex $b_{2,1}$, the canonical directed edge is $\p{b_{1,5} \to b_{2,1}}$ in $R_4$, which gives ${\cal P}R(b_{2,1})=R_4$ and $q\p{b_{2,1}} = -$.
\end{example}

With the charges of all vertices fixed, we can now assign values to the unlabeled interior edges in \smash{$\Gamma^{\p{M}}_{\ba}$}. 
\begin{definition}[Values of unlabeled edges]\label{edge-value-definition}
 Let $\Gamma_{\ba} \in \TSP\p{\mathcal{P}_{\ba}}$ and let $\Gamma^{\p{M}}_{\ba}$ be its corresponding $M$-graph as in Definition \ref{m-loop-tsp}. Set
  \begin{equation}\label{unlabeled-edge-set}
    \mathcal{E}_0\p{\Gamma^{\p{M}}_{\ba}}
      \coloneqq \set{e \suchthat e\text{ is an unlabeled, non-boundary edge in }\Gamma^{\p{M}}_{\ba}}.
  \end{equation}
Given an arbitrary $\bsig\in\{+,-\}^\fn$, we assign charges to the vertices of \smash{$\Gamma^{\p{M}}_{\ba}$} according to the rule described in \Cref{vertex-charge}. Then, we define the value function $ f_{\bsig} \colon \mathcal{E}_0(\Gamma^{\p{M}}_{\ba}) \to \C$ in the following way.
  \begin{enumerate}
    \item If $e = \p{a_k, b}$, i.e., $e$ is an external edge, then 
      \begin{equation}\label{f-external}
        f_{\bsig}\p{e} \coloneqq \Thetagen^{\p{q\p{a_k}, q\p{b}}}_{t,a_kb}.
      \end{equation}
    \item If $e = \p{b_1, b_2}$, i.e., $e$ is an internal edge, then 
      \begin{equation}\label{f-internal}
        f_{\bsig}\p{e}
          \coloneqq \p{S_t\Thetagen_t^{\p{q\p{b_1}, q\p{b_2}}}}_{b_1b_2}
          = \p{S_t\Thetagen_t^{\p{q\p{b_2}, q\p{b_1}}}}_{b_2b_1},
      \end{equation}
      where the second equality is due to the charge symmetry (\ref{Theta-almost-symmetry}).
  \end{enumerate}
  We emphasize that $\Thetagen_t$ (defined under  \Cref{assm:general}) need not satisfy the charge symmetry, so the order of the lower indices in (\ref{f-external}) is important.
\end{definition}

\begin{definition}[$M$-loops]\label{M-loop-value-definition}
  In the setting of Definition \ref{m-loop-tsp}, for each $i \in \qqq{\fm}$, recall the notation $C^{\p{b_i}}$ from (\ref{b_i-cycle}). Then, we define the \emph{$M$-loop generated by $b_i$} as
  \begin{equation}\label{eq:Mloop_defgen}
    \cal F\p{b_i} \equiv \Mgen_{\bm{\sigma}{\p{b_i}}, \ba{\p{b_i}}}^{(\delta_i)}
      := \avg{\prod_{j=1}^{\delta_i} \left[M\p{\sigma_j{(b_i)}}F_{a_j(b_i)}\right]},
  \end{equation}
  where the parameters $\bm{\sigma}{\p{b_i}}\in\{+,-\}^{\delta_i}$ and $\ba{\p{b_i}}\in\qqq{N}^{\delta_i}$ are given as follows. 
  \begin{enumerate}
    \item For $j \in \qqq{\delta_i}$, take the edge $e_j{\p{b_i}} = \p{b_{i,j-1} \to b_{i,j}}$, with the cyclic convention that $b_{i,0} = b_{i,\delta_i}$. By the construction in Definition \ref{m-loop-tsp}, there is exactly one region $R_{k_j}$ that contains $e_j$, and we set ${\sigma}_j{\p{b_i}} = {\sigma}_{k_j}$. In other words, 
    $e_j{\p{b_i}}$ is an edge on the directed path $p_{k_j}$ as in Definition \ref{canonical-directed-edge} with the same charge as the region $R_{k_j}$.
    \item For $j \in \qqq{\delta_i}$, we set $a_j{\p{b_i}} = b_{i,j}$.
  \end{enumerate}
  In simpler terms, $\cal F\p{b_i}$ is an $M$-loop that shares the same orientation as  $\mathcal{P}_{\ba}$, with each edge having the same color as the region it lies in. See \Cref{Fig2:charge-example} for an example of how to assign charges and directions to the $M$-loops.
  \begin{figure}[h]
    \centering
     \scalebox{0.8}{
    \begin{tikzpicture}
      \foreach \i/\t/\co/\charge in {1/60/red/+, 2/120/blue/-, 3/180/red/+, 4/240/blue/-, 5/300/red/+, 6/360/blue/-} {
        \coordinate (a\i) at (\t+90:3.5);
        \fill (a\i) circle (1pt);
        \draw[\co, dashed,
        arrows={->[scale=1.5]},
    shorten >= 2pt
        ] (a\i) -- (\t+150:3.5);
        \node at (\t+90:3.8) {$a_{\i}$};
        \node at (\t+60:3.5) {$\charge$};
      }
      
      \coordinate (c11) at ($(b1)!0.45!(a6)$);
      \coordinate (c12) at ($(b1)!0.45!(a1)$);
      \coordinate (c13) at ($(b1)!0.45!(a2)$);
      \coordinate (c14) at ($(b1)!0.45!(a3)$);
      \coordinate (c15) at ($(b1)!0.45!(b2)$);
      
      \coordinate (c21) at ($(b2)!0.25!(b1)$);
      \coordinate (c22) at ($(b2)!0.45!(a4)$);
      \coordinate (c23) at ($(b2)!0.45!(a5)$);
      
      \foreach \i/\nc in {1/5, 2/3} {
        \foreach \j in {1, 2, ..., \nc} {
          \fill (c\i\j) circle (1pt);
        }
      }
      
      \foreach \i/\j/\co in {1/2/blue, 2/3/red, 3/4/blue, 4/5/red, 5/1/red} {
        \draw[\co, arrows={->[scale=1.5]},
    shorten >= 2pt] (c1\i) -- (c1\j);
      }
      
      \foreach \i/\j/\co in {1/2/red, 2/3/blue, 3/1/red} {
        \draw[\co, arrows={->[scale=1.5]},
    shorten >= 2pt] (c2\i) -- (c2\j);
      }
      
      \draw (c11) node[xshift=14pt]{$b_{1,1}$} -- (a6);
      \draw (c12) node[above=6pt]{$b_{1,2}$} -- (a1);
      \draw (c13) node[xshift=2pt, yshift=-14pt]{$b_{1,3}$} -- (a2);
      \draw (c14) node[xshift=14pt]{$b_{1,4}$} -- (a3);
      
      \draw (c21) node[yshift=12pt]{$b_{2,1}$} -- (c15) node[xshift=8pt, yshift=-12pt]{$b_{1,5}$};
      
      \draw (c22) node[below left]{$b_{2,2}$} -- (a4);
      \draw (c23) node[xshift=12pt, yshift=-8pt]{$b_{2,3}$} -- (a5);
      
      \node at ($0.2*(c11) + 0.2*(c12) + 0.2*(c13) + 0.2*(c14) + 0.2*(c15)$) {$M$};
      \node at ($0.333*(c21) + 0.333*(c22) + 0.333*(c23)$) {$M$};
    \end{tikzpicture}}
    \caption{Example of how to color and orient the $M$-loops.} \label{Fig2:charge-example}
  \end{figure}
\end{definition}

\begin{remark}
  From \Cref{unique-interior-edge}, it is not hard to see that for $j \in \qqq{\delta_i}$, we have $q\p{b_{i,j-1}} = \sigma_j{\p{b_i}}$. In other words, given an $M$-edge $e_j{\p{b_i}} = \p{b_{i,j-1} \to b_{i,j}}$, the charge of $b_{i,j-1}$ is the same as the charge of $e_j{\p{b_i}}$. This provides a more convenient, though perhaps less natural, method for assigning charges to internal vertices.
\end{remark}

With the above preparations, we are now ready to define the values of the $M$-graphs corresponding to the canonical partitions. 

\begin{definition}\label{M-graph-value-definition}
  In the setting of Definitions \ref{m-loop-tsp} and \ref{edge-value-definition}, 
  define
  \begin{equation}\label{M-graph-value-unsummed}
    \Gamma^{\p{M;\mathbf{b}}}_{t,\bsig,\ba}
      \coloneqq \p{\prod_{i=1}^\fm \cal F\p{b_i}} \cdot  \prod_{e \in \mathcal{E}_0(\Gamma^{\p{M}}_{\ba})} f_{\bsig}\p{e}
        .
  \end{equation}
Here, $f_{\bsig}\p{e}$ may depend on $t$, while $\cal F\p{b_i}$ does not. Then, we define
  \begin{equation}\label{M-graph-value}
    \Gamma^{\p{M}}_{t,\bm{\sigma},\ba}
      \coloneqq \sum_{i=1}^\fm \sum_{j=1}^{\delta_i} \sum_{b_{i,j} \in \qqq{N}} \Gamma^{\p{M;\mathbf{b}}}_{t,\bsig,\ba}.
  \end{equation}
Note $\Gamma^{\p{M}}_{\ba}$ represents a graph, while $\Gamma^{\p{M;\mathbf{b}}}_{t,\bsig,\ba}$ and $\Gamma^{\p{M}}_{t,\bm{\sigma},\ba}$ are complex functions of $t$, $\bm{\sigma}$, $\ba$, and $\bfb$. 
Moreover, the boundary edges are irrelevant to the graph values---one can think that each of them has a value of 1.
\end{definition}

\begin{example}\label{n=6,one-graph}
We show how to assign the value to the graph from Figure \ref{Fig2:charge-example}. Recall that the charges of all vertices have been fixed in \Cref{charge-example}.
  \begin{enumerate}
    \item Since $q\p{a_1} = +$ and $q\p{b_{1,2}} = -$, by \eqref{f-external}, we have 
      \(
        f_{\bsig}\p{\p{a_1, b_{1,2}}} = \Thetagen^{\p{+,-}}_{t,a_1b_{1,2}}.
      \)
    \item 
    Since $q\p{a_2} = -$ and $q\p{b_{1,3}} = +$, by \eqref{f-external}, we have 
      \(
        f_{\bsig}\p{\p{a_2, b_{1,3}}} = \Thetagen^{\p{-,+}}_{t,a_2b_{1,3}}.
      \)
    \item For the edges between $M$-loops, we have $q\p{b_{2,1}} = q\p{b_{1,5}} = -$, so by \eqref{f-internal}, we have 
      \[
        f_{\bsig}\p{\p{b_{2,1}, b_{1,5}}}
          = \p{S_t\Thetagen_t^{\p{-,-}}}_{b_{2,1}b_{1,5}}
          = \p{S_t\Thetagen_t^{\p{-,-}}}_{b_{1,5}b_{2,1}}.
      \]
  \end{enumerate}
  Assigning values to all edges in \Cref{Fig2:charge-example} as above, we finally get that 
  \begin{align*}
    &\Gamma^{\p{M;\mathbf{b}}}_{t,\bsig,\ba}
      = \Thetagen^{\p{+,-}}_{t,a_1b_{1,2}}
        \Thetagen^{\p{-,+}}_{t,a_2b_{1,3}}
        \Thetagen^{\p{+,-}}_{t,a_3b_{1,4}}
        \Thetagen^{\p{-,+}}_{t,a_4b_{2,2}}
        \Thetagen^{\p{+,-}}_{t,a_5b_{2,3}}
        \Thetagen^{\p{-,+}}_{t,a_6b_{1,1}} \\
        &\times
          \avg{M^*F_{b_{1,1}} MF_{b_{1,2}} M^*F_{b_{1,3}} MF_{b_{1,4}} M^*F_{b_{1,5}}}
          \p{S_t\Thetagen_t^{\p{-,-}}}_{b_{1,5}b_{2,1}}
          \avg{M^*F_{b_{2,1}} M^*F_{b_{2,2}} MF_{b_{2,3}}}.
  \end{align*}
\end{example}

\begin{figure}[h]
    \hspace{\stretch{1}}
     \scalebox{0.7}{
    \begin{tikzpicture}
      \coordinate (a_1) at (-2.7, 0.0);\fill (a_1) circle (1pt);\node at (-2.97, 0.0) {$a_1$};\coordinate (a_2) at (-0.0, -2.7);\fill (a_2) circle (1pt);\node at (-0.0, -2.97) {$a_2$};\coordinate (a_3) at (2.7, -0.0);\fill (a_3) circle (1pt);\node at (2.97, -0.0) {$a_3$};\coordinate (a_4) at (0.0, 2.7);\fill (a_4) circle (1pt);\node at (0.0, 2.97) {$a_4$};\coordinate (d_{9,1}) at (-1.238, -0.562);\fill (d_{9,1}) circle (1pt);\coordinate (d_{9,2}) at (-0.563, -1.238);\fill (d_{9,2}) circle (1pt);\coordinate (d_{5,3}) at (1.238, 0.562);\fill (d_{5,3}) circle (1pt);\coordinate (d_{5,4}) at (0.563, 1.238);\fill (d_{5,4}) circle (1pt);\coordinate (d_{5,9}) at (0.45, 0.45);\fill (d_{5,9}) circle (1pt);\coordinate (d_{9,5}) at (-0.45, -0.45);\fill (d_{9,5}) circle (1pt);\node at (-1.62, -1.62) {$\sigma_{2}$};\node at (1.62, -1.62) {$\sigma_{3}$};\node at (1.62, 1.62) {$\sigma_{4}$};\node at (-1.62, 1.62) {$\sigma_{1}$};\node at (-0.72, -0.72) {$M$};\node at (0.72, 0.72) {$M$};\draw (a_1) -- (d_{9,1});\draw (a_1) [dashed]-- (a_2);\draw (d_{9,1}) -- (d_{9,2});\draw (a_2) -- (d_{9,2});\draw (a_2) [dashed]-- (a_3);\draw (d_{9,2}) -- (d_{9,5});\draw (a_3) -- (d_{5,3});\draw (a_3) [dashed]-- (a_4);\draw (d_{5,3}) -- (d_{5,4});\draw (a_4) -- (d_{5,4});\draw (a_4) [dashed]-- (a_1);\draw (d_{5,4}) -- (d_{5,9});\draw (d_{5,9}) -- (d_{9,5});\draw (d_{5,9}) -- (d_{5,3});\draw (d_{9,5}) -- (d_{9,1});
          \node[above] at (d_{9,1}) {$b_{1,1}$};
          \node[right] at (d_{9,2}) {$b_{1,2}$};
          \node[below] at (d_{5,3}) {$b_{2,1}$};
          \node[left] at (d_{5,4}) {$b_{2,2}$};
          \node[xshift=-12pt, yshift=2pt] at (d_{5,9}) {$b_{2,3}$};
          \node[xshift=12pt, yshift=-4pt] at (d_{9,5}) {$b_{1,3}$};
    \end{tikzpicture}
    \hspace{\stretch{1}}
    \begin{tikzpicture}
      \coordinate (a_1) at (-2.7, 0.0);\fill (a_1) circle (1pt);\node at (-2.97, 0.0) {$a_1$};\coordinate (a_2) at (-0.0, -2.7);\fill (a_2) circle (1pt);\node at (-0.0, -2.97) {$a_2$};\coordinate (a_3) at (2.7, -0.0);\fill (a_3) circle (1pt);\node at (2.97, -0.0) {$a_3$};\coordinate (a_4) at (0.0, 2.7);\fill (a_4) circle (1pt);\node at (0.0, 2.97) {$a_4$};\coordinate (d_{8,1}) at (-1.237, 0.563);\fill (d_{8,1}) circle (1pt);\coordinate (d_{7,2}) at (0.562, -1.238);\fill (d_{7,2}) circle (1pt);\coordinate (d_{7,3}) at (1.238, -0.563);\fill (d_{7,3}) circle (1pt);\coordinate (d_{8,4}) at (-0.562, 1.238);\fill (d_{8,4}) circle (1pt);\coordinate (d_{8,7}) at (-0.45, 0.45);\fill (d_{8,7}) circle (1pt);\coordinate (d_{7,8}) at (0.45, -0.45);\fill (d_{7,8}) circle (1pt);\node at (-1.62, -1.62) {$\sigma_{2}$};\node at (1.62, -1.62) {$\sigma_{3}$};\node at (1.62, 1.62) {$\sigma_{4}$};\node at (-1.62, 1.62) {$\sigma_{1}$};\node at (-0.75, 0.75) {$M$};\node at (0.75, -0.75) {$M$};\draw (a_1) -- (d_{8,1});\draw (a_1) [dashed]-- (a_2);\draw (d_{8,1}) -- (d_{8,7});\draw (a_2) -- (d_{7,2});\draw (a_2) [dashed]-- (a_3);\draw (d_{7,2}) -- (d_{7,3});\draw (a_3) -- (d_{7,3});\draw (a_3) [dashed]-- (a_4);\draw (d_{7,3}) -- (d_{7,8});\draw (a_4) -- (d_{8,4});\draw (a_4) [dashed]-- (a_1);\draw (d_{8,4}) -- (d_{8,1});\draw (d_{8,7}) -- (d_{7,8});\draw (d_{8,7}) -- (d_{8,4});\draw (d_{7,8}) -- (d_{7,2});
          \node[xshift=2pt, yshift=-10pt] at (d_{8,1}) {$b_{1,1}$};
          \node[left] at (d_{7,2}) {$b_{2,1}$};
          \node[above] at (d_{7,3}) {$b_{2,2}$};
          \node[xshift=12pt] at (d_{8,4}) {$b_{1,3}$};
          \node[xshift=12pt] at (d_{8,7}) {$b_{1,2}$};
          \node[xshift=-12pt, yshift=-4pt] at (d_{7,8}) {$b_{2,3}$};
    \end{tikzpicture}
    \hspace{\stretch{1}}
    
    \hspace{\stretch{1}}
    \begin{tikzpicture}
      \coordinate (a_1) at (-2.7, 0.0);\fill (a_1) circle (1pt);\node at (-2.97, 0.0) {$a_1$};\coordinate (a_2) at (-0.0, -2.7);\fill (a_2) circle (1pt);\node at (-0.0, -2.97) {$a_2$};\coordinate (a_3) at (2.7, -0.0);\fill (a_3) circle (1pt);\node at (2.97, -0.0) {$a_3$};\coordinate (a_4) at (0.0, 2.7);\fill (a_4) circle (1pt);\node at (0.0, 2.97) {$a_4$};\coordinate (d_{6,1}) at (-0.9, -0.0);\fill (d_{6,1}) circle (1pt);\coordinate (d_{6,2}) at (-0.0, -0.9);\fill (d_{6,2}) circle (1pt);\coordinate (d_{6,3}) at (0.9, -0.0);\fill (d_{6,3}) circle (1pt);\coordinate (d_{6,4}) at (0.0, 0.9);\fill (d_{6,4}) circle (1pt);\node at (-1.62, -1.62) {$\sigma_{2}$};\node at (1.62, -1.62) {$\sigma_{3}$};\node at (1.62, 1.62) {$\sigma_{4}$};\node at (-1.62, 1.62) {$\sigma_{1}$};\node at (0.0, 0.0) {$M$};\draw (a_1) -- (d_{6,1});\draw (a_1) [dashed]-- (a_2);\draw (d_{6,1}) -- (d_{6,2});\draw (a_2) -- (d_{6,2});\draw (a_2) [dashed]-- (a_3);\draw (d_{6,2}) -- (d_{6,3});\draw (a_3) -- (d_{6,3});\draw (a_3) [dashed]-- (a_4);\draw (d_{6,3}) -- (d_{6,4});\draw (a_4) -- (d_{6,4});\draw (a_4) [dashed]-- (a_1);\draw (d_{6,4}) -- (d_{6,1});
          \node[xshift=-6pt, yshift=-8pt] at (d_{6,1}) {$b_{1,1}$};
          \node[xshift=12pt, yshift=-6pt] at (d_{6,2}) {$b_{1,2}$};
          \node[xshift=6pt, yshift=8pt] at (d_{6,3}) {$b_{1,3}$};
          \node[xshift=-12pt, yshift=6pt] at (d_{6,4}) {$b_{1,4}$};
    \end{tikzpicture}}
    \hspace{\stretch{1}}
    \caption{$M$-graphs for $\fn = 4$.}\label{n=4,all-graphs}
  \end{figure}

\begin{example}\label{example:n=4}
We now describe all $M$-graphs for the $\fn=4$ case. From Figure \ref{n=4,all-graphs}, we get the following values:
  \begin{align*}
    &\Thetagen^{\p{\sigma_1,\sigma_2}}_{t,a_1b_{1,1}}
      \Thetagen^{\p{\sigma_2,\sigma_3}}_{t,a_2b_{1,2}}
      \Thetagen^{\p{\sigma_3,\sigma_4}}_{t,a_3b_{2,1}}
      \Thetagen^{\p{\sigma_4,\sigma_1}}_{t,a_4b_{2,2}}
      \p{S_t\Thetagen_t^{\p{\sigma_1,\sigma_3}}}_{b_{1,3}b_{2,3}} \\
      &\hspace{0.5in} \times
        \avg{M\p{\sigma_1}F_{b_{1,1}} M\p{\sigma_2}F_{b_{1,2}} M\p{\sigma_3}F_{b_{1,3}}}
        \avg{M\p{\sigma_3}F_{b_{2,1}} M\p{\sigma_4}F_{b_{2,2}} M\p{\sigma_1}F_{b_{2,3}}}, \\
    &\Thetagen^{\p{\sigma_1,\sigma_2}}_{t,a_1b_{1,1}}
      \Thetagen^{\p{\sigma_2,\sigma_3}}_{t,a_2b_{2,1}}
      \Thetagen^{\p{\sigma_3,\sigma_4}}_{t,a_3b_{2,2}}
      \Thetagen^{\p{\sigma_4,\sigma_1}}_{t,a_4b_{1,3}}
      \p{S_t\Thetagen_t^{\p{\sigma_4,\sigma_2}}}_{b_{1,2}b_{2,3}} \\
      &\hspace{0.5in} \times
        \avg{M\p{\sigma_1}F_{b_{1,1}} M\p{\sigma_2}F_{b_{1,2}} M\p{\sigma_4}F_{b_{1,3}}}
        \avg{M\p{\sigma_2}F_{b_{2,1}} M\p{\sigma_3}F_{b_{2,2}} M\p{\sigma_4}F_{b_{2,3}}}, \\
    &\Thetagen^{\p{\sigma_1,\sigma_2}}_{t,a_1b_{1,1}}
      \Thetagen^{\p{\sigma_2,\sigma_3}}_{t,a_2b_{1,2}}
      \Thetagen^{\p{\sigma_3,\sigma_4}}_{t,a_3b_{1,3}}
      \Thetagen^{\p{\sigma_4,\sigma_1}}_{t,a_4b_{1,4}} \\
      &\hspace{0.5in} \times
        \avg{M\p{\sigma_1}F_{b_{1,1}} M\p{\sigma_2}F_{b_{1,2}} M\p{\sigma_3}F_{b_{1,3}} M\p{\sigma_4}F_{b_{1,4}}}.
  \end{align*}
\end{example}

\subsection{Tree representation}\label{sec_tree}

We now introduce the primitive equation for primitive loops under the more general \Cref{assm:general}. Later, in \Cref{subsec:estimates}, we will see that the equation (\ref{primitive-equation}) will reduce to (\ref{pro_dyncalK}) after introducing the block averages. 

\begin{definition}[Primitive loops]\label{general-primitive-equation}
We define the primitive loops of length $\fn\ge 2$ as the unique solution to the following system of \emph{primitive equations}:
  \begin{equation}\label{primitive-equation}
    \partial_t \Kgen_{t,\bm{\sigma},\ba}^{(\fn)}
      = \sum_{1 \leq k < \ell \leq \fn} \sum_{c, d}
        \p{(\mathcal{G}_L)^{\p{c}}_{k,\ell} \circ \Kgen^{(\fn)}_{t,\bm{\sigma},\ba}}
        S_{cd}
        \p{(\mathcal{G}_R)^{\p{d}}_{k,\ell} \circ \Kgen^{(\fn)}_{t,\bm{\sigma},\ba}},
  \end{equation}
  with the following initial condition at $t=0$ (recall the $M$-loop defined in \eqref{eq:Mloop_defgen}):
  \begin{equation}\label{initial-condition}
    \Kgen_{0,\bm{\sigma},\ba}^{(\fn)} = \Mgen_{\bm{\sigma},\ba}^{(\fn)}.
  \end{equation}
  Here, $\mathcal{G}_L$ and $\mathcal{G}_R$ are cut-and-glue operators understood as actions on $\p{\bm{\sigma}, \ba}$:
  \begin{align*}
(\mathcal{G}_L)^{\p{c}}_{k,\ell}\p{\bm{\sigma}, \ba}&\coloneqq \p{\p{\sigma_1, \ldots, \sigma_k, \sigma_\ell, \ldots, \sigma_\fn}, \p{a_1, \ldots, a_{k-1}, c, a_\ell, \ldots, a_\fn}}, \\(\mathcal{G}_R)^{\p{d}}_{k,\ell}\p{\bm{\sigma}, \ba}
&\coloneqq \p{\p{\sigma_k, \ldots, \sigma_\ell}, \p{a_k, \ldots, a_{\ell-1}, d}}.
  \end{align*}
  Note that these operators are simply reindexing of the $\cutL$ and $\cutR$ operators from Definition \ref{Def:oper_loop}, with a different notation to avoid confusion. 
\end{definition}

The next lemma gives the tree representation formula for the primitive loops.
\begin{lemma}\label{tree-representation}
When $\fn=2$, the primitive loop is given by 
\be\label{Kn2sol_gen}\Kgen^{(2)}_{t,(\sigma_1,\sigma_2)}=\Thetagen_t^{(\sig_1,\sig_2)}\MMgen^{(\sig_1,\sig_2)},\ee 
where we regard $\Kgen^{(2)}_{t,(\sigma_1,\sigma_2)}$ as a matrix with entries $(\Kgen^{(2)}_{t,\p{\sigma_1,\sigma_2}})_{ab} \equiv \Kgen^{(2)}_{t,\p{\sigma_1,\sigma_2},\p{a,b}}$. For primitive loops of length $\fn\ge 3$, we have the following representation formula:
  \begin{equation}\label{eq:tree_rep}
\Kgen_{t,\bm{\sigma},\ba}^{(\fn)}
      = \sum_{\Gamma \in \TSP\p{\mathcal{P}_{\ba}}} \Gamma^{\p{M}}_{t,\bm{\sigma},\ba}.
  \end{equation}
\end{lemma}
\begin{proof}
Similar to \eqref{Kn2sol}, we can verify that \eqref{Kn2sol_gen} satisfies the primitive equation \eqref{primitive-equation} and the initial condition \eqref{initial-condition}  for $\fn=2$.
It remains to show that for $\fn\ge 3$, \eqref{eq:tree_rep} satisfies the primitive equation (\ref{primitive-equation}) along with the initial condition (\ref{initial-condition}). By definition, we have
  \smash{\(
    \Thetagen_0 = I\) and \(   S_0\Thetagen_0 = 0.
  \)}
  Thus, at $t = 0$, the nonzero graphs must contain no unlabeled edges between different $M$-loops. The only such graph is the $M$-graph corresponding to the star graph, which has exactly one internal vertex of degree $\fn$. It is straightforward to see that this graph has a value of \smash{$\Mgen_{\bm{\sigma},\ba}^{(\fn)}$},   showing that \smash{${\Kgen}^{(\fn)}_{0,\bm{\sigma},\ba}$} satisfies the initial condition (\ref{initial-condition}). 


  We now verify (\ref{primitive-equation}) with the following identities: 
  \begin{align}
    \partial_t \Thetagen_t^{\p{\sigma,\sigma'}}
      = \Kgen^{(2)}_{t,\p{\sigma,\sigma'}} S \Thetagen_t^{\p{\sigma,\sigma'}},\quad \partial_t \Big(S_t\Thetagen_t^{\p{\sigma, \sigma'}}\Big)
      = \Big(\Thetagen_t^{\p{\sigma',\sigma}}\Big)^\top S\Thetagen_t^{\p{\sigma,\sigma'}}.
        \label{theta-derivative} 
  \end{align}
  Note that, on the RHS above, we have $S$ instead of $S_t$. Recalling \Cref{edge-value-definition}, we partition $\mathcal{E}_0(\Gamma_{\ba}^{(M)})$ into two subsets: 
  \begin{align*}
    \mathcal{E}_{\Thetagen}
      &\coloneqq \set{\p{a_k, c_k} \in \mathcal{E}_0\p{\Gamma_{\ba}^{(M)}} \suchthat 1 \leq k \leq \fn} ,\quad
      \mathcal{E}_{S\Thetagen}
      \coloneqq \set{\text{internal edges in } \mathcal{E}_0\p{\Gamma_{\ba}^{(M)}} }.
  \end{align*}
  We then calculate $\partial_t \Gamma^{\p{M;\mathbf{b}}}_{t,\bm{\sigma},\ba}$ as:
  \begin{align}
    \partial_t \Gamma^{\p{M;\mathbf{b}}}_{t,\bm{\sigma},\ba}
      &= \sum_{k=1}^\fn \frac{\Gamma^{\p{M;\mathbf{b}}}_{t,\bm{\sigma},\ba}}{\Thetagen^{\p{\sigma_k,\sigma_{k+1}}}_{t,a_kc_k}} \partial_t \Thetagen^{\p{\sigma_k,\sigma_{k+1}}}_{t,a_kc_k} + \sum_{\p{e_1, e_2} \in \mathcal{E}_{S\Thetagen}} \frac{\Gamma^{\p{M;\mathbf{b}}}_{t,\bm{\sigma},\ba}}{(S_t\Thetagen_t^{\p{q\p{e_1}, q\p{e_2}}})_{e_1e_2}} \partial_t \p{S_t\Thetagen_t^{\p{q\p{e_1}, q\p{e_2}}}}_{e_1e_2}.\label{eq:differentiate}
  \end{align}
  
  For the first term on the RHS of \eqref{eq:differentiate}, using (\ref{theta-derivative}), we get
  \begin{equation}
    \frac{\Gamma^{\p{M;\mathbf{b}}}_{t,\bm{\sigma},\ba}}{\Thetagen^{\p{\sigma_k,\sigma_{k+1}}}_{t,a_kc_k}} \partial_t \Thetagen^{\p{\sigma_k,\sigma_{k+1}}}_{t,a_kc_k}
      = \sum_{c, d}
          \Kgen^{(2)}_{t,\p{\sigma_k,\sigma_{k+1}},\p{a_k,c}}
          S_{cd}
          \Thetagen^{\p{\sigma_k,\sigma_{k+1}}}_{t,dc_k}
          \frac{\Gamma^{\p{M;\mathbf{b}}}_{t,\bm{\sigma},\ba}}{\Thetagen^{\p{\sigma_k,\sigma_{k+1}}}_{t,a_kc_k}}.
  \end{equation}
  Here, $\Thetagen^{\p{\sigma_k,\sigma_{k+1}}}_{t,dc_k} \Gamma^{\p{M;\mathbf{b}}}_{t,\bm{\sigma},\ba} / \Thetagen^{\p{\sigma_k,\sigma_{k+1}}}_{t,a_kc_k}$ is the graph $\Gamma^{\p{M;\mathbf{b}}}_{t,\bm{\sigma},\ba}$ but with the vertex $a_k$ replaced by $d$ as in Figure \ref{differentiating-Theta}. Moreover, it is clear that $q\p{d} = q\p{a_k} = \sigma_k$. Hence, we get   
  \begin{align*}
    \sum_{k=1}^\fn \frac{\Gamma^{\p{M;\mathbf{b}}}_{t,\bm{\sigma},\ba}}{\Thetagen^{\p{\sigma_k,\sigma_{k+1}}}_{t,a_kc_k}} \partial_t \Thetagen^{\p{\sigma_k,\sigma_{k+1}}}_{t,a_kc_k}
      &= \sum_{k=1}^\fn \sum_{c, d}
          {\Kgen}^{(2)}_{t,\p{\sigma_k,\sigma_{k+1}},\p{a_k,c}}
          S_{cd}
          \Gamma^{\p{M;\mathbf{b}}}_{t,\bm{\sigma},\p{a_1, \ldots, a_{k-1}, d, a_{k+1}, \ldots, a_\fn}} \\
      &= \sum_{k=1}^\fn \sum_{c, d}
          \Gamma^{\p{M;\mathbf{b}}}_{t,(\mathcal{G}_L)^{\p{d}}_{k,k+1}\p{\bm{\sigma}, \ba}}
          S_{dc}
          \Gamma^{\p{M;\mathbf{b}}}_{t,(\mathcal{G}_R)^{\p{c}}_{k,k+1}\p{\bm{\sigma}, \ba}}, \numberthis\label{theta-derivative-result}
  \end{align*}
  where we used the symmetry of $S$ in the second step. 
  \begin{figure}[h]
    \centering
    \scalebox{0.9}{
    \begin{tikzpicture}
      \coordinate (ak) at (-2, 0);
      \coordinate (ak-1) at (2, 3);
      \coordinate (ak+1) at (2, -3);
      
      \coordinate (b11) at (1, 0);
      \coordinate (b10) at (2, 1);
      \coordinate (b12) at (2, -1);
      
      \fill (ak-1) circle (1pt) node[above]{$a_{k-1}$};
      \fill (ak) circle (1pt) node[left]{$a_k$};
      \fill (ak+1) circle (1pt) node[below]{$a_{k+1}$};
      
      \fill (b11) circle (1pt) node[below, xshift=-8pt]{$b_{1,1}$};
      \fill (b12) circle (1pt) node[right]{$b_{1,2}$};
      \fill (b10) circle (1pt) node[right]{$b_{1,0}$};
      
      \draw (ak-1) [dashed]-- (ak) node[pos=0.5, above left]{$\sigma_k$};
      \draw (ak) [dashed]-- (ak+1)  node[pos=0.5, below left]{$\sigma_{k+1}$};
      
      \draw (b11) -- (ak) node[pos=0.5, above]{$\partial_t$};
      
      \draw (b10) -- (b11) -- (b12);
      
      \draw[dashed] (b12) [dotted]-- (ak+1);
      \draw[dashed] (b10) [dotted]-- (ak-1);
      
      \node at (1.75, 0) {$M$};
    \end{tikzpicture}
    \begin{tikzpicture}
      \coordinate (ak-1) at (2, 3);
      \coordinate (ak+1) at (2, -3);
      
      \coordinate (b11) at (1, 0);
      \coordinate (b10) at (2, 1);
      \coordinate (b12) at (2, -1);
      
      \coordinate (d) at (-0.5, 0);
      
      \fill (ak-1) circle (1pt) node[above]{$a_{k-1}$};
      \fill (ak+1) circle (1pt) node[below]{$a_{k+1}$};
      
      \fill (b11) circle (1pt) node[below, xshift=-8pt]{$b_{1,1}$};
      \fill (b12) circle (1pt) node[right]{$b_{1,2}$};
      \fill (b10) circle (1pt) node[right]{$b_{1,0}$};
      
      \fill (d) circle (1pt) node[below, xshift=-4pt]{$d$};
      
      \draw (ak-1) [dashed]-- (d) node[pos=0.5, above left]{$\sigma_k$};
      \draw (d) [dashed]-- (ak+1)  node[pos=0.5, below left]{$\sigma_{k+1}$};
      
      \draw (b11) -- (d) -- +(-0.25, 0);
      
      \draw (b10) -- (b11) -- (b12);
      
      \draw[dashed] (b12) [dotted]-- (ak+1);
      \draw[dashed] (b10) [dotted]-- (ak-1);
      
      \node at (1.75, 0) {$M$};
      
      \node[anchor=east, yshift=-4.5pt] at (-0.8, 0) {$= \displaystyle\sum_{c, d} \Kgen^{(2)}_{t,\p{\sigma_k,\sigma_{k+1}}, \p{a_k, c}} S_{cd}$};
      \node at (1.75, 0) {$M$};
    \end{tikzpicture}
    }
    \caption{The relevant subgraph of $\Gamma^{\protect\p{M}}_{\ba}$ when differentiating $\Thetagen^{\protect\p{\sigma_k,\sigma_{k+1}}}_{t,a_kb_{1,1}}$.}\label{differentiating-Theta}
  \end{figure}
  
  For the second term on the RHS of \eqref{eq:differentiate}, using (\ref{theta-derivative}), we get
  \begin{align*}    \frac{\Gamma^{\p{M;\mathbf{b}}}_{t,\bm{\sigma},\ba} \partial_t \big(S_t\Thetagen_t^{\p{q\p{e_1}, q\p{e_2}}}\big)_{e_1e_2}}{\big(S_t\Thetagen_t^{\p{q\p{e_1}, q\p{e_2}}}\big)_{e_1e_2}} 
      &= \sum_{c, d} \frac{\Gamma^{\p{M;\mathbf{b}}}_{t,\bm{\sigma},\ba}}{\big(S_t\Thetagen_t^{\p{q\p{e_1}, q\p{e_2}}}\big)_{e_1e_2}}
        \Thetagen^{\p{q\p{e_2},q\p{e_1}}}_{t,ce_1}
        S_{cd}
        \Thetagen^{\p{q\p{e_1}, q\p{e_2}}}_{t,de_2}.
  \end{align*}
  Without loss of generality, suppose $\p{e_1, e_2} = \p{b_{2,1}, b_{1,1}} = R_k \cap R_\ell$ with $k < \ell$. Referring to Figure \ref{differentiating-STheta}, we see that $q\p{c} = q\p{e_2} = \sigma_\ell$ and $q\p{d} = q\p{e_1} = \sigma_k$. Thus, we have  $\Thetagen^{\p{q\p{e_2},q\p{e_1}}}_{t,ce_1} = \Thetagen^{\p{q\p{c},q\p{e_1}}}_{t,ce_1}$ and $\Thetagen^{\p{q\p{e_1}, q\p{e_2}}}_{t,de_2} = \Thetagen^{\p{q\p{d}, q\p{e_2}}}_{t,de_2}$, which shows that  
  \begin{align*}
  \frac{\Gamma^{\p{M;\mathbf{b}}}_{t,\bm{\sigma},\ba} \partial_t\big(S_t\Thetagen_t^{\p{q\p{e_1}, q\p{e_2}}}\big)_{e_1e_2} }{\big(S_t\Thetagen_t^{\p{q\p{e_1}, q\p{e_2}}}\big)_{e_1e_2}}   
    &= \sum_{c, d} \Gamma^{\p{M;\mathbf{b}}}_{t,\p{\sigma_k, \ldots, \sigma_\ell},\p{a_k, \ldots, a_{\ell-1}, c}}
        S_{cd}
        \Gamma^{\p{M;\mathbf{b}}}_{t,\p{\sigma_1, \ldots, \sigma_k, \sigma_\ell, \ldots, \sigma_\fn},\p{a_1, \ldots, a_{k-1}, d, a_\ell, \ldots, a_\fn}} \\
      &= \sum_{c, d} \Gamma^{\p{M;\mathbf{b}}}_{t,(\mathcal{G}_L)^{\p{d}}_{k,\ell}\p{\bm{\sigma},\ba}}
        S_{dc}
        \Gamma^{\p{M;\mathbf{b}}}_{t,(\mathcal{G}_R)^{\p{c}}_{k,\ell}\p{\bm{\sigma},\ba}}. \numberthis\label{stheta-derivative-result}
  \end{align*}
  \begin{figure}[h]
    \topskip0pt
    \centering
    \scalebox{0.9}{
    \begin{tikzpicture}
      \coordinate (ak) at (-2, 2);
      \coordinate (ak-1) at (2, 3);
      \coordinate (al-1) at (-2, -2);
      \coordinate (al) at (2, -3);
      
      \coordinate (b11) at (1, 0);
      \coordinate (b10) at (2, 1);
      \coordinate (b12) at (2, -1);
      
      \coordinate (b21) at (-1, 0);
      \coordinate (b20) at (-2, -1);
      \coordinate (b22) at (-2, 1);
      
      \fill (ak-1) circle (1pt) node[above]{$a_{k-1}$};
      \fill (ak) circle (1pt) node[above]{$a_k$};
      \fill (al) circle (1pt) node[below]{$a_\ell$};
      \fill (al-1) circle (1pt) node[below]{$a_{\ell-1}$};
      
      \fill (b11) circle (1pt) node[below, xshift=-8pt]{$b_{1,1}$};
      \fill (b12) circle (1pt) node[right]{$b_{1,2}$};
      \fill (b10) circle (1pt) node[right]{$b_{1,0}$};
      
      \fill (b21) circle (1pt) node[below, xshift=6pt]{$b_{2,1}$};
      \fill (b22) circle (1pt) node[left]{$b_{2,2}$};
      \fill (b20) circle (1pt) node[left]{$b_{2,0}$};
      
      \draw (ak-1) [dashed]-- (ak) node[pos=0.5, above]{$\sigma_k$};
      \draw (al) [dashed]-- (al-1)  node[pos=0.5, below]{$\sigma_\ell$};
      
      \draw (b11) -- (b21) node[pos=0.5, above]{$\partial_t$};
      
      \draw (b10) -- (b11) -- (b12);
      \draw (b22) -- (b21) -- (b20);
      
      \draw[dashed] (b12) [dotted]-- (al);
      \draw[dashed] (b20) [dotted]-- (al-1);
      \draw[dashed] (b10) [dotted]-- (ak-1);
      \draw[dashed] (b22) [dotted]-- (ak);
      
      \node at (-1.75, 0) {$M$};
      \node at (1.75, 0) {$M$};
    \end{tikzpicture}
    \begin{tikzpicture}
      \coordinate (ak) at (-4, 2);
      \coordinate (ak-1) at (2, 3);
      \coordinate (al-1) at (-4, -2);
      \coordinate (al) at (2, -3);
      
      \coordinate (b11) at (1, 0);
      \coordinate (b10) at (2, 1);
      \coordinate (b12) at (2, -1);
      
      \coordinate (b21) at (-3, 0);
      \coordinate (b20) at (-4, -1);
      \coordinate (b22) at (-4, 1);
      
      \coordinate (c) at (-1.75, 0);
      \coordinate (d) at (-0.5, 0);
      
      \fill (ak-1) circle (1pt) node[above]{$a_{k-1}$};
      \fill (ak) circle (1pt) node[above]{$a_k$};
      \fill (al) circle (1pt) node[below]{$a_\ell$};
      \fill (al-1) circle (1pt) node[below]{$a_{\ell-1}$};
      
      \fill (b11) circle (1pt) node[below, xshift=-8pt]{$b_{1,1}$};
      \fill (b12) circle (1pt) node[right]{$b_{1,2}$};
      \fill (b10) circle (1pt) node[right]{$b_{1,0}$};
      
      \fill (b21) circle (1pt) node[below, xshift=6pt]{$b_{2,1}$};
      \fill (b22) circle (1pt) node[left]{$b_{2,2}$};
      \fill (b20) circle (1pt) node[left]{$b_{2,0}$};
      
      \fill (c) circle (1pt) node[below]{$c$};
      \fill (d) circle (1pt) node[below, xshift=-4pt]{$d$};
      
      \draw (ak-1) [dashed]-- (d) node[pos=0.5, above left]{$\sigma_k$};
      \draw (c) [dashed]-- (ak) node[pos=0.5, above right]{$\sigma_k$};
      \draw (al) [dashed]-- (d)  node[pos=0.5, below left]{$\sigma_\ell$};
      \draw (c) [dashed]-- (al-1)  node[pos=0.5, below right]{$\sigma_\ell$};
      
      \draw (b11) -- (d) -- +(-0.25, 0);
      \draw (b21) -- (c) -- +(0.25, 0);
      
      \draw (b10) -- (b11) -- (b12);
      \draw (b22) -- (b21) -- (b20);
      
      \draw[dashed] (b12) [dotted]-- (al);
      \draw[dashed] (b20) [dotted]-- (al-1);
      \draw[dashed] (b10) [dotted]-- (ak-1);
      \draw[dashed] (b22) [dotted]-- (ak);
      
      \node at (-3.75, 0) {$M$};
      \node at (1.75, 0) {$M$};
      \node at (-1.125, 0) {$S_{cd}$};
      \node at (-5.5, 0) {$= \displaystyle\sum_{c, d}$};
    \end{tikzpicture}
    }
    \caption{The relevant subgraph of $\Gamma^{\protect\p{M}}_{\ba}$ when differentiating $\protect(S_t\Thetagen_t^{\protect\p{\sigma_k,\sigma_\ell}})_{b_{2,1}b_{1,1}}$.}\label{differentiating-STheta}
  \end{figure}
  
  Putting (\ref{theta-derivative-result}) and (\ref{stheta-derivative-result}) together, we get
  \begin{equation}\label{full-derivative}
    \partial_t \Gamma^{\p{M;\mathbf{b}}}_{t,\bm{\sigma},\ba}
      = \sum_{1 \leq k < \ell \leq \fn} \sum_{e \in \mathcal{E}_0(\Gamma_{\ba}^{(M)}): e = R_k \cap R_\ell} 
      \sum_{c, d} \Gamma^{\p{M;\mathbf{b}}}_{t,(\mathcal{G}_L)^{\p{c}}_{k,\ell}\p{\bm{\sigma},\ba}}
        S_{cd}
        \Gamma^{\p{M;\mathbf{b}}}_{t,(\mathcal{G}_R)^{\p{d}}_{k,\ell}\p{\bm{\sigma},\ba}}.
  \end{equation}
  Note that for fixed $k < \ell$ and $c, d \in \qqq{N}$, the following mapping is a bijection:
  \begin{equation}\label{reindex}
    \p{\bm{\sigma}, \ba}
      \mapsto \p{\bm{\sigma}_c', \ba_c', \bm{\sigma}_d'', \ba_d''} = \p{(\mathcal{G}_L)^{\p{c}}_{k,\ell}\p{\bm{\sigma},\ba}, (\mathcal{G}_R)^{\p{d}}_{k,\ell}\p{\bm{\sigma},\ba}}.
  \end{equation}
 Thus, after summing over $\Gamma_{\ba} \in \TSP\p{\cal P_{\ba}}$, $b_{i,j} \in \qqq{N}$, and re-indexing with (\ref{reindex}) and (\ref{full-derivative}), we obtain that 
  \begin{align*}
    \partial_t \sum_{\Gamma \in \TSP\p{\mathcal{P}_{\ba}}} \Gamma^{\p{M}}_{t,\bm{\sigma},\ba}
      &= \sum_{1 \leq k < \ell \leq \fn} \sum_{c, d}
        \sum_{\Gamma' \in \TSP(\mathcal{P}_{\ba_c'})} \Gamma^{\p{M}}_{t,\bm{\sigma}_c',\ba_c'} \times 
        S_{cd} \times 
        \sum_{\Gamma'' \in \TSP(\mathcal{P}_{\ba_d''})} \Gamma^{\p{M}}_{t,\bm{\sigma}_d'',\ba_d''}.
  \end{align*}
  This shows that \eqref{eq:tree_rep} satisfies the primitive equation \eqref{primitive-equation}. 
\end{proof}

\subsection{Ward's identity}

In this subsection, we derive a ``Ward's identity" for the $\Kgen$-loops under \Cref{assm:general}, which will, in particular, imply the identity \eqref{WI_calK} as a special case. We begin with the following identity for the $\Theta$-propagator, which we will also refer to as its ``Ward's identity". 

\begin{lemma}
  If $\sigma \neq \sigma'$, then we have
  \begin{equation}\label{ward-M}
    \sum_{a, b} \Thetagen^{(\sigma,\sigma')}_{t,ab} M_{xb}\p{\sigma} M_{by}\p{\sigma'} = \frac{(\Im M)_{xy}}{\eta_t}.
  \end{equation}
\end{lemma}

\begin{proof}
  Our starting point is the following resolvent identity that can be derived directly from \eqref{eq:Mt}: 
  \begin{equation}\label{resolvent-M}
    (\Im M)_{xy} = \sum_{a} M_{xa}(\sigma)M_{ay}(\sigma')   \Big(\sum_b tS_{ab} \Im M_{bb} + \eta_t\Big). 
  \end{equation}
  Setting $x = y$, we get the self-consistent equation
  \[
    \Im M_{xx} = \sum_b \big(\MMgen^{\p{\sigma',\sigma}}S_t\big)_{xb} \p{\Im M_{bb}} + \sum_a \eta_t \MMgen^{\p{\sigma',\sigma}}_{xa}.
  \]
  Solving this equation and using (\ref{Theta-almost-symmetry}), we get
  \begin{equation}\label{ward-M-diagonal}
    \Im M_{xx} = \sum_a \big(\Thetagen_t^{(\sigma,\sigma')} \MMgen^{(\sigma,\sigma')}\big)_{ax} \eta_t.
  \end{equation}
  Substituting (\ref{ward-M-diagonal}) into (\ref{resolvent-M}) and using the identity $I + \Thetagen_t^{(\sigma,\sigma')} \MMgen^{(\sigma,\sigma')} S_t = \Thetagen_t^{(\sigma,\sigma')}$, 
  we obtain (\ref{ward-M}).
\end{proof}

With \eqref{ward-M} and the equation
\[
    \Thetagen_t^{(\sigma,\sigma')} - \Thetagen_t^{(\sigma,\sigma'')}
      = \Thetagen_t^{(\sigma,\sigma'')} \big(\MMgen^{(\sigma,\sigma')} - \MMgen^{(\sigma,\sigma'')}\big) S_t\Thetagen_t^{(\sigma,\sigma')},
  \]
we immediately derive the following identities:
for each $\sigma \in \set{+, -}$,  
  \begin{align}
 \Thetagen_{t,ab}^{\p{+,\sigma}} - \Thetagen_{t,ab}^{\p{-,\sigma}} 
      &= 2\ii\eta_t \sum_{a', b', x, y} \Thetagen^{\p{-,\sigma}}_{t,ax} \Thetagen^{\p{-,+}}_{t,a'b'} \p{S_t\Thetagen_t^{\p{+,\sigma}}}_{yb} \avg{M\p{\sigma}F_y M^*F_{b'} MF_x}, \label{ward-Theta-1} \\
\Thetagen_{t,ab}^{\p{\sigma,+}} - \Thetagen_{t,ab}^{\p{\sigma,-}}
      &= 2\ii\eta_t \sum_{a', b', x, y} \Thetagen^{\p{\sigma,-}}_{t,ax} \Thetagen^{\p{-,+}}_{t,a'b'} \p{S_t\Thetagen_t^{\p{\sigma,+}}}_{yb} \avg{M\p{\sigma}F_x M^*F_{b'} MF_y}. \label{ward-Theta-2}
  \end{align}

In contrast to the proof of Lemma 3.6 in \cite{Band1D}, our proof of Ward's identity does not use the dynamics given by the primitive equation (\ref{primitive-equation}) but instead relies only on the graphical representations of the \smash{$\Kgen$}-loops and the above identities \eqref{ward-M}, \eqref{ward-Theta-1} and \eqref{ward-Theta-2}. This approach allows us to establish a subtler Ward's identity that applies to individual graphs $\Gamma^{(M)}$, as stated in \Cref{ward-identity-gamma} below. We can then derive the corresponding identity for the \smash{$\Kgen$}-loops by summing over $\Gamma$, leading to \Cref{ward-identity}. 
The key observation in our proof is that the $\TSP$ graphs on $(\fn + 1)$ vertices can be constructed from $\TSP$ graphs on $\fn$ vertices, which is the content of  \Cref{slice-decomposition}.
\begin{definition}
  Let $\ba = \p{a_1, \ldots, a_\fn}$ and fix $\Gamma \in \TSP\p{\mathcal{P}_{\ba}}$. We define the \textbf{slices} of $\Gamma$ as follows.
  \begin{enumerate}
    \item Let $p_1$ be the canonical directed path in $R_1$, as defined in Definition \ref{canonical-directed-edge}, but specified in $\Gamma_\ba$ rather than in its $M$-graph. 
    Suppose $p_1$ can be written as 
    \be\label{eq:p1}p_1 = \p{a_\fn \to b_1 \to \cdots \to b_{k-1} \to a_1}.\ee 
    \item For each edge of the form $\p{b_{i-1}, b_i}$, $1 \leq i \leq k$ (with the convention $b_0 = a_\fn$ and $b_k = a_1$), we define $\Gamma_{\p{b_{i-1}, b_i}} \in \TSP(\mathcal{P}_{\p{\ba,a_{\fn+1}}})$ by adding new vertices $a_{\fn+1}, b'$ and replacing $\p{b_{i-1}, b_i}$ with the subgraph consisting of three edges $\p{b_{i-1}, b'},\ \p{b', b_i},$ and $ \p{a_{\fn+1}, b'}$.
    \item For each vertex $b_i$, $1 \leq i \leq k - 1$ (i.e., we exclude the external vertices $a_1, a_\fn$), we define $\Gamma_{\p{b_i}} \in \TSP(\mathcal{P}_{\p{\ba, a_{\fn+1}}})$ by simply adding the vertex $a_{\fn+1}$ and the edge $\p{a_{\fn+1}, b_i}$.
  \end{enumerate}
  See \Cref{Fig:slice} for illustrations of the above slicing operations. We then define the collection of slices of $\Gamma$ by
  \begin{equation}
    \slice\p{\Gamma}
      \coloneqq \set{\Gamma_{\p{b_{i-1},b_i}} \suchthat 1 \leq i \leq k} \cup \set{\Gamma_{\p{b_i}} \suchthat 1 \leq i \leq k-1}. \label{gamma-slices}
  \end{equation}
  \begin{figure}[h]
    \centering
    \scalebox{0.9}{
    \begin{tikzpicture}
      \coordinate (a1) at (180:2);
      \coordinate (an) at (90:2);
      \coordinate (b) at (0, 0);
      
      \fill (a1) circle (1pt) node[left]{$a_1$};
      \fill (an) circle (1pt) node[above]{$a_\fn$};
      \fill (b) circle (1pt) node[below right]{$b_1$};
      
      \draw (an) [dashed]-- (a1);
      \draw (a1) -- (b);
      \draw (b) -- (an);
      
      \node at (-1, -0.5) {$\Gamma$};
    \end{tikzpicture}
    \begin{tikzpicture}
      \coordinate (a1) at (180:2);
      \coordinate (an) at (90:2);
      \coordinate (an+1) at (135:2);
      \coordinate (b) at (0, 0);
      \coordinate (c) at (90:1);
      
      \fill (a1) circle (1pt) node[left]{$a_1$};
      \fill (an) circle (1pt) node[above]{$a_\fn$};
      \fill (b) circle (1pt) node[below right]{$b_1$};
      \fill (an+1) circle (1pt) node[above left]{$a_{\fn+1}$};
      \fill (c) circle (1pt);
      
      \draw (an) [dashed]-- (an+1);
      \draw (an+1) [dashed]-- (a1);
      \draw (a1) -- (b);
      \draw (b) -- (c) -- (an);
      \draw (an+1) -- (c);
      
      \node at (-1, -0.5) {$\Gamma_{\p{a_\fn,b_1}}$};
      \node at (-3, 1) {$\to$};
    \end{tikzpicture}
    $\quad$
    \begin{tikzpicture}
      \coordinate (a1) at (180:2);
      \coordinate (an) at (90:2);
      \coordinate (an+1) at (135:2);
      \coordinate (b) at (0, 0);
      
      \fill (a1) circle (1pt) node[left]{$a_1$};
      \fill (an) circle (1pt) node[above]{$a_\fn$};
      \fill (b) circle (1pt) node[below right]{$b_1$};
      \fill (an+1) circle (1pt) node[above left]{$a_{\fn+1}$};
      
      \draw (an) [dashed]-- (an+1);
      \draw (an+1) [dashed]-- (a1);
      \draw (a1) -- (b);
      \draw (b) -- (an);
      \draw (an+1) -- (b);
      
      \node at (-1, -0.5) {$\Gamma_{\p{b_1}}$};
    \end{tikzpicture}
    $\quad$
    \begin{tikzpicture}
      \coordinate (a1) at (180:2);
      \coordinate (an) at (90:2);
      \coordinate (an+1) at (135:2);
      \coordinate (b) at (0, 0);
      \coordinate (c) at (180:1);
      
      \fill (a1) circle (1pt) node[left]{$a_1$};
      \fill (an) circle (1pt) node[above]{$a_\fn$};
      \fill (b) circle (1pt) node[below right]{$b_1$};
      \fill (an+1) circle (1pt) node[above left]{$a_{\fn+1}$};
      \fill (c) circle (1pt);
      
      \draw (an) [dashed]-- (an+1);
      \draw (an+1) [dashed]-- (a1);
      \draw (a1) -- (c);
      \draw (c) -- (b) -- (an);
      \draw (an+1) -- (c);
      
      \node at (-1, -0.5) {$\Gamma_{\p{b_1,a_1}}$};
    \end{tikzpicture}
    }
    \caption{$\slice\protect\p{\Gamma}$ is created by ``slicing" up $R_1$ into two pieces.}\label{Fig:slice}
  \end{figure}
\end{definition}

\begin{claim}\label{slice-decomposition}
$\TSP(\mathcal{P}_{\p{\ba, a_{\fn+1}}})$ can be expressed as a disjoint union
  \[
    \TSP\p{\mathcal{P}_{\p{\ba, a_{\fn+1}}}}
      = \bigsqcup_{\Gamma \in \TSP\p{\mathcal{P}_{\ba}}} \slice\p{\Gamma}.
  \]
\end{claim}
\begin{proof}
  It is not hard to see that the slices are graphs in $\TSP(\mathcal{P}_{\p{\ba, a_{\fn+1}}})$. Given $\Gamma' \in \TSP(\mathcal{P}_{\p{\ba, a_{\fn+1}}})$, we can construct a $\Gamma\in \TSP\p{\mathcal{P}_{\ba}}$ by removing all edges and vertices connected to $a_{\fn+1}$ except for the external vertices $a_1, a_\fn$ and adding the edge $\p{a_\fn, a_1}$. 
  In other words, if the partition of $\mathcal{P}_{\p{\ba, a_{\fn+1}}}$ that gives $\Gamma'$ is $\set{R_1, R_2, \ldots, R_{\fn+1}}$, then the partition of $\mathcal{P}_{\ba}$ that gives $\Gamma$ is essentially $\set{R_1 \cup R_{\fn+1}, R_2, \ldots, R_\fn}$.
\end{proof}

We are ready to derive the Ward's identity for {$\Gamma^{(M)}_{t,\p{+,\bm{\sigma}},\ba}-\Gamma^{(M)}_{t,\p{-,\bm{\sigma}},\ba}.$} 
Note that these two graphs differ only in the edges within the region $R_1$. To transform \smash{$\Gamma^{(M)}_{t,\p{+,\bm{\sigma}},\ba}$} into {$\Gamma^{(M)}_{t,\p{-,\bm{\sigma}},\ba}$}, we will sequentially replace each $+$ in these edges with a $-$, one at a time. Each such replacement corresponds to a graph in $\slice\p{\Gamma}$. 
Thus, the bulk of the following proof is to ensure that charges are assigned correctly throughout this process.

\begin{lemma}[Ward's identity]\label{ward-identity-gamma}
  Let $\ba = \p{a_1, \ldots, a_\fn}$, $\bm{\sigma} \in \set{+, -}^{\fn-1}$, and pick any $\Gamma \in \TSP\p{\mathcal{P}_{\ba}}$. Let $\Gamma^{\p{M}}$ be its corresponding $M$-graph. Then, we have
  \[
    \frac{1}{2\ii\eta_t} \p{\Gamma^{(M)}_{t,\p{+,\bm{\sigma}},\ba} - \Gamma^{(M)}_{t,\p{-,\bm{\sigma}},\ba}}
      = \sum_{\Gamma' \in \slice\p{\Gamma}} \sum_{a_{\fn+1}} \p{\Gamma'}^{(M)}_{t,\p{+,\bm{\sigma},-},\p{\ba,a_{\fn+1}}}.
  \]
\end{lemma}
\begin{proof}
  Let $\mathbf{b}$ denote the internal vertices of $\Gamma$. We consider the edges of \smash{$\Gamma^{\p{M;\bfb}}_{t, \p{\sigma_1, \bm{\sigma}},\ba}$} along the path in \eqref{eq:p1}. 
  More precisely, we define
  \[
    \mathcal{G}^{\p{\sigma_1; \bfb}}
      \coloneqq \Thetagen^{\p{\sigma_\fn,\sigma_1}}_{t,a_\fn b_{1,1}}
        \p{\prod_{i=2}^{k-1} M^{\p{\sigma_1}}_{b_{i-1,1}b_{i-1,2}} \p{S_t\Thetagen_t^{\p{\tau_i,\sigma_1}}}_{b_{i-1,2}b_{i,1}}}
        M^{\p{\sigma_1}}_{b_{k-1,1}b_{k-1,2}}
        \Thetagen^{\p{\sigma_1,\sigma_2}}_{t,a_1b_{k-1,2}},
  \]
  where we used the notation $M^{(\sigma)}\equiv M(\sigma)$ for $\sig\in\{+,-\}$, $\tau_i:=q(b_{i-1,2})$ for $2\le i \le k-1$, and $b_{i,j}$ for $1\le i \le k-1$ to represent the vertices on the $M$-loops obtained from the replacements of the vertices $b_i$ in $\Gamma_{\ba}$. 
  \begin{figure}[h]
    \begin{tikzpicture}
      \coordinate (an) at (0, 3);
      \coordinate (a1) at (-3, 0);
      \coordinate (b11) at (0.1, 2);
      \coordinate (b12) at (0, 1);
      \coordinate (b21) at (-1, 0);
      \coordinate (b22) at (-2, -0.1);
      
      \fill (an) circle (1pt) node[above]{$a_\fn$};
      \fill (a1) circle (1pt) node[left]{$a_1$};
      \fill (b11) circle (1pt) node[right]{$b_{1,1}$};
      \fill (b12) circle (1pt) node[right]{$b_{1,2}$};
      \fill (b21) circle (1pt) node[below]{$b_{2,1}$};
      \fill (b22) circle (1pt) node[below]{$b_{2,2}$};
      
      \draw[blue, arrows={->[scale=1.5]},
    shorten >= 2pt] (an) [dashed]-- (a1) node[black, pos=0.5, above left]{$+$};
      \draw (an) -- (b11);
      \draw[blue, arrows={->[scale=1.5]},
    shorten >= 2pt] (b11) -- (b12);
      \draw (b12) -- (b21);
      \draw[blue, arrows={->[scale=1.5]},
    shorten >= 2pt] (b21) -- (b22);
      \draw (b22) -- (a1);
    \end{tikzpicture}
 
\caption{When $\sigma_1 = +$ and $k = 3$, we have $\mathcal{G}^{\protect\p{+;\bfb}} = \Thetagen^{\protect\p{\sigma_\fn ,+}}_{t,a_\fn b_{1,1}} M_{b_{1,1}b_{1,2}} \protect (S_t\Thetagen^{\protect\p{\tau,+}})_{t,b_{1,2}b_{2,1}} M_{b_{2,1}b_{2,2}} \Thetagen^{\protect\p{+,\sigma_2}}_{t,a_1b_{2,2}}$.}\label{calG-k=3}
  \end{figure}
  Note that these are the only factors depending on $\sigma_1$, and that they are listed in the order in which their indices appear in $p_1$. For brevity, we denote
  \[
    g^{\p{\sigma_1; \bfb}}_i \coloneqq \begin{cases}
        \Thetagen^{\p{\sigma_\fn ,\sigma_1}}_{t,a_\fn b_{1,1}} &\text{if } i = 1,  \\[1ex]
        M^{\p{\sigma_1}}_{b_{p,1}b_{p,2}}         &\text{if } 1 < i < 2k - 1,\ i = 2p, \\[1ex]
        \big(S_t\Thetagen_t^{\p{\tau_i,\sigma_1}}\big)_{b_{p-1,2}b_{p,1}}  &\text{if } 1 < i < 2k - 1,\ i = 2p - 1, \\[1ex]
        \Thetagen^{\p{\sigma_1,\sigma_2}}_{t,a_1b_{k-1,2}} &\text{if } i = 2k - 1,
      \end{cases}
  \]
  so that we can express $\mathcal{G}^{\p{\sigma_1; \bfb}}$ as 
  \[
    \mathcal{G}^{\p{\sigma_1; \bfb}} = \prod_{i=1}^{2k-1} g^{\p{\sigma_1; \bfb}}_i.
  \]
We then have the telescoping sum
  \begin{equation}\label{telescoping-ward}
  \begin{split}
    \mathcal{G}^{\p{+; \bfb}} - \mathcal{G}^{\p{-; \bfb}}
      &= \sum_{j=1}^{2k-1} \p{\prod_{i=1}^{j-1} g_i^{\p{-;\bfb}} \cdot \p{g_j^{\p{+; \bfb}} - g_j^{\p{-; \bfb}}} \cdot \prod_{i=j+1}^{2k-1} g_i^{\p{+; \bfb}}}
      \eqqcolon \sum_{j=1}^{2k-1} \mathcal{G}^{\p{\Delta;\bfb}}_j.      
  \end{split}
  \end{equation}
  Notice that given $1 \leq j \leq 2k - 1$, each factor to the left of $j$ has $\sigma_1 = -$, whereas each factor to the right of $j$ has $\sigma_1 = +$. We will see that these factors are the edges in $R_{\fn+1}$ and $R_{1}$, respectively, of the resulting slice of $\Gamma$.
  
  Let $1 \leq j \leq 2k - 1$. We split up the computation into two cases: when $j$ is even (in which case we deal with an $M$-edge) and when $j$ is odd (in which case we deal with a \smash{$\Thetagen_t$ or $S_t\Thetagen_t$-edge}). In the first case, (\ref{ward-M}) yields that 
  \[
    g_j^{\p{+; \bfb}} - g_j^{\p{-; \bfb}}
      = 2\ii\eta_t \sum_{a_{\fn+1}, b'} \Thetagen^{\p{-,+}}_{t,a_{\fn+1}b'} M^*_{b_{p,1}b'} M_{b'b_{p,2}}.
  \]
  Referring to Figure \ref{ward-even-i}, we readily observe that the $M$-edges have the correct colors and all vertices carry the correct charges; specifically, $q\p{b_{p,1}} = q\p{a_{\fn+1}}= -$, $q\p{b'} =  q\p{a_1} = +$, $q(a_\fn)=\sig_\fn$, and $q(b_{p,2})$ remains unchanged. 
  It is also easy to see that \smash{$g_i^{\p{\pm; \bfb}}$}, $i \neq j$, have the correct colors as well. In other words, we have shown that
  \[
    \mathcal{G}^{\p{\Delta;\mathbf{b}}}_j \frac{\Gamma^{\p{M;\mathbf{b}}}_{t,\p{\sigma_1, \bm{\sigma}},\ba}} {\mathcal{G}^{\p{\sigma_1; \mathbf{b}}}}
      = 2\ii\eta_t \sum_{a_{\fn+1}, b'} \p{\Gamma_{\p{b_{p-1}}}}^{\p{M; \mathbf{b}}}_{t, \p{+, \bm{\sigma}, -}, (\ba,a_{\fn+1})}.
  \]
  \begin{figure}[h]
    \centering
    \begin{tikzpicture}
      \coordinate (an) at (0, 3);
      \coordinate (a1) at (-3, 0);
      \coordinate (b11) at (0, 1);
      \coordinate (b12) at (-1, 0);
      
      \fill (an) circle (1pt) node[above]{$a_\fn$};
      \fill (a1) circle (1pt) node[left]{$a_1$};
      \fill (b11) circle (1pt) node[right]{$b_{p,1}$};
      \fill (b12) circle (1pt) node[below]{$b_{p,2}$};
      
      \draw[dashed, arrows={->[scale=1.5]},
    shorten >= 2pt] (an) -- (a1) node[black, pos=0.5, above left]{$\sigma_1$};
      \draw[dotted] (an) -- (b11);
      \draw[arrows={->[scale=1.5]},
    shorten >= 2pt] (b11) -- (b12);
      \draw[dotted] (b12) -- (a1);
      
      \node at (135:0.25) {$M$};
    \end{tikzpicture}
    \begin{tikzpicture}
      \coordinate (an) at (0, 3);
      \coordinate (a1) at (-3, 0);
      \coordinate (an+1) at (135:3);
      \coordinate (b11) at (0, 1);
      \coordinate (b12) at (-1, 0);
      \coordinate (b') at (135:1);
      
      \fill (an) circle (1pt) node[above]{$a_\fn$};
      \fill (a1) circle (1pt) node[left]{$a_1$};
      \fill (an+1) circle (1pt) node[above left]{$a_{\fn+1}$};
      \fill (b11) circle (1pt) node[right]{$b_{p,1}$};
      \fill (b12) circle (1pt) node[below]{$b_{p,2}$};
      \fill (b') circle (1pt) node[above, xshift=2pt, yshift=2pt]{$b'$};
      
      \draw[dashed, red, arrows={->[scale=1.5]},
    shorten >= 2pt] (an) -- (an+1) node[black, pos=0.5, above left]{$-$};
      \draw[dashed, blue, arrows={->[scale=1.5]},
    shorten >= 2pt] (an+1) -- (a1) node[black, pos=0.5, above left]{$+$};
      \draw[red, arrows={->[scale=1.5]},
    shorten >= 2pt] (b11) -- (b');
      \draw[dotted] (an) -- (b11);
      \draw[blue, arrows={->[scale=1.5]},
    shorten >= 2pt] (b') -- (b12);
      \draw[dotted] (b12) -- (a1);
      \draw (an+1) -- (b');
      
      \node at (135:0.25) {$M$};
      \node[anchor=east] at (-3.5, 1.5) {$\to \displaystyle 2\ii\eta_t \sum_{a_{\fn+1}, b'}$};
    \end{tikzpicture}
    \caption{When $i$ is even, we apply (\ref{ward-M}).}\label{ward-even-i}
  \end{figure}
  
When $j$ is odd and $j\notin\{1,2k-1\}$, suppose $\p{b_{p-1,2}, b_{p,1}} = R_1 \cap R_\ell$. Then, the charge $\tau_j$ is given by $\tau_j = \sigma_\ell$, so (\ref{ward-Theta-2}) yields
  \begin{equation}\label{ward-STheta-odd}
    g_j^{\p{+; \bfb}} - g_j^{\p{-; \bfb}}
      = 2\ii\eta_t \sum_{a_{\fn+1}, b', x, y} \p{S_t\Thetagen_t^{\p{\sigma_\ell,-}}}_{b_{p-1,2}x} \Thetagen^{\p{-,+}}_{t,a_{\fn+1}b'} \p{S_t\Thetagen_t^{\p{\sigma_\ell,+}}}_{yb_{p,1}}\avg{M\p{\sigma_\ell}F_x M^*F_{b'} MF_y}.      
  \end{equation}
  Referring to Figure \ref{ward-odd-i}, we can check that the $M$-edges have the correct color and that the charges of $b_{p-1,2}, x, b', y, b_{p,1}, a_{\fn+1}$ are consistent with the charges on the $\Thetagen_t$ and $S_t\Thetagen_t$-factors. The cases $j \in\{ 1, 2k - 1\}$ can be handled in exactly the same way, except the first $(S_t\Thetagen_t)$-factor in (\ref{ward-STheta-odd}) is replaced by a $\Thetagen$-factor. This leads to that
  \[
    \mathcal{G}^{\p{\Delta;\mathbf{b}}}_j \frac{\Gamma^{\p{M;\mathbf{b}}}_{t, \p{\sigma_1, \bm{\sigma}},\ba}}{\mathcal{G}^{\p{\sigma_1; \bfb}}}
      = 2\ii\eta_t \sum_{a_{\fn+1}, b', x, y} \p{\Gamma_{\p{b_{p-1}, b_p}}}^{\p{M; \mathbf{b}}}_{t,\p{+, \bm{\sigma}, -}, (\ba,a_{\fn+1})}.
  \]
  \begin{figure}[h]
    \centering
    \begin{tikzpicture}
      \coordinate (an) at (0, 3);
      \coordinate (a1) at (-3, 0);
      \coordinate (b12) at (0, 1);
      \coordinate (b21) at (-1, 0);
      
      \fill (an) circle (1pt) node[above]{$a_\fn$};
      \fill (a1) circle (1pt) node[left]{$a_1$};
      \fill (b12) circle (1pt) node[right]{$b_{p-1,2}$};
      \fill (b21) circle (1pt) node[below]{$b_{p,1}$};
      
      \draw[dashed, arrows={->[scale=1.5]},
    shorten >= 2pt] (an) -- (a1) node[black, pos=0.5, above left]{$\sigma_1$};
      \draw[dotted] (an) -- (b12);
      \draw[arrows={->[scale=1.5]},
    shorten >= 2pt] (b12) -- (b21);
      \draw[dotted] (b21) -- (a1);
      
      \node at (135:0.25) {$R_\ell$};
    \end{tikzpicture}
    \begin{tikzpicture}
      \coordinate (an) at (0, 3);
      \coordinate (a1) at (-3, 0);
      \coordinate (an+1) at (135:3);
      \coordinate (b11) at (0, 1);
      \coordinate (b12) at (-1, 0);
      \coordinate (b') at (135:2);
      \coordinate (x) at (-0.5, 1.25);
      \coordinate (y) at (-1.25, 0.5);
      
      \fill (an) circle (1pt) node[above]{$a_\fn$};
      \fill (a1) circle (1pt) node[left]{$a_1$};
      \fill (an+1) circle (1pt) node[above left]{$a_{\fn+1}$};
      \fill (b11) circle (1pt) node[right]{$b_{p-1,2}$};
      \fill (b12) circle (1pt) node[below]{$b_{p,1}$};
      \fill (b') circle (1pt) node[above, xshift=3pt]{$b'$};
      \fill (x) circle (1pt) node[above, xshift=2pt]{$x$};
      \fill (y) circle (1pt) node[left, yshift=-2pt]{$y$};
      
      \draw[dashed, red, arrows={->[scale=1.5]},
    shorten >= 2pt] (an) -- (an+1) node[black, pos=0.5, above left]{$-$};
      \draw[dashed, blue, arrows={->[scale=1.5]},
    shorten >= 2pt] (an+1) -- (a1) node[black, pos=0.5, above left]{$+$};
      \draw[arrows={->[scale=1.5]},
    shorten >= 2pt] (b11) -- (x);
      \draw[red, arrows={->[scale=1.5]},
    shorten >= 2pt] (x) -- (b');
      \draw[blue, arrows={->[scale=1.5]},
    shorten >= 2pt] (b') -- (y);
      \draw[arrows={->[scale=1.5]},
    shorten >= 2pt] (y) -- (x);
      \draw[dotted] (an) -- (b11);
      \draw[arrows={->[scale=1.5]},
    shorten >= 2pt] (y) -- (b12);

      \draw[dotted] (b12) -- (a1);
      \draw (an+1) -- (b');
      
      \node at (135:1.5) {$M$};
      \node at (135:0.75) {$R_\ell$};
      \node[anchor=east] at (-3.5, 1.5) {$\to \displaystyle 2\ii\eta_t \sum_{a_{\fn+1}, b', x, y}$};
    \end{tikzpicture}
    \caption{When $i$ is odd, we apply (\ref{ward-Theta-2}).}\label{ward-odd-i}
  \end{figure}
  
  Putting everything together and summing over the $b_{i,j}$ vertices, we get
  \[
    \Gamma^{\p{M}}_{t,\p{+,\bm{\sigma}},\ba} - \Gamma^{\p{M}}_{t,\p{-,\bm{\sigma}},\ba}
      = \sum_{\mathbf b} \sum_j \mathcal{G}^{\p{\Delta;\mathbf{b}}}_j \frac{\Gamma^{\p{M;\mathbf{b}}}_{t, \p{\sigma_1, \bm{\sigma}},\ba}}{\mathcal{G}^{\p{\sigma_1; \mathbf{b}}}}
      = 2\ii\eta_t \sum_{\Gamma' \in \slice\p{\Gamma}} \sum_{a_{\fn+1}} \p{\Gamma'}^{\p{M}}_{t,\p{+, \bm{\sigma}, -},\p{\ba,a_{\fn+1}}},
  \]
  which concludes the proof of \Cref{ward-identity-gamma}.
\end{proof}
Combining Lemma \ref{ward-identity-gamma} with \Cref{slice-decomposition}, we immediately obtain the following corollary.
\begin{corollary}\label{ward-identity}
  Let $\ba = \p{a_1, \ldots, a_\fn}$ and $\bm{\sigma} \in \set{+, -}^{\fn-1}$. Then, we have
  \[
    \frac{1}{2\ii\eta_t} \p{\Kgen^{(\fn)}_{t,\p{+,\bm{\sigma}},\ba} - \Kgen^{(\fn)}_{t,\p{-,\bm{\sigma}},\ba}}
      = \sum_{a_{\fn+1}} \Kgen^{(\fn+1)}_{t,\p{+,\bm{\sigma},-},\p{\ba,a_{\fn+1}}}.
  \]
\end{corollary}

\subsection{Core decomposition and molecule}\label{subsec:molecule} 

Before proving the sum-zero property and the $\cal K$-loop bounds in \Cref{ML:Kbound}, we first introduce the \emph{core decomposition} of {$\Gamma^{(M)}_\ba$} and the concept of \emph{molecule} in this subsection. 
Define the subset 
\[
\Z^{\mathrm{off}}_\fn \coloneqq \set{\set{k, \ell} \suchthat 1 \leq k < \ell \leq \fn,\, \abs{k - \ell \mod \fn} \neq 1 }.
\]
Given $\Gamma \in \TSP\p{\mathcal{P}_{\ba}}$, if $R_k \cap R_\ell\ne \emptyset$ with $\{k,\ell\}\in \Z^{\mathrm{off}}_\fn$, then we say that $R_k$ and $ R_\ell$ are \emph{non-trivial neighbors} (as opposed to the trivial neighbor pairs $R_{k}$ and $R_{k+1}$).

\begin{lemma}[Recursive decomposition]\label{recursive-decomposition-lemma}
For $\Gamma \in \TSP\p{\mathcal{P}_{\ba}}$, assume that $R_k$ and $ R_\ell$ are non-trivial neighbors. Let $\p{c, c'} = R_k \cap R_\ell$ denote the shared edge in \smash{$\Gamma^{(M)}_\ba$}. 
Suppose $c$ and $c'$ represent the vertices in the cycles obtained by replacing the internal vertices $b$ and $b'$ in $\Gamma_\ba$, respectively.  
Then, we have the decomposition
  \begin{equation}\label{recursive-decomposition}
    \Gamma_{t,\bm{\sigma},\ba}^{\p{M; \mathbf{b}}}
      = \frac{\p{\Gamma_c}^{\p{M; \mathbf{b}_c}}_{t,(\mathcal{G}_L)^{\p{a}}_{k\ell}\p{\bm{\sigma}, \ba}}}{\Thetagen_{t,ac}^{\p{\sigma_\ell, \sigma_k}}}
        \p{S_t\Thetagen_t^{\p{\sigma_k, \sigma_\ell}}}_{cc'}
        \frac{\p{\Gamma_{c'}}^{\p{M; \mathbf{b}_{c'}}}_{t,(\mathcal{G}_R)^{(a')}_{k\ell}\p{\bm{\sigma}, \ba}}}{\Thetagen_{t,a'c'}^{\p{\sigma_k, \sigma_\ell}}},
  \end{equation}
  where the $\TSP$ graphs $\Gamma_c$ and $\Gamma_{c'}$ are defined as follows. Since $\Gamma$ is a tree, removing the edge $\p{b, b'}$ results in two connected subtrees: 
one connected to $b$ and the other connected to $b'$. Denote these subtrees as $\mathcal{T}_c$ and $\mathcal{T}_{c'}$, respectively.
Then, we define 
\[\Gamma_c \coloneqq \mathcal{T}_c \cup \p{c, a} \in \TSP\Big(\mathcal{P}_{(\mathcal{G}_L)^{\p{a}}_{k\ell}\p{\ba}}\Big),\quad  \Gamma_{c'}\coloneqq \mathcal{T}_{c'} \cup \p{c', a'} \in \TSP\Big(\mathcal{P}_{(\mathcal{G}_R)^{\p{a'}}_{k\ell}\p{\ba}}\Big).\]
This decomposition is unique in the sense that for every such edge $\p{c, c'}$, there is exactly one pair $\p{\Gamma_c, \Gamma_{c'}}$ satisfying (\ref{recursive-decomposition}), up to relabeling of vertices. We refer readers to Figure \ref{recursive-decomposition-figure} for an illustration of this decomposition.
\end{lemma}
\begin{proof}
It is not hard to check that the charges of the vertices and $M$-edges in $\Gamma^{\p{M}}_c$ and \smash{$\Gamma^{\p{M}}_{c'}$} are consistent with those of $\Gamma^{\p{M}}$, with which we readily conclude the proof.
  \begin{figure}[h]
    \centering
    \scalebox{0.9}{
    \begin{tikzpicture}
      \coordinate (ak) at (-2, 2);
      \coordinate (ak-1) at (2, 3);
      \coordinate (al-1) at (-2, -2);
      \coordinate (al) at (2, -3);
      
      \coordinate (b11) at (1, 0);
      \coordinate (b10) at (2, 1);
      \coordinate (b12) at (2, -1);
      
      \coordinate (b21) at (-1, 0);
      \coordinate (b20) at (-2, -1);
      \coordinate (b22) at (-2, 1);
      
      \fill (ak-1) circle (1pt) node[above]{$a_{k-1}$};
      \fill (ak) circle (1pt) node[above]{$a_k$};
      \fill (al) circle (1pt) node[below]{$a_\ell$};
      \fill (al-1) circle (1pt) node[below]{$a_{\ell-1}$};
      
      \fill (b11) circle (1pt) node[below, xshift=-4pt]{$c'$};
      \fill (b12) circle (1pt) node[right]{$b_{1,2}$};
      \fill (b10) circle (1pt) node[right]{$b_{1,0}$};
      
      \fill (b21) circle (1pt) node[below, xshift=2pt]{$c$};
      \fill (b22) circle (1pt) node[left]{$b_{2,2}$};
      \fill (b20) circle (1pt) node[left]{$b_{2,0}$};
      
      \draw[arrows={->[scale=1.5]},
    shorten >= 2pt,dashed] (ak-1) -- (ak) node[pos=0.5, above]{$\sigma_k$};
      \draw[arrows={->[scale=1.5]},
    shorten >= 2pt,dashed] (al-1) -- (al)  node[pos=0.5, below]{$\sigma_\ell$};
      
      \draw[green] (b11) -- (b21);
      
      \draw (b10) -- (b11) -- (b12);
      \draw (b22) -- (b21) -- (b20);
      
      \draw[dotted] (b12) -- (al);
      \draw[dotted] (b20) -- (al-1);
      \draw[dotted] (b10) -- (ak-1);
      \draw[dotted] (b22) -- (ak);
      
      \node at (-1.75, 0) {$M$};
      \node at (1.75, 0) {$M$};
    \end{tikzpicture}
    \begin{tikzpicture}
      \coordinate (ak) at (-2, 2);
      \coordinate (ak-1) at (2, 3);
      \coordinate (al-1) at (-2, -2);
      \coordinate (al) at (2, -3);
      
      \coordinate (b11) at (1, 0);
      \coordinate (b10) at (2, 1);
      \coordinate (b12) at (2, -1);
      
      \coordinate (b21) at (-1, 0);
      \coordinate (b20) at (-2, -1);
      \coordinate (b22) at (-2, 1);
      
      \coordinate (shift) at (0, -4);
      \coordinate (ak-1) at ($(ak-1) + (shift)$);
      \coordinate (al) at ($(al) + (shift)$);
      \coordinate (b12) at ($(b12) + (shift)$);
      \coordinate (b10) at ($(b10) + (shift)$);
      \coordinate (d') at ($(b21) + (shift)$);
      \coordinate (c') at ($(b11) + (shift)$);
      
      \fill (ak-1) circle (1pt) node[above]{$a_{k-1}$};
      \fill (ak) circle (1pt) node[above]{$a_k$};
      \fill (al) circle (1pt) node[below]{$a_\ell$};
      \fill (al-1) circle (1pt) node[below]{$a_{\ell-1}$};
      \fill (c') circle (1pt) node[below, xshift=-4pt]{$c'$};
      \fill (d') circle (1pt) node[left]{$a'$};
      
      \fill (b11) circle (1pt) node[right]{$a$};
      \fill (b12) circle (1pt) node[right]{$b_{1,2}$};
      \fill (b10) circle (1pt) node[right]{$b_{1,0}$};
      
      \fill (b21) circle (1pt) node[below, xshift=2pt]{$c$};
      \fill (b22) circle (1pt) node[left]{$b_{2,2}$};
      \fill (b20) circle (1pt) node[left]{$b_{2,0}$};
      
      \draw[dashed,arrows={->[scale=1.5]},
    shorten >= 2pt] (b11) -- (ak) node[pos=0.5, above right]{$\sigma_k$};
      \draw[dashed,arrows={->[scale=1.5]},
    shorten >= 2pt] (al-1) -- (b11)  node[pos=0.5, below right]{$\sigma_\ell$};
      
      \draw[dashed,arrows={->[scale=1.5]},
    shorten >= 2pt] (ak-1) -- (d') node[pos=0.5, above left]{$\sigma_k$};
      \draw[dashed,arrows={->[scale=1.5]},
    shorten >= 2pt] (d') -- (al)  node[pos=0.5, below left]{$\sigma_\ell$};
      
      \draw[green] (b11) -- (b21);
      \draw[green] (c') -- (d');

      \draw (b10) -- (c') -- (b12);
      \draw (b22) -- (b21) -- (b20);
      
      \draw[dotted] (b12) -- (al);
      \draw[dotted] (b20) -- (al-1);
      \draw[dotted] (b10) -- (ak-1);
      \draw[dotted] (b22) -- (ak);
      
      \node at (-1.75, 0) {$M$};
      \node at ($(1.75, 0) + (shift)$) {$M$};
      \node[anchor=east] at (-3.0, -4.0) {$\to$};
    \end{tikzpicture}
    }
    \caption{The left graph is $\Gamma^{\protect\p{M}}$ with $\protect\p{c, c'} = R_k \cap R_\ell$, and the top and bottom graphs on the RHS are $\Gamma_c^{\protect\p{M}}$ and $\Gamma_{c'}^{\protect\p{M}}$, respectively. 
    The green edges represent the three $\Thetagen_t$ and $S_t\Thetagen_t$ edges written explicitly in (\ref{recursive-decomposition}). }\label{recursive-decomposition-figure}
  \end{figure}
\end{proof}

\begin{definition}[Core]\label{core-definition}
 For $\Gamma \in \TSP\p{\mathcal{P}_{\ba}}$, let $\bfb = \p{b_1, \ldots, b_\fm}$ be its internal indices, and let $\bm{d}=(d_1, \ldots, d_\fn)$ be the internal vertices connected to $a_1, \ldots, a_\fn$, respectively. Recall that $\delta_i$ denotes the degree of $b_i$ in $\Gamma$. 
Then, we define the \textbf{core} of $\Gamma^{\p{M}}$ as
  \[
    \Sigmagen^{\p{\Gamma; M}}_{t, \bm{\sigma}, \bm{d}}
      \coloneqq \sum_{i=1}^\fm \sum_{j=1}^{\delta_i}  \mathbf 1_{b_{i,j} \not\in \bm{d}} \sum_{b_{i,j} \in \qqq{N}} \frac{\Gamma^{\p{M; \bfb}}_{t,\bm{\sigma},\ba}}{\prod_{i=1}^\fn \Thetagen_{t,a_id_i}^{(\sigma_i,q(d_i))}},
  \]
  where we have slightly abused the notation: in the indicator function $\mathbf 1$, $b_{i,j}$ is treated as a vertex and $b_{i,j}\notin\bm{d}$ means that $b_{i,j}\notin\{d_1,\ldots, d_\fn\}$. 
  Simply speaking, the core is constructed graphically by removing all the external edges (i.e., those corresponding to \smash{$\Thetagen_t$}) from \smash{$\Gamma^{\p{M; \bfb}}_{t,\bm{\sigma},\ba}$}, and then summing over all the internal vertices, excluding the vertices $d_1, \ldots, d_\fn$.
\end{definition}


Next, we will present a corresponding decomposition for the cores of $\Kgen$-loops. This decomposition is similar to that in \Cref{recursive-decomposition-lemma}, but we must also keep track of all pairs of non-trivial neighbors. Informally, when ``splitting" each graph through an edge $R_k \cap R_\ell$ between non-trivial neighbors, every other non-trivial neighbor pair will belong to either the ``left" graph or the ``right" graph, as these pairs are not permitted to ``cross" each other.

\begin{definition} \label{k-sigma-pi}
Let $\ba = \p{a_1, \ldots, a_\fn}$. Given $\Gamma \in \TSP\p{\mathcal{P}_{\ba}}$ and $\bm{\sigma} \in \set{+, -}^{\fn}$, define the set of index pairs  corresponding to the ``long" internal edges (i.e., the {$S_t\Thetagen^{(\sig,\sig')}_t$} edges with $\sig\ne\sig'$) as 
\[
{\cal F}_{\text{long}}(\Gamma, \boldsymbol{\sigma}) := \left\{\{k,\ell\} \in \Z^{\mathrm{off}}_\fn: R_k\cap R_\ell\neq \emptyset ,\ \sigma_k\ne \sigma_\ell \right\}.
\]
Given any subset \(\pi \subset \mathbb{Z}_\fn^{\mathrm{off}}\), we denote by
\[
\TSP(\mathcal{P}_{\ba}, \boldsymbol{\sigma}, \pi) := \big\{\Gamma \in \TSP(\mathcal{P}_{\ba}) : {\cal F}_{\text{long}}(\Gamma, \boldsymbol{\sigma}) = \pi\big\}
\]
the subset of graphs in \(\TSP(\mathcal{P}_{\ba})\) with \(\pi\) labeling their long internal edges (note $\pi$ can be $\emptyset$). We define
  \[
    \Kgen^{\p{\pi}}_{t,\bm{\sigma},\ba}
      \coloneqq \sum_{\Gamma \in \TSP\p{\mathcal{P}_{\ba}, \bm{\sigma}, \pi}} \Gamma^{\p{M}}_{t,\bm{\sigma},\ba}.
  \]
  We analogously define the corresponding \emph{core} as
  \[
    \Sigmagen^{\p{\pi}}_{t,\bm{\sigma},\bm{d}}
      \coloneqq \sum_{\Gamma \in \TSP\p{\mathcal{P}_{\ba}, \bm{\sigma}, \pi}} \Sigmagen^{\p{\Gamma; M}}_{t, \bm{\sigma}, \bm{d}},
  \]
  where $\Sigmagen^{\p{\Gamma; M}}_{t, \bm{\sigma}, \bm{d}}$ is defined in Definition \ref{core-definition}. When $\pi = \emptyset$, we call $\Sigmagen^{\p{\emptyset}}$ as a \textbf{molecule}.
\end{definition}

\begin{lemma}[Core decomposition]\label{core-decomposition}
Without loss of generality, suppose $\set{1, r} \in \pi$ and $\p{c, c'} = R_1 \cap R_r$. Swapping $c, c'$ if necessary, we may ensure that $q\p{c} = \sigma_1$. Then, we have the \emph{core decomposition}
  \begin{equation}\label{recursive-decomposition-Sigma}
\Sigmagen^{\p{\pi}}_{t,\bm{\sigma},\bm{d}} = \sum_{c, c'} \Sigmagen^{(\pi^{\p{c}})}_{t,\bm{\sigma}^{\p{c}},\bm{d}^{\p{c}}} 
        \p{S_t\Thetagen_t^{\p{\sigma_1, \sigma_r}}}_{cc'}
\Sigmagen^{(\pi^{(c')})}_{t,\bm{\sigma}^{(c')},\bm{d}^{(c')}} ,
  \end{equation}
  where $\pi_0 \coloneqq \pi \setminus \set{\set{1, r}},$ and
  \begin{align}
&\pi^{(c)} \coloneqq \set{\set{k - r + 1, \ell - r + 1} \mid \set{k, \ell} \in \pi_0,\ r \leq k < \ell \leq \fn},\ \ \pi^{(c')} \coloneqq \set{\set{k, \ell} \in \pi_0 \mid 1 \leq k < \ell \leq r}, \label{eq:pic}\\
  &\big(\bm{\sigma}^{\p{c}}, \bm{d}^{\p{c}}\big)
        \coloneqq \left(\mathcal{G}_L\right)^{\p{c}}_{1r}\p{\bm{\sigma}, \bm{d}}, \ \
    \big(\bm{\sigma}^{(c')}, \bm{d}^{(c')}\big)
        \coloneqq \left(\mathcal{G}_R\right)^{(c')}_{1r}\p{\bm{\sigma}, \bm{d}}.\label{eq:sigmac}
  \end{align}
Note that we have the partition
  \(
    \pi =(\pi^{(c)} + r - 1) \cup \pi^{(c')} \cup  \set{\set{1, r}}.
  \)
\end{lemma}
\begin{proof}
By Lemma \ref{recursive-decomposition-lemma} and \Cref{core-definition}, we have
  \begin{equation}\label{recursive-decomposition-Sigma-pre}
    \Sigmagen^{\p{\pi}}_{t,\bm{\sigma},\bm{d}}
      = \sum_{c, c'} \sum_{\Gamma \in \TSP\p{\mathcal{P}_{\ba}, \bm{\sigma}, \pi}}
        \Sigmagen^{\p{\Gamma_c; M}}_{t, \bm{\sigma}^{\p{c}}, \bm{d}^{\p{c}}}
        \p{S_t\Thetagen_t^{\p{\sigma_1, \sigma_r}}}_{cc'}
        \Sigmagen^{\p{\Gamma_{c'}; M}}_{t, \bm{\sigma}^{(c')}, \bm{d}^{(c')}},
  \end{equation}
Let $\ba^{\p{c}}
        \coloneqq \left(\mathcal{G}_L\right)^{\p{c}}_{1r}(\ba)$ and $\ba^{(c')}
        \coloneqq \left(\mathcal{G}_R\right)^{(c')}_{1r}(\ba)$. Then, we further define
  \begin{align*}
    \TSP^{\p{c}}
      \coloneqq \TSP\p{\mathcal{P}_{\ba^{\p{c}}}, \bm{\sigma}^{\p{c}}, \pi^{\p{c}}}, \quad
    \TSP^{(c')}
      \coloneqq \TSP\p{\mathcal{P}_{\ba^{(c')}}, \bm{\sigma}^{(c')}, \pi^{(c')}}.
  \end{align*}
  By reindexing, equation (\ref{recursive-decomposition-Sigma-pre}) becomes
  \begin{align*}
    \Sigmagen^{\p{\pi}}_{t,\bm{\sigma},\bm{d}}
      &= \sum_{c, c'} 
        \bigg(\sum_{\Gamma \in \TSP^{\p{c}}} \Sigmagen^{\p{\Gamma; M}}_{t, \bm{\sigma}^{\p{c}}, \bm{d}^{\p{c}}}\bigg)
        \p{S_t\Thetagen_t^{\p{\sigma_1, \sigma_r}}}_{cc'}
        \bigg(\sum_{\Gamma' \in \TSP^{(c')}} \Sigmagen^{\p{\Gamma'; M}}_{t, \bm{\sigma}^{(c')}, \bm{d}^{(c')}}\bigg),
  \end{align*}
  which is exactly (\ref{recursive-decomposition-Sigma}),
\end{proof}

\subsection{Block reduction to the $\cK$-loops}\label{subsec:estimates}

For the remainder of this section, we return to the RBSOs defined in Definition \ref{def: BM}. First, we observe that under the flow in \Cref{def_flow}, the block Anderson and Wegner orbital models satisfy \Cref{assm:general} due to \eqref{eq:mt_stay}. Next, we can easily derive the $\cal{K}$-loops (as defined in \Cref{Def_Ktza}) from the \smash{$\Kgen$}-loops (as defined in \cref{general-primitive-equation}) as follows.


\begin{lemma}\label{lem:KgentoK}
Given $\fn\ge 2$, for each $t \in \br{0, 1}$, $\bm{\sigma} \in \set{+, -}^\fn$, and $\ba =([a_1],\ldots, [a_\fn]) \in (\Zn)^\fn$, we have
\begin{equation}\label{K-rescaling}
    \cK^{(\fn)}_{t,\bm{\sigma},\ba}
        = W^{-d\fn} \sum_{x_1\in[a_1],\ldots,x_\fn\in[a_\fn]}\Kgen^{(\fn)}_{t,\bm{\sigma},\bx}= \sum_{\bx}\Kgen^{(\fn)}_{t,\bm{\sigma},\bx}\cdot \prod_{k=1}^\fn E_{[a_k]}(x_k),
  \end{equation}
  where $\bx=(x_1,\ldots, x_\fn)$ and we abbreviate that $E_{[a_k]}(x_k):=(E_{[a_k]})_{x_kx_k}$. As a consequence, we can derive Ward's identity \eqref{WI_calK} for $\cal K$-loops from \Cref{ward-identity}.
\end{lemma}
\begin{proof}
First, with \eqref{initial-condition}, we can check that \eqref{K-rescaling} satisfies the initial condition \eqref{eq:initial_K}. Next, using equation \eqref{primitive-equation} and the formula for $S$ in \eqref{eq:SBA} or \eqref{eq:SWO}, we can check that \smash{$\cK^{(\fn)}_{t,\bm{\sigma},\ba}$} in \eqref{K-rescaling} satisfies
\begin{align*}
\partial_t \cK^{(\fn)}_{t,\bm{\sigma},\ba}&=\sum_{1 \leq k < l \leq \fn} W^{-d\fn} \sum_{x_1\in[a_1],\ldots,x_\fn\in[a_\fn]}\sum_{c, d}         \p{(\mathcal{G}_L)^{\p{c}}_{k,l} \circ \Kgen^{(\fn)}_{t,\bm{\sigma},\bx}}
        W^{-d}S^{\LK}_{[c][d]}
        \p{(\mathcal{G}_R)^{\p{d}}_{k,l} \circ \Kgen^{(\fn)}_{t,\bm{\sigma},\bx}} \\
&=W^d\sum_{1 \leq k < l \leq \fn} W^{-d(\fn+k-l+1)}\sum_{(\mathcal{G}_L)^{\p{c}}_{k,l}(\bx)\in \cutL^{[c]}_{k,l}(\ba)}         \p{(\mathcal{G}_L)^{\p{c}}_{k,l} \circ \Kgen^{(\fn)}_{t,\bm{\sigma},\bx}}
        \times  S^{\LK}_{[c][d]} \\
&\qquad\qquad \ \times W^{-d(l-k+1)}\sum_{(\mathcal{G}_R)^{(d)}_{k,l}(\bx)\in \cutR^{[d]}_{k,l}(\ba)}
        \p{(\mathcal{G}_R)^{(d)}_{k,l} \circ \Kgen^{(\fn)}_{t,\bm{\sigma},\bx}} \\
&=     W^d \sum_{1\le k < l \le \fn} \sum_{[c], [d]\in \Zn} \left( \cutL^{[c]}_{k, l} \circ \mathcal{K}^{(\fn)}_{t, \boldsymbol{\sigma}, \ba} \right) S^{\LK}_{[c][d]} \left( \cutR^{[d]}_{k, l} \circ \mathcal{K}^{(\fn)}_{t, \boldsymbol{\sigma}, \ba} \right) ,        
\end{align*}
which is \eqref{pro_dyncalK}. Above, \smash{$(\mathcal{G}_L)^{\p{c}}_{k,l}(\bx)\in \cutL^{[c]}_{k,l}(\ba)$} should be interpreted as the summation of the vertices in \smash{$(\mathcal{G}_L)^{\p{c}}_{k,l}(\bx)$} over the blocks in {$\cutL^{[c]}_{k,l}(\ba)$}, and a similar notation applies to \smash{$(\mathcal{G}_R)^{\p{d}}_{k,l}(\bx)\in \cutR^{[d]}_{k,l}(\ba)$}.
\end{proof}

Recall Definitions \ref{def_Theta} and \ref{def:propagator-entrywise} and the $M$-loops defined in \eqref{eq:KMloop} and \eqref{initial-condition}. In the setting of RBSOs, we can express the $M$-loops and $\Theta$-propagators as  
\begin{align}\label{eq:BRedK}
& {\cal M}^{(k)}_{\boldsymbol{\sigma}, \mathbf b}=\sum_{\bx} {\Mgen}^{(k)}_{\boldsymbol{\sigma}, \bx}\prod_{i=1}^k E_{[b_i]}(x_i),\quad \forall \bx=(x_1,\ldots, x_k),\  \mathbf b=([b_1],\ldots,[b_k]),\\
\label{eq:BRedK2}
&W^d \sum_x E_{[a]}(x)\Thetagen_{t,xy}^{(\sig_1,\sig_2)}= \Theta_{t,[a][b]}^{(\sig_1,\sig_2)},\  W^d\big(S_t\Thetagen_{t}^{(\sig_1,\sig_2)}\big)_{xy}= \big(S_t^{\LK}\Theta_{t}^{(\sig_1,\sig_2)}\big)_{[a][b]},\quad \forall x\in[a], y\in [b].
\end{align}
Inspired by these relations, we define the block reductions of  \smash{$\Kgen^{(\fn)}_{t, \bm{\sigma}, \ba}$, $\Kgen^{\p{\pi}}_{t, \bm{\sigma}, \ba}$}, and $\Sigmagen^{\p{\pi}}_{t, \bm{\sigma}, \bm{d}}$ (recall \Cref{k-sigma-pi}), denoted by \smash{$\cK^{(\fn;B)}_{t, \bm{\sigma}, \ba}$, $\cK^{\p{\pi;B}}_{t, \bm{\sigma}, \ba}$, and $\Sigma^{\p{\pi;B}}_{t, \bm{\sigma}, \bm{d}}$}, as follows: 
\begin{itemize}
    \item[(1)] Replace every $\Thetagen_{t,ab}^{(\sig,\sig')}$ edge (representing an external edge) in each graph with a $\Theta_{t,[a][b]}^{(\sig,\sig')}$ edge.
   \item[(2)] Replace every $(S_t\Thetagen_t^{(\sig,\sig')})_{ab}$ edge (representing an unlabeled internal edge) in each graph with an $(S_t^{\LK}\Theta_t^{(\sig,\sig')})_{[a][b]}$ edge.
   \item[(3)] Replace each  ${\Mgen}^{(k)}_{\boldsymbol{\sigma}, \bx}$-loop with a loop representing $W^{(k-1)d}{\cal M}^{(k)}_{\boldsymbol{\sigma}, \mathbf b}$ with $x_i\in[b_i]$ for all $i\in\qqq{k}$. Here, we have chosen the $W^{(k-1)d}$ scaling such that  
   \be\label{eq:sizeMloop}
\qquad  W^{(k-1)d}{\cal M}^{(k)}_{\boldsymbol{\sigma}, \mathbf b} = \frac{1}{W^d}\sum_{x_k\in[b_k]} \left[\left(\prod_{i=1}^{k-1}M(\sig_i)I_{[b_i]}\right)M(\sig_k)\right]_{x_k x_k}=\OO(1),
\ee
where we used the notation $I_{[a]}$ in \eqref{def:Ia} and the fact that $\|M(\sig)\|=\OO(1)$ by \eqref{eq:M-msc}. 
\end{itemize}
Then, we let 
$$\cal{K}^{\p{\pi}}_{t, \bm{\sigma}, \ba} \coloneqq \cK^{\p{\pi;B}}_{t, \bm{\sigma}, \ba},\quad \Sigma^{\p{\pi}}_{t, \bm{\sigma}, \bm{d}}:= \Sigma^{\p{\pi;B}}_{t, \bm{\sigma}, \bm{d}} \ .$$
With \eqref{K-rescaling}, \eqref{eq:BRedK}, and \eqref{eq:BRedK2}, we observe that: 
\be\label{K-pi-rescaling}
    \cal{K}^{(\fn)}_{t,\bm{\sigma},\ba} = W^{-d\p{\fn-1}} \cK^{(\fn;B)}_{t,\bm{\sigma},\ba}=W^{-d\p{\fn-1}}\sum_\pi \cK^{\p{\pi}}_{t, \bm{\sigma}, \ba}.
\ee
To see this, for \smash{$\cal{K}^{(\fn)}_{t,\bm{\sigma},\ba}$} defined in \eqref{K-rescaling}, we pick an arbitrary graph from the tree representation formula \eqref{eq:tree_rep} for \smash{$\Kgen^{(\fn)}_{t,\bm{\sigma},\bx}$}. 
In this graph, every unlabeled \smash{$\Thetagen_{t}$ or $S_t\Thetagen_{t}$} edge contributes an edge in (1) or (2), along with a $W^{-d}$ factor, while each \smash{$\Mgen$}-loop contributes a loop in (3) along with a $W^{d}$ factor. Due to the tree structure, we have 
$\#\{\text{unlabeled edges}\}-\#\{M\text{-loops}\}=\fn-1,$  
which leads to \eqref{K-pi-rescaling}.

With the block reduction representation in \eqref{K-pi-rescaling}, using the estimates (\ref{prop:ThfadC_short}) and \eqref{eq:sizeMloop}, we can immediately derive the following bound on \emph{pure loops} where all charges are identical.

\begin{claim}[Pure loop estimate]\label{claim:pure}
  Let $\bm{\sigma} \in \set{+, -}^\fn$ be such that $\sigma_1 = \sigma_2 = \cdots = \sigma_\fn$. Then, we have
  \[
    \cal{K}^{(\fn)}_{t,\bm{\sigma},\ba} = \OO\p{W^{-d(\fn-1)}}.
  \]    
\end{claim}
Next, by repeatedly applying Ward's identity, we obtain the following bound from the pure loop estimate.
\begin{corollary}\label{K-ward-bound}
  Let $\bm{\sigma} \in \set{+, -}^\fn$. Then, we have that 
  \[
    \sum_{\br{a_2},\ldots,\br{a_\fn}} \mathcal{K}^{(\fn)}_{t,\bm{\sigma},\ba}=W^{-d\p{\fn-1}}\sum_\pi \sum_{\br{a_2},\ldots,\br{a_\fn}} \mathcal{K}^{(\pi)}_{t,\bm{\sigma},\ba} = \OO\Big(\big(W^d\eta_t\big)^{-\fn+1}\Big).
  \]
\end{corollary}
\begin{proof}
We apply Ward's identity \eqref{WI_calK} repeatedly. In each step, we gain a $(W^d\eta_t)^{-1}$ factor while reducing the number of external vertices by 1. We continue this process until every resulting loop is a pure loop. Then, applying \Cref{claim:pure} and \eqref{prop:ThfadC_short} concludes the proof. 
\end{proof}
We now state and prove a \emph{sum-zero property} for $\Sigma^{\p{\pi}}$. It will be the main tool for the proof of \Cref{ML:Kbound}.  

\begin{lemma}[Sum zero property]\label{sum-zero}
Let $\fn \in 2\N$, $\fn \geq 4$, and consider an alternating loop $\bm{\sigma}$ such that $\sig_k\ne \sig_{k+1}$ for all $k\in\qqq{\fn}$, where we adopt the convention that $\sig_{\fn+1}=\sig_1$. Then, the single-molecule graph with $\pi=\emptyset$ (recall \Cref{k-sigma-pi}) satisfies the sum-zero property:
  \begin{equation}\label{sum-zero-property}
    \frac{1}{n^d} \sum_{\bm{d}} \Sigma^{\p{\emptyset}}_{t, \bm{\sigma}, \bm{d}} = \OO\p{\eta_t},
  \end{equation}
where recall that $\bm{d}=([d_1], \ldots, [d_\fn])$ denote the internal vertices connected to $[a_1], \ldots,$ $ [a_\fn]$, respectively. 
\end{lemma}
\begin{proof}
  We will prove a slightly more general sum zero property by induction: for all $\fn \in 2\N$ and $\pi\subset \Z^{\mathrm{off}}_\fn$,  
  \begin{equation}\label{sum-zero-inductive-hypothesis}
    \frac{1}{n^d} \sum_{\bm{d}} \Sigma^{\p{\pi}}_{t,\bm{\sigma},\bm{d}}
      = \OO\p{\eta_t}.
  \end{equation}   
If only the scenario $\pi = \emptyset$ is possible, then by \Cref{K-ward-bound}, we have that  
\be\label{sumzero_derv}\frac{1}{(1-t)^\fn} \sum_{\bm{d}} \Sigma^{\p{\emptyset}}_{t, \bm{\sigma}, \bm{d}} = \sum_{\ba} \mathcal{K}^{(\emptyset)}_{t,\bm{\sigma},\ba}   =\OO\p{n^d\eta_t^{-\fn+1}}.\ee
In the first step, we used that for any $\sig\ne \sig'$ and $[y]\in \Zn$, $\sum_{[x]}\Theta^{(\sig,\sig')}_{[x][y]}=(1-t)^{-1}$ by \eqref{eq:WardM1}. Together with the fact that $1-t\asymp \eta_t$, the equation \eqref{sumzero_derv} implies \eqref{sum-zero-inductive-hypothesis}. 
 As a special case, this concludes the base case with $\fn = 4$ where only $\pi = \emptyset$ is possible. 
 For the general case with $\fn > 4$, suppose we have proved \eqref{sum-zero-inductive-hypothesis} for all alternating loops of length $k<\fn$ and $\pi\subset \Z^{\mathrm{off}}_k$. 
  
We first consider the $\pi\ne\emptyset$ case. Without loss of generality, suppose $\set{1, r} \in \pi$, and let $\p{c, c'} = R_1 \cap R_r$ (note $c,c'$ now denote blocks in \smash{$\Zn$}). Swapping $c$ and $c'$ if necessary, we can assume that $\p{q\p{c}, q(c')} = \p{+, -} = \p{\sigma_1, \sigma_r}$. Applying \Cref{core-decomposition} and the block reduction, we get
  \[
    \sum_{\bm{d}} \Sigma^{\p{\pi}}_{t,\bm{\sigma},\bm{d}}
      = \sum_{c, c'} \sum_{\bm{d}^{\p{c}} \setminus \set{c}}\Sigma^{(\pi^{\p{c}})}_{t,\bm{\sigma}^{\p{c}},\bm{d}^{\p{c}}}   
        \cdot \p{S_t^{\LK}\Theta_t^{\p{+, -}}}_{cc'}
        \cdot \sum_{\bm{d}^{(c')} \setminus \set{c'}}\Sigma^{(\pi^{(c')})}_{t,\bm{\sigma}^{(c')},\bm{d}^{(c')}} .
  \]
Since $\bm{\sigma}$ is alternating, $r$ must be be even. Then, both $\bm{\sigma}^{\p{c}}$ and $\bm{\sigma}^{(c')}$ are alternating loops of lengths $\fn - r + 2 \in 2\N$ and $r \in 2\N$, respectively. Thus, we can apply the induction hypothesis and derive that 
  \begin{align*}
    \frac{1}{n^d} \sum_{\bm{d}} \Sigma^{\p{\pi}}_{t,\bm{\sigma},\bm{d}}
      &= \bigg(\frac{1}{n^d} \sum_{\bm{d}^{\p{c}}}\Sigma^{(\pi^{\p{c}})}_{t,\bm{\sigma}^{\p{c}},\bm{d}^{\p{c}}} \bigg)
        \cdot \frac{1}{n^d} \sum_{c, c'} \p{S_t^{\LK}\Theta_t^{\p{+, -}}}_{cc'}\cdot 
        \bigg(\frac{1}{n^d} \sum_{\bm{d}^{(c')}}\Sigma^{(\pi^{(c')})}_{t,\bm{\sigma}^{(c')},\bm{d}^{(c')}} \bigg) \\
      &= \OO\p{\eta_t \cdot \frac{t}{1 - t} \cdot \eta_t} = \OO\p{\eta_t},
  \end{align*}
  where, in the first step, we used that the first and third factors on the RHS do not depend on $c$ or $c'$ due to the block translation invariance. 

  To deal with the $\pi = \emptyset$ case, we adopt a similar argument as in \eqref{sumzero_derv} and decompose \smash{$\mathcal{K}^{(\fn)}_{t,\bm{\sigma},\ba}$} as
  \begin{align}
    \frac{W^{d(\fn-1)}}{n^d} \sum_{\ba} \mathcal{K}^{(\fn)}_{t,\bm{\sigma},\ba}
      &= \p{\sum_{\ba} \prod_{i=1}^\fn \Theta_{t,[a_i][d_i]}^{\p{\sigma_i,\sigma_{i+1}}}} \bigg(\frac{1}{n^d} \sum_{\bm{d}}\Sigma^{\p{\emptyset}}_{t,\bm{\sigma},\bm{d}}
        + \frac{1}{n^d}\sum_{\pi \neq \emptyset}  \sum_{\bm{d}}\Sigma^{\p{\pi}}_{t,\bm{\sigma},\bm{d}}\bigg) \nonumber\\
      &= \p{1 - t}^{-\fn} \bigg(\frac{1}{n^d} \sum_{\bm{d}}\Sigma^{\p{\emptyset}}_{t,\bm{\sigma},\bm{d}}
        + \frac{1}{n^d}\sum_{\pi \neq \emptyset}  \sum_{\bm{d}}\Sigma^{\p{\pi}}_{t,\bm{\sigma},\bm{d}}\bigg).\label{eq:piempty}
  \end{align}
  By \Cref{K-ward-bound}, we have that 
  $$\frac{W^{d(\fn-1)}}{n^d}\sum_{\ba} \mathcal{K}^{(\fn)}_{t,\bm{\sigma},\ba}   \lesssim \frac{W^{d(\fn-1)}}{n^d} \cdot n^d\p{W^d\eta_t}^{-\fn+1}=\eta_t^{-\fn+1}.$$
  Plugging this bound and the estimate (\ref{sum-zero-inductive-hypothesis}) established for the $\pi\ne \emptyset$ case into \eqref{eq:piempty}, we conclude that 
  \[
    \frac{1}{n^d} \sum_{\bm{d}}\Sigma^{\p{\emptyset}}_{t,\bm{\sigma},\bm{d}}
      = -\frac{1}{n^d} \sum_{\pi \neq \emptyset} \sum_{\bm{d}}\Sigma^{\p{\pi}}_{t,\bm{\sigma},\bm{d}} + \p{1 - t}^\fn \OO\p{\eta_t^{-\fn+1}}
      = \OO\p{\eta_t}.
  \]
  This completes the induction step and hence concludes the proof of \Cref{sum-zero}.
\end{proof}


We are now prepared to utilize the sum-zero property in \Cref{sum-zero} to complete the proof of \Cref{ML:Kbound}. In this proof, we will use the following simple estimate regarding the core.

\begin{claim}\label{core-decay}
  We have that
  \be\label{eq:core}
    \sum_{\br{d_2}, \ldots, \br{d_\fn}} \Sigma^{\p{\emptyset}}_{t, \bm{\sigma}, \bm{d}} =\OO\p{\eta_t},\quad \text{and}\quad 
    \sum_{\br{d_2}, \ldots, \br{d_\fn}} \abs{\Sigma^{\p{\emptyset}}_{t, \bm{\sigma}, \bm{d}}} = \OO\p{\eta_t + \lambda^2}.
  \ee
\end{claim}
\begin{proof}
The first identity in \eqref{eq:core} is an immediate consequence of \eqref{sum-zero-property} by the block translation invariance. For the second identity, note that all edges (i.e., the $M$-edges and \smash{$\Theta_t^{(\sig,\sig)}$}-edges) in \smash{$\Sigma^{\p{\emptyset}}$} are short-range. Hence, by the estimates \eqref{Mbound_AO2} and \eqref{prop:ThfadC_short}, there exists a constant $C>0$ such that whenever $\max_{i, j} \abs{\br{d_i} - \br{d_j}}>0$,
  \[
    \abs{\Sigma^{\p{\emptyset}}_{t, \bm{\sigma}, \bm{d}}}
      \le  C\lambda^2 \p{C\heta}^{\max_{i, j} \abs{\br{d_i} - \br{d_j}}-1}+W^{-D},
  \]
  for any large constant $D>0$. This gives that
  \be\label{eq:offSigma}
    \sum_{\br{d_2}, \ldots, \br{d_\fn}}^{([d_1])}  \abs{\Sigma^{\p{\emptyset}}_{t, \bm{\sigma}, \bm{d}}} \lesssim \lambda^2,\quad \text{where}\ \ \sum_{(\br{d_2}, \ldots, \br{d_\fn})}^{([d_1])}:= \sum_{(\br{d_2}, \ldots, \br{d_\fn}) \in (\Zn)^{\fn-1}\setminus \{([d_1],\ldots, [d_1])\}}.
  \ee
  For the diagonal term with $\br{d_1} = \cdots = \br{d_\fn}$, using \Cref{sum-zero} along with the fact that \smash{$\Sigma^{\p{\emptyset}}_{t, \bm{\sigma}, \p{\br{d_1}, \ldots, \br{d_1}}}$} does not depend on $[d_1]$ due to block translation invariance, we obtain that 
  \begin{align*}
    \Sigma^{\p{\emptyset}}_{t, \bm{\sigma}, \p{\br{d_1}, \ldots, \br{d_1}}} &= \frac{1}{n^d}\sum_{[d_1]}\Sigma^{\p{\emptyset}}_{t, \bm{\sigma}, \p{\br{d_1}, \ldots, \br{d_1}}} =\OO\p{\eta_t} - \frac{1}{n^d}\sum_{[d_1]}\sum_{\br{d_2}, \ldots, \br{d_\fn}}^{([d_1])}\Sigma^{\p{\emptyset}}_{t, \bm{\sigma}, \bm{d}} =\OO(\eta_t+\lambda^2).  
  \end{align*}
  Combined with \eqref{eq:offSigma}, this concludes the second identity in \eqref{eq:core}. 
\end{proof}
\begin{proof}[\bf Proof of Lemma \ref{ML:Kbound}]
Using the definitions \eqref{Kn2sol} for $2$-$\cK$ loops and \eqref{Kn3sol} for $3$-$\cK$ loops (noting that an alternating loop is not possible when $\fn=3$), we can easily show that \eqref{eq:bcal_k} holds for $\fn\in\{2,3\}$ by applying the estimates \eqref{prop:ThfadC}, \eqref{prop:ThfadC_short}, and \eqref{Mbound_AO2}. For $\fn\ge 4$, if $\bsig$ is a pure loop, we can easily derive \eqref{eq:bcal_k} from the estimates \eqref{Mbound_AO2} and \eqref{prop:ThfadC_short}, as all edges in the $\fn$-$\cK$ loop are short-range. Otherwise, we can assume that $\sig_1\ne\sig_2$, up to the relabeling of vertices. By \eqref{K-pi-rescaling}, it suffices to prove the following estimate: 
  \begin{equation}\label{Kbound-inductive-hypothesis}
\sum_{\br{d_2}, \ldots, \br{d_\fn}} \p{\prod_{i=2}^\fn \Theta^{\p{\sigma_i, \sigma_{i+1}}}_{t,\br{a_i}\br{d_i}}}\Sigma^{\p{\pi}}_{t, \bm{\sigma}, \bm{d}}\prec \p{\ell_t^d\eta_t}^{-\fn+2} \Big({\min_{2 \leq i \leq \fn} \avg{\br{a_i} - \br{d_1}}}\Big)^{-d},\quad  \forall\pi \subset \mathbb{Z}_\fn^{\mathrm{off}}\, .
\end{equation}
Summing this estimate over $[d_1]$ and applying \eqref{prop:ThfadC} to {$\Theta_{t,[a_1][d_1]}^{(\sig_1,\sig_2)}$} concludes \eqref{eq:bcal_k}. 

We will prove \eqref{Kbound-inductive-hypothesis} by induction in $\fn$. Suppose that \eqref{Kbound-inductive-hypothesis} holds for all $\cK$-loops of length $k<\fn$ and $\pi\subset \Z^{\mathrm{off}}_k$. 
We begin with the case $\pi = \emptyset$ and further divide the problem into the following two cases: 
(i) $\bsig$ is an alternating loop; (ii) there exists $j \in \qqq{2,\fn}$ such that $\sigma_j = \sigma_{j+1}$ (with the convention $\sigma_{\fn+1} = \sigma_1$). 
In the second case, {$\Theta_{t,\br{a_j}\br{d_j}}^{\p{\sigma_j, \sigma_{j+1}}}$} is a short edge. Then, applying \eqref{Mbound_AO2} and \eqref{prop:ThfadC}, we obtain that 
  \[
    \prod_{i=2}^\fn \Theta_{t,\br{a_i}\br{d_i}}^{\p{\sigma_i, \sigma_{i+1}}} \prec \p{\ell_t^d\eta_t}^{-\fn+2} \avg{[a_j]-[d_j]}^{-d},
  \]
  which implies that
  \begin{align*}
    \sum_{\br{d_2}, \ldots, \br{d_n}} \p{\prod_{i=2}^\fn \Theta^{\p{\sigma_i, \sigma_{i+1}}}_{t,\br{a_i}\br{d_i}}}\Sigma^{\p{\emptyset}}_{t, \bm{\sigma}, \bm{d}}
      &\prec \p{\ell_t^d\eta_t}^{-\fn+2} \sum_{\br{d_2}, \ldots, \br{d_n}} \abs{\Sigma^{\p{\emptyset}}_{t, \bm{\sigma}, \bm{d}}}\avg{[a_j]-[d_j]}^{-d} \prec \p{\ell_t^d\eta_t}^{-\fn+2} \avg{[a_j]-[d_1]}^{-d}.
  \end{align*}
  Here, in the second step, we used that the core $\Sigma^{\p{\emptyset}}_{t, \bm{\sigma}, \bm{d}}$ consists of short edges only. 
  
  
It remains to consider the more challenging case (i), to deal with which we will need to use 
the estimates in \eqref{eq:core}. For brevity, we introduce the following notations:
  \[
    \br{s_i} \coloneqq \br{d_i} - \br{d_1},
    \quad f\p{\br{a_i}, \br{s_i}} \coloneqq \Theta^{\p{\sigma_i, \sigma_{i+1}}}_{t,\br{a_i}(\br{d_1}+\br{s_i})}.
  \]
  Then, the left-hand side (LHS) of (\ref{Kbound-inductive-hypothesis}) can be written as
  \[
    \sum_{\br{d_2}, \ldots, \br{d_n}} \p{\prod_{i=2}^\fn \Theta^{\p{\sigma_i, \sigma_{i+1}}}_{t,\br{a_i}\br{d_i}}}\Sigma^{\p{\emptyset}}_{t, \bm{\sigma}, \bm{d}}
      = \sum_{\br{s_2}, \ldots, \br{s_n}} \p{\prod_{i=2}^\fn f\p{\br{a_i}, \br{s_i}}} \Sigma^{\p{\emptyset}}_{t, \bm{\sigma}, \bm{d}}.
  \]
  We further decompose
  \[
    f\p{\br{a_i}, \br{s_i}}
      = f_0\p{\br{a_i}, \br{s_i}}
        + f_1\p{\br{a_i}, \br{s_i}}
        + f_2\p{\br{a_i}, \br{s_i}},
  \]
  where we have defined
  \begin{align*}
    f_0\p{\br{a_i}, \br{s_i}}
      &\coloneqq f\p{\br{a_i}, \br{0}}, \quad
    f_1\p{\br{a_i}, \br{s_i}}
      \coloneqq \frac{1}{2} \br{f\p{\br{a_i}, \br{s_i}} - f\p{\br{a_i}, -\br{s_i}}}, \\
    f_2\p{\br{a_i}, \br{s_i}}
      &\coloneqq \frac{1}{2} \br{f\p{\br{a_i}, \br{s_i}} + f\p{\br{a_i}, -\br{s_i}}}
        - f\p{\br{a_i}, \br{0}} .
  \end{align*}
  By (\ref{prop:ThfadC}), (\ref{prop:BD1}), and (\ref{prop:BD2}), they satisfy the following estimates:
  \be
  \begin{split}
    &f_0\p{\br{a_i}, \br{s_i}}
      \prec \p{\ell_t^d \eta_t}^{-1},\quad
    f_1\p{\br{a_i}, \br{s_i}}
      \prec \frac{\p{\eta_t + \lambda^2}^{-1} |[s_i]|}{\avg{\br{a_i} - \br{d_1}}^{d-1}}, \quad f_2\p{\br{a_i}, \br{s_i}}
      \prec \frac{\p{\eta_t + \lambda^2}^{-1}|[s_i]|^2}{\avg{\br{a_i} - \br{d_1}}^{d}}.\label{f012}
  \end{split}\ee
  Then, the LHS of (\ref{Kbound-inductive-hypothesis}) can be decomposed as
  \[
    \sum_{\br{d_2}, \ldots, \br{d_\fn}} \p{\prod_{i=2}^\fn \Theta^{\p{\sigma_i, \sigma_{i+1}}}_{t,\br{a_i}\br{d_i}}}\Sigma^{\p{\emptyset}}_{t, \bm{\sigma}, \bm{d}}
      = \sum_{ \xi_2, \ldots, \xi_\fn \in \qqq{0,2}} \sum_{\br{s_2}, \ldots, \br{s_\fn}} \p{\prod_{i=2}^\fn f_{\xi_i}\p{\br{a_i}, \br{s_i}}} \Sigma^{\p{\emptyset}}_{t, \bm{\sigma}, \bm{d}}.
  \]
  We will prove (\ref{Kbound-inductive-hypothesis}) for each fixed sequence $\p{\xi_2, \ldots, \xi_\fn}$. We divide the proof into the following cases. 
  \begin{enumerate}
 \item 
If $\xi_i = 2$ for at least one $i$, then using \eqref{f012}, we get that for any constant $D>0$,
      \begin{align*}
         \sum_{\br{s_2}, \ldots, \br{s_\fn}} \p{\prod_{i=2}^\fn f_{\xi_i}\p{\br{a_i}, \br{s_i}}} \Sigma^{\p{\emptyset}}_{t, \bm{\sigma}, \bm{d}}
          \prec &~\frac{\p{\ell_t^d \eta_t}^{-\fn+2}}{\eta_t + \lambda^2} \Big({\min_{2 \leq i \leq \fn} \avg{\br{a_i} - \br{d_1}}}\Big)^{-d}  \sum_{\br{d_2}, \ldots, \br{d_\fn}} \abs{\Sigma^{\p{\emptyset}}_{t, \bm{\sigma}, \bm{d}}}+ W^{-D}.  
      \end{align*}
      Here, we also used that the summation over $|[s_i]|\ge (\log W)^2$ is exponentially small since \smash{$\Sigma^{\p{\emptyset}}_{t, \bm{\sigma}, \bm{d}}$} consists of short edges only. Applying \eqref{eq:core} cancels the factor \smash{$\p{\eta_t + \lambda^2}^{-1}$}, thereby yielding (\ref{Kbound-inductive-hypothesis}).


\item Suppose $\xi_i = 1$ for at least two indices $i$. Without loss of generality, suppose $\xi_2=\xi_3=1$. Then, using \eqref{f012} and \eqref{eq:core}, we find that for any constant $D>0$,
      \begin{align}\label{eq:twoone}
     \sum_{\br{s_2}, \ldots, \br{s_\fn}} \p{\prod_{i=2}^\fn f_{\xi_i}\p{\br{a_i}, \br{s_i}}} \Sigma^{\p{\emptyset}}_{t, \bm{\sigma}, \bm{d}}\prec &~ \frac{\p{\ell_t^d \eta_t}^{-\fn+3}}{\eta_t + \lambda^2} \frac{\mathbf 1(|[a_2]-[d_1]|\le (\log W)^2\ell_t)}{\avg{[a_2]-[d_1]}^{d-1}\avg{[a_3]-[d_1]}^{d-1}}  + W^{-D},
      \end{align}
      where we also used the exponential decay of the $\Theta_t$-propagators beyond the scale $\ell_t$ due to \eqref{prop:ThfadC}. By using $\eta_t + \lambda^2\gtrsim \eta_t\ell_t^2$, we can derive \eqref{Kbound-inductive-hypothesis} from \eqref{eq:twoone} for both dimensions $d=1$ and $d=2$. 
      

\item If $\xi_k=1$ for some $2\le k\le \fn$ and all other $\xi_i$'s are equal to zero, then using $f_1([a_k],[s_k])=-f_1([a_k],-[s_k])$, we see that the corresponding term vanishes.


\item Finally, if $\xi_2 = \cdots = \xi_\fn = 0$, then using \eqref{f012}, the exponential decay of the $\Theta_t$-propagators, and the first identity in \eqref{eq:core}, we get that for any constant $D>0$,
      \begin{align*}
        &\sum_{\br{s_2}, \ldots, \br{s_n}} \p{\prod_{i=2}^\fn f_0\p{\br{a_i}, \br{s_i}}} \Sigma^{\p{\emptyset}}_{t, \bm{\sigma}, \bm{d}}
          = \p{\prod_{i=2}^\fn f\p{\br{a_i}, \br{0}}} \sum_{\br{d_2}, \ldots, \br{d_\fn}} \Sigma^{\p{\emptyset}}_{t, \bm{\sigma}, \bm{d}} \\
          &\prec \eta_t\p{\ell_t^d \eta_t}^{-\fn+1}\mathbf 1\p{|[a_2]-[d_1]|\le (\log W)^2\ell_t} +W^{-D}\prec  \p{\ell_t^d \eta_t}^{-\fn+2}\avg{[a_2]-[d_1]}^{-d},
      \end{align*}
      which implies (\ref{Kbound-inductive-hypothesis}).
\end{enumerate}

  \medskip  
  In sum, we have shown that (\ref{Kbound-inductive-hypothesis}) holds when $\pi = \emptyset$, which also establishes the base case $\fn=4$ for the induction argument.   
  For a general $\fn\ge 4$ and $\pi\ne \emptyset$, we apply \Cref{core-decomposition} and use the induction hypothesis. Specifically, suppose we have the decomposition \eqref{recursive-decomposition-Sigma}, whose block reduction takes the form
  \begin{equation}\label{recursive-decomposition-Sigma3}
\Sigma^{\p{\pi}}_{t,\bm{\sigma},\bm{d}} = \sum_{c, c'} \Sigma^{(\pi^{\p{c}})}_{t,\bm{\sigma}^{\p{c}},\bm{d}^{\p{c}}} \p{S_t^{\LK}\Theta_t^{\p{\sigma_1, \sigma_r}}}_{[c][c']}
\Sigma^{(\pi^{(c')})}_{t,\bm{\sigma}^{(c')},\bm{d}^{(c')}} ,
  \end{equation}
 where $\pi^{(c)}$ and $\pi^{(c')}$ are defined as in \eqref{eq:pic}, and $(\bm{\sigma}^{\p{c}}, \bm{d}^{\p{c}})$ and $(\bm{\sigma}^{(c')}, \bm{d}^{(c')})$ now take the form 
  \begin{align}
  &\big(\bm{\sigma}^{\p{c}}, \bm{d}^{\p{c}}\big)
        \coloneqq \cutL^{\p{[c]}}_{1r}\p{\bm{\sigma}, \bm{d}}, \ \
    \big(\bm{\sigma}^{(c')}, \bm{d}^{(c')}\big)
        \coloneqq \cutR^{([c'])}_{1r}\p{\bm{\sigma}, \bm{d}}.\label{eq:sigmac2}
  \end{align}
By the induction hypothesis \eqref{Kbound-inductive-hypothesis}, we have 
  \begin{align*}
      \sum_{\bm{d}^{(c)}\setminus \{[c]\}} \p{\prod_{i=r}^{\fn} \Theta^{\p{\sigma_i, \sigma_{i+1}}}_{t,\br{a_i}[d_{i}]}}\Sigma^{(\pi^{(c)})}_{t,\bm{\sigma}^{(c)},\bm{d}^{(c)}}
      &\prec \frac{\p{\ell_t^d\eta_t}^{-\fn+r}}{\big(\min_{r \leq i \leq \fn} \avg{\br{a_i} - \br{c}}\big)^{d}} \, ,\\
      \sum_{\bm{d}^{(c')}\setminus \{[d_{1}]\}} \p{\Theta^{\p{\sigma_r, \sigma_{1}}}_{t,[c][c']}\prod_{i=2}^{r-1} \Theta^{\p{\sigma_i, \sigma_{i+1}}}_{t,\br{a_i}[d_{i}]}}    \Sigma^{(\pi^{(c')})}_{t,\bm{\sigma}^{(c')},\bm{d}^{(c')}}
      &\prec \frac{\p{\ell_t^d\eta_t}^{-r+2}}{\big(\avg{[c] - [d_{1}]} \wedge\min_{2 \leq i \leq r-1} \avg{[a_i] - [d_{1}]}\big)^{d}} \, .
  \end{align*}
Then, with the decomposition \eqref{recursive-decomposition-Sigma3} and the above two estimates, 
we can write the LHS of \eqref{Kbound-inductive-hypothesis} as 
\begin{align*}
& \sum_{\bm{d}^{(c)}\setminus \{[c]\}} \sum_{\bm{d}^{(c')}\setminus \{[d_{1}]\}}  \p{\prod_{i=r}^{\fn} \Theta^{\p{\sigma_i, \sigma_{i+1}}}_{t,\br{a_i}[d_{i}]}}\Sigma^{(\pi^{(c)})}_{t,\bm{\sigma}^{(c)},\bm{d}^{(c)}} 
\p{t\Theta^{\p{\sigma_r, \sigma_{1}}}_{t,[c][c']}\prod_{i=2}^{r-1} \Theta^{\p{\sigma_i, \sigma_{i+1}}}_{t,\br{a_i}[d_{i}]}}    \Sigma^{(\pi^{(c')})}_{t,\bm{\sigma}^{(c')},\bm{d}^{(c')}}\\
&      \prec \p{\ell_t^d\eta_t}^{-\fn+2}\sum_{[c]}\Big(\min_{r \leq i \leq \fn} \avg{\br{a_i} - \br{c}}\Big)^{-d}\Big(\avg{[c] - [d_{1}]} \wedge\min_{2 \leq i \leq r-1} \avg{[a_i] - [d_{1}]}\Big)^{-d}\\
&\prec \p{\ell_t^d\eta_t}^{-\fn+2} \Big(\min_{2 \leq i \leq \fn} \avg{\br{a_i} - [d_{1}]}\Big)^{-d},
\end{align*}  
which gives the inductive assumption \eqref{Kbound-inductive-hypothesis}. This completes the proof of \Cref{ML:Kbound}. 
\end{proof}

\section{Proof of main results}\label{sec:proof}

The main results can be derived easily from the following key theorems about the $G$-loop estimates.

\begin{theorem}[$G$-loop estimates]\label{ML:GLoop}
For the block Anderson model in \Cref{def: BM}, suppose the assumptions of \Cref{zztE} hold with $t_0\le 1 - W^{\fd}\eta_*$, where $\fd$ is an arbitrarily small constant. Then, for each fixed $\fn\in \N$, the following estimates hold uniformly in $t\in [0,t_0]$: 
\be\label{Eq:L-KGt}
 \max_{\boldsymbol{\sigma}, \ba}\left|{\cL}^{(\fn)}_{t, \boldsymbol{\sigma}, \ba}-{\cal K}^{(\fn)}_{t, \boldsymbol{\sigma}, \ba}\right|\prec (W^d\ell_t^d\eta_t)^{-\fn} .  
 \ee
 Together with \eqref{eq:bcal_k}, it implies the following estimate:
\be
\max_{\boldsymbol{\sigma}, \ba}\left|{\cal L}^{(\fn)}_{t, \boldsymbol{\sigma}, \ba} \right|\prec (W^d\ell_t^d\eta_t)^{-\fn+1}. \label{Eq:L-KGt2} 
\ee
\end{theorem}

\begin{theorem}[$2$-$G$ loop estimates]\label{ML:GLoop_expec}
In the setting of \Cref{ML:GLoop}, the expectation of a $2$-$G$ loop satisfies a better bound uniformly in $t\in [0,t_0]$: 
\begin{equation}\label{Eq:Gtlp_exp}
 \max_{\boldsymbol{\sigma}, \ba}\left|\mathbb E{\cal L}^{(2)}_{t, \boldsymbol{\sigma}, \ba}-{\cal K}^{(2)}_{t, \boldsymbol{\sigma}, \ba}\right|\prec (W^d\ell_t^d\eta_t)^{-3}
 . 
\end{equation}
Moreover, for $\boldsymbol{\sigma}=(+,-)$ and $ \ba=([a_1], [a_2])$, we have the following decay estimate uniformly in $t\in [0,t_0]$: 
for any large constant $D>0$, 
\begin{equation}\label{Eq:Gdecay}
 \left| {\cal L}^{(2)}_{t, \boldsymbol{\sigma}, \ba}-{\cal K}^{(2)}_{t, \boldsymbol{\sigma}, \ba}\right|\prec (W^d\ell_t^d\eta_t)^{-2}\exp \left(-\left|\frac{ [a_1]-[a_2] }{\ell_t}\right|^{1/2}\right)+W^{-D}. 
\end{equation}
\end{theorem}

\begin{theorem}[Local law for $G_t$]\label{ML:GtLocal}
In the setting of \Cref{ML:GLoop}, the following local law holds uniformly in $t\in [0,t_0]$ (recall that we abbreviate $\lambda\equiv \lambda_0$ and $E\equiv E_0$, as mentioned below \Cref{zztE}): 
\begin{equation}\label{Gt_bound}
 \|G_{t}-M(E,\lambda) \|_{\max} \prec (W^d\ell_t^d\eta_t)^{-1/2}.
   \end{equation} 
\end{theorem}


We now outline the strategy for the proofs of Theorems \ref{ML:GLoop}, \ref{ML:GLoop_expec} and \ref{ML:GtLocal}. At $t=0$, we have $G_{0}(\sigma)=M(\sigma)$. Together with Definitions \ref{Def:G_loop} and \ref{Def_Ktza}, it implies that for any fixed $\fn\in \N$:   
$${\cL}^{(\fn)}_{0, \boldsymbol{\sigma},\ba}= {\cK}^{(\fn)}_{0, \boldsymbol{\sigma},\ba} = \cM^{(\fn)}_{ \boldsymbol{\sigma},\ba},\quad \forall \; \boldsymbol{\sigma}\in \{+,-\}^\fn,\ \ \ba\in (\Zn)^\fn. 
$$
Now, for $t\in[0,t_0]$, we will establish the following theorem.

\begin{theorem}\label{lem:main_ind}For the block Anderson model in \Cref{def: BM}, suppose the assumptions of \Cref{zztE} hold with $t_0\le 1 - W^{\fd}\eta_*$.
Suppose that the estimates \eqref{Eq:L-KGt}, \eqref{Eq:Gtlp_exp}, \eqref{Eq:Gdecay}, and \eqref{Gt_bound} hold at some fixed $s\in [0,t_0]$: 
\begin{itemize}
\item[(a)] {\bf $G$-loop estimate}: for each fixed $\fn\in \N$, 
\be\label{Eq:L-KGt+IND}
 \max_{\boldsymbol{\sigma}, \ba}\left|{\cal L}^{(\fn)}_{s, \boldsymbol{\sigma}, \ba}-{\cal K}^{(\fn)}_{s, \boldsymbol{\sigma}, \ba}\right|\prec (W^d\ell_s^d\eta_s)^{-\fn};
\ee
\item[(b)] {\bf 2-$G$ loop estimate}: 
for $\boldsymbol{\sigma}\in\{(+,-),(-,+)\}$ and $ \ba=([a_1],[a_2]),$ and for any large constant $D>0$,
\be\label{Eq:Gdecay+IND}
\left| {\cal L}^{(2)}_{s, \boldsymbol{\sigma}, \ba}-{\cal K}^{(2)}_{s, \boldsymbol{\sigma}, \ba}\right|\prec (W^d\ell_s^d\eta_s)^{-2}\exp \left(- \left|\frac{ [a_1]-[a_2] }{\ell_s}\right|^{1/2}\right)+W^{-D}  \, ;
\ee
\item[(c)] {\bf Local law}: 
\be \label{Gt_bound+IND}
 \|G_{s}-M(E,\lambda)\|_{\max} \prec (W^d\ell_s^d\eta_s)^{-1/2};
\ee

\item[(d)] {\bf Expected $2$-$G$ loop estimate}:
 \be \label{Eq:Gtlp_exp+IND}
 \max_{\boldsymbol{\sigma}, \ba}\left|\mathbb E{\cal L}^{(2)}_{s, \boldsymbol{\sigma}, \ba}-{\cal K}^{(2)}_{s, \boldsymbol{\sigma}, \ba}\right|\prec (W^d\ell_s^d\eta_s)^{-3}. 
\ee
\end{itemize}
Then, for any $t\in [s,t_0]$ satisfying that
\begin{equation}\label{con_st_ind}
(W^d\ell_t^d\eta_t)^{-\frac{1}{100}} \le  \frac{1-t}{1-s} < 1, 
\end{equation}
the estimates \eqref{Eq:L-KGt}, \eqref{Eq:Gtlp_exp}, \eqref{Eq:Gdecay}, and \eqref{Gt_bound} hold. 
In addition, if we do not assume \eqref{Eq:Gtlp_exp+IND}, then the estimates \eqref{Eq:L-KGt}, \eqref{Eq:Gdecay}, and \eqref{Gt_bound} still hold at $t$. 
\end{theorem}

With \Cref{lem:main_ind}, we can prove Theorems \ref{ML:GLoop}, \ref{ML:GLoop_expec}, and \ref{ML:GtLocal} easily by induction in $t$. The proof of \Cref{lem:main_ind} will be divided into six steps, where the details for each step will be provided in Section \ref{Sec:Stoflo}. 

\medskip 
\noindent 
\textbf{Step 1} (A priori $G$-loop bound): We will show that $\fn$-$G$ loops satisfy the a priori bound 
 \begin{equation}\label{lRB1}
   {\cal L}^{(\fn)}_{u,\boldsymbol{\sigma}, \ba}\prec (\ell_u^d/\ell_s^d)^{(\fn-1)}\cdot 
   (W^d\ell_u^d\eta_u)^{-\fn+1},\quad  \forall s\le u\le t.
\end{equation}
Furthermore,  the following weak local law holds: 
\begin{equation}\label{Gtmwc}
    \|G_u-M\|_{\max}\prec  (W^d\ell_u^d\eta_u)^{-1/4},\quad \forall s\le u\le t .
\end{equation} 

\medskip
\noindent 
\textbf{Step 2} (Sharp local law and a priori $2$-$G$ loop estimate): 
The following sharp local law holds: 
\begin{equation}\label{Gt_bound_flow}
     \|G_u-M\|_{\max}\prec  (W^d \ell_u^d \eta_u)^{-1/2},\quad \forall s\le u\le t.
\end{equation} 
Hence, the local law  \eqref{Gt_bound} holds at time $t$.   
In addition, for all $\boldsymbol{\sigma}\in\{(+,-),(-,+)\}$, $ \ba=([a_1],[a_2])\in(\Zn)^2$, and $u\in[s,t]$, we have 
\begin{equation}\label{Eq:Gdecay_w}
\left| {\cal L}^{(2)}_{u, \boldsymbol{\sigma}, \ba}-{\cal K}^{(2)}_{u, \boldsymbol{\sigma}, \ba}\right| \prec \left(\eta_s/\eta_u\right)^4\cdot (W^d\ell_u^d\eta_u)^{-2}\exp \left(- \left|\frac{ [a_1]-[a_2] }{\ell_u}\right|^{1/2}\right)+W^{-D}
\end{equation}
for any large constant $D>0$.   

   \medskip
 \noindent 
\textbf{Step 3}   (Sharp $G$-loop bound): 
 The following sharp bound on $\fn$-$G$ loops holds for each fixed $\fn\in \N$: 
\begin{equation}\label{Eq:LGxb}
\max_{\boldsymbol{\sigma}, \ba}\left| {\cal L}^{(\fn)}_{u, \boldsymbol{\sigma}, \ba} \right|
\prec 
  (W^d\ell_u^d\eta_u)^{-\fn+1} ,\quad \forall s\le u\le t .
\end{equation}

\medskip
 \noindent 
\textbf{Step 4}  (Sharp $(\cL-\cK)$-loop estimate): 
The following sharp estimate on $ ({\cal L}-{\cal K})$-loops holds for each fixed $\fn\in \N$: 
\begin{equation}\label{Eq:L-KGt-flow}
 \max_{\boldsymbol{\sigma}, \ba}\left|{\cal L}^{(\fn)}_{u, \boldsymbol{\sigma}, \ba}-{\cal K}^{(\fn)}_{u, \boldsymbol{\sigma}, \ba}\right|\prec (W^d\ell_u^d\eta_u)^{-\fn},\quad \forall s\le u\le t .
\end{equation}
Hence, the $G$-loop estimate \eqref{Eq:L-KGt} holds at time $t$.

\medskip
 \noindent 
\textbf{Step 5}  (Sharp 2-$G$ loop estimate): For all $\boldsymbol{\sigma}\in\{(+,-),(-,+)\}$, $ \ba=([a_1],[a_2])\in(\Zn)^2,$ and $u\in[s,t]$, the following estimate holds for any large constant $D>0$:
\begin{equation}\label{Eq:Gdecay_flow}
\left| {\cal L}^{(2)}_{u, \boldsymbol{\sigma}, \ba}-{\cal K}^{(2)}_{u, \boldsymbol{\sigma}, \ba}\right| \prec   (W^d\ell_u^d\eta_u)^{-2}\exp \left(- \left|\frac{ [a_1]-[a_2] }{\ell_u}\right|^{1/2}\right)+W^{-D}. 
\end{equation}
It shows that the estimate \eqref{Eq:Gdecay} holds at time $t$.

\medskip
\noindent 
\textbf{Step 6} (Expected 2-$G$ loop estimate): 
The following estimate holds: 
\begin{equation}\label{Eq:Gtlp_exp_flow}
 \max_{\boldsymbol{\sigma}, \ba}\left|\mathbb E{\cal L}^{(2)}_{u, \boldsymbol{\sigma}, \ba}-{\cal K}^{(2)}_{u, \boldsymbol{\sigma}, \ba}\right|\prec (W^d\ell_u^d\eta_u)^{-3},\quad  
 \forall s \le u \le t .
\end{equation}
Hence, the estimate \eqref{Eq:Gtlp_exp} holds at time $t$. In Steps 1--5, the induction hypothesis \eqref{Eq:Gtlp_exp+IND} will not be used. In other words, the estimates \eqref{Eq:L-KGt}, \eqref{Eq:Gdecay}, and \eqref{Gt_bound} still hold at $t$ without assuming  \eqref{Eq:Gtlp_exp+IND}.

\medskip 

We remark that the estimates established in each step hold uniformly in $u\in[s,t]$ (recall \Cref{stoch_domination}) due to a standard $N^{-C}$-net argument. For simplicity of presentation, we will not emphasize this uniformity at every step of the proof.

Before concluding this section, we present the proof of our main results, \Cref{thm_locallaw} and \Cref{thm_diffu}.
\begin{proof}[\bf Proof of \Cref{thm_locallaw} and \Cref{thm_diffu}]
For a fixed $z=E+\ii \eta$ with $|E|\le 2-\kappa$ and $W^\e\eta_* \le \eta \le 1$, we can choose the deterministic flow as in \Cref{zztE} such that \eqref{eq:zztE} and \eqref{GtEGz} hold. Then, with \eqref{GtEGz} and the fact that $\ell_{t_0}\asymp \ell(\eta)$ by the second estimate in \eqref{eq:E0bulk}, we observe that \Cref{ML:GtLocal}, \Cref{ML:GLoop} (with $\fn=1$), and \Cref{ML:GLoop_expec} respectively give the entrywise local law \eqref{locallaw}, the averaged local law \eqref{locallaw_aver}, and the quantum diffusion estimates in \Cref{thm_diffu} (recall \eqref{eq:avg_GTheta}) at each \emph{fixed $z$}. To extend these estimates uniformly to all $z$, we can use a standard $N^{-C}$-net argument, whose detail we omit. 
\end{proof}

\section{Resolvent entry estimates}\label{appd:MDE}

Our proof of \Cref{lem:main_ind} depends crucially on the following lemma, which bounds the resolvent entries through 2-$G$ loops.  

\begin{lemma}\label{lem_GbEXP}
In the setting of \Cref{lem:main_ind}, fix any $t\in [0,t_0]$ and suppose we have the following estimates for any $\Psi$ satisfying the assumptions \eqref{eq:Psi1D}--\eqref{sym_cond}: there exists a small constant $\e_0>0$ such that 
\begin{equation}\label{def_asGMc}
\|G_t - M\|_{\max} \prec W^{-\e_0},
\end{equation}
and there exist deterministic control parameters $\Psi_t>0$ and $\Phi_t([a],[b])>0$ such that  
\be\label{eq:def_Psit}
\max_{[a],[b]\in \Zn} \cL^{(2)}_{t,(+,-),([a],[b])}+W^{-d}\prec \Psi_t^2,\quad {\cal L}^{(2)}_{t, (-,+), ([a], [b])}\prec \Phi_t([a],[b]),\quad \forall [a],[b]\in \Zn .
\ee
Then, the following estimates hold: 
\begin{itemize}
    \item[(a)] {\bf Entrywise local law}: We have 
    \begin{align}\label{GiiGEX} 
 \|G_t-M\|_{\max} \prec \Psi_t \, . 
\end{align}
\item[(b)] {\bf Averaged local law}: For any $[a],[b]\in \Zn$, we have that (recall $I_{[a]}$ defined in \eqref{def:Ia})
\begin{align}\label{GavLGEX}
\left\langle  \left(G_t - M\right)E_{[a]}\right\rangle \prec \Psi_t^2, \quad  W^{-d} \left\langle  \left(G_t - M\right)I_{[a]}MI_{[b]}\right\rangle  \prec \Psi_t^2\|I_{[a]}MI_{[b]}\| \, . 
\end{align}
\item[(c)] {\bf Entrywise decay estimate}:
There exists a constant $c>0$ depending on $\e_A$ (recall \eqref{eq:cond_A12}) such that for any large constant $D>0$, we have 
\begin{align}\label{GijGEX}
\max_{x\in [a], y \in [b]} |(G_t-M)_{xy}|^{2}  \prec & \sum_{[a'], [b'] \in \Zn}  \Phi_t([a'],[b']) W^{-c(|[a']-[a]|+|[b']-[b]|)} \nonumber\\
&+\Psi_t^2 W^{-c|[a]-[b]|}+W^{-D}, \quad  \forall \ [a],[b]\in \Zn \, .
\end{align}
\end{itemize}
\end{lemma}

Our primary focus is on the proof of \eqref{GijGEX}, while the proofs of \eqref{GiiGEX} and \eqref{GavLGEX} basically follow the argument developed in \cite{He2018} as we will explain later. We first introduce some notations. 

Since the time $t$ does not play any role in the following proof, we will abbreviate $G\equiv G_t$, $V\equiv V_t$, and $z\equiv z_t$ for clarity of presentation. 
Moreover, we will use the following simplified notation of generalized matrix entries: given a matrix $\cal A$ and any vectors $\mathbf u,\bv$, we denote $ \cal A_{\mathbf u\mathbf v}:= \bu^* \cal A\bv$ and $\cal A_{x\mathbf v}:= \mathbf e_x^* \cal A\bv$, where $\mathbf e_x$ is the standard basis unit vector along the $x$-th direction. For $k\ge 0$, we define a sequence of subsets $\cal V_{k}$ of vectors inductively as follows. First, let subset $\Omega_0:=\{\mathbf e_x: x\in \ZL\}$. Next, given \smash{$\Omega_k$}, we define \smash{$\Omega_{k+1}$} as 
\begin{align}\label{eq:Omegak+1}
\Omega_{k+1}:=\cup_{[a]} \Omega_{k+1}([a]) ,\quad \text{where}\quad \Omega_{k+1}([a]):=\left\{  I_{[a]}\Psi\bv:\bv\in \Omega_{k}\right\}.
\end{align}
Finally, we define 
\[\cal V_k:=\cup_{i=1}^k \Omega_i,\quad \cal V_k([a]):=\{\mathbf e_x:x\in[a]\}\cup \left(\cup_{i=1}^k \Omega_i([a])\right).\] 
By removing the zero vector, we can assume that all vectors in $\cal V_k$ are \emph{nonzero}.  
Similar to \cite[Section 5.1]{EKS_Forum}, we introduce the following notation of $\vertiii{\cdot}_{(k)}$-norms with respect to the vectors in $\cal V_k$: given an arbitrary $N\times N$ matrix $\cal A$, we define 
\be\label{eq:Aknorm}\vertiii{\cal A}_{(k)}:= \max_{\bv,\bw\in\cal V_{k}}\frac{  |\cal A_{\bv \bw}|}{\|\bv\|_2\|\bw\|_2}  .\ee
For our proof below, we define a class of deterministic control parameters $\xi_{k}\ge W^{-d/2}$ such that
\be\label{eq:controlxi0} 
\vertiii{G-M}_{(k)}\prec \xi_k,  \quad \forall k\ge 0. 
\ee
Without loss of generality, we may assume that these parameters $\xi_k$ are \emph{non-decreasing} in $k$. 
Moreover, due to the trivial bound $\|G\|\le (\im z_t)^{-1}\ll N$, we can always take $\xi_k\le N$.  
Using the assumption \eqref{eq:cond_A12}, the following expansion of $M$:
\be\label{eq:expandM}  M=-\sum_{k=1}^\infty \frac{(\lambda \Psi)^{k}}{(z+m)^{k+1}},
\ee
and the fact that $\Psi|_{[x][y]}=0$ for $|[x]-[y]|>1$, we can derive that for any $\bv,\bw\in \cal V_k$, small constant $c\in (0,\e_A)$, and large constant $D>0$,
\be\label{eq:vktovk+1}
\begin{split}
\frac{|\bv^*(G-M)M\bw|}{\|\bv\|_2\|\bw\|_2 } &\prec \sum_{l=0}^\infty (2dCW^{-\e_A})^l \xi_{k+l}  + W^{-D} \lesssim \sum_{l=0}^\infty W^{-cl} \xi_{k+l}+ W^{-D}.
\end{split}
\ee
This bound also applies to \(|\bv^*(G-M)I_{[a]}M\bw|\), \(|\bw^*M(G-M)\bv|\), and \(|\bw^*MI_{[a]}(G-M)\bv|\), as well as to all analogous terms obtained by replacing each occurrence of  $(G-M)$ or $M$ with $(G-M)^*$ or $M^*$, respectively.

Our proof will also use the $T$-variables defined as 
\be\label{eq:Tvariable}
T_{xy}:=\sum_{\al}S_{x\al}|G_{\al y}|^2=(G^*E_{[x]}G)_{yy},\quad \wt T_{xy}:=(GE_{[y]}G^*)_{xx},\ee
which has been a key object in previous studies of random band matrices (see, e.g., \cite{delocal,PartI,PartII,Band1D_III,BandI,BandII,BandIII,RBSO}). 
We also introduce \emph{generalized $T$-variables} for an arbitrary deterministic vector $\bv\in \C^N$, defined as
\be\label{eq:gen_T}
T_{x\bv}:=\sum_{\al}S_{x\al}|G_{\al \bv}|^2=(G^*E_{[x]}G)_{\bv\bv},\quad \wt T_{x\bv}:=\sum_{\al}S_{x\al}|G_{\al \bv}|^2=(GE_{[x]}G^*)_{\bv\bv}.
\ee
It is easy to check by definition and the bound \eqref{Mbound_AO} that 
\be\label{eq:T(G-M)}
T_{xy}\prec \|G-M\|_{\max}^2+W^{-d}\prec \xi_0^2,\quad T_{x\bv}\prec \max_{x}|(G-M)_{x\bv}|^2+W^{-d}\|\bv\|_2^2 \, ,
\ee 
and similar bounds for \smash{$\wt T_{xy}$ and $\wt T_{x\bv}$}. 
Throughout the proof, we will use the following key observation, which follows from the rotational invariance of the GUE blocks of $V$. Suppose we have established that $T_{xy} \prec \Psi^2$ for some deterministic parameter $\Psi > 0$, and let $\bv$ be an arbitrary deterministic unit vector in some $\cal V_k([y])$. 
Then, there exists a $W^d\times W^d$ deterministic unitary matrix $U$ such that $(I_n \otimes U)\bv = \mathbf{e}_y$, which implies 
\be\label{eq:Txv} T_{x\bv} =\mathbf e_y^* \left(\cal U G^*\cal U^*\right) E_{[x]}\left(\cal U  G \cal U^*\right) \mathbf e_y, \quad \text{where}\quad \cal U:=I_n\otimes U. \ee
Under the unitary transformation by $\cal U$, the matrix $\mathcal{U} V \mathcal{U}^*$ has the same distribution as $V$, due to the rotational invariance of the GUE blocks. Moreover, $\mathcal{U} \Psi \mathcal{U}^*$ satisfies the same assumptions \eqref{eq:Psi1D}--\eqref{eq:cond_A12} as the original $\Psi$. 
Regarding the symmetry condition \eqref{sym_cond}, we observe that the Hermitian properties of $A_1$ and $A_2$ are preserved under unitary transformation. Additionally, when $\Psi$ is symmetric, the matrix $U$ is (real) orthogonal, and thus $\mathcal{U} \Psi \mathcal{U}^*$ remains symmetric. 
As a consequence, the same bound applies to  $T_{x\bv}$. Furthermore, by the charge symmetry of $G$ under the transformation $z\mapsto \bar z$, an analogous bound also holds for \smash{$\wt T_{x\bv}$}. This yields that
\be\label{eq:Txy-Txv}
T_{xy}\prec \Psi^2 \implies T_{x\bv} + \wt T_{x\bv}\prec \Psi^2 \|\bv\|_2^2,\quad \forall \bv\in \cal V_k([y]), \ k\ge 0 .
\ee

\subsection{Proof of \eqref{GijGEX}}\label{appd:MDE1} 

Before proceeding to the proof of \eqref{GijGEX}, we first record the following a priori estimate.
\begin{lemma}\label{lem:apriori}
    In the setting of \Cref{lem_GbEXP}, suppose $\|G-M\|_{\max} \prec \xi_0$ for a deterministic control parameter $W^{-d/2}\le\xi_0\le W^{-\e_0}$. Then, for each fixed $k\ge 1$, we have that \smash{$\vertiii{G-M}_{(k)}\prec \xi_0$}.
\end{lemma}
\begin{proof}
Under Definitions \ref{def_flow} and  \ref{Def:stoch_flow}, we can write $G_t$ and $M$ as
$$ G_t = (\lambda\Psi + V_t - z_t)^{-1},\quad M=(\lambda \Psi - E - m)^{-1} = (\lambda \Psi - z_t - t m)^{-1},$$
which gives the following relation:
\be\label{G-M}
G_t - M = - M(tm + V_t) G_t.
\ee 
Fix any $\bv,\bw\in \cal V_k$. We bound the moments $\E |(G-M)_{\bv \bw}|^{2p}=\E |\Gc_{\bv \bw}|^{2p}$ for any fixed $p\in \N$, where we recall the notation \eqref{Eq:defwtG}. Using \eqref{G-M} and Gaussian integration by parts, we obtain that 
\begin{align*} \E |\Gc_{\bv \bw}|^{2p}&= -\E\left[(M(tm+V_t)G)_{\bv \bw}(\Gc_{\bv \bw})^{p-1} (\Gc^{-}_{\bv \bw})^{p}\right]\\
&= t\sum_{[a]}\sum_{\al\in [a]}M_{\bv \al}\E\left[\qq{(G-M)E_{[a]}} G_{\al \bw} (\Gc_{\bv \bw})^{p-1} (\Gc^{-}_{\bv \bw})^{p}\right] \\
& + (p-1)t\sum_{\al,\beta}M_{\bv \al}S_{\al\beta}\E\left[G_{\al \bw}G_{\bv \beta}G_{\beta \bw} (\Gc_{\bv \bw})^{p-2} (\Gc^{-}_{\bv \bw})^{p}\right]\\
&+ pt\sum_{\al,\beta}M_{\bv \al}S_{\al\beta}\E\left[|G_{\beta \bw}|^2 \overline G_{\bv \al} |\Gc_{\bv \bw}|^{2p-2} \right]=:I_1+I_2+I_3, \label{eq:Gcvy}\numberthis
\end{align*}
where $\Gc^{-}_{\bv \bw}:=\Gc^*_{\bw\bv}$ denotes the complex conjugate of $\Gc_{\bv y}$.
Using \eqref{Mbound_AO} and \eqref{eq:vktovk+1}, we obtain that 
\begin{align}\label{eq:Maddv}
    \sum_{\al\in [a]}M_{\bv\al} G_{\al \bw}= (MI_{[a]}M)_{\bv \bw} +\bv^*MI_{[a]}(G-M)  \bw \prec \p{1+ \zeta_{k,c}}\|\bv\|_2\|\bw\|_2, 
\end{align}
where we abbreviate \smash{\(\zeta_{k,c} :=\sum_{l=0}^\infty \xi_{k+l}e^{-cl}\)}. 
Combining \eqref{eq:Maddv} with the estimate $\qq{(G-M)E_{[a]}}\prec \xi_0$ gives the following bound on $I_1$:
\[I_1\prec \xi_0(1+ \zeta_{k,c})\|\bv\|_2\|\bw\|_2\cdot \E |\Gc_{\bv y}|^{2p-1} .\]

For $I_2$ and $I_3$, by using \eqref{eq:T(G-M)} and \eqref{eq:Txy-Txv}, we have that for any fixed $k\ge 0$,
\be\label{eq:controlxi1} 
\max_{x}\left(T_{x\bv}+\wt T_{x\bv}\right)\prec \xi_0^2\|\bv\|_2^2  ,  \quad \forall \bv \in \cal V_k. 
\ee
Using \eqref{eq:Maddv} and \eqref{eq:controlxi1}, we obtain that 
\begin{align*}
 \sum_{\al,\beta\in [a]}M_{\bv\al}S_{\al\beta} G_{\al \bw} G_{\bv \beta}G_{\beta \bw} \le \big|(MI_{[a]} G)_{\bv \bw}\big|(G^*E_{[a]}G)_{\bw\bw}^{1/2}(GE_{[a]}G^*)_{\bv\bv}^{1/2}\prec \xi_0^2(1 + \zeta_{k,c})\|\bv\|_2^2\|\bw\|_2^2. 
\end{align*}
Thus, we can bound $I_2$ by 
\[I_2\prec \xi_0^2(1+ \zeta_{k,c})\|\bv\|_2^2\|\bw\|_2^2\cdot \E |\Gc_{\bv y}|^{2p-2} .\]
A similar argument yields the same bound for $I_3$. In sum, we obtain that 
\begin{align*} \E |\Gc_{\bv \bw}|^{2p}\prec \xi_0(1+ \zeta_{k,c})\|\bv\|_2\|\bw\|_2 \E |\Gc_{\bv \bw}|^{2p-1} + \xi_0^2(1+\zeta_{k,c})\|\bv\|_2^2\|\bw\|_2^2  \E |\Gc_{\bv \bw}|^{2p-2}.
\end{align*}
Applying H{\"o}lder's, Young's, and Markov's inequalities, we derive from the above bound that 
\begin{align} \label{eq:Maddv2}
\left(\E |\Gc_{\bv \bw}|^{2p}\right)^{\frac{1}{2p}}\prec \xi_0(1+ \zeta_{k,c})\|\bv\|_2\|\bw\|_2 \implies  \frac{|\Gc_{\bv \bw}|}{\|\bv\|_2\|\bw\|_2}\prec \xi_0(1+\zeta_{k,c}).
\end{align}
Then, taking a union bound over $\bv,\bw\in \cal V_{k}$, we get a self-improving estimate for any fixed $k\ge 0$:
$$\vertiii{G-M}_{(l)}\prec \xi_{l} \ \ \forall l\ge k \Rightarrow \ \vertiii{G-M}_{(k)}\prec \xi_0 (1+\zeta_{k,c})  \lesssim \xi_0\Big(1+\sum_{l=0}^K \xi_{k+l}W^{-cl}\Big) ,$$
where we have taken $K\in \N$ sufficiently large such that $W^{-c(K+1)}N\le 1$. Iterating this estimate for $r$ many steps yields that 
$$\vertiii{G-M}_{(k)}\prec \xi_0 + W^{-(c\wedge\e_0)r}\sum_{l=0}^{rK} \xi_l W^{-(c\wedge\e_0)l}.  $$
This concludes the proof if we take $r>(c\wedge\e_0)^{-1}\log_W N$.     
\end{proof}

We now define another class of deterministic control parameters $0< \xi_k([a],[b])\le \xi_0$ such that the following bounds hold: 
\be\label{eq:controlxi} 
\max_{\bv\in \cal V_k([a])}\max_{\bw\in \cal V_k([b])} \frac{|(G-M)_{\bv \bw}|\vee |(G-M)_{\bw\bv}|}{\|\bv\|_2\|\bw\|_2}\prec \xi_k([a],[b]).\ee
Again, without loss of generality, we may assume that the control parameters $\xi_k([a],[b])$ are \emph{non-decreasing} in $k$.  We will also use a similar estimate as in \eqref{eq:vktovk+1}: for any $\bv\in \cal V_k([x])$, $\bw\in \cal V_k([y])$, and \smash{$[a]\in\Zn$}, 
\be\label{eq:vktovk+1 decay}
\begin{split}
	\frac{|\bv^*MI_{[a]}(G-M)\bw|}{\|\bv\|_2\|\bw\|_2 } &\prec  \mathbf \sum_{l=0}^\infty W^{-cl} \xi_{k+l}([a],[y])1(|[x]-[a]|\le l)+ W^{-D}  \\
	&\lesssim W^{-c|[x]-[a]|} \xi_{k+K}([a],[y])+ W^{-D},
\end{split}
\ee
where we have taken the constant $K\in \N$ sufficiently large depending on $c$, $D$, and $\log_W N$. 
This bound also applies to \(|\bw^*(G-M)I_{[a]}M\bv|\), as well as to all analogous terms obtained by replacing each occurrence of  $(G-M)$ or $M$ with $(G-M)^*$ or $M^*$, respectively.

Given $\bv\in \cal V_k([x])$ and $\bw\in \cal V_k([y])$, we now bound the RHS of equation \eqref{eq:Gcvy}. 
Using \eqref{eq:controlxi} and \eqref{eq:vktovk+1 decay}, we can bound $I_1$ as follows for any small constant $c\in (0,\e_A)$ and large constant $D>0$: 
\begin{align}
I_1&=t\sum_{[a]}(MI_{[a]}M)_{\bv \bw}\E\left[\qq{\Gc E_{[a]}}  (\Gc_{\bv \bw})^{p-1} (\Gc^{-}_{\bv \bw})^{p}\right] + t\sum_{[a]}\E\left[\qq{\Gc E_{[a]}} (MI_{[a]}\Gc)_{\bv \bw} (\Gc_{\bv \bw})^{p-1} (\Gc^{-}_{\bv \bw})^{p}\right] \nonumber \\
& \prec \xi_0 \Big( W^{-c|[x]-[y]|}+ \sum_{[a]}W^{-c|[x]-[a]|}\xi_{k+K}\p{[a],[y]}+W^{-D}\Big)\|\bv\|_2\|\bw\|_2\cdot \E |\Gc_{\bv y}|^{2p-1} \, .\label{eq;I1}
\end{align}
Next, using \eqref{Mbound_AO} and \eqref{eq:vktovk+1 decay}, we can bound the term $I_2$ as 
\begin{align*}
I_2&=(p-1)t \sum_{[a]}\E\left\{\left[(MI_{[a]}M)_{\bv \bw} + (MI_{[a]}\Gc)_{\bv \bw}\right](GE_{[a]}G)_{\bv \bw} (\Gc_{\bv \bw})^{p-2} (\Gc^{-}_{\bv \bw})^{p}\right\} \\
& \prec \|\bv\|_2^2\|\bw\|_2^2 \cdot  \E |\Gc_{\bv \bw}|^{2p-2} \cdot \sum_{[a]}\Big( W^{-c|[x]-[a]|-c|[a]-[y]|}+ W^{-c|[x]-[a]|}\xi_{k+K}\p{[a],[y]}\Big)\\
&\qquad \times \left(\xi_0 W^{-d/2} W^{-c|[a]-[y]|} + W^{-d}W^{-c|[x]-[a]|} \xi_{k+K}\p{[a],[y]}+\xi_0 \cdot \xi_{k}\p{[a],[y]}+W^{-D}\right)  \\
&\prec \left[\xi_0^2 W^{-c|[x]-[y]|}+\xi_0\cal E_{k+K,c/2}^2([x], [y])+W^{-D}\right]\|\bv\|_2^2\|\bw\|_2^2 \cdot \E |\Gc_{\bv \bw}|^{2p-2} , \numberthis \label{eq;I2}
\end{align*}
where in the second step, we use the following estimate by \eqref{Mbound_AO}, \eqref{eq:T(G-M)}, \Cref{lem:apriori}, and \eqref{eq:vktovk+1 decay}:
\begin{align*}
&(GE_{[a]}G)_{\bv \bw}= (GE_{[a]}M)_{\bv \bw}+(ME_{[a]}\Gc)_{\bv \bw} +(\Gc E_{[a]}\Gc)_{\bv \bw}\prec  \|E_{[a]}M\mathbf w\|_2 (GE_{[a]}G^*)_{\bv\bv}^{1/2}\\
&\qquad + W^{-d}\|\bv\|_2\|\bw\|_2 \cdot W^{-c|[x]-[a]|}\xi_{k+K}\p{[a],[y]}+ \|\bv\|_2\|\bw\|_2\cdot \xi_k\p{[x],[a]}\xi_{k}\p{[a],[y]} + W^{-D}\|\bv\|_2\|\bw\|_2 \\
&\prec \|\bv\|_2\|\bw\|_2\left(\xi_0 W^{-\frac d2} W^{-c|[a]-[y]|} + W^{-d}W^{-c|[x]-[a]|}\xi_{k+K}\p{[a],[y]}+\xi_k\p{[x],[a]}\xi_{k}\p{[a],[y]}+W^{-D}\right),\numberthis \label{eq:GVG}
\end{align*}
and in the third step, we introduce the function $\cal E_{k,c}$ as 
\be\label{eq:defEkc}\cal E_{k,c}([x], [y]):=\sum_{[\al],[\beta]}W^{-c|[x]-[\al]|}\xi_k\p{[\al],[\beta]}W^{-c|[\beta]-[y]|}\ee
for any $k\ge 0$ and positive constant $c$.  
With a similar argument, we can check that 
\begin{align}
I_3- I_0&\prec \left[\xi_0^2 W^{-c|[x]-[y]|}+\xi_0 \cal E_{k+K,c/2}^2([x], [y])+W^{-D}\right]\|\bv\|_2^2\|\bw\|_2^2\cdot \E |\Gc_{\bv\bw}|^{2p-2},\label{eq;I3}
\end{align}
where $I_0$ is the ``leading term" defined as
\be\label{def;I0}I_0:= pt\sum_{\al,\beta}|M_{\bv\al}|^2 S_{\al\beta}\E\left[|G_{\beta \bw}|^2 |\Gc_{\bv \bw}|^{2p-2}\right].\ee

We continue bounding $I_0$ using \eqref{G-M} and Gaussian integration by parts: 
\begin{align*}
I_0=&~ - pt\sum_{\al,\beta}|M_{\bv \al}|^2 S_{\al\beta}\E\left[(G(tm+V)M)_{\beta \bw} \overline G_{\beta \bw} |\Gc_{\bv \bw}|^{2p-2}\right] \\
= &~pt^2\sum_{\al,\beta}\sum_{[b]}\sum_{\mu\in[b]}|M_{\bv \al}|^2 M_{\mu\bw} S_{\al\beta}\E\left[\qq{(G-M) E_{[b]}} G_{\beta \mu}\overline G_{\beta \bw} |\Gc_{\bv \bw}|^{2p-2}\right]\\
&+ pt^2\sum_{\al,\beta,\gamma, \mu}|M_{\bv \al}|^2 M_{\mu\bw} S_{\al\beta}S_{\gamma \mu}\E\left[ |G_{\beta \gamma}|^2\overline G_{\mu \bw} |\Gc_{\bv \bw}|^{2p-2}\right]\\
&+p(p-1)t^2 \sum_{\al,\beta,\gamma,\mu}|M_{\bv \al}|^2 M_{\mu\bw} S_{\al\beta}S_{\gamma \mu}\E\left[G_{\beta \gamma} \overline G_{\beta \bw} G_{\bv \mu}G_{\gamma \bw}(\Gc_{\bv \bw})^{p-2}(\Gc_{\bv \bw}^-)^{p-1}\right] \\
&+p(p-1)t^2 \sum_{\al,\beta,\gamma,\mu}|M_{\bv \al}|^2 M_{\mu\bw} S_{\al\beta}S_{\gamma \mu}\E\left[ G_{\beta \gamma}\overline G_{\beta \bw} \overline G_{\bv \gamma }\overline G_{\mu \bw}(\Gc_{\bv \bw})^{p-1}(\Gc_{x\bw}^-)^{p-2}\right]\\
=:&~I_{01} + I_{02} + I_{03} + I_{04} \, . 
\end{align*}
Again, by applying \eqref{Mbound_AO} and \eqref{eq:vktovk+1 decay}, and using a similar argument as above for the term $I_2$, we deduce that 
\begin{align}
I_{01}&\prec \left[\xi_0^2W^{-c|[x]-[y]|}+\xi_0 \cal E^2_{k+K,c}([x], [y])+W^{-D} \right]\|\bv\|_2^2\|\bw\|_2^2\cdot \E |\Gc_{\bv\bw}|^{2p-2} , \label{eq;I01}\\
I_{02}-J_0&\prec \left[\xi_0^3 W^{-c|[x]-[y]|}+ \xi_0 \cal E^2_{k+K,c/2}([x], [y])+W^{-D}\right]\|\bv\|_2^2\|\bw\|_2^2\cdot \E |\Gc_{\bv\bw}|^{2p-2} , \label{eq;I02}\\
|I_{03}|+|I_{04}| &\prec \left[\xi_0^3 W^{-c|[x]-[y]|}+\xi_0 \cal E^3_{k+K,c/2}([x], [y])+W^{-D}\right]\|\bv\|_2^3\|\bw\|_2^3\cdot \E |\Gc_{\bv\bw}|^{2p-3} , \label{eq;I0304}
\end{align}
where $J_0$ is the ``leading term" of $I_{02}$ defined as 
$$J_0:=pt^2\sum_{\al,\beta,\gamma, \mu}|M_{\bv\al}|^2 |M_{\mu\bw}|^2 S_{\al\beta}S_{\gamma \mu}\E\left[ |G_{\beta \gamma}|^2 |\Gc_{\bv\bw}|^{2p-2}\right].$$
To illustrate how to get the bounds \eqref{eq;I01}--\eqref{eq;I0304}, we take $I_{03}$ as an example---the same argument applies to $I_{04}$, and the proofs of \eqref{eq;I01} and \eqref{eq;I02} are simpler. 
We need to bound 
\begin{align*}
&~\sum_{[a],[b]}\sum_{\al,\beta\in[a]}\sum_{\gamma,\mu\in[b]}|M_{\bv \al}|^2 M_{\mu\bw} S_{\al\beta}S_{\gamma \mu}G_{\beta \gamma} \overline G_{\beta \bw} G_{\bv \mu}G_{\gamma \bw} \\
=&~\sum_{[a],[b]}\sum_{\al,\beta\in[a]}|M_{\bv \al}|^2  S_{\al\beta}  \overline G_{\beta \bw} (GI_{[b]}M)_{\bv\bw}(GE_{[b]}G)_{\beta \bw}\\
=&~\sum_{[a],[b]}\sum_{\al,\beta\in[a]}|M_{\bv \al}|^2  S_{\al\beta}  \overline G_{\beta \bw} \left[(MI_{[b]}M)_{\bv\bw}+(\Gc I_{[b]}M)_{\bv\bw}\right](GE_{[b]}G)_{\beta \bw}=:J_{1}+J_2.
\end{align*}
For the term $J_1$ with the factor $(MI_{[b]}M)_{\bv\bw}$, using \eqref{Mbound_AO}, \eqref{eq:T(G-M)}, \Cref{lem:apriori}, and the estimate \eqref{eq:GVG} for $(GE_{[b]}G)_{\beta \bw}$, we can bound it by 
\begin{align*}
J_1&\prec \|\bv\|_2^3\|\bw\|_2^2\sum_{[a],[b]} \xi_0^2 W^{-d/2}(G^*E_{[a]}G)_{\bw\bw}^{1/2}\cdot W^{-2c|[x]-[a]|-c|[x]-[b]|-c|[b]-[y]|}\\
&\prec \|\bv\|_2^3\|\bw\|_2^3\cdot \xi_0^3 W^{-d/2} W^{-c|[x]-[y]|}.\numberthis \label{eq:J1}
\end{align*}
For the term $J_2$, we decompose it as
\begin{align*}
J_{21}+J_{22}:=&~\sum_{[a],[b]}\sum_{\al,\beta\in[a]}|M_{\bv \al}|^2  S_{\al\beta}  \left(\overline M_{\beta \bw}+\Gc^-_{\beta \bw}\right) (\Gc I_{[b]}M)_{\bv\bw} (GE_{[b]}G)_{\beta \bw} .
\end{align*}
Again, the term $J_{12}$ can be bounded using \eqref{Mbound_AO}, \Cref{lem:apriori}, and \eqref{eq:GVG}:
\begin{align*}
J_{21}=&~W^{-d}\sum_{[a],[b]}\sum_{\al\in[a]}|M_{\bv \al}|^2  (\Gc I_{[b]}M)_{\bv\bw} (M^* I_{[a]}GE_{[b]}G)_{\bw \bw} \\
\prec &~\|\bv\|_2^3\|\bw\|_2^3 \sum_{[a],[b]} W^{-2c|[x]-[a]|-c|[b]-[y]|-c|[a]-[y]|} \xi_0^3\prec \|\bv\|_2^3\|\bw\|_2^3\cdot \xi_0^3 W^{-c|[x]-[y]|}.\numberthis \label{eq:J12}
\end{align*}
Next, using \eqref{Mbound_AO}, \eqref{eq:vktovk+1 decay}, and \eqref{eq:GVG}, we can control $J_{22}$ by 
\begin{align*}
J_{22}\prec &~\|\bv\|_2^3\|\bw\|_2^3 \sum_{[a],[b]} W^{-2c|[x]-[a]|} \xi_{k}([a],[y])\br{W^{-c|[b]-[y]|}\xi_{k+K}([x],[b])+W^{-D}}\\
&~ \quad  \times\left(\xi_0 W^{-d/2} W^{-c|[b]-[y]|} + W^{-d}W^{-c|[a]-[b]|}\xi_{k+K}\p{[b],[y]}+\xi_k\p{[a],[b]} \xi_{k}\p{[b],[y]}+W^{-D}\right)\\
\prec &~ \|\bv\|_2^3\|\bw\|_2^3 \left(\xi_0^4 W^{-c|[x]-[y]|}+\xi_0\cal E_{k+K,c/2}^3([x],[y])+n^dW^{-D}\right).\numberthis \label{eq:J22}
\end{align*}
Combining \eqref{eq:J1}--\eqref{eq:J22}, we obtain \eqref{eq;I0304} for $|I_{03}|$ (by choosing $D$ sufficiently large). 
Finally, for the term $J_0$, note that \smash{$\sum_{\beta,\gamma}S_{\al\beta}S_{\gamma \mu}|G_{\beta \gamma}|^2=\cL^{(2)}_{t,(-,+),([a],[b])}$} for $\al\in [a],\mu\in [b]$. Then, with \eqref{eq:def_Psit} and \eqref{Mbound_AO}, we get 
\begin{align}\label{eq:J0}
J_0&\prec \sum_{[a],[b]}\sum_{\al\in[a],\mu\in[b]}|M_{\bv\al}|^2|M_{\mu\bw}|^2\Phi_{t}([a],[b]) \cdot \E |\Gc_{\bv\bw}|^{2p-2} \nonumber\\
&\prec \sum_{[a],[b]}W^{-2c|[x]-[a]|-2c|[b]-[y]|} \Phi_{t}([a],[b])\cdot \E |\Gc_{\bv\bw}|^{2p-2}
\end{align}

Combining the estimates \eqref{eq;I1}, \eqref{eq;I2}, \eqref{eq;I3}, \eqref{eq;I01}--\eqref{eq;I0304}, and \eqref{eq:J0}, we conclude that  
\begin{align*}
&\E |\Gc_{xy}|^{2p} \prec \left(\xi_0 W^{-c|[x]-[y]|}  + \xi_0 \cal E_{k+K,c}([x], [y])+W^{-D}\right)\|\bv\|_2\|\bw\|_2 \cdot \E |\Gc_{xy}|^{2p-1} \\
&+\Big(\sum_{[a],[b]}W^{-2c|[x]-[a]|-2c|[b]-[y]|} \Phi_{t}([a],[b])+ \xi_0^2  W^{-c|[x]-[y]|}+\xi_0 \cal E_{k+K,c/2}^2([x], [y])+W^{-D}\Big)\|\bv\|_2^2\|\bw\|_2^2\cdot \E |\Gc_{xy}|^{2p-2}\\
&+ \left( \xi_0^3 W^{-c|[x]-[y]|}+\xi_0 \cal E^3_{k+K,c/2}([x], [y]) +W^{-D}\right)\|\bv\|_2^3\|\bw\|_2^3\cdot \E |\Gc_{xy}|^{2p-3}  .
\end{align*}
Applying H{\"o}lder's and Young's inequalities, we derive from the above bound that 
\begin{align*}
\frac{\big(\E |\Gc_{\bv\bw}|^{2p}\big)^{\frac{1}{2p}}}{\|\bv\|_2\|\bw\|_2}  &\prec  \sum_{[a],[b]}W^{-c|[x]-[a]|-c|[b]-[y]|} [\Phi_{t}([a],[b])]^{\frac12} + \xi_0 W^{-\frac{c}{3}|[x]-[y]|} + \xi_0^{\frac 1 3}\cal E_{k+K,c/2}([x], [y]) +W^{-D} .
\end{align*}
A similar bound holds for $\E |\Gc_{\bw\bv}|^{2p}$ by symmetry. Then, applying Markov's inequality, we obtain a self-improving estimate for the control parameters $\xi_k\p{[a],[b]}$. That is, if $G-M$ satisfies the bounds in \eqref{eq:controlxi} with \smash{$\xi_k\p{[a],[b]}=\xi_k^{(0)}\p{[a],[b]}\le \xi_0$}, then we can derive a better bound for any fixed $k\ge 0$: 
\begin{align}
\max_{\bv\in \cal V_k([a])}\max_{\bw\in \cal V_k([b])} \frac{|\Gc_{\bv \bw}|\vee |\Gc_{\bw\bv}|}{\|\bv\|_2\|\bw\|_2}\prec  \xi_k^{(1)}\p{[a],[b]}:= &~\sum_{[a'],[b']}W^{-c|[a]-[a']|-c|[b]-[b']|} \left[\Phi_{t}([a'],[b'])\right]^{1/2}   \nonumber\\
&~+ \xi_0 W^{-\frac{c}{3}|[a]-[b]|} + \xi_0^{\frac 1 3}\cal E_{k+K,c/2}([a], [b]) +W^{-D} ,\label{eq:selfimprove_xi}
\end{align}
where $\cal E_{k+K,c/2}([a], [b])$ is defined in terms of $\xi^{(0)}_{k+K}([a],[b])$ as in \eqref{eq:defEkc}. Then, taking \smash{$\xi_k\p{[a],[b]}=\xi^{(1)}_k\p{[a],[b]}$} as the input in the previous argument, we get an even better bound: 
\begin{align*}
\max_{\bv\in \cal V_k([a])}\max_{\bw\in \cal V_k([b])} \frac{|\Gc_{\bv \bw}|\vee |\Gc_{\bw\bv}|}{\|\bv\|_2\|\bw\|_2}\prec \xi^{(2)}_k\p{[a],[b]}:= &~\sum_{[a'],[b']}W^{-\frac{c}{2}|[a]-[a']|-\frac{c}{2}|[b]-[b']|} \left[\Phi_{t}([a'],[b'])\right]^{1/2}   \nonumber\\
&~+ \xi_0 W^{-\frac{c}{3}|[a]-[b]|} + \xi_0^{\frac 2 3}\cal E_{k+2K,c/2}([a], [b])+W^{-D} ,
\end{align*}
where $\cal E_{k+2K,c/2}$ is defined in terms of $\xi_{k+2K}^{(0)}\p{[a],[b]}$. Iterating this process for $r$ many times, we get that
\begin{align*}
\max_{x\in[a],y\in[b]}|\Gc_{xy}|\prec \xi^{(r)}_0\p{[a],[b]}:= &~\sum_{[a'],[b']}W^{-\frac{c}{2}|[a]-[a']|-\frac{c}{2}|[b]-[b']|} \left[\Phi_{t}([a'],[b'])\right]^{1/2}   \nonumber\\
&~+ \xi_0 W^{-\frac{c}{3}|[a]-[b]|} + \xi_0^{\frac r 3}\cal E_{rK,c/2}([a], [b]) + W^{-D}.
\end{align*}
Given any constant $D>0$, after $r=\lceil 3D/\e_0\rceil$ many iterations, we obtain that 
\begin{align*}
\max_{x\in[a],y\in[b]}|\Gc_{xy}|\prec \sum_{[a'],[b']}W^{-\frac{c}{2}|[a]-[a']|-\frac{c}{2}|[b]-[b']|} \left[\Phi_{t}([a'],[b'])\right]^{1/2} + \xi_0 W^{-\frac{c}{3}|[a]-[b]|}   + W^{-D}.
\end{align*}
Finally, using \eqref{GiiGEX}, which allows us to take $\xi_0=\Psi_t$, we conclude \eqref{GijGEX}.

\subsection{Proof of \eqref{GiiGEX} and \eqref{GavLGEX}}
If there is no $\xi_0 W^{-\frac{c}{3}|[x]-[y]|}$ term in \eqref{eq:selfimprove_xi}, then taking the maximum over $\bv,\bw\in \cal V_0$, we would get a self-improving estimate for $\xi_0$, which allows us to bound $\xi_0$ by $\Psi_t$ via an iteration argument as in \Cref{appd:MDE1}. 
We see that one source of this troublesome term is the $I_1$ term in \eqref{eq;I1}, which contains a light weight \smash{$\qq{\Gc E_{[\al]}}$} that can only be bounded by $\xi_0$ at this stage.
To deal with this issue, we adopt the argument developed in \cite{He2018}, that is, instead of bounding \smash{$\Gc$} directly, we first control the expression
$$ \Pi(G):=I+z_t G + \cal S_t(G) G- (\lambda\Psi)G.$$
Here, $\cal S_t(\cdot):=t\cal S(\cdot)$ with the operator $\cal S$ defined in \eqref{eq:opS}; in other words, $\cal S_t$ is the linear operator defined in terms of the variance matrix $S_t$ of $V_t$. 
With \eqref{def_G0t} and the definition of $\Pi(G)$, we can write that 
\be\label{eq:G-M}  
\Gc=G - M= - M\left[ \Pi(G) -(\cal S_t(G)-t m)G\right] . 
\ee
Then, the key to the proof is to control $\Pi(G)$ using the bounds on the $2$-$G$ loops given in \eqref{eq:def_Psit}. 
For our later proof of \eqref{Gtmwc} in Step 1, we first establish the following lemma, which does not assume \eqref{def_asGMc} or \Cref{lem:apriori}.

\begin{lemma}\label{lem G<T new}
Let $\xi_k,\phi_k,\Phi_k \ge W^{-d/2} $ be a sequence of deterministic control parameters such that
\be\label{eq:controPi} 
\vertiii{G-M}_{(k)}\prec \xi_k, \quad  \vertiii{\Pi(G)}_{(k)}\prec \phi_k, \quad \max_{x\in\ZL}\max_{\bv\in \cal V_k}\left(T_{x\bv}+ \wt T_{x\bv}\right)/\|\bv\|_2^2 \prec \Phi_k^2 \, .  
\ee
Then, for each fixed $k\ge 0$, we have 
\be\label{eq:G-Mk} 
\vertiii{\Pi(G)}_{(k)} \prec \phi_k^{1/2}\left(\Phi_{k}+\sqrt{(1+\xi_k)\Phi_k}\right)+\Phi_{k}\p{1+\sqrt{(1+\xi_k)\Phi_k}}  \, .  
\ee
\end{lemma}

\begin{proof}
Given $\bw\in \cal V_k$ and $\bv\in \cal V_k([a])$ for some \smash{$[a]\in \Zn$}, we bound the moments $\E |\Pi(G)_{\bv\bw}|^{2p}$ for any fixed $p\in \N$. With the identity $ I_N +z_t G-(\lambda\Psi)G=VG$ and applying Gaussian integration by parts, we obtain  
\begin{align} 
\E |\Pi(G)_{\bv \bw}|^{2p}&= \E\left[(VG)_{\bv\bw}\Pi(G)_{\bv\bw}^{p-1} \overline \Pi(G)_{\bv\bw}^{p}\right]+ \sum_{x\in[a]}\bar\bv(x)\E  \left[\cal S_t(G)_{xx} G_{x\bw}\Pi(G)_{\bv\bw}^{p-1} \overline \Pi(G)_{\bv\bw}^{p}\right] \nonumber\\
&= t\sum_{x,\al\in[a]}\bar\bv(x)\E\left[S_{x\al}G_{\al \bw} \partial_{\al x}\left(\Pi(G)_{\bv\bw}^{p-1} \overline \Pi(G)_{\bv\bw}^{p}\right)\right].\label{eq:simple_cumu}
\end{align}
Direct calculations yield that 
\begin{align}
&\partial_{\al x}\Pi(G)_{\bv\bw}=-\Pi(G)_{\bv \al} G_{x\bw} + \bar\bv(\al) G_{x\bw} - t G_{\bv\bw} (GE_{[a]}G)_{x\al} ,\label{eq:decomp_PiG}\\
&\partial_{\al x}\overline \Pi(G)_{\bv\bw}=-\overline \Pi(G)_{\bv x} \overline G_{\al \bw} + \bar\bv(x) \overline G_{\al \bw} - t \overline G_{\bv\bw} (G^* E_{[a]}G^*)_{x\al}  .\label{eq:decomp_PiG2}
\end{align}
Plugging them into \eqref{eq:simple_cumu} and using the bounds in \eqref{eq:controPi}, we can obtain that
\begin{align*}
\E |\Pi(G)_{\bv\bw}|^{2p} \prec \|\bv\|_2^2\|\bw\|_2^2 \br{\phi_k\Phi_k^2+\phi_k(1+\xi_{k})\Phi_k+ \Phi_k^2 + (1+\xi_k)\Phi_k^3}  
\cdot \E |\Pi(G)_{xy}|^{2p-2}.
\end{align*}
Applying H{\"o}lder's and Young's inequalities, we derive from the above bound that 
\begin{align*}
 \p{\E  |\Pi(G)_{\bv\bw}|^{2p}}^{\frac{1}{2p}}  &\prec \|\bv\|_2\|\bw\|_2\br{\phi_k\Phi_k^2+\phi_k(1+\xi_{k})\Phi_k+ \Phi_k^2 + (1+\xi_k)\Phi_k^3}^{\frac 12} .
\end{align*}
Then, using Markov's inequality and taking a union bound over $\bv,\bw\in \cal V_k$, we get \eqref{eq:G-Mk}.
\end{proof}

Under the assumption \eqref{def_asGMc}, we can choose $\xi_k=\xi_0\le W^{-\e_0}$ and $\Phi_k\le \xi_0\le W^{-\e_0}$ by \Cref{lem:apriori} and \eqref{eq:T(G-M)}. 
We also have the trivial bound \(\vertiii{\Pi(G)}_{(k)}\le N^2.\)
From \eqref{eq:G-Mk}, we get a self-improving estimate:
$$\vertiii{\Pi(G)}_{(k)}\prec \phi_k \ \Rightarrow \ \|\Pi(G)\|_{\max} \prec \left(\phi_k \Phi_k\right)^{1/2} + \Phi_k.$$
Then, starting from $\phi_k=N^2$, iterating this estimate for $\OO(1)$ many times, we get that $\vertiii{\Pi(G)}_{(k)} \prec \Phi_k$ for any fixed $k\ge 0$. 
With a similar argument as that in \eqref{eq:vktovk+1 decay}, we can get that for any $\bw\in \cal V_k$ and $\bv\in \cal V_k([x])$:
\be\label{eq:max_Pi}
[M\Pi(G)]_{\bv\bw} \prec \sum_{[a]}W^{-c|[x]-[a]|}\Phi_{k+K}\|\bv\|_2\|\bw\|_2 + W^{-cK}\|\bv\|_2\|\bw\|_2 \lesssim  \Phi_{k+K}\|\bv\|_2\|\bw\|_2,
\ee
where we take $K$ sufficiently large in the second step such that \smash{$W^{-cK}\le W^{-d/2}\le \Phi_{k+K}$}. 
Applying \eqref{eq:max_Pi} to \eqref{eq:G-M}, we obtain that for all \smash{$x,y\in\Z_L^d$}:
\begin{align}
    G_{xy} - M_{xy}&= \OO_\prec(\Phi_K) + t[M  (\cal S(G)-m)(G-M)]_{xy}+  t[M(\cal S(G)-m)M]_{xy} \nonumber\\
    &= \OO_\prec(\Phi_K + W^{-\e_0} \|\cal S(G)- m\|_{\max})+  t\sum_{\al,\beta} M_{x\al}M_{\al y}S_{\al\beta}(G-M)_{\beta\beta}\, .\label{eq:MDE}
\end{align}
Solving the above equation for the vector of the diagonal entries of \smash{$\Gc$}, \smash{${\Lambda}:=(\Gc_{xx}: x\in \ZL)\in \C^N$}, we get
\be\label{eq:linearM} 
\|\Lambda\|_{\max} \prec \|(1-t M^+S)^{-1}\|_{\infty\to\infty}\left(\Phi_K + W^{-\e_0} \|\Lambda\|_{\max}\right),\ee
where we recall the matrix $M^+$ defined in \eqref{def:Theta}. We claim that 
\be\label{eq:S+infinf}
\|(1-tM^+S)^{-1}\|_{\infty\to\infty} \lesssim 1. 
\ee
In fact, if we replace $M^+$ by $m_{sc}^2$, then it is known that $\|(1-tm_{sc}^2 S)^{-1}\|_{\infty\to\infty}\lesssim 1$ (see e.g., Lemma 4.2 of \cite{PartII} for a proof of this fact). On the other hand, \eqref{eq:M-msc} tells that $M^+$ is a perturbation of $m_{sc}^2I_N$ in the sense of $L^\infty\to L^\infty$ norm. This leads to the bound \eqref{eq:S+infinf}. Applying \eqref{eq:S+infinf} to \eqref{eq:linearM} gives that 
\be\label{eq:G-Mmax} \|\Lambda\|_{\max} \prec \Phi_K + W^{-\delta_0} \|\Lambda\|_{\max} \ \implies \ \|\Lambda\|_{\max} \prec \Phi_K \ \implies \ \|G-M\|_{\max}\prec \Phi_K.\ee
Here, in the second step, we have inserted $\|\Lambda\|_{\max} \prec \Phi_K$ back to \eqref{eq:MDE} to get the conclusion.

To conclude \eqref{GiiGEX}, it remains to control the $T$-variables with $2$-$G$ loops. In other words, we need to bound $ T_{x\bv}$ and \smash{$\wt T_{x\bv}$} for $\bv\in \cal V_k$. By \eqref{eq:Txy-Txv}, it suffices to control the moments $ \E T_{xy}^p$ for each fixed $p\in \N$: 
$$ \E T_{xy}^p=-\sum_\al S_{x\al} \E\left[(G(tm+V)M)_{\al y}\overline G_{\al y} \cdot T_{xy}^{p-1}\right].$$
We can estimate the RHS using Gaussian integration by parts, following a similar approach to the estimation of $I_0$ in \eqref{def;I0}. However, the derivation here is significantly simpler than that below \eqref{def;I0}, as it only requires the max-norms and does not involve tracking decay factors. Therefore, we omit the details. This argument gives a self-improving estimate for $\Phi_0$ in terms of $\Psi_t$, iterating which for $\OO(1)$ many times yields 
\be\label{eq:boundT}
\max_{x,y}T_{xy}\prec \Psi_t^2 \implies  \max_{x}\max_{\bv\in \cal V_k}\p{T_{x\bv}+\wt T_{x\bv}}/\|\bv\|_2^2\prec \Psi_t^2,
\ee
showing that we can take $\Phi_k=\Psi_t$ for any $k\ge 0$. Combining this with \eqref{eq:G-Mmax} completes the proof of \eqref{GiiGEX}.
More generally, by applying \eqref{GiiGEX}, \eqref{eq:max_Pi}, and \eqref{eq:boundT} to \eqref{eq:G-M}, we obtain that for any $\bv, \bw \in \mathcal{V}_k$, 
\be\label{eq:G-Mmax2} \Gc_{\bv\bw}= - \left[ M\Pi(G)\right]_{\bv\bw} +\left[M(\cal S_t(G)-t m)G\right]_{\bv\bw} \prec \Psi_t \implies 
\vertiii{G-M}_{(k)}\prec \Psi_t .
\ee


For the proof of \eqref{GavLGEX}, with the estimates in \eqref{eq:boundT} and \eqref{eq:G-Mmax2} as inputs and following the arguments in the proof of \cite[Proposition 3.2]{He2018}, we can establish the following result: 
\be\label{eq:aver}
W^{-d}\tr\left[ M\Pi(G)I_{[a]}\right]\prec \Psi_t^2,\quad W^{-d}\tr\left[ M\Pi(G)I_{[a]}MI_{[b]}\right]\prec \Psi_t^2\|I_{[a]}MI_{[b]}\|,\quad \forall [a],[b]\in \Zn \, .
\ee
More precisely, whenever an estimate on $\Gc_{\bv\bw}$ is required, we apply the bound \eqref{eq:G-Mmax2}. Additionally, in the proof of \cite[Proposition 3.2]{He2018}, Ward's identity is used in several instances; instead, we can bound the relevant terms using $T$-variables, which can further be controlled by \eqref{eq:boundT}. 
In fact, our argument is significantly simpler than that in \cite{He2018}, as we only need to handle terms arising from Gaussian integration by parts, whereas \cite{He2018} also deals with terms from more general cumulant expansions. Therefore, we omit the details.

Given \eqref{eq:aver}, we first prove the first estimate in \eqref{GavLGEX}. We introduce another control parameter $\theta\ge W^{-d}$ such that 
\(\max_{[a]}|\qq{(G-M)E_{[a]}}|\prec \theta.\) 
Then, using \eqref{eq:G-M}, \eqref{eq:G-Mmax2}, and \eqref{eq:aver} we get that 
\begin{align}\label{eq:aver_pf1}
\qq{(G-M)E_{[a]}}&= - W^{-d}\tr\br{M\Pi(G)I_{[a]}} +tW^{-d} \sum_{[b]}\qq{(G-M)E_{[b]}}\tr\left(MI_{[b]}GI_{[a]}\right)  \nonumber\\
&=W^{-d} \sum_{[b]}\qq{(G-M)E_{[b]}}\tr \left(MI_{[b]}MI_{[a]}\right)+\OO_\prec(\Psi_t^2+\theta \Psi_t) \nonumber\\
&= m_{sc}^2\qq{(G-M)E_{[a]}} + \OO_\prec\left(\Psi_t^2+\theta \Psi_t + \theta \heta \right).
\end{align}
In the third step, we have used \eqref{eq:M-msc} and \eqref{Mbound_AO} to get that 
$$W^{-d}\sum_{x\in[a]}\sum_{[b]\ne [a]}|\tr(MI_{[b]}MI_{[a]})|\lesssim \heta^2,\quad  W^{-d} \tr(MI_{[a]}MI_{[a]})=m_{sc}^2 + \OO(\heta).$$ 
From \eqref{eq:aver_pf1}, we get $\qq{(G-M)E_{[a]}}\prec \Psi_t^2+\theta \p{\Psi_t + \heta }$. Then, taking a union bound in $[a]$, we derive the following self-improving estimate for $\theta$:
$$\max_{[a]}|\qq{(G-M)E_{[a]}}|\prec \theta \ \implies \ \max_{[a]}|\qq{(G-M)E_{[a]}}|\prec\Psi_t^2+\theta \p{\Psi_t + \heta} .$$
Iterating it for $\OO(1)$ many times yields the first estimate in \eqref{GavLGEX}, 
which further implies that $\|\cal S_t(G)-t m\|_{\max}\prec \Psi_t^2.$ 
Plugging this estimate into \eqref{eq:G-M} and using \eqref{eq:G-Mmax2}, we obtain that 
\begin{align}\label{eq:aver_pf} W^{-d} \left\langle \left(G_t - M\right)I_{[a]}MI_{[b]} \right\rangle = -W^{-d}\tr\br{M\Pi(G)I_{[a]}MI_{[b]}}
+\OO_\prec\p{\Psi_t^2\|I_{[a]}MI_{[b]}\|}\prec \Psi_t^2\|I_{[a]}MI_{[b]}\|,
\end{align}
where we used \eqref{eq:aver} in the second step. This concludes the second estimate in \eqref{GavLGEX}.

\section{Analysis of the loop hierarchy}\label{Sec:Stoflo}

This section is devoted to the proof of \Cref{lem:main_ind}, following the six-step strategy summarized in \Cref{sec:proof}. Our proof closely parallels the approach used in \cite[Section 5]{Band1D}, based on the tools developed in the preceding sections. Specifically, we rely on the flow (\Cref{zztE}), the loop hierarchy (\Cref{lem:SE_basic}), the deterministic estimates in \Cref{lem:propM,lem_propTH}, various properties of $\cK$-loops---including the upper bound (\Cref{ML:Kbound}), Ward's identity (\Cref{lem_WI_K}), and the tree representation (\Cref{tree-representation})---as well as the resolvent entry estimate (\Cref{lem_GbEXP}).
The main difference from the argument in \cite[Section 5]{Band1D} occurs in Step 1, for which we will provide a more detailed explanation. For Steps 2--5, we outline the main ideas in the body of the text, while deferring the technical details and proofs of some key results to \Cref{sec:main_appd} for the reader’s convenience. Finally, the proof of Step 6 is largely independent of the previous steps and is therefore presented separately in \Cref{sec:pf_step6} in the appendix.


\subsection{Step 1: A priori $G$-loop bound}

Our proof crucially relies on the following continuity estimate, \Cref{lem_ConArg}. To make the dependence on the spectral parameter $z$ and the coupling parameter $\lambda$ explicit, we will, for the remainder of this subsection, denote the resolvent by \(G_t(z, \lambda)=(V_t+\lambda\Psi-z)^{-1}\) and the corresponding $G$-loop by \smash{${\cal L}^{(\fn)}_{t, \boldsymbol{\sigma}, \ba}(z, \lambda)$}. Moreover, we define the $T$-variables $(T_t)_{x\bv}(z,\lambda)$ and \smash{$(\wt T_t)_{x\bv}(z,\lambda)$} in terms of the resolvent $G_t(z,\lambda)$ as in \eqref{eq:gen_T}.

\begin{lemma}[Continuity estimate]\label{lem_ConArg}
Fix any \( c\le s \leq t \leq 1\) for a constant \(c > 0\). Given any $\lambda$ satisfying the condition \eqref{eq:cond_A12}, let $\lambda_s:=\lambda\cdot \sqrt{s/t}$. 
Then, we have the following continuity estimates for the $G$-loops and generalized resolvent entries. 
\begin{enumerate}

\item Assume that the following bound holds at time \(s\) for each fixed \(\fn \in \mathbb{N}\):  
\begin{equation}\label{55}
 \max_{\boldsymbol{\sigma}, \ba} \left|{\cal L}^{(\fn)}_{s, \boldsymbol{\sigma}, \ba}\p{z_s,\lambda_s}\right| \prec \left( W^d \ell_{s}^d \eta_{s} \right)^{-\fn+1}. 
\end{equation}
Then, for each fixed \(\fn \in \mathbb{N}\) with $\fn\ge 2$, we have 
\begin{align}\label{res_lo_bo_eta}
\max_{\boldsymbol{\sigma}, \ba} 
\left|{\cal L}_{t, \boldsymbol{\sigma}, \ba}^{(\fn)}\p{z_t,\lambda}\right| \prec 
\left(\frac{\ell_{t}^d}{\ell_{s}^d}\right)^{\fn-1}\frac{\max_{[a]}\langle \im G_t(z_t,\lambda) E_{[a]}\rangle}{\left( W^d \ell_{t}^d \eta_{t} \right)^{\fn-1}}.
\end{align}
Moreover, for an arbitrary deterministic unit vector $\bv$, we have 
\begin{align}\label{eq:cont_T}
(T_t)_{x\bv}(z_t,\lambda)+(\wt T_t)_{x\bv}(z_t,\lambda) \prec (T_s)_{x\bv}(z_t,\lambda_s)+(\wt T_s)_{x\bv}(z_t,\lambda_s)+\frac{\ell_{t}^d}{\ell_{s}^d}\frac{\im (G_t)_{\bv\bv}(z_t,\lambda)}{W^d \ell_{t}^d \eta_{t} }.
\end{align}

\item Given any deterministic unit vectors $\bv,\bw\in \C^N$, suppose 
\be\label{eq:ImGs}\im (G_s)_{\bv\bv}(z_s,\lambda_s)\lesssim 1,\quad \im (G_s)_{\bw\bw}(z_s,\lambda_s)\lesssim 1, \quad \left|(G_s)_{\bv\bw}(z_s,\lambda_s)\right|\lesssim 1,
\ee
with high probability. Then, the following estimates hold with high probability at time $t$: 
\be\label{eq:ImGt}
\im (G_t)_{\bv\bv}(z_t,\lambda)\lesssim {\eta_s}/{\eta_t},\quad (G_t)_{\bv\bw}(z_t,\lambda)\lesssim {\eta_s}/{\eta_t}.
\ee
 
\end{enumerate}

\end{lemma}
\begin{proof}
The continuity estimate \eqref{res_lo_bo_eta} has been established as Lemma 5.1 in \cite{Band1D}. Notably, the proof there does not rely on the specifics of the underlying random matrix model and thus carries over directly to our setting. 
The only difference lies in the choice of $\lambda_s$ at time $s$, which arises from the following simple observation:
\be\label{eq:Hslambdas} H_s(\lambda_s)=V_s+\lambda_s\Psi \stackrel{d}{=} \sqrt{{s}/{t}}\cdot \left( V_t + \lambda \Psi\right)= \sqrt{{s}/{t}}\cdot H_t(\lambda).\ee
The estimate \eqref{eq:cont_T} appears as an intermediate step in the proof of \eqref{res_lo_bo_eta} in the case $\fn=2$.

To show \eqref{eq:ImGt}, we choose the spectral parameter
\be\label{eq:wtzzz}\wt z:=\re  z_t+\ii \sqrt{{t}/{s}}\cdot \im z_s ,\quad \wt z_s:=\sqrt{{s}/{t}}\cdot \wt z . \ee
First, we claim that with high probability,
\be\label{eq:ImGs2}\im (G_s)_{\bv\bv}(\wt z_s,\lambda_s)\lesssim 1,\quad \im (G_s)_{\bw\bw}(\wt z_s,\lambda_s)\lesssim 1.
\ee
Under the flow \eqref{eq:zt}, we can check that $|\wt z_s-z_s| \lesssim\eta_s$.
Then, applying the Cauchy-Schwarz inequality to the resolvent expansion for $G_s(\wt z_s)-G_s(z_s)$ (where we omit the parameter $\lambda_s$) together with Ward's identity, we obtain that 
\be\label{eq:ward_app}
|(G_s)_{\bv\bv}(\wt z_s)-(G_s)_{\bv\bv}(z_s)| \le |\wt z_s-z_s|\sum_ x|(G_s)_{\bv x}(\wt z_s)||(G_s)_{x\bv}(z_s)| \lesssim \left[\im (G_s)_{\bv\bv}(\wt z_s)\cdot \im (G_s)_{\bv\bv}(z_s)\right]^{1/2},
\ee
which in turn implies that
\be\nonumber\im (G_s)_{\bv\bv}(\wt z_s) \lesssim \im (G_s)_{\bv\bv}(z_s) + \left[\im (G_s)_{\bv\bv}(\wt z_s)\cdot \im (G_s)_{\bv\bv}(z_s)\right]^{1/2}\implies \im (G_s)_{\bv\bv}(\wt z_s)\lesssim\im (G_s)_{\bv\bv}(z_s),\ee
thereby completing the proof of \eqref{eq:ImGs2} under \eqref{eq:ImGs}. 
Next, using \eqref{eq:Hslambdas} and \eqref{eq:wtzzz}, we can check that \[G_s(\wt z_s,\lambda_s)\stackrel{d}{=}\sqrt{t/s}\cdot G_t(\wt z,\lambda_s).\]
Noting that $\re \wt z= \re z_t$, the first estimate in \eqref{eq:ImGt} follows immediately from \eqref{eq:ImGs2} and the fact that the function $\eta\mapsto\eta \im (G_{t})_{\bv\bv}(\re z_t + \ii \eta)$ is monotonically increasing in $\eta$. More generally, this gives that the following bound holds uniformly in \(\eta\in [\eta_t,\im \wt z]\) with high probability:
\be\label{eq:ImGt222}
\im (G_t)_{\bv\bv}(\re z_t + \ii \eta,\lambda)\lesssim {\eta_s}/{\eta} .
\ee

For the second estimate in \eqref{eq:ImGt}, applying the Cauchy-Schwarz inequality to $\partial_\eta (G_t)_{\bv\bv}(\re z_t + \ii \eta)$ (where we omit the parameter $\lambda$) along with Ward's identity, we obtain that 
\[\left|\frac{\dd}{\dd\eta}(G_t)_{\bv\bw}(\re z_t + \ii \eta) \right|
\lesssim \frac{1}{\eta}\left[\im (G_t)_{\bv \bv}(\re z_t + \ii \eta)\cdot \im (G_t)_{\bw\bw}(\re z_t + \ii \eta)|\right]^{1/2}\lesssim \frac{\eta_s}{\eta^2},
\]
uniformly in \(\eta\in [\eta_t,\im \wt z]\) with high probability, where we applied \eqref{eq:ImGt222} in the second step. Integrating this bound over $\eta$, we conclude the second bound in \eqref{eq:ImGt}.  
\end{proof}

In \Cref{lem_ConArg}, we have obtained a priori bound \eqref{res_lo_bo_eta} on the $G$-loops. To complete the proof of Step 1, we still need to establish the weak local law, which in turn yields an estimate on $\max_{[a]}\langle \im G_t(z_t,\lambda) E_{[a]}\rangle$. 
For this purpose, we use a standard \smash{$N^{-C}$}-net argument, based on the following \Cref{lem G<T}. 


\begin{lemma}\label{lem G<T}
In the setting of \Cref{lem:main_ind}, let $\e_0$ be a fixed constant. Suppose that, for a fixed integer  $K\ge\lceil 2d/\e_A\rceil$, the following estimate holds for all $k\in\qqq{0,K}$:
\be\label{initialGT}
\max_{\bv,\bw\in \cal V_k}\frac{|G_{\bv\bw}|}{\|\bv\|_2\|\bw\|_2}\prec \frac{\eta_s}{\eta_t},\quad \max_{x\in\ZL,\bv\in \cal V_k} \frac{(T_t)_{x\bv}+ (\wt T_t)_{x\bv}}{\|\bv\|_2^2} \prec \Phi^2,\quad \text{where}\quad \Phi^2:=\frac{\eta_s\ell_{t}^d}{\eta_t\ell_{s}^d}\frac{1}{W^d \ell_{t}^d \eta_{t} } \, . 
\ee
Define the event $\Xi=\left\{\|G_t-M\|_{\max} \le W^{-\e_0} \right\}$ for a constant $\e_0>0$. 
Then, under the condition \eqref{con_st_ind}, for any constants $\e,D>0$, we have that
\be\label{eq:boundG*}
\P\p{\mathbf 1(\Xi)\|G_t-M\|_{\max} \ge W^{\e}\sqrt{\eta_s/\eta_t}\cdot \Phi }\le W^{-D}.
\ee
\end{lemma}
\begin{proof}
By \Cref{lem G<T new}, under the condition \eqref{initialGT}, we obtain the self-improving estimate:
\[\vertiii{\Pi(G)}_{(k)}\prec \phi_k 
\implies \vertiii{\Pi(G)}_{(k)} \prec \phi_k^{1/2}\sqrt{\frac{\eta_s}{\eta_t}\Phi} + \Phi+\sqrt{\frac{\eta_s}{\eta_t}}\Phi^{3/2}, \quad \forall k \in\qqq{0,K}. \]
Note that under the conditions $1-t\ge W^\fd\eta_*$ and \eqref{con_st_ind}, we have $\Phi':=\sqrt{{\eta_s}/{\eta_t}}\cdot \Phi\le W^{-\fd/4}$. Thus, iterating the above estimate for $\OO(1)$ many times, we derive that $\vertiii{\Pi(G)}_{(k)}\prec\Phi'$ for each $k \in \qqq{0,K}$. Then, with a similar argument as in \eqref{eq:max_Pi}, we obtain the following bound for any constant $c\in (\e_A/2,\e_A)$:
\be\label{eq:controlPi2}  \|M\Pi(G)\|_{\max}=\vertiii{M\Pi(G)}_{(0)}\prec \Phi_{K} + W^{-cK}\le \Phi' .  \ee
With \eqref{eq:controlPi2}, applying a similar argument as that below \eqref{eq:MDE}, we obtain that 
\(\mathbf 1(\Xi)\|G-M\|_{\max} \prec \Phi'\), which concludes the proof. 
\end{proof}

Given \Cref{lem G<T,lem_ConArg}, the proof of Step 1 is similar to that in \cite[Section 5.1]{Band1D}. 
\begin{proof}[\bf Step 1: Proof of \eqref{lRB1} and \eqref{Gtmwc}] 
In the first case with $t> s \ge 1/2$, in addition to \eqref{Eq:L-KGt+IND} and \eqref{Gt_bound+IND}, in the proof of \Cref{lem_GbEXP}, we have also established the following bound at time $s$ (recall \eqref{eq:boundT}):
\[\max_{x}\max_{\bv\in \cal V_k}\p{(T_s)_{x\bv}+(\wt T_s)_{x\bv}}/\|\bv\|_2^2\prec (W^d\ell_s^d\eta_s)^{-1}.\] 
Then, applying \Cref{lem_ConArg}, we can verify the condition \eqref{initialGT}, and hence obtain \eqref{eq:boundG*} by \Cref{lem G<T}. It is obvious that the same estimate holds with $t$ replaced by $u$. Using a standard $N^{-C}$-net argument (cf.~\cite[Section 5.3]{Semicircle}), we can readily derive the weak local law \eqref{Gtmwc}. Applying \eqref{res_lo_bo_eta} then concludes \eqref{lRB1}. 

In case II with $0\le s<t \le 1/2$, we have $\sup_{u\in[s,t]}\|G_{u}\|\lesssim 1$, which  immediately implies that \smash{\(
{\cal L}^{(\fn)}_{u,\boldsymbol{\sigma}, \ba} \lesssim W^{-d(\fn-1)}\)} for all $u\in [s,t]$ (so \eqref{lRB1} follows) and verifies the condition \eqref{initialGT}. 
Then, the weak local law \eqref{Gtmwc} follows from \Cref{lem G<T} along with an $N^{-C}$-net argument as in the first case.

Finally, if $0\le s <1/2 < t$, we can first show that \eqref{lRB1} and \eqref{Gtmwc} hold for $u\in [s,1/2]$ as in the second case. Then, applying the argument in the first case with $s$ replaced by $1/2$, we can further show that \eqref{lRB1} and \eqref{Gtmwc} hold for $u\in [1/2,t]$. 
\end{proof}

\subsection{Dynamics of the $G$-loops}

For any fixed $\fn\in \N$, combining the loop hierarchy \eqref{eq:mainStoflow} and the equation \eqref{pro_dyncalK} for primitive loops, we obtain (recall that $S^{\LK}=I$):
\begin{align}\label{eq_L-K-1}
&\dd(\mathcal{L} - \mathcal{K})^{(\fn)}_{t, \boldsymbol{\sigma}, \ba} 
    = W^d \sum_{1 \leq k < l \leq \fn} \sum_{[a]} 
   (\mathcal{L} - \mathcal{K})^{(\fn+k-l+1)}_{t, \cutL^{[a]}_{k, l}\left(\boldsymbol{\sigma},\, \ba\right)} 
   \mathcal{K}^{(l-k+1)}_{t,\cutR^{[a]}_{k,l}\left(\boldsymbol{\sigma} ,\ba\right)}\, \dd t \nonumber\\
&+ W^d \sum_{1 \leq k < l \leq \fn} \sum_{[a]} \mathcal{K}^{(\fn+k-l+1)}_{t, \cutL^{[a]}_{k, \,l}\left(\boldsymbol{\sigma},\ba\right)} 
   \left(\cL-\cK\right)^{(l-k+1)}_{t,\cutR^{[a]}_{k, l}\left(\boldsymbol{\sigma},\ba\right)} \, \dd t + \mathcal{E}^{(\fn)}_{t, \boldsymbol{\sigma}, \ba}\dd t + 
    \dd\mathcal{B}^{(\fn)}_{t, \boldsymbol{\sigma}, \ba} 
    +\mathcal{W}^{(\fn)}_{t, \boldsymbol{\sigma}, \ba}
    \dd t,
\end{align}
where $\mathcal{E}^{(\fn)}_{t, \boldsymbol{\sigma}, \ba}$ is defined by 
\begin{equation}\label{def_ELKLK}
\mathcal{E}^{(\fn)}_{t, \boldsymbol{\sigma}, \ba} :=
    W^d \sum_{1 \leq k < l \leq \fn} \sum_{[a]} 
   (\mathcal{L} - \mathcal{K})^{(\fn+k-l+1)}_{t, \cutL^{[a]}_{k, l}\left(\boldsymbol{\sigma},\, \ba\right)} 
   (\cL-\mathcal{K})^{(l-k+1)}_{t,\cutR^{[a]}_{k,l}\left(\boldsymbol{\sigma} ,\ba\right)}\,  .
\end{equation}
We can rearrange the first two terms on the RHS of \eqref{eq_L-K-1} according to the lengths of $\cK$-loops and rewrite them as
\(\sum_{\lenk=2}^\fn [\OK^{(\lenk)} (\mathcal{L} - \mathcal{K})]^{(\fn)}_{t, \boldsymbol{\sigma}, \ba}\dd t,\) 
where $\OK^{(\lenk)}$ is a linear operator defined as 
\begin{align}\label{DefKsimLK}
\left[\OK^{(\lenk)} (\mathcal{L} - \mathcal{K})\right]_{t, \boldsymbol{\sigma}, \ba}^{(\fn)}:= &~W^d \sum_{1\le k < l \leq \fn : l-k=\lenk-1} \sum_{[a]} 
   (\mathcal{L} - \mathcal{K})^{(\fn-\lenk+2)}_{t, \cutL^{[a]}_{k, l}\left(\boldsymbol{\sigma},\, \ba\right)} 
   \mathcal{K}^{(\lenk)}_{t,\cutR^{[a]}_{k,l}\left(\boldsymbol{\sigma} ,\ba\right)} \nonumber\\
+&~  W^d \sum_{1 \leq k < l \leq \fn:l-k=\fn-\lenk+1} \sum_{[a]} \mathcal{K}^{(\lenk)}_{t, \cutL^{[a]}_{k, \,l}\left(\boldsymbol{\sigma},\ba\right)} 
   \left(\cL-\cK\right)^{(\fn-\lenk+2)}_{t,\cutR^{[a]}_{k, l}\left(\boldsymbol{\sigma},\ba\right)}  .
\end{align}
Taking out the leading term with $\lenk =2$, we can rewrite \eqref{eq_L-K-1} as
\begin{align}\label{eq_L-Keee}
    \dd(\mathcal{L} - \mathcal{K})^{(\fn)}_{t, \boldsymbol{\sigma}, \ba} =&~ \left[\OK^{(2)} (\mathcal{L} - \mathcal{K})\right]^{(\fn)}_{t, \boldsymbol{\sigma}, \ba} \, \dd t+\sum_{\lenk=3}^\fn \left[\OK^{(\lenk)} (\mathcal{L} - \mathcal{K})\right]^{(\fn)}_{t, \boldsymbol{\sigma}, \ba}\, \dd t  + \mathcal{E}^{(\fn)}_{t, \boldsymbol{\sigma}, \ba}\dd t + 
    \dd\mathcal{B}^{(\fn)}_{t, \boldsymbol{\sigma}, \ba} 
    +
    \mathcal{W}^{(\fn)}_{t, \boldsymbol{\sigma}, \ba}\dd t.
\end{align}

\begin{definition}[Evolution kernel]\label{DefTHUST}
For $t\in [0,1]$ and $\bsig=(\sigma_1,\ldots,\sig_\fn)\in \{+,-\}^\fn$, define the linear operator ${\thn}^{(\fn)}_{t, \boldsymbol{\sigma}}$ acting on $\fn$-dimensional tensors ${\cal A}: (\Zn)^{\fn}\to \mathbb C$ as: 
\begin{align}\label{def:op_thn}
    \left({{\thn}}^{(\fn)}_{t, \boldsymbol{\sigma}} \circ \mathcal{A}\right)_{\ba} & = \sum_{i=1}^\fn \sum_{[b_i]\in\Zn} \left(\frac{M^{(\sig_i,\sig_{i+1})}}{1 - t M^{(\sig_i,\sig_{i+1})}}\right)_{[a_i] [b_i]}  \mathcal{A}_{\ba^{(i)}([b_i])}, 
\end{align} 
where $\ba=([a_1],\ldots, [a_\fn])\in (\Zn)^\fn$, $\ba^{(i)}([b_i])$ is defined as 
\begin{align*}    
    \ba^{(i)}([b_i]):= ([a_1], \ldots, [a_{i-1}], [b_i], [a_{i+1}], \ldots, [a_\fn]),
\end{align*}
and we recall the convention that $\sigma_{\fn+1}=\sig_1$. 
Then, the evolution kernel corresponding to ${{\thn}}^{(\fn)}_{t, \boldsymbol{\sigma}}$ is given by  
    \begin{align}\label{def_Ustz}
        \left(\mathcal{U}_{s, t, \boldsymbol{\sigma}}^{(\fn)} \circ \mathcal{A}\right)_{\ba} = \sum_{\mathbf b = ([b_1], \ldots, [b_{\fn}])\in (\Zn)^\fn} \prod_{i=1}^\fn \left(\frac{1 - s M^{(\sig_i,\sig_{i+1})}}{1 - t M^{(\sig_i,\sig_{i+1})}}\right)_{[a_i] [b_i]} \cdot \mathcal{A}_{\mathbf{b}}. 
    \end{align}
\end{definition}

By the definition of the 2-$\cK$ loop in \eqref{Kn2sol}, we observe that 
\begin{align}
\frac{\dd }{\dd t}\left(\mathcal{U}_{s, t, \boldsymbol{\sigma}}^{(\fn)} \circ \mathcal{A}\right)_{\ba} =     \left({\thn}_{t, \boldsymbol{\sigma}}^{(\fn)} \circ \mathcal{U}_{s, t, \boldsymbol{\sigma}}^{(\fn)} \circ \mathcal{A}\right)_{\ba} = \left[\OK^{(2)} (\mathcal{U}_{s, t, \boldsymbol{\sigma}}^{(\fn)} \circ \mathcal{A})\right]^{(\fn)}_{t, \boldsymbol{\sigma}, \ba}.
\end{align}
With Duhamel's principle, we can derive from \eqref{eq_L-Keee} the following expression for any $s\le t$:
\begin{align}\label{int_K-LcalE}
    &(\mathcal{L} - \mathcal{K})^{(\fn)}_{t, \boldsymbol{\sigma}, \ba}  =
    \left(\mathcal{U}^{(\fn)}_{s, t, \boldsymbol{\sigma}} \circ (\mathcal{L} - \mathcal{K})^{(\fn)}_{s, \boldsymbol{\sigma}}\right)_{\ba} + \sum_{l_\mathcal{K} =3}^\fn \int_{s}^t \left(\mathcal{U}^{(\fn)}_{u, t, \boldsymbol{\sigma}} \circ \Big[\OK^{(\lenk)} (\mathcal{L} - \mathcal{K})\Big]^{(\fn)}_{u, \boldsymbol{\sigma}}\right)_{\ba} \dd u \nonumber \\
    &+ \int_{s}^t \left(\mathcal{U}^{(\fn)}_{u, t, \boldsymbol{\sigma}} \circ \mathcal{E}^{(\fn)}_{u, \boldsymbol{\sigma}}\right)_{\ba} \dd u  + \int_{s}^t \left(\mathcal{U}^{(\fn)}_{u, t, \boldsymbol{\sigma}} \circ \cW^{(\fn)}_{u, \boldsymbol{\sigma}}\right)_{\ba} \dd u + \int_{s}^t \left(\mathcal{U}^{(\fn)}_{u, t, \boldsymbol{\sigma}} \circ \dd \cB^{(\fn)}_{u, \boldsymbol{\sigma}}\right)_{\ba} .
\end{align}
Furthermore, let $T$ be a stopping time with respect to the matrix Brownian motion $\{H_t\}$ in \eqref{MBM}, and denote $\tau:= T\wedge t$. Then, we have a stopped version of \eqref{int_K-LcalE}:  
\begin{align}\label{int_K-L_ST}
& (\mathcal{L} - \mathcal{K})^{(\fn)}_{\tau, \boldsymbol{\sigma}, \ba}  =
    \left(\mathcal{U}^{(\fn)}_{s, \tau, \boldsymbol{\sigma}} \circ (\mathcal{L} - \mathcal{K})^{(\fn)}_{s, \boldsymbol{\sigma}}\right)_{\ba} + \sum_{l_\mathcal{K} =3}^\fn \int_{s}^\tau \left(\mathcal{U}^{(\fn)}_{u, \tau, \boldsymbol{\sigma}} \circ \Big[\OK^{(\lenk)} (\mathcal{L} - \mathcal{K})\Big]^{(\fn)}_{u, \boldsymbol{\sigma}}\right)_{\ba} \dd u \nonumber \\
    &+ \int_{s}^\tau \left(\mathcal{U}^{(\fn)}_{u, \tau, \boldsymbol{\sigma}} \circ \mathcal{E}^{(\fn)}_{u, \boldsymbol{\sigma}}\right)_{\ba} \dd u  + \int_{s}^\tau \left(\mathcal{U}^{(\fn)}_{u, \tau, \boldsymbol{\sigma}} \circ \cW^{(\fn)}_{u, \boldsymbol{\sigma}}\right)_{\ba} \dd u + \int_{s}^\tau \left(\mathcal{U}^{(\fn)}_{u, \tau, \boldsymbol{\sigma}} \circ \dd \cB^{(\fn)}_{u, \boldsymbol{\sigma}}\right)_{\ba}.  
\end{align}
We will utilize the above two equations to estimate $(\mathcal{L} - \mathcal{K})$-loops. For the analysis, we also introduce the following notation that corresponds to the quadratic variation of the martingale term (recall \eqref{def_Edif}).

\begin{definition}
\label{def:CALE} 
For $t\in [0,1]$ and $\bsig=(\sigma_1,\ldots,\sig_\fn)\in \{+,-\}^\fn$,
we introduce the $(2\fn)$-dimensional tensor
  \begin{align*}
 \left( \cB\otimes  \cB \right)^{(2\fn)}_{t, \boldsymbol{\sigma}, \ba, \ba'} := &
 \sum_{k=1}^\fn \left( \cB\times  \cB \right)^{(k)}_{t, \,\boldsymbol{\sigma}, \ba, \ba'},\ \ \forall \ba=([a_1],\ldots, [a_\fn]) ,\ \ba'=([a_1'],\ldots, [a_\fn']), 
 \end{align*}
 where $\left( \cB\times  \cB \right)^{(k)}_{t, \,\boldsymbol{\sigma}, \ba, \ba'}$ is defined as 
\be\label{defEOTE}
 \left( \cB\times  \cB \right)^{(k)}_{t, \boldsymbol{\sigma}, \ba, \ba'} :
 =   W^d \sum_{[b],[b']} S^{\LK}_{[b][b']}{\cal L}^{(2\fn+2)}_{t, (\boldsymbol{\sigma}\times\overline\bsig)^{(k) },(\ba\times \ba')^{(k )}([b],[b'])}.
\ee
Here, ${\cal L}^{(2\fn+2)}$ denotes a $(2\fn+2)$-loop obtained by cutting the $k$-th edge of \smash{${\cal L}^{(\fn)}_{t,\boldsymbol{\sigma},\ba}$} and then gluing it (with indices $\ba$) with its conjugate loop (with indices $\ba'$) along the new vertices $[b]$ and $[b']$. Formally, its expression is written as: 
\begin{align*}
    &{\cal L}^{(2\fn+2)}_{t, (\boldsymbol{\sigma}\times\overline\bsig)^{(k) },(\ba\times \ba')^{(k )}([b],[b'])}:=  \bigg \langle \prod_{i=k}^\fn \left(G_{t}(\sigma_i) E_{[a_i]}\right) \cdot \prod_{i=1}^{k-1} \left(G_{t}(\sigma_i) E_{[a_i]}\right) \cdot G_t(\sigma_k) E_{[b]} G_t(-\sigma_k) \\
    &\qquad \qquad \times \prod_{i={1}}^{k-1} \left(E_{[a_{k-i}']} G_{t}(-\sigma_{k-i}) \right)\cdot \prod_{i=k}^{\fn} \left(E_{[a_{\fn+k-i}']} G_{t}(-\sigma_{\fn+k-i}) \right) E_{[b']}\bigg\rangle ,
\end{align*}
where the notations $(\boldsymbol{\sigma}\times\overline \bsig)^{(k) }$ and $(\ba\times \ba')^{(k )}$ represent (with $\overline\bsig$ denoting $(-\sig_1,\ldots,-\sig_\fn)$)
\begin{align}\label{def_diffakn_k}
   & (\ba\times \ba')^{(k )}([b],[b'])=( [a_k],\ldots, [a_\fn], [a_1],\ldots [a_{k-1}], [b], [a'_{k-1}],\ldots [a_1'], [a_\fn']\cdots [a'_{k}],[b']),\nonumber \\
  &  (\boldsymbol{\sigma}\times \overline\bsig)^{(k )}=(  \sigma_k, \ldots \sigma_\fn,   \sigma_1,\ldots ,\sigma_{k}, -\sigma_{k}, \ldots, -\sigma_1 ,   -\sigma_\fn, \ldots, -\sigma_{k }).
\end{align}
Here, we remark that the symbols ``$\otimes$" and ``$\times$" in the above notations were introduced to emphasize the symmetric structures of the notations---they do not represent any kind of ``product". 
Note that for the block Anderson model, we must have $[b]=[b']$ since $S^{\LK}=I_n$.  
\end{definition}

Under the above notation, the following lemma is an immediate consequence of the Burkholder-Davis-Gundy inequality. 
\begin{lemma}[Estimation of the martingale term]\label{lem:DIfREP}
Let $T$ be a stopping time with respect to the matrix Brownian motion $\{H_t\}$, and denote $\tau:= T\wedge t$. Then, for any fixed $p\in \N$, we have that 
 \begin{equation}\label{alu9_STime}
   \mathbb{E} \left [    \int_{s}^\tau\left(\mathcal{U}^{(\fn)}_{u, \tau, \boldsymbol{\sigma}} \circ \dd\cB^{(\fn)}_{u, \boldsymbol{\sigma}}\right)_{\ba} \right ]^{2p} 
   \lesssim 
   \mathbb{E} \left(\int_{s}^\tau 
   \left(\left(
   \mathcal{U}^{(\fn)}_{u, \tau,  \boldsymbol{\sigma}}
   \otimes 
   \mathcal{U}^{(\fn)}_{u, \tau,  \overline{\boldsymbol{\sigma}}}
   \right) \;\circ  \;
   \left( \cB \otimes  \cB \right)^{(2\fn)}
   _{u, \boldsymbol{\sigma}  }
   \right)_{\ba, \ba}\dd u 
   \right)^{p}
  \end{equation} 
  where $\mathcal{U}^{(\fn)}_{u, \tau,  \boldsymbol{\sigma}} \otimes    \mathcal{U}^{(\fn)}_{u, \tau,  \overline{\boldsymbol{\sigma}}}$ denotes the tensor product of the evolution kernel in \eqref{def_Ustz}: 
  \begin{align*}
  &\left[\left(
   \mathcal{U}^{(\fn)}_{u, \tau,  \boldsymbol{\sigma}}
   \otimes 
   \mathcal{U}^{(\fn)}_{u, \tau,  \overline{\boldsymbol{\sigma}}}
   \right)\circ \cal A\right]_{\ba,\ba'} =
   \sum_{\mathbf{b},\mathbf{b}'\in(\Zn)^{\fn}} \; \prod_{i=1}^\fn \left(\frac{1 - u M^{(\sig_i,\sig_{i+1})}}{1 - \tau M^{(\sig_i,\sig_{i+1})}}\right)_{[a_i] [b_i]} \prod_{i=1}^\fn \left(\frac{1 - u M^{(-\sig_{i},-\sig_{i+1})}}{1 - \tau M^{(-\sig_{i},-\sig_{i+1})}}\right)_{[a'_i] [b'_i]} \mathcal{A}_{\mathbf{b},\mathbf{b}'}
   \end{align*}
  for any $(2\fn)$-dimensional tensor ${\cal A}: (\Zn)^{2\fn}\to \mathbb C$ and $\mathbf b=([b_1],\ldots, [b_\fn])$, $\mathbf b'=([b_1'],\ldots, [b_\fn'])$.
\end{lemma}


For the proofs in subsequent steps, we will need to control the terms in equation \eqref{int_K-L_ST}, which requires the following estimates regarding the evolution kernel in \Cref{DefTHUST}. 
We first have a simple bound on the $({\infty\to \infty})$-norm of \smash{$\mathcal U_{t,\bsig,\ba}^{(\fn)}$}, whose proof will be given in \Cref{sec:pf_sum_Ndecay}. 

\begin{lemma}
\label{lem:sum_Ndecay}
Let \smash{${\cal A}: (\Zn)^{\fn}\to \mathbb C$} be an $\fn$-dimensional tensor for a fixed $\fn\in \N$ with $\fn\ge 2$. Then, for each $0\le s \le t < 1$, we have that  
\begin{align}\label{sum_res_Ndecay}
   \| {\cal U}^{(\fn)}_{s,t,\boldsymbol{\sigma}}\;\circ {\cal A}\|_{\infty} \lesssim \left(  \eta_s/ \eta_t\right)^{\fn }\cdot \|{\cal A}\|_{\infty} , 
\end{align}
where the $L^\infty$-norm of ${\cal A}$ is defined as $\|{\cal A}\|_{\infty}=\max_{\ba\in (\Zn)^\fn}|\cal A_{\ba}|$.
\end{lemma}

If the tensor $\cal A$ exhibits faster-than-polynomial decay at scales larger than $\ell_s$, then we obtain a stronger bound than the $({\infty\to \infty})$-norm bound given by \eqref{sum_res_Ndecay}. 
This bound can be further improved if $\cal A$ satisfies certain sum zero property or parity symmetry.    
We postpone the proof of this lemma to \Cref{sec:pf_sum-decay}. 
 
\begin{lemma}\label{lem:sum_decay}
Let ${\cal A}: (\Zn)^{\fn}\to \mathbb C$ be an $\fn$-dimensional tensor for a fixed $\fn\in \N$ with $\fn\ge 2$. Suppose it satisfies the following decay property for some small constant $\e\in(0,1)$ and large constant $D>1$:  
\begin{equation}\label{deccA0}
\max_{i,j\in \Zn}|[a_i]-[a_j]|\ge W^{\e}\ell_s \ \ \text{for} \ \ \ba=([a_1],\ldots, [a_\fn])\in (\Zn)^{\fn}  \implies  |\cal A_{\ba}|\le W^{-D} .
\end{equation}
Fix any $0\le s \le t < 1$ such that $(1-t)/(1-s)\ge W^{-d}$. 
There exists a constant $C_\fn>0$ that does not depend on  $\e$ or $D$ such that the following bound holds  (note  $\eta_t\ell_t^d\lesssim\eta_s\ell_s^d$ by the definition \eqref{eq:ellt}):
\begin{align}\label{sum_res_1}
    \left\|{\cal U}^{(\fn)}_{s,t,\boldsymbol{\sigma}} \circ {\cal A}\right\|_\infty \le W^{C_\fn\e}\frac{\ell_t^d }{\ell_s^d }\left(\frac{\ell_s^d \eta_s}{\ell_t^d \eta_t}\right)^{\fn} \|{\cal A}\|_\infty
    +W^{-D+C_\fn}.
\end{align}
In addition, stronger bounds hold in the following cases: 
\begin{itemize}
\item[(I)] If we have $\sigma_1=\sigma_2$ for $\boldsymbol{\sigma}=(\sigma_1,\cdots, \sigma_\fn),$ then
\begin{align}\label{sum_res_2_NAL}
    \left\|{\cal U}^{(\fn)}_{s,t,\boldsymbol{\sigma}} \circ {\cal A}\right\|_\infty \le W^{C_\fn\e}  \left(\frac{\ell_s^d\eta_s}{\ell_t^d\eta_t}\right)^{\fn-1}  \|{\cal A}\|_{\infty} 
   +W^{-D+C_\fn}.
\end{align} 

\item[(II)] If ${\cal A}$ satisfies the following sum zero property: 
\begin{align}\label{sumAzero}
 \sum_{[a_2],\ldots,[a_\fn]\in \Zn}{\cal A}_{\ba}=0 \quad \forall [a_1]\in \Zn,
\end{align}
then we have that 
\begin{align}\label{sum_res_2}
    \left\|{\cal U}^{(\fn)}_{s,t,\boldsymbol{\sigma}} \circ {\cal A}\right\|_\infty \le W^{C_\fn\e} \frac{\ell_t^{d-1}}{\ell_s^{d-1}} \left(\frac{\ell_s^d\eta_s}{\ell_t^d\eta_t}\right)^{\fn}  \|{\cal A}\|_{\infty} 
   +W^{-D+C_\fn}.
\end{align} 
If ${\cal A}$ further satisfies the following parity symmetry:
\be\label{eq:A_zero_sym}
\quad \cal A_{([a],[a]+[b_2],\ldots, [a]+[b_\fn])} = \cal A_{([a],[a]-[b_2],\ldots, [a]-[b_\fn])},\quad \forall [a],[b_2],\ldots, [b_\fn]\in \Zn, 
\ee
then the bound \eqref{sum_res_2} can be improved as follows:
\begin{align}\label{sum_res_2_sym}
    \left\|{\cal U}^{(\fn)}_{s,t,\boldsymbol{\sigma}} \circ {\cal A}\right\|_\infty \le W^{C_\fn\e} \left(\frac{\ell_s^d\eta_s}{\ell_t^d\eta_t}\right)^{\fn}  \|{\cal A}\|_{\infty} 
   +W^{-D+C_\fn}.
\end{align}
\end{itemize}

\end{lemma}

\subsection{Step 2: Sharp local law and a priori 2-$G$ loop estimate}

In this section, we focus on the $2$-$G$-loop with $\fn=2$ and $\boldsymbol{\sigma}=(+,-)$, while the case $\boldsymbol{\sigma}=(-,+)$ also follows by taking the matrix transposition. 

Given $u\in[s,t]$, $0\le \ell\le n$, and a sufficiently large constant $D>0$, we introduce the following functions 
\begin{align}
      {\cal T}^{(\cal L-K)}_u (\ell)&:= \max_{[a],[b] : \; |[a]-[b]|\ge \ell} \left| ({\cal L-\cK})^{(2)}_{u,  (+,-),  ( [a], [b])}\right|, 
      \label{def_WTu} \\
{\cal T} _{u, D} (\ell) &:=  (W^d\ell_u^d\eta_u)^{-2}
\exp \Big(- \big( \ell /\ell_u \big)^{1/2} \Big)+W^{-D},\label{def_WTuD}
\end{align} 
that control the tail behavior of $\cL-\cK$. Note ${\cal T}^{(\cal L-K)}_u$ and ${\cal T} _{u, D}$ are non-increasing functions in $\ell$:
 $$ 
 {\cal T}^{(\cal L-K)}_u (\ell_1)\ge {\cal T}^{(\cal L-K)}_u (\ell_2),\quad {\cal T} _{u,D} (\ell_1)\ge{\cal T} _{u,D}(\ell_2),\quad \text{for}\quad 0\le \ell_1\le \ell_2.
 $$
Furthermore, denote the ratio between them by  
\begin{equation}\label{def_Ju}
     {\cal J} _{u,D} (\ell):=   {\cal T}^{(\cal L-K)}_u (\ell) \big/{\cal T}_{u,D} (\ell) +1\, .
\end{equation} 
Then, to show \eqref{Eq:Gdecay_w}, it suffices to prove that for all $u\in [s,t]$ and any large constant $D>0$: 
 \begin{align}\label{shoellJJ}
\max_{0\le \ell \le n}   {\cal J} _{u,D} (\ell) \prec (\eta_s/\eta_u)^4 .
  \end{align}
For the proof of \eqref{shoellJJ}, we let ${\cal J}^* _{u,D}\ge 1$ be a \emph{deterministic parameter} such that 
\be\label{eq:def_new_J*}
\max_{0\le \ell \le n}   {\cal J} _{u,D} (\ell) \prec {\cal J}^* _{u,D}.
\ee
Moreover, we introduce an intermediate 
scale parameter
\be\label{eq:ell*u}
  \ell^{*}_u := (\log W)^{3/2}\ell_u, \quad u\in[s,t] .
\ee
Note that on the scale $\ell_u^*$, the propagators $\Theta_u^{\bsig}$ is exponentially small (i.e., faster than any polynomial decay) by \Cref{lem_propTH}, whereas ${\cal T}_{u,D}$ is not (i.e., slower than any polynomial decay).

We now estimate the 2-$G$ loops by bounding the terms in \eqref{int_K-L_ST} with $\fn=2$.

\begin{lemma}\label{lem_dec_calE} In the setting of  \Cref{lem:main_ind}, suppose the estimates \eqref{lRB1} and \eqref{Gtmwc} 
hold. Fix any $u\in [s,t]$, $D\ge 10$, $\bsig=(+,-)$, and $\ba=([a_1],[a_2])$, we have  
\begin{align}\label{res_deccalE_lk}
{\cal E}^{(2)}_{u, \boldsymbol{\sigma},\ba} / 
  {\cal T}_{t,D}(|[a_1]-[a_2]|) \prec  \eta_u^{-1}  \left(W^d\ell_u^d \eta_u\right)^{-1}
 \left({\cal J}^{*}_{u,D} \right)^{2}.
\end{align} 
\end{lemma}
\begin{proof}
    The proof of this lemma follows directly from the definition of ${\cal J}^{*}_{u,D}$ and is the same as that of equation (5.34) in \cite{Band1D}. Hence, we omit the details. 
\end{proof}
  \begin{lemma}\label{lem_dec_calE2}  Under the assumptions of \Cref{lem_dec_calE}, we have that  
\begin{align}
      \cW_{u, \boldsymbol{\sigma},\ba}^{(2)} 
    /  
 {\cal T}_{t,D}(|[a_1]-[a_2]|)
 & \prec \eta_u^{-1} \left(\ell^d_u/\ell^d_s\right)^2{\bf 1}\left(|[a_1]-[a_2]|\le \ell_u^*\right)  +   \eta_u^{-1} \left(W^d\ell^d_u \eta_u \right)^{-1/3} 
       \left({\cal J}^{*}_{u,D} \right)^{3}.
 \label{res_deccalE_wG}
\end{align} 
  \end{lemma}

  \begin{lemma}\label{lem_dec_calE3}  Under the assumptions of \Cref{lem_dec_calE}, assume further that $\ba'=([a_1'],[a_2'])$ satisfies
\begin{align} 
\max_{i=1}^2 |[a_i]-[a_i']|\le \ell_t^*. \label{57}
\end{align}
Then, we have that 
\begin{align}
\label{res_deccalE_dif}
    \left(
   \mathcal{B}\otimes  \mathcal{B} \right)^{(4)}
   _{u, \bsig, \ba ,\ba'} /  
 {\cal T}_{t,D}^2(|[a_1]-[a_2]|)
 & \prec  \eta_u^{-1} \left(\frac{\ell^d_u}{\ell^d_s}\right)^5 {\bf 1}(|[a_1]-[a_2]|\le 4\ell_t^*)  +  \eta_u^{-1}{\left(W^d \ell_u^d\eta_u\right)^{-1/3} } {\left({\cal J}^{*}_{u,D} \right)^{3}}.   
\end{align} 
  \end{lemma}

The proof of \Cref{lem_dec_calE2,lem_dec_calE3} depends on \Cref{lem_GbEXP} and will be given in \Cref{sec:pf_dec_calE2,sec:pf_dec_calE3}, respectively.  
Using Lemmas \ref{lem_dec_calE}--\ref{lem_dec_calE3}, along with the evolution kernel estimate in \Cref{lem:sum_Ndecay}, we can complete Step 2 in the proof of \Cref{lem:main_ind}. Define the stopping time
\begin{align}\label{eq:def_TTT}
T:=\inf \Big\{u\ge s: \max_{0\le \ell \le n}   {\cal J} _{u,D} (\ell) \ge W^\e\left(\eta_s / \eta_t\right)^4\Big\} .
\end{align}
Then, the estimates \eqref{Gt_bound_flow} and \eqref{Eq:Gdecay_w} are immediate consequences of the following lemma, whose proof will be deferred to \Cref{sec:pf_step2}.

\begin{lemma}\label{lem:pf_step2} 
In the setting of \Cref{lem:main_ind}, suppose \Cref{lem_dec_calE,lem_dec_calE2,lem_dec_calE3} holds. Then, for sufficiently small constant $\e>0$, we have that $T\ge t$ with high probability. Furthermore, we have a slightly stronger bound that will also be used in Step 5: 
\begin{align}\label{53}
\left({\cal L}-{\cal K}\right)^{(2)}_{t, \bsig, \ba} \prec  {\cal T}_{t, D}(|[a_1]-[a_2]|)\cdot  \big [ (\eta_s/\eta_t)^{5/2}  \cdot {\bf 1}(|a_1-a_2|\le 6\ell_t^*)+1 \big ].
\end{align}
\end{lemma}  

\begin{proof}[\bf Step 2: Proof of \eqref{Gt_bound_flow} and \eqref{Eq:Gdecay_w}] 
By the definitions \eqref{shoellJJ} and \eqref{Eq:Gdecay_w}, the estimate \eqref{Eq:Gdecay_w} follows readily from the fact that $T\ge t$ with high probability. 
Combining \eqref{Eq:Gdecay_w} with \Cref{ML:Kbound}, we obtain that
\begin{equation}
\label{lk2safyas}
 \max_{[a], [b]}{\cal L}^{(2)}_{u, (+,-), ([a],[b])}\prec \left(\eta_s/\eta_u\right)^4\cdot (W^d\ell_u^d\eta_u)^{-2}+(W^d\ell_u^d\eta_u)^{-1}\le 2(W^d\ell_u^d\eta_u)^{-1}
\end{equation}
for all $u\in [s,t]$, where we used the condition  \eqref{con_st_ind} in the second step. 
Then, applying \Cref{lem_GbEXP} (with 
the weak local law in \eqref{Gtmwc} verifying the assumption \eqref{def_asGMc}), we conclude the proof of \eqref{Gt_bound_flow}. 
\end{proof}

\subsection{Step 3: Sharp $G$-loop bound}\label{sec_sumzero}

In Step 2 of the proof of \Cref{lem:main_ind}, we have established an exponential decay of the 2-$G$ loops beyond the scale $\ell_u$ as shown in \eqref{Eq:Gdecay_w}. 
With \Cref{lem_GbEXP}, we can easily extend this decay to general $G$-loops.

\begin{definition}[Fast decay property]\label{Def_decay}
Let ${\cal A}: (\Zn)^{\fn}\to \mathbb C$ be an $\fn$-dimensional tensor for a fixed $\fn\in \N$ with $\fn\ge 2$. Given $u\in[s,t]$ and constants $\e,D>0$, we say $\cal A$ satisfies the $(u, \e, D)$-decay property if 
\begin{equation}\label{deccA}
    \max_{i,j\in \qqq{\fn}}|[a_i]-[a_j]|\ge \ell_u W^{\e}\ \implies \ {\cal A}_{\ba}=\OO(W^{-D})\quad \text{for}\quad \ba=([a_1],[a_2],\ldots, [a_\fn]).
\end{equation}
\end{definition}

It is easy to see that the $G$-loops satisfy the $(u, \e, D)$-decay property under the estimates \eqref{Gt_bound_flow} and \eqref{Eq:Gdecay_w} for any constants $\e,D>0$.
The following claim is an immediate consequence of \Cref{lem_propTH,lem_GbEXP}.

\begin{claim}
\label{lem_decayLoop} 
Suppose \eqref{Gt_bound_flow} and \eqref{Eq:Gdecay_w} hold. For any $\fn\ge 2$, $\bsig\in \{+,-\}^\fn$, $u\in[s,t]$, and constants $\e,D>0$, the $G$-loops \smash{${\cal L}^{(\fn)}_{u,\bsig,\ba}$} and primitive loops \smash{${\cal K}^{(\fn)}_{u,\bsig,\ba}$} satisfy the $(u, \e, D)$-decay property with probability $1-\OO(W^{-D'})$ for any large constant $D'>0$. In other words, we have that
\begin{align}\label{res_decayLK}
\mathbb P\left( \max_{\boldsymbol{\sigma}}  \Big(\left|{\cal L}^{(\fn)}_{u,\boldsymbol{\sigma},\ba}\right|+\left|{\cal K}^{(\fn)}_{u,\boldsymbol{\sigma},\ba}\right|\Big)\cdot{\bf 1}\left(\max_{  i, j\in \qqq{\fn} } |[a_i]-[a_j]|\ge W^{\e}\ell_u \right) \ge W^{-D}\right)\le W^{-D'}. 
\end{align}
 \end{claim}

Due to the fast decay property of the $G$-loops and $\cK$-loops, when the evolution kernels act on them, we can apply \Cref{lem:sum_decay} to bound their $(\infty\to \infty)$-norms, which leads to a crucial improvement over \Cref{lem:sum_Ndecay}. 

For any $\fn\ge 1$, let $\Xi^{({\cal L})}_{t, \fn}\ge 1$ and $\Xi^{({\cal L}-{\cal K})}_{t, \fn}\ge 1$ be \emph{deterministic} control parameters for $G$-loops and $(\cL-\cK)$-loops of length $\fn$ such that the following bounds hold: 
\begin{align}\label{def:XiL}
	\wh\Xi^{({\cal L})}_{t, \fn} &:= \max_{\boldsymbol{\sigma}\in \{+, -\}^\fn}\max_{\ba\in (\Zn)^\fn} \left|{\cal L}^{(\fn)}_{t, \boldsymbol{\sigma}, \ba} \right|\cdot \left(W^d\ell_t^d \eta_t\right)^{\fn-1}\prec \Xi^{({\cal L})}_{t, \fn},\\
	\wh\Xi^{({\cal L}-{\cal K})}_{t, \fn} &:=\max_{\boldsymbol{\sigma}\in \{+, -\}^\fn}\max_{\ba\in (\Zn)^\fn} \left|({\cal L}-{\cal K})^{(\fn)}_{t, \boldsymbol{\sigma}, \ba} \right|\cdot \left(W^d\ell_t^d \eta_t\right)^{\fn} \prec \Xi^{({\cal L}-{\cal K})}_{t, \fn}.\label{def:XIL-K}
\end{align}
We can control the terms on the RHS of equation \eqref{int_K-LcalE} using these parameters, the initial estimate \eqref{Eq:L-KGt+IND},  \Cref{lem:DIfREP} for the martingale term, and the evolution kernel estimates in \Cref{lem:sum_decay} due to the fast decay property shown in \Cref{lem_decayLoop}. 
In particular, for non-alternating loops $\bsig$, using Case I of \Cref{lem:sum_decay}, we can readily establish the following lemma, whose proof will be postponed to \Cref{sec:pf_STOeq_NQ}. 

\begin{lemma}\label{lem:STOeq_NQ} 
Under the assumptions of \Cref{lem:main_ind}, suppose the estimates \eqref{Gt_bound_flow} and \eqref{Eq:Gdecay_w} hold uniformly in $u\in[s,t]$. Fix any $\fn\ge 2$ and $\bsig\in \{+,-\}^\fn$ satisfying 
\begin{equation}\label{NALsigm}
\sigma_k=\sigma_{k+1} \quad \text{for some}\quad  k\in\qqq{\fn}.
\end{equation} 
(Recall that $\sig_{\fn+1}=\sig_1$ as a convention.) 
Then, we have the following estimate: 
\begin{align}\label{am;asoiuw}
&\max_{\ba\in (\Zn)^\fn} \left|({\cal L}-{\cal K})^{(\fn)}_{t, \boldsymbol{\sigma}, \ba} \right|\cdot \left(W^d\ell_t^d \eta_t\right)^{\fn} \\
&\prec \sup_{u\in [s,t]}
  \left(\max_{k\in \qqq{2,\fn-1} }\Xi^{({\cal L-\cal K})}_{u,k}+\max_{k\in \qqq{2,\fn} }\frac{\Xi^{({\cal L-\cal K})}_{u,k} \Xi^{({\cal L-\cal K})}_{u,\fn-k+2}}{W^d\ell_u^d\eta_u}+\Xi^{({\cal L})}_{u,\fn+1}+(\Xi^{(\cal L)}_{u, {2\fn+2}})^{1/2}\right).\nonumber
\end{align} 
\end{lemma}

To deal with the case with alternating signs, where \eqref{NALsigm} does not occur: 
\begin{equation}\label{NALsig_diff}
\sigma_k=- \sigma_{k+1}, \quad \forall \ k\in\qqq{\fn}\, ,
\end{equation}
we introduce another key tool---a sum zero operator ${\cal Q}_t$.

\begin{definition}[Sum zero operator]\label{Def:QtPt} 
Let ${\cal A}: (\Zn)^{\fn}\to \mathbb C$ be an $\fn$-dimensional tensor for a fixed $\fn\in \N$ with $\fn\ge 2$. Define the partial sum operator ${\cal P}$ as 
$$
\left( {\cal P} \circ {\cal A}\right)_{[a_1]}= \sum_{[a_i]: i\in\qqq{2,\fn}}  {\cal A}_{\ba}. 
$$
Note a tensor $\cal A$ satisfies the sum zero property in \eqref{sumAzero} if and only if $ {\cal P} \circ {\cal A}= 0$. Given $t\in[0,1)$, we define 
\be\label{eq:sumzero_op}
 \left({\cal Q}_{t}\circ {\cal A}\right)_{\ba}={\cal A}_{\ba}-\left({\cal P} \circ {\cal A}\right)_{[a_1]}\dthn^{(\fn)}_{t, \ba},\quad 
 \dthn^{(\fn)}_{t, \ba}:=(1-t)^{\fn-1}\prod_{i=2}^\fn \Theta^{(+,-)}_{t,[a_1][a_i]}.\ee
Since $\sum_{[a_i]}\Theta^{(+,-)}_{t,[a_1][a_i]}=(1-t)^{-1}$ by \eqref{eq:WardM1}, we can check that $ {\cal P} \circ  \dthn^{(\fn)}_{t, \ba}=1$, ${\cal P} \circ {\cal Q}_t =0$, and 
\be\label{eq:sum0PA}
{\cal P} \circ {\cal A} = 0\ \ \implies\ \ {\cal P} \circ \left(\thn^{(\fn)}_{t, \boldsymbol{\sigma}}  \circ {\cal A}\right) = 0,
\ee
where we recall the operator \(\thn^{(\fn)}_{t, \boldsymbol{\sigma}} \) defined in \Cref{DefTHUST}. In other words, if \({\cal A}\) satisfies the sum zero property \eqref{sumAzero}, then so does \smash{$\thn^{(\fn)}_{t, \boldsymbol{\sigma}}  \circ {\cal A}$}. 
\end{definition}

We will use the sum zero operator to get improved estimates on the terms on the RHS of \eqref{int_K-LcalE} when \eqref{NALsig_diff} holds. 
Roughly speaking, we will decompose a tensor $\cal A$ as \smash{${\cal A}_{\ba}= ({\cal Q}_{t}\circ {\cal A})_{\ba} + \left({\cal P} \circ {\cal A}\right)_{[a_1]}\dthn^{(\fn)}_{t, \ba}$} using \eqref{eq:sumzero_op}. For the first part, we can get an improvement by using \eqref{sum_res_2}, while for the second part, we can apply Ward's identity to ${\cal P} \circ {\cal A}$. 
This leads to the following lemma, whose proof is postponed to \Cref{sec:pf_STOeq_Qt_weak}. 

\begin{lemma}\label{lem:STOeq_Qt_weak} 
Under the assumptions of \Cref{lem:main_ind}, suppose the estimates \eqref{Gt_bound_flow} and \eqref{Eq:Gdecay_w} hold uniformly in $u\in[s,t]$. 
For any fixed $\fn\ge 2$, we have that: 
\begin{align}\label{am;asoi222_weak}
 \wh\Xi^{({\cal L-\cal K})}_{t,\fn}
& \prec \frac{\ell_t^{d-1}}{\ell_s^{d-1}}\sup_{u\in [s,t]}
  \left(\max_{k\in \qqq{1,\fn-1} }\Xi^{({\cal L-\cal K})}_{u,k}+\max_{k\in \qqq{2, \fn} }\frac{ \Xi^{({\cal L-\cal K})}_{u,k}\Xi^{({\cal L-\cal K})}_{u,\fn-k+2}}{W^d\ell_u^d\eta_u} +\Xi^{({\cal L})}_{u,\fn+1}+\left(\Xi^{(\cal L)}_{u, {2\fn+2}}\right)^{1/2}\right).
\end{align}
\end{lemma}

Compared to \eqref{am;asoiuw}, the estimate \eqref{am;asoi222_weak} is weaker by a factor $\ell_t/\ell_s$ when $d=2$, due to the additional \smash{$\ell_t^{d-1}/\ell_s^{d-2}$} factor arising in \eqref{sum_res_2}. To complete Step 3, it is necessary to eliminate this factor. 
As mentioned in the introduction, addressing this issue requires the exploration of a CLT-type cancellation mechanism. 
We formalize this in the following lemma. Owing to its similarity to the argument in \cite[Section 7]{Band2D}, we choose to defer the proof to \Cref{sec:add_d=2}. 
Nonetheless, for the reader’s convenience, the proof is presented in a self-contained manner.


\begin{lemma}\label{lem:STOeq_Qt} 
In the setting of \Cref{lem:STOeq_Qt_weak}, we have the following estimate: 
\begin{align}\label{am;asoi222}
 \wh\Xi^{({\cal L-\cal K})}_{t,\fn}
 \prec \sup_{u\in [s,t]}
  \left(\; \max_{k\in \qqq{1,\fn-1} }\Xi^{({\cal L-\cal K})}_{u,k}+\max_{k\in \qqq{2, \fn} }\frac{\Xi^{({\cal L-\cal K})}_{u,k}\Xi^{({\cal L-\cal K})}_{u,\fn-k+2}}{W^d\ell_u^d\eta_u} +\Xi^{({\cal L})}_{u,\fn+1}+\left(\Xi^{(\cal L)}_{u, {2\fn+2}}\right)^{1/2}\right).
\end{align}
\end{lemma}

We are now ready to complete the proof of Step 3.  
In this step, besides the assumptions in \Cref{lem:main_ind}, we have also established the a priori $G$-loop bound \eqref{lRB1} in Step 1, as well as the sharp local law \eqref{Gt_bound_flow} and the weak $2$-$G$ loop estimate \eqref{Eq:Gdecay_w} in Step 2. Furthermore, by \eqref{Gt_bound_flow}, \eqref{Eq:Gdecay_w}, and \eqref{GavLGEX}, we have that 
\begin{align}\label{eq:res_ELK_n=1} 
  \max_{[a]}  \left|\langle (G_u-M) E_{[a]}\rangle\right|
  \prec (W^d\ell_u^d\eta_u)^{-1}  ,
  \end{align} 
which implies that $\wh\Xi_{u,1}^{(\cL-\cK)}\prec 1$ uniformly in $u\in[s,t]$. 
With these inputs, in Lemma \ref{lem:STOeq_Qt}, we have established the following estimate uniformly in $u\in [s,t]$: 
\begin{align}\label{am;asoi333}
\sup_{v\in[s,u]} \wh \Xi^{({\cal L-\cal K})}_{v,\fn}
\prec \sup_{v\in [s,u]}
  \left(\; \max_{r\in \qqq{2,\fn-1} }\Xi^{({\cal L-\cal K})}_{v,r}+\max_{r\in \qqq{2, \fn} } \frac{\Xi^{({\cal L-\cal K})}_{v,r}\Xi^{({\cal L-\cal K})}_{v,\fn-r+2}}{W^d\ell_v^d\eta_v }+\Xi^{({\cal L})}_{v,\fn+1}+\left(\Xi^{(\cal L)}_{v, {2\fn+2}}\right)^{1/2}\right).
\end{align}
(By Lemma \ref{lem:STOeq_Qt}, the estimate \eqref{am;asoi333} holds for each fixed $u\in[s,t]$. Again, applying an $N^{-C}$-net argument extends it uniformly to all $u\in [s,t]$.)
Similar to the argument in Section 5.6 of \cite{Band1D}, we will iterate the estimate \eqref{am;asoi333} to derive the sharp bound \eqref{Eq:LGxb} on $G$-loops, that is, for each fixed $\fn\in \N$, 
\be\label{eq:Xiiter}
\sup_{u\in [s,t]} \wh\Xi^{(\cL)}_{u,\fn} \prec 1.
\ee

By the $\cK$-loop bound \eqref{eq:bcal_k}, $\wh\Xi^{(\cal L)}_{u,\fn}$ and $\wh\Xi^{(\cal L-\cal K)}_{u,\fn}$ bound each other as follows:
\begin{align}\label{rela_XILXILK}
    \wh\Xi^{(\cal L)}_{u,\fn}\prec 1+\left(W^d\ell_u^d\eta_u\right)^{-1}\cdot \wh\Xi^{({\cal L-\cal K})}_{u, \fn},\quad \quad \wh\Xi^{({\cal L-\cal K})}_{u, \fn}\prec \left(W^d\ell_u^d\eta_u\right) \left(\wh\Xi^{(\cal L)}_{u,\fn}+1\right).
\end{align}
Moreover, by the a priori $G$-loop bound in \eqref{lRB1}, we have the following initial bound uniformly in $u\in[s,t]$:
\begin{align}\label{sef8w483r324}
  \wh\Xi^{(\cal L)}_{u,\fn}\prec (\ell_u^d/\ell_s^d)^{\fn-1}, \quad 
 \wh\Xi^{({\cal L-K})}_{u, \fn}\prec (\ell_u^d/\ell_s^d)^{\fn-1}\cdot (W^d\ell_u^d\eta_u) .
\end{align}
Then, for $u\in [s,t]$, we define the control parameter
\begin{equation}\label{adsyzz0s8d6}
      \Psi_{u}(\fn,k;[s,t]):=  (W^d\ell_s^d\eta_s)^{1/2}+(\ell_t^d/\ell_s^d)^{\fn-1}\times
    \begin{cases}
        W^d\ell_u^d\eta_u , &  \ \text{if}\ k=0\\
        (W^d\ell_s^d\eta_s)^{1-k/4}, &   \ \text{if}\  k\ge 1 
    \end{cases}.
\end{equation} 
Note that when $k\ge 1$, $\Psi_{u}(\fn,k;[s,t])$ does not depend on $u$. 
The iterations will be performed in both $\fn$ and $k$ at the same time. We summarize the result of each iteration in the next lemma, whose proof will be deferred to \Cref{sec:pf_iterations}. 

\begin{lemma}\label{lem:iterations}
    In the setting of \Cref{lem:main_ind}, suppose \eqref{Eq:Gdecay_w} and \eqref{am;asoi333} hold uniformly in $u\in[s,t]$. Fix any $\fn,k\in \N$ with $\fn\ge 2$ and $k\ge 1$. Suppose the following estimate holds uniformly in $u\in[s,t]$: 
    \be\label{eq:iteration_induc}
    \sup_{v\in[s,u]}\wh \Xi^{({\cal L-K})}_{v,r}\prec   \sup_{v\in[s,u]}\Psi_v(r,l;[s,u])\ee
    for all $(r,l)$ satisfying one of the following conditions: (1) $l=k$ and $2\le r \le \fn-1$; (2) $l=k-1$ and $2\le r\le \fn+2$. Then, we have the following estimate uniformly in $u\in[s,t]$:
    \be\label{eq:iteration_improve}
    \sup_{v\in[s,u]}\wh \Xi^{({\cal L-K})}_{v,\fn}\prec   \Psi_u(\fn,k;[s,u]).\ee
\end{lemma}

Roughly speaking, this lemma states if we have already established a ``good" bound for every shorter $G$-loop of length $r\le \fn-1$ (case (1)) or a ``weaker" bound for $G$-loops of length $r\le \fn+2$ (case (2)), then we can derive the ``good" bound for all $G$-loops of length $\fn$. 
Using \Cref{lem:iterations}, the iterative argument leading to the proof of \eqref{Eq:LGxb} in Step 3 proceeds as follows:


\begin{proof}[\bf Step 3: Proof of \eqref{Eq:LGxb}]
By \eqref{sef8w483r324}, we initially have a weak bound for $G$-loops of arbitrarily large lengths, meaning that \eqref{eq:iteration_induc} holds with $l=0$ for every $r\in \N$. Applying \Cref{lem:iterations} once, we obtain a slightly improved bound \eqref{eq:iteration_improve} with $k=1$ and $\fn=2$. Then, continuing the iteration in $\fn$ while keeping $k=1$ fixed, we establish the bound \eqref{eq:iteration_improve} for each fixed $\fn\in \N$ with $k=1$. 
Next, applying the iteration in \Cref{lem:iterations} again yields an even stronger bound \eqref{eq:iteration_improve} with $\fn=2$ and $k=2$. Repeating the iteration in $\fn$ while keeping $k=2$ fixed, we can establish the bound \eqref{eq:iteration_improve} for every fixed $\fn\in \N$ with $k=2$. This process continues, progressively improving the bound to \eqref{eq:iteration_improve} for each fixed $\fn\in \N$ with $k=3$, and so forth.

For any given $(\fn,k)\in \N^2$, by repeating the above argument for $\OO(1)$ times, we conclude that the estimate \eqref{eq:iteration_improve} holds. As a special case, it gives that  
\smash{\( \sup_{u\in[s,t]}\wh \Xi^{({\cal L-K})}_{u,\fn}\prec   \Psi_t(\fn,k;[s,t]). \)}
In particular, by condition \eqref{con_st_ind}, if we choose $k>2+\fn/25$, then $(\ell_t^d/\ell_s^d)^{\fn-1}\cdot  (W^d\ell_s^d\eta_s)^{1-k/4} \ll (W^d\ell_s^d\eta_s)^{1/2}$, which leads to 
$$\sup_{u\in[s,t]}\wh\Xi^{({\cal L-K})}_{u,\fn}\prec \Psi_t(\fn,k;[s,t]) \lesssim (W^d\ell_s^d\eta_s)^{1/2}.$$
Together with \eqref{rela_XILXILK}, it implies the estimate \eqref{eq:Xiiter} under the condition \eqref{con_st_ind}, thereby completing the proof of the $G$-loop bound in  \eqref{Eq:LGxb}.
\end{proof}

\subsection{Step 4: Sharp $(\cL-\cK)$-loop estimate}
 
For the proof of \eqref{Eq:L-KGt-flow}, we recall the estimate \eqref{am;asoi333} established in Step 3. 
In this step, using the sharp $G$-loop bound \eqref{Eq:LGxb}, we can choose the parameters \smash{$\Xi^{(\cal L)}_{v,\fn+1}=1$ and $\Xi^{(\cal L)}_{v,2\fn+2}=1$}. 
As a consequence, we obtain that for any $\fn\ge 2$: 
 \begin{align} \label{saww02}
 \sup_{v\in[s,u]}\wh\Xi_{v,\fn}^{(\mathcal{L}-\mathcal{K})}  \prec  \sup_{v\in [s,u]}
  \left(\max_{r\in\qqq{2,\fn-1}}\Xi^{({\cal L-\cal K})}_{v,r}+(W^d\ell_v^d\eta_v)^{-1}\max_{r\in\qqq{2,\fn} }\Xi^{({\cal L-\cal K})}_{v,r}\Xi^{({\cal L-\cal K})}_{v,\fn-r+2}\right).
\end{align}
On the other hand, by the rough bound \eqref{Eq:Gdecay_w} on $2$-$G$ loops, we have 
\be \label{saww02_ini}
\sup_{v\in [s,u]}\wh\Xi_{v,2 }^{(\mathcal{L}-\mathcal{K})}\prec (\eta_s/\eta_u)^{4}.
\ee
Hence, we can choose \smash{$\Xi_{v,2 }^{(\mathcal{L}-\mathcal{K})}\equiv (\eta_s/\eta_u)^{4}$} for $v\in[s,u]$. 
Then, by taking $\fn=2$ in \eqref{saww02} and using this parameter along with the condition \eqref{con_st_ind}, we derive the following self-improving estimate for $2$-$G$ loops:  
$$\sup_{v\in[s,u]}\wh\Xi_{v,2}^{(\mathcal{L}-\mathcal{K})}  \prec  1+ (W^d\ell_u^d\eta_u)^{-1} \sup_{v\in [s,u]}
  \left(\Xi^{({\cal L-\cal K})}_{v,2}\right)^2 \prec 1+ (W^d\ell_u^d\eta_u)^{-1/2} \sup_{v\in [s,u]}
\Xi^{({\cal L-\cal K})}_{v,2}.$$ 
Iterating this estimate for $\OO(1)$ many times gives
\(\sup_{v\in[s,u]}\wh\Xi_{v,2}^{(\mathcal{L}-\mathcal{K})}\prec 1=:\Xi_{v,2}^{(\mathcal{L}-\mathcal{K})}.\) 
Starting from this initial bound, by applying \eqref{saww02} inductively in $\fn$, we derive that 
\smash{\(\sup_{v\in[s,u]}\wh\Xi_{v,\fn}^{(\mathcal{L}-\mathcal{K})}\prec 1\)} for any fixed $\fn\in\N$, which concludes the estimate \eqref{Eq:L-KGt-flow}.

\subsection{Step 5: Sharp 2-$G$ loop estimate}
 
To show that \eqref{Eq:Gdecay_flow} holds uniformly in $u\in[s,t]$, note we have established in Step 4 that 
$$
\max_{\bsig\in\{+,-\}^2}\max_{\ba\in(\Zn)^2}\left({\cal L}-{\cal K}\right)^{(2)}_{u,\bsig,\ba}\prec (W^d\ell_u^d\eta_u)^{-2} 
$$
uniformly in $u\in[s,t]$. It already implies that 
$$
\left({\cal L}-{\cal K}\right)^{(2)}_{u,\bsig,\ba} \prec {\cal T}_{u, D}(|[a_1]-[a_2]|) \quad \text{for}\quad |[a_1]-[a_2]|=\OO(\ell_t^*).
$$
On the other hand, when $|[a_1]-[a_2]|\ge 6 \ell_t^*$, \eqref{Eq:Gdecay_flow} is a direct consequence of \eqref{53}.


\appendix

\section{Proof of some results in \Cref{Sec:Stoflo}}\label{sec:main_appd}

\subsection{Proof of \Cref{lem:sum_Ndecay}} \label{sec:pf_sum_Ndecay}
Notice the following identity:
\be\label{eq:decompU}\frac{1 - s M^{(\sig_i,\sig_{i+1})}}{1 - t M^{(\sig_i,\sig_{i+1})}} = 1 + \Xi^{(i)},\quad \text{where}\quad \Xi^{(i)}:=(t-s)\Theta_t^{(\sig_i,\sig_{i+1})} M^{(\sig_i,\sig_{i+1})} .\ee
From \eqref{eq:WardM2}, we have $\|\Theta_t^{(\sig_i,\sig_{i+1})}\|_{\infty\to \infty}\lesssim (1-t)^{-1}$, and from \eqref{Mbound_AO2}, it follows that $\|M^{(\sig_i,\sig_{i+1})}\|_{\infty\to \infty}\lesssim 1$. Combining these two bounds yields that 
\be\label{Xi_infint}
\|\Xi^{(i)}\|_{\infty\to \infty}\lesssim (t-s)/(1-t).
\ee
Together with the  identity \eqref{eq:decompU}, this implies that 
$$\left\|\frac{1 - s M^{(\sig_i,\sig_{i+1})}}{1 - t M^{(\sig_i,\sig_{i+1})}}\right\|_{\infty\to \infty}\lesssim \frac{1-s}{1-t}\asymp \frac{\eta_s}{\eta_t}.$$
With this estimate, we easily conclude the proof using the definition \eqref{def_Ustz}.

\subsection{Proof of \Cref{lem:sum_decay}}\label{sec:pf_sum-decay}
The corresponding lemma for the setting of 1D random band matrices has been proved as Lemma 7.3 in \cite{Band1D}. In our setting, the proof for the $d=1$ case is exactly the same by using the estimates in \Cref{lem_propTH}. Hence, we will focus on the $d=2$ case below. 

With the decomposition  \eqref{eq:decompU}, we can express \smash{$\mathcal{U}^{(\fn)}_{s, t, \boldsymbol{\sigma}} \circ \mathcal{A}$} as
\begin{align}\label{eq:decomp_U2}
      \left(\mathcal{U}^{(\fn)}_{s, t, \boldsymbol{\sigma}} \circ \mathcal{A}\right)_{\ba}&=\sum_{\mathbf b\in (\Zn)^\fn} \prod_{i=1}^\fn \left( \delta_{[a_i][b_i]}+\Xi^{(i)}_{[a_i][b_i]} \right)  \cdot \mathcal{A}_{\mathbf b} =\sum_{\mathbf b\in (\Zn)^\fn}\sum_{A\subset \qqq{\fn}} \prod_{i\in A} \delta_{[a_i][b_i]} \cdot \prod_{i\in A^c} \Xi^{(i)}_{[a_i][b_i]}  \cdot \mathcal{A}_{\mathbf{b}}\, .
\end{align}
Using the estimates \eqref{Mbound_AO2} and \eqref{prop:ThfadC}, we can show that 
\be\label{eq:decayXi}
\Xi^{(i)}_{[a_i][b_i]}\prec \frac{\eta_s}{\eta_t  {\ell}_t^d} e^{-c |[a_i]-[b_i]|/ {\ell}_t}+ W^{-D}
\ee
for any large constant $D>0$. 
We claim that for any subset $A$ with $|A|=k\in \qqq{1,\fn}$, 
\be\label{sum_res_1_red0}
\sum_{\mathbf b\in (\Zn)^\fn} \prod_{i\in A} \delta_{[a_i][b_i]} \cdot \prod_{i\in A^c} \Xi^{(i)}_{[a_i][b_i]}  \cdot \mathcal{A}_{\boldsymbol{b}} \le W^{C\e}\left(\frac{\ell_s^d \eta_s}{\ell_t^d \eta_t}\right)^{\fn-k} \|{\cal A}\|_{\infty}+ W^{-D+C },
\ee
for a constant $C$ that does not depend on $\e$ or $D$. On the other hand, for $A= \emptyset$, we claim that 
\be\label{sum_res_1_red}
\sum_{\mathbf b\in (\Zn)^\fn} \prod_{i=1}^\fn \Xi^{(i)}_{[a_i][b_i]}  \cdot \mathcal{A}_{\boldsymbol{b}} \le W^{C\e} \frac{\ell_t^d }{\ell_s^d }\left(\frac{\ell_s^d \eta_s}{\ell_t^d \eta_t}\right)^{\fn} \|{\cal A}\|_{\infty}+ W^{-D+C}.
\ee
For \eqref{sum_res_1_red0}, we assume that $A=\qqq{1,k}$ without loss of generality, and let $\mathbf a'$ and $\mathbf b'$ denote $\mathbf a'=([a_{1}],\ldots,[a_{k}])$ and $\mathbf b'=([b_{k+1}],\ldots,[b_{\fn}])$, respectively. Then, we can bound the LHS of \eqref{sum_res_1_red0} as 
\begin{align*}
\sum_{\mathbf b'\in (\Zn)^{\fn-k}} \prod_{i=k+1}^\fn  \Xi^{(i)}_{[a_i][b_i]}  \cdot \mathcal{A}_{\ba',\mathbf{b}'} &\prec  \left(\frac{\eta_s}{\eta_t\ell_t^d }\right)^{\fn-k} \|\cal A\|_\infty \sum_{\mathbf b'} \mathbf 1\left(\max_{i=k+1}^{\fn}|[b_i]-[a_1]|\le W^\e\ell_s \right) +  W^{-D+d(\fn-k)}\\
    &\lesssim (W^{d\e})^{\fn-k}\left(\frac{\ell_s^d \eta_s}{\ell_t^d \eta_t}\right)^{\fn-k} \|\cal A\|_\infty  + W^{-D+d(\fn-k)} , 
\end{align*}
where we used \eqref{eq:decayXi}, the decay property \eqref{deccA0} for $\cal A_{\mathbf b}$, and the estimate \eqref{Xi_infint}, together with the condition $(1-t)/(1-s)\ge W^{-d}$. 
This concludes \eqref{sum_res_1_red0} for any constant $C>d(\fn-k)$.  
For \eqref{sum_res_1_red}, we have that 
\begin{align}
   \sum_{\mathbf b}\prod_{i=1}^{\fn} \Xi^{(i)}_{[a_i][b_i]} \cdot \mathcal{A}_{\mathbf b} &= \sum_{\mathbf b}\prod_{i=1}^{\fn} \Xi^{(i)}_{[a_i][b_i]} \cdot \mathcal{A}_{\mathbf b}\cdot  \mathbf 1\left(\max_{i\ne j}|[b_i]-[b_j]|\le W^\e\ell_s \right) + \OO\left( W^{-D+d\fn}\right) \nonumber\\
   &\prec  \left(\frac{\eta_s}{\eta_t\ell_t^d }\right)^{\fn-1} \|\cal A\|_\infty \sum_{[b_1]} \left|\Xi^{(1)}_{[a_1][b_1]}\right| \sum_{\mathbf b'} \mathbf 1\left(\max_{i=2}^{\fn}|[b_i]-[b_1]|\le W^\e\ell_s \right) +  W^{-D+d\fn} \nonumber\\
    &\lesssim W^{d (\fn-1) \e }\left(\frac{\eta_s\ell_s^d}{\eta_t\ell_t^d}\right)^{\fn-1} \|\cal A\|_\infty \sum_{[b_1]} \left|\Xi^{(1)}_{[a_1][b_1]}\right| + W^{-D+d\fn} \nonumber\\
    &\le W^{d\fn \e}\frac{\eta_s}{\eta_t}\left(\frac{\eta_s\ell_s^d}{\eta_t\ell_t^d}\right)^{\fn-1}\|\cal A\|_\infty + W^{-D+d\fn} , \label{sum_res_deriv_red}
\end{align} 
where $\mathbf b'$ denotes $\mathbf b'=([b_{2}],\ldots,[b_{\fn}])$. Here, in the first step, we used the decay property \eqref{deccA0} for $\cal A_{\mathbf b}$ and the bound \eqref{Xi_infint}; in the second step, we applied \eqref{eq:decayXi} for $i\in\qqq{2,\fn}$; in the last step, we applied \eqref{Xi_infint} again to bound {$\sum_{[b_1]} |\Xi^{(1)}_{[a_1][b_1]}|$} by $\OO( \eta_s/\eta_t)$. This gives \eqref{sum_res_1_red} for any constant $C>d\fn$.  Combining \eqref{sum_res_1_red0} and \eqref{sum_res_1_red} concludes the proof of \eqref{sum_res_1}.   


For the estimates \eqref{sum_res_2_NAL}, \eqref{sum_res_2}, and \eqref{sum_res_2_sym}, we notice that when $|A| \ge 1$, \eqref{sum_res_1_red0} already gives a good enough bound. 
Hence, we only need to focus on the case where $A=\emptyset$. 
For Case I, due to \eqref{prop:ThfadC_short}, we can use the better bound {$\sum_{[b_1]}|\Xi^{(1)}_{[a_1][b_1]}|=\OO(1)$} in the last step of \eqref{sum_res_deriv_red}, which gives \eqref{sum_res_2_NAL}. For the estimate \eqref{sum_res_2} in Case II, we need to show that
\be\label{sum_res_2_red}
\sum_{\mathbf b\in (\Zn)^\fn} \prod_{i=1}^\fn \Xi^{(i)}_{[a_i][b_i]}  \cdot \mathcal{A}_{\boldsymbol{b}} \le W^{C_\fn\e} \frac{\ell_t^{d-1}}{\ell_s^{d-1}}\left(\frac{\ell_s^d\eta_s}{\ell_t^d\eta_t}\right)^{\fn} \|{\cal A}\|_{\infty} 
   +W^{-D+C_\fn}.
\ee
It suffices to assume that $\sigma_i\ne \sigma_{i+1}$ for all $i\in \qqq{\fn}$. We adopt a similar argument as that in the proof of \cite[Lemma 7.3]{Band1D}, that is, we decompose $\Xi^{(i)}$ as 
\be\nonumber \Xi^{(i)}_{[a_i][b_i]}=\Xi^{(i)}_{[a_i][b_1]} + \wh \Xi^{(i)}_{[a_i];[b_1][b_i]},\quad\text{with}\quad \wh \Xi^{(i)}_{[a_i];[b_1][b_i]}=\Xi^{(i)}_{[a_i][b_i]}-\Xi^{(i)}_{[a_i][b_1]}. \ee
Then, we expand the LHS of \eqref{sum_res_2_red} as 
\begin{align} \label{eq:decompXii}
\sum_{A: A^c\subset \qqq{2,\fn}}f(A),\quad \text{with}\quad 
f(A):=\sum_{\mathbf b\in (\Zn)^\fn}\prod_{i\in A}  \Xi^{(i)}_{[a_i][b_1]} \cdot \prod_{i\in A^c} \wh \Xi^{(i)}_{[a_i];[b_1][b_i]}  \cdot \mathcal{A}_{\mathbf{b}}. 
\end{align}
By the sum zero property \eqref{sumAzero}, the leading term with $A^c=\emptyset$ vanishes. For the remaining terms, we will use  \eqref{eq:decayXi} to control the factors \smash{$\Xi^{(i)}_{[a_i][b_1]}$} and apply \eqref{prop:BD1} and \eqref{Mbound_AO2} to control the factors \smash{$\wh \Xi^{(i)}_{[a_i];[b_1][b_i]}$}:
\be\label{eq:Xibb}
\Xi^{(i)}_{[a_i][b_1]} \le \frac{W^c \eta_s }{\eta_t\ell_t^d}, \quad   \wh \Xi^{(i)}_{[a_i];[b_1][b_i]} 
\le  \frac{W^{c}\eta_s}{ \eta_t\ell_t^2}\frac{|[b_i]-[b_1]|}{\qq{[a_i]-[b_1]}+\qq{[a_i]-[b_i]}},
\ee
for any small constant $c>0$.
Without loss of generality, suppose $2\notin A$. Then, using the estimates in \eqref{eq:Xibb} and the decay property \eqref{deccA0} for $\cal A$, by a similar argument as in \eqref{sum_res_deriv_red}, we can bound $f(A)$ by
\begin{align}
f(A)\lesssim  \left(W^{d\e+c}\frac{\ell_s^d\eta_s}{\ell_t^d\eta_t}\right)^{\fn-1}\|\cal A\|_\infty 
\sum_{[b_1]} \left|\Xi^{(1)}_{[a_1][b_1]}\right| \frac{W^\e\ell_s}{\qq{[a_2]-[b_1]}} + W^{-D}L^{d\fn} .\label{eq:bddfA}
\end{align}
In the derivation, we have also used that
$$ \sum_{[b_i]}\left|\wh \Xi^{(i)}_{[a_i];[b_1][b_i]}\right| 
\lesssim  \frac{n^dW^c\eta_s  }{\lambda^2+\eta_t} \le L^d,$$
by applying \eqref{eq:Xibb} in the first step and using $\lambda^2 \ge W^{-d+2\fd}$ in the second step due to \eqref{eq:cond-lambda2}.
Then, with \eqref{eq:decayXi}, we can further bound \eqref{eq:bddfA} as
\begin{align*}
f(A)&\lesssim  \left(W^{d\e+c}\frac{\ell_s^d\eta_s}{\ell_t^d\eta_t}\right)^{\fn} \|\cal A\|_\infty \cdot \frac{1}{\ell_s} 
\sum_{[b_1]} \frac{\exp(-c|[a_1]-[b_1]]|/\ell_t)}{\qq{[a_2]-[b_1]}} + W^{-D}L^{d\fn} \\
&\lesssim \frac{\ell_t}{\ell_s} \left(W^{d\e+c}\frac{\ell_s^d\eta_s}{\ell_t^d\eta_t}\right)^{\fn} \|\cal A\|_\infty + W^{-D}L^{d\fn}.
\end{align*}
This gives the estimate \eqref{sum_res_2_red} when $d=2$, which further concludes \eqref{sum_res_2}. 

For the proof of \eqref{sum_res_2_sym}, we notice that if $|A^c|\ge 2$, then the above derivation already gives a good enough bound. For example, if $2,3\notin A$, then using \eqref{eq:Xibb} and using the decay property \eqref{deccA0} for $\cal A$, we obtain that 
\begin{align}
f(A)&\lesssim  \left(W^{d\e+c}\frac{\ell_s^d\eta_s}{\ell_t^d\eta_t}\right)^{\fn} \|\cal A\|_\infty   
\sum_{[b_1]} \frac{\exp(-c|[a_1]-[b_1]]|/\ell_t)}{\qq{[a_2]-[b_1]}\qq{[a_3]-[b_1]}} + W^{-D}L^{d\fn} \nonumber\\
&\prec \left(W^{d\e+c}\frac{\ell_s^d\eta_s}{\ell_t^d\eta_t}\right)^{\fn} \|\cal A\|_\infty + W^{-D}L^{d\fn}.\label{eq:fA_A>2}
\end{align}
It remains to deal with the case $A^c=\{i\}$ for some $i\in\qqq{2,\fn}$. Using the symmetry in \eqref{eq:A_zero_sym}, we obtain that    
$$\sum_{[b_i]}\wh \Xi^{(i)}_{[a_i];[b_1][b_i]}\cal A_{\mathbf b}=\sum_{[b_i]}\frac{1}{2}\left(\Xi^{(i)}([a_i],[b_i])+\Xi^{(i)}([a_i],2[b_1]-[b_i])-2\Xi^{(i)}([a_i],[b_1])\right)\cal A_{\mathbf b}. $$
Using \eqref{prop:BD2} (with $d=2$) and \eqref{Mbound_AO2}, we can bound the coefficient by
$$\Xi^{(i)}([a_i],[b_i])+\Xi^{(i)}([a_i],2[b_1]-[b_i])-2\Xi^{(i)}([a_i],[b_1]) \prec \frac{\eta_s}{\eta_t\ell_t^2}\frac{|[b_i]-[b_1]|^2}{\qq{[a_i]-[b_1]}^2}.$$
Combining this with \eqref{eq:Xibb} and the decay property \eqref{deccA0} for $\cal A$, we can bound $f(A)$ in a similar way as in \eqref{eq:fA_A>2}, with the denominator $\qq{[a_2]-[b_1]}\qq{[a_3]-[b_1]}$ replaced by $\qq{[a_i]-[b_1]}^2$. This concludes \eqref{sum_res_2_sym}.

\subsection{Proof of \Cref{lem_dec_calE2}} \label{sec:pf_dec_calE2}

Recall the scale parameter defined in \eqref{eq:ell*u}. In the following proof, we will tacitly use the following claim, whose proof is identical to that of \cite[Lemma 5.6]{Band1D}, relying on \Cref{lem_propTH}.

\begin{claim}[Lemma 5.6 of \cite{Band1D}]\label{lem_dec_calE_0}
In the setting of \Cref{lem:main_ind}, for any large constant $D>0$ and small constant $\delta>0$, the following estimates hold if $|[b] -[a]|\ge \delta \ell_t^*$:
 \be \label{Kell*}
\left({\Theta}_t^{\bsig}\right)_{[a][b]} \le W^{-D} ,\quad    \left(\frac{1 - u M^{\bsig}}{1 - t M^{\bsig}}\right)_{[a] [b]} \le W^{-D},\quad \forall \bsig\in \{+,-\}^2,\ u\in [0,t];
\ee
\be
{\cal L}^{(2)}_{t
, \bsig,([a], [b])} \prec {\cal J}^*_{t
,D}\cdot{\cal T}_{t,D}(|[a]-[b]|) ,\quad \text{for}\quad  \bsig=(+,-). \label{auiwii}
\ee
Furthermore, for any constant $C>0$, we have 
\begin{equation}\label{Tell*}
    {\cal T}_{t,D} \left(\ell-C\ell_t^*\right)\prec   {\cal T}_{t,D} \left(\ell \right),\quad \forall \ell\ge 0.
\end{equation} 

\end{claim}

By the definition \eqref{def_EwtG}, we have that 
\begin{align}\label{eq:calW}
\mathcal{W}^{(2)}_{t, \boldsymbol{\sigma}, \ba} = W^d \sum_{[a]} \langle \Gc_u E_{[a]}\rangle {\cal L}^{(3)}_{u, (+,+,-),([a],[a_1], [a_2])}+c.c., 
 \end{align}
where $c.c.$ denotes the complex conjugate of the preceding term. By \eqref{GavLGEX}, we have  
\begin{align}
    \langle \Gc_u E_{[a]}\rangle 
    &\prec (W^{d}\ell_u^d\eta_u)^{-1} + {\cal J}^*_{u,D}\cdot (W^d\ell_u^d\eta_u)^{-2},\label{eq:GcEa}
\end{align}
where we have used \Cref{ML:Kbound} and the definition of ${\cal J}^* _{u,D}$ to get the control parameter $\Psi_u$ as follows: 
\begin{align}
    \max_{[a],[b]\in \Zn} \cL^{(2)}_{u,(+,-),([a],[b])}+W^{-d} &\prec \max_{[a],[b]\in \Zn} (\cL-\cK)^{(2)}_{u,(+,-),([a],[b])}+(W^{d}\ell_u^d\eta_u)^{-1} \nonumber\\
    &\prec (W^{d}\ell_u^d\eta_u)^{-1} + {\cal J}^*_{u,D}\cdot (W^d\ell_u^d\eta_u)^{-2}=:\Psi_u^2 .\label{eq:Psiu}
\end{align}
Plugging \eqref{eq:GcEa} into \eqref{eq:calW}, we obtain that  
 \begin{align}\label{jiizziy}
\mathcal{W}^{(2)}_{t, \boldsymbol{\sigma}, \ba} \prec (\ell_u^d\eta_u)^{-1} \sum_{[a]}   \left|{\cal L}^{(3)}_{u, (+,+,-),([a],[a_1],[a_2])}\right|\left(1+{\cal J}^*_{u,D}\cdot (W^d\ell_u^d\eta_u)^{-1}\right).
\end{align}

To bound \eqref{jiizziy}, we introduce another scale parameter
\be\label{eq:elldagger}
\ell_u^{\dagger}:=(\log W)^{3/2}\ell^*_u  = (\log W)^3 \ell_u.
\ee
By the definition \eqref{def_WTuD}, ${\cal T}_{u,D}(\ell_u^{\dagger})\prec W^{-D}$ for any large constant $D>0$. 
If $|[a]-[a_1]|\le  \ell_u^{\dagger}$, then the bound \eqref{lRB1} established in the Step 1 already implies that 
\begin{align}\label{alkkj}
 \sum_{[a]:|[a]-[a_1]|\le  \ell_u^{\dag}} \left|{\cal L}^{(3)}_{u, (+,+,-),([a],[a_1],[a_2])}\right|\prec (W^d\ell_s^d\eta_u)^{-2} \cdot (\ell_u^\dag)^d \prec  \ell_u^d (W^d\ell_s^d\eta_u)^{-2}  .
\end{align}
If $|[a]-[a_1]|>\ell_u^{\dagger}$, with \eqref{Gtmwc}, we can get the trivial bound
\begin{align}\label{eq:G3loop}
    {\cal L}^{(3)}_{u, (+,+,-),([a],[a_1],[a_2])}&\prec W^{-2d}\sum_{x_1\in [a_1],y\in [a]}|(G_u)_{yx_1}|  \le  W^{-2d}\sum_{x_1\in [a_1],y\in [a]}\left(|(G_u-M)_{yx_1}| + |M_{yx_1}|\right).
\end{align}
We use the estimates \eqref{GijGEX} and \eqref{Mbound_AO} to bound the first and second terms on the RHS of \eqref{eq:G3loop}, respectively, 
which gives that when $|[a]-[a_1]|\ge \ell_u^{\dag}$,
\be\label{eq:bdd-G3loop}
{\cal L}^{(3)}_{u, (+,+,-),([a],[a_1],[a_2])}\prec \bigg\{\sum_{[a'], [b'] \in \Zn}  \Phi_u([a'], [b']) W^{-c(|[a']-[a]|+|[b']-[a_1]|)}\bigg\}^{1/2} + W^{-D}
\ee
for any constant $D>0$. Here, we can take \be\label{eq:choosePhi}\Phi_u([a'],[b'])=\cal J^*_{u,D}\cdot \cal T_{u,D}(|[a']-[b']|) ,\quad \forall [a'],[b']\in \Zn \ \text{with}\ |[a']-[b']|\gtrsim \ell_u^*, \ee 
by equation \eqref{auiwii}. Then, using ${\cal T}_{u,D}(|[a']-[b']|)\prec W^{-2D}$ for $|[a']-[b']|\gtrsim \ell_u^{\dag}$, we can bound \eqref{eq:bdd-G3loop} by 
$$
{\cal L}^{(3)}_{u, (+,+,-),([a],[a_1],[a_2])}\prec ({\cal J}^*_{u,D})^{1/2}\cdot W^{-D}
$$
for any large constant $D>0$. Taking $D$ sufficiently large, we obtain that 
\begin{align}\label{alkkj2}
 \sum_{[a]: |[a]-[a_1]| > \ell_u^{\dag}}\left|{\cal L}^{(3)}_{u, (+,+,-),([a],[a_1],[a_2])}\right|\prec  ({\cal J}^*_{u,D})^{1/2}\cdot W^{-10d}. 
\end{align}

Now, assume $|[a_1]-[a_2]|\le \ell_u^*$ for \eqref{jiizziy}. In this case, plugging \eqref{alkkj} and \eqref{alkkj2} into \eqref{jiizziy}, we get  
\begin{align}\label{82jjdopasj}
\cW^{(2)}_{u, \boldsymbol{\sigma}, \mathbf{a}} /{\cal T}_{u,D}(|[a_1]-[a_2]|)\prec 
 \eta_u^{-1}\left(\ell_u^d / \ell_s^d\right)^2 \left(1+  (\mathcal{J}_{u, D}^*)^2 \cdot \left(W^d \ell_u^d \eta_u\right)^{-1 }\right).
 \end{align}
Together with  condition \eqref{con_st_ind}, it implies the estimate \eqref{res_deccalE_wG} for $|[a_1]-[a_2]|\le \ell_u^*$. 

For the case $|[a_1]-[a_2]|\ge \ell_u^*$, if $|[a]-[a_1]|> \ell_u^{\dag}$, then \eqref{alkkj2} already gives a good enough bound. Similarly, the summation of ${\cal L}^{(3)}_{u, (+,+,-),([a],[a_1], [a_2])}$ over the region \smash{$\{[a]:|[a]-[a_2]|> \ell_u^{\dag}\}$} satisfies the same bound as in \eqref{alkkj2}. Hence, in the remaining proof, it suffices to assume that \be\label{eq:consa1a2}|[a]-[a_1]|\vee |[a]-[a_2]| \le \ell_u^\dag.\ee 
Then, we divide the proof according to the following two cases:  
 $$
 (1): |[a]-[a_1]|\wedge |[a]-[a_2]| \le \ell_u^*/2,\quad \text{and}\quad 
 (2): |[a]-[a_1]|\wedge |[a]-[a_2]|\ge \ell_u^*/2.
 $$
In Case (1), we assume without loss of generality that $|[a]-[a_1]|\le \ell_u^*/2$. Then, by the triangle inequality, we have $\ell_u^*/2 \le |[a]-[a_2]|\le \ell_u^{\dag}$. Applying the Cauchy-Schwarz inequality, we can bound that 
\begin{align}
&\cal L^{(3)}_{u, (+,+,-),([a],[a_1],[a_2])} \le  W^{-3d}\sum_{x_1\in[a_1],x_2\in[a_2],y\in[a]}|G_{yx_1}| \left(|G_{x_2y}|^2 + |G_{x_1 x_2}|^2\right) \nonumber\\
&\lesssim \cL^{(2)}_{u,(+,-),([a],[a_2])} \cdot \max_{y\in [a]} W^{-d}\sum_{x_1\in[a_1]}|G_{yx_1}|+ \cL^{(2)}_{u,(+,-),([a_2],[a_1])} \cdot \max_{x_1\in [a_1]} W^{-d}\sum_{y\in[a]}|G_{y x_1}| ,\label{kskjw}
\end{align}
where we have abbreviated that $G_u\equiv G$.  
Using \eqref{GiiGEX}, we can bound that 
\begin{align}\label{jaysw2002}
    W^{-d}\sum_{x_1\in[a_1]}|G_{yx_1}| + W^{-d}&\sum_{y\in[a]}|G_{y x_1}| \prec \Psi_u + W^{-d}\sum_{x_1\in[a_1]}|M_{yx_1}|+ W^{-d}\sum_{y\in[a]}|M_{yx_1}| \nonumber\\
    &\prec (W^d\ell_u^d\eta_u)^{-1/2} + ({\cal J}^*_{u,D})^{1/2}(W^d\ell_u^d\eta_u)^{-1} ,
\end{align}
where in the second step we have used $\Psi_u$ from \eqref{eq:Psiu} and applied 
\eqref{Mbound_AO} to bound \smash{$W^{-d}\sum_{x_1\in[a_1]}|M_{yx_1}|$} and \smash{$W^{-d}\sum_{y\in[a]}|M_{yx_1}|$} by $\OO(W^{-d})$.  
Inserting the bound \eqref{jaysw2002} into \eqref{kskjw}, we obtain that
\begin{align*}
    \cal L^{(3)}_{u, (+,+,-),([a],[a_1],[a_2])} \prec \left( \cL^{(2)}_{u,(+,-),([a],[a_2])} + \cL^{(2)}_{u,(+,-),([a_2],[a_1])}\right)\left( (W^d\ell_u^d\eta_u)^{-1/2} + ({\cal J}^*_{u,D})^{1/2}(W^d\ell_u^d\eta_u)^{-1} \right).
\end{align*}
Then, applying \eqref{auiwii} to the two $2$-$G$ loops gives that   
\begin{align*}
    \cal L^{(3)}_{u, (+,+,-),([a],[a_1],[a_2])} \prec&~ {\cal J}^*_{u,D}\left[ \cal T_{u,D}(|[a]-[a_2]|) + \cal T_{u,D}(|[a_1]-[a_2]|) \right]  \left( (W^d\ell_u^d\eta_u)^{-1/2} + ({\cal J}^*_{u,D})^{1/2}(W^d\ell_u^d\eta_u)^{-1} \right). 
\end{align*}
Furthermore, since $|[a]-[a_2]|\ge |[a_1]-[a_2]|-\ell_u^*/2$, we have that ${\cal T}_{u,D}(|[a]-[a_2]|) \prec  {\cal T}_{u,D}(|[a_1]-[a_2]|)$ by \eqref{Tell*}. Thus, we get from the above equation that 
\begin{align}    \label{lwjufw}
 \sum_{[a]:|[a]-[a_1]|\le \ell_u^*/2,|[a]-[a_2]|\le \ell_u^\dag}  & \cal L^{(3)}_{u, (+,+,-),([a],[a_1],[a_2])}  \prec ({\cal J}^*_{u,D})^{3/2} 
 {\cal T}_{u,D}(|[a_1]-[a_2]|) \cdot \ell_u^d (W^d\ell_u^d\eta_u)^{-1/2}. 
\end{align}

For Case (2) with $|[a_1]-[a_2]|\ge \ell_u^*$ and $ |[a]-[a_1]|\wedge |[a]-[a_2]|\ge \ell_u^*/2$, we need to bound 
\begin{align}\label{LGGGXXX}
\left|\cal L^{(3)}_{u, (+,+,-),([a],[a_1],[a_2])} \right| \le W^{-3d} \sum_{x_1\in[a_1],x_2\in[a_2],y\in[a]} |G_{x_1x_2}||G_{x_2 y}||G_{yx_1}| .
\end{align}
Applying the estimate \eqref{GijGEX} with  $\Phi_u$ given by \eqref{eq:choosePhi}, we can bound $|G_{x_1x_2}|$ as 
\begin{align}
    |G_{x_1x_2}|^2 &\lesssim |(G-M)_{x_1x_2}|^2 + |M_{x_1x_2}|^2 
     \prec \cal J_{u,D}^* \sum_{[a'], [b'] \in \Zn}\cal T_{u,D}(|[a']-[b']|)W^{-c(|[a']-[a_1]|+|[b']-[a_2]|)}  + W^{-2D} \nonumber\\
    &\prec \cal J_{u,D}^* \cal T_{u,D}(|[a_1]-[a_2]|) \, , 
    \label{GGTLJ}
\end{align} 
where in the last step, we used \eqref{Tell*} to control the sum over $[a']$ and $[b']$, and applied \eqref{Mbound_AO} to bound $|M_{x_2x_1}|^2$.  Using \eqref{GGTLJ} and similar bounds for $|G_{x_2 y}|$ and $|G_{yx_1}|$, we can bound the summation of \eqref{LGGGXXX} over $[a]$ by 
\begin{align}\label{GGTLJ2}
\sum_{[a]}^\star   \cal L^{(3)}_{u, (+,+,-),([a],[a_1],[a_2])} \prec &~ ({\cal J}^*_{u,D})^{3/2}{\cal T}_{u,D}^{1/2}(|[a_1]-[a_2]|)  \sum_{[a]}^\star {\cal T}_{u,D}^{1/2}(|[a_1]-[a]|){\cal T}_{u,D}^{1/2}(|[a_2]-[a]|),
\end{align}
where $\sum^\star$ refers to the summation of $[a]$ over the region $\{|[a]-[a_1]|\ge \ell_u^*/2, |[a]-[a_2]|\ge \ell_u^*/2\}$. With basic calculus, it is straightforward to check that 
\be\label{eq:Calculus1} \sum_{[a]} \exp\left( -\frac{|[a_1]-[a]|^{1/2}+|[a_2]-[a]|^{1/2}}{2\ell_u^{1/2}}\right) \lesssim \ell_u^d \exp\left( -\frac{|[a_1]-[a_2]|^{1/2}}{2\ell_u^{1/2}}\right),\ee
with which we can bound \eqref{GGTLJ2} by 
 \begin{align}\label{lwjufw2}
\sum_{[a]}^\star   \cal L^{(3)}_{u, (+,+,-),([a],[a_1],[a_2])} &\prec ({\cal J}^*_{u,D})^{3/2}{\cal T}_{u,D}(|[a_1]-[a_2]|)   \cdot  \ell_u^d (W^d\ell_u^d\eta_u)^{-1}.
 \end{align}

Finally, plugging the estimates \eqref{lwjufw} and \eqref{lwjufw2} into \eqref{jiizziy} completes the proof of  \eqref{res_deccalE_wG} when $|[a_1]-[a_2]|\ge \ell_u^*$.

\subsection{Proof of \Cref{lem_dec_calE3}}\label{sec:pf_dec_calE3} 
By \Cref{def:CALE}, we can write $\left( \cB\otimes  \cB \right)^{(4)}_{t, \boldsymbol{\sigma}, \ba, \ba'}$ as 
\begin{align*}
 \left( \cB\otimes  \cB \right)^{(4)}_{t, \boldsymbol{\sigma}, \ba, \ba'} =  &
 W^d \sum_{k=1}^2 \sum_{[b]}\cal L_k([b]),\quad \text{with}\quad \cal L_k([b]):={\cal L}^{(6)}_{t, (\boldsymbol{\sigma}\times\overline\bsig)^{(k) },(\ba\times \ba')^{(k )}([b],[b])},
 \end{align*}
where $(\boldsymbol{\sigma}\times\overline\bsig)^{(k) }$ and $(\ba\times \ba')^{(k )}([b],[b])$, $k\in\{1,2\}$, are defined by 
\begin{align*}    &(\boldsymbol{\sigma}\times\overline\bsig)^{(1)}=(+,-,+,-,+,-),\quad  (\ba\times \ba')^{(1)}([b],[b])=([a_1],[a_2],[b],[a'_2],[a'_1],[b]),\\
&(\boldsymbol{\sigma}\times\overline\bsig)^{(2)}=(-,+,-,+,-,+),\quad  (\ba\times \ba')^{(1)}([b],[b])=([a_2],[a_1],[b],[a'_1],[a'_2],[b]).
\end{align*}
Without loss of generality, by symmetry, we only need to bound the loop $\cal L_1$. Moreover, for simplicity of presentation, we will abbreviate $G_u\equiv G$.  
We divide the proof into two cases.

\medskip

\noindent 
{\bf Case 1: $|a_1-a_2|\le 4\ell_t^*$.} 
When $|[b]-[a_1]|\ge \ell_u^{\dag}$, using that ${\cal T}_{u,D}(\ell)$ is exponentially small when $\ell\gtrsim \ell_u^\dag$, we can easily obtain that 
\begin{align}\label{alkkj3_0}
\sum_{[b]:|[b]-[a_1]|\ge \ell_u^\dag} {\cal L}_1([b])\prec ({\cal J}^*_{u,D})^{1/2}\cdot W^{-10d},
\end{align}
using an argument similar to that leading to \eqref{alkkj2}. When $|[b]-[a_1]|\le \ell_u^{\dag}$, we use the $6$-$G$ loop bound \eqref{lRB1} established in Step 1 to get that  
\begin{align}\label{alkkj3}
\sum_{[b]:|[b]-[a_1]|\le \ell_u^{\dag}} {\cal L}_1([b]) \prec (\ell_u^d/\ell_s^d)^5 (W^d\ell_u^d\eta_u)^{-5} \cdot \ell_u^d   .
\end{align}
Combining \eqref{alkkj3_0} and \eqref{alkkj3}, we conclude \eqref{res_deccalE_dif} when $|[a_1]-[a_2]|\le 4\ell_t^*$.  

\medskip
\noindent {\bf Case 2: $|a_1-a_2|\ge 4\ell_t^*$.} 
Without loss of generality, by symmetry, it suffices to assume that  $|[b] -[a_1]|\le |[b]-[a_2]|$. Recall that $|[a'_i]-[a_i]|\le \ell_t^*$ by the assumption \eqref{57}. By the triangle inequality, we have that 
$$|[a_2']-[a_1']| \ge 2\ell_t^*, \quad |[a_2]-[b]| \ge 2\ell_t^*,\quad |[a_2']-[b]| \ge \ell_t^* . $$
Hence, ${\cal L}_1$ contains at least four ``long edges", and we correspondingly bound ${\cal L}_1$ as:  
\begin{align}\label{L6G6X}
{\cal L}^{(1)} \le &~ \max_{\substack{x_1\in[a_1],x_2\in[a_2], y\in[b]\\ x_1'\in[a_1'],x_2'\in[a_2']}} \left|G_{x_1x_2}\right|
\left|G_{x_2y}\right| 
|G_{yx_2'}|
|G_{x_2'x_1'}| \cdot W^{-2d} \sum_{x_1\in[a_1],x_1'\in [a_1']}\Big|\left(G^\dagger E_{[b]} G\right)_{ x_1' x_1}\Big| . 
\end{align}
We claim that 
\begin{align} \label{GGTLJ4G}
\left|G_{x_1x_2}\right|
\left|G_{x_2y}\right| 
|G_{yx_2'}| |G_{x_2'x_1'}| \prec 
\left({\cal J}^*_{u,D}\right)^2 
{\cal T} _{t,D}(|[a_1]-[a_2]|){\cal T} _{t,D}( |[b]-[a_2]|).  
\end{align}
The proof of this bound is based on the estimates \eqref{GijGEX} and \eqref{Tell*}. Since the argument is very similar to that of \eqref{GGTLJ} and almost the same as that for equation (5.64) in \cite{Band1D}, we omit the details.


It remains to bound the summation in \eqref{L6G6X}. We divide it into two cases:
\begin{align}
\text{(i)}: |[b]-[a_1]|\wedge |[b]-[a_1']| \le \ell_u^*, \quad \text{and}\quad \text{(ii)}: |[b]-[a_1]|\wedge |[b]-[a_1']|\ge \ell_u^* .
\end{align} 
Applying the Cauchy-Schwarz inequality and the 4-$G$ loop bound established in \eqref{lRB1}, we can bound that 
\begin{align}\label{eq:CS4loop}
     W^{-2d} \sum_{x_1\in[a_1],x_1'\in [a_1']}\Big|\left(G^\dagger E_{[b]} G\right)_{x_1' x_1}\Big| &\le 
     \left(\cL^{(4)}_{u, (-,+,-,+),([b],[a_1],[b],[a_1'])}\right)^{1/2} \prec (\ell_u^d/\ell_s^d)^{3/2}\cdot (W^d\ell_u^d\eta_u)^{-3/2}.  
\end{align} 
In case (i), if $|[b]-[a_1]|\le \ell_u^*$, then we have $|[b]-[a_2]|\ge  |[a_1]-[a_2]|-\ell_u^*$; if $|[b]-[a_1']|\le \ell_u^*$, then 
$$|[b]-[a_2]|\ge |[a_1']-[a_2']|-|[b]-[a_1']|-|[a_2]-[a_2']|\ge |[a_1]-[a_2]| - 4\ell_u^*,$$
where we also used $|[a_1']-[a_2']| \ge |[a_1]-[a_2]|-|[a_1']-[a_1]|-|[a_2']-[a_2]|\ge |[a_1]-[a_2]|-2\ell_u^*$ in the second step. In any case, by \eqref{Tell*}, we have that ${\cal T} _{t,D}\left( |[b]-[a_2]| \right) \prec {\cal T} _{t,D}\left(  |[a_1]-[a_2]| \right)$. 
Hence, plugging \eqref{eq:CS4loop} and \eqref{GGTLJ4G} into \eqref{L6G6X}, we obtain that 
\begin{align}\label{eq:boundL1_i}
 W^d \sum_{[b]}^* {\cal L}_1([b])
 \prec &\eta_u^{-1}(\ell_u^d/\ell_s^d)^{3/2} \cdot\left(W^d \ell_u^d \eta_u\right)^{-1/2}
\cdot\left({\cal J}^*_{u,D}\right)^2 {\cal T}^2_{t,D}(|[a_1]-[a_2]|).
\end{align}
where $\sum^*$ refers to the summation of $[b]$ over the region $\{|[b]-[a_1]|\le |[b]-[a_2]|,\ |[b]-[a_1]|\wedge |[b]-[a_1']| \le \ell_u^*\}$. 

In case (ii), we bound  $(G^\dagger E_{[b]}G)_{x_1'x_1}$ by 
$$
\left|(G^\dagger E_{[b]}G)_{x_1'x_1}\right|\le \max_{y\in [b]}|G_{yx_1}||G_{yx'_1}|.
$$
Since $|[b]-[a_1]|\wedge |[b]-[a_1']|\ge \ell_u^*$, using again a similar argument as that for \eqref{GGTLJ} (based on the estimates \eqref{GijGEX}, \eqref{auiwii}, and \eqref{Tell*}), we can derive that  
$$
\left|(G^\dagger E_bG)_{x_1'x_1}\right|
\prec {\cal J}^*_{u,D} {\cal T}_{t,D}(|[b]-[a_1]|) .
$$
Plugging this bound and \eqref{GGTLJ4G} into \eqref{L6G6X}, we obtain that   
\begin{align}
W^d\sum_{[b]}^{\star} {\cal L}_1
 &\prec W^d \left( {\cal J}^*_{u,D}\right)^3 {\cal T}_{t,D}(|[a_1]-[a_2]|)\sum_{[b]} {\cal T}_{t,D}(|[b]-[a_1]|) {\cal T}_{t,D}(|[b]-[a_2]|) \nonumber \\
&\prec \eta_t^{-1} \left( {\cal J}^*_{u,D}\right)^3 \cdot (W^d\ell_t^d\eta_t)^{-1} 
\cdot {\cal T}_{t,D}^{2}(|[a_1]-[a_2]|),\label{eq:boundL1_ii}
 \end{align}
 where $\sum^\star$ refers to the summation of $[b]$ over the region $\{|[b]-[a_1]|\le |[b]-[a_2]|,\ |[b]-[a_1]|\wedge |[b]-[a_1']| \ge \ell_u^*\}$, and in the second step we used a similar fact as in \eqref{eq:Calculus1}:
 \be\label{eq:Calculus2} \sum_{[a]} \exp\left( -\frac{|[a_1]-[a]|^{1/2}+|[a_2]-[a_1]|^{1/2}}{\ell_t^{1/2}}\right) \lesssim \ell_t^d \exp\left( -\frac{|[a_1]-[a_2]|^{1/2}}{\ell_t^{1/2}}\right).\ee

Combining the bounds \eqref{eq:boundL1_i} and \eqref{eq:boundL1_ii}, we obtain \eqref{res_deccalE_dif} for the $|a_1-a_2|\ge 4\ell_t^*$ case by using the condition \eqref{con_st_ind}. 

\subsection{Proof of \Cref{lem:pf_step2}}\label{sec:pf_step2}
The proof of \Cref{lem:pf_step2} depends on the following  evolution kernel estimate. It shows that given $0\le s\le t< 1$, if a two-dimensional tensor $\cal A$ decays exponentially on the scale $\ell_s$, then \smash{${\cal U}^{(2)}_{s,t,\bsig}\circ \cal A$} decays on the scale $\ell_t$ for $\bsig\in\{(+,-),(-,+)\}$. 

\begin{lemma}
\label{TailtoTail}
For any $t\in[0,1)$ and $\ell\ge 0$, we introduce the notation  
 $${\cal T}_{t}(\ell) := (W^d\ell_t^d\eta_t)^{-2} \exp \big(- \left| \ell /\ell_t\right|^{1/2} \big).
$$
For $\bsig\in\{(+,-),(-,+)\}$ and $s\in [0,1)$, suppose ${\cal A}_{\ba}$ satisfies that 
$$
|{\cal A}_{\ba}|\le {\cal T}_{s}(|[a_1]-[a_2]|)+W^{-D},\quad \forall \ba=([a_1],[a_2])\in(\Zn)^2,
$$
for some constant $D>0$. 
Then, for any $t\in[s,1)$, we have that 
\begin{align}
\label{neiwuj} 
\left({\cal U}^{(2)}_{s,t,\boldsymbol{\sigma}} \circ 
{\cal A}\right)_{\ba} & \prec
 {\cal T}_{t}( |[a_1]-[a_2]|)+W^{-D} (\eta_s/\eta_t)^2 \quad \text{if}\ \ |[a_1]-[a_2]|\ge (\log W)^{3/2}\ell_t.
\end{align}
 \end{lemma}
\begin{proof}
A similar lemma for the setting of 1D RBM has been proved as Lemma 7.2 in \cite{Band1D}, and exactly the same arguments apply also to our setting by using the estimates in \Cref{lem_propTH}. So we omit the details. 
\end{proof}

Note that when $\fn=2$, the second term on the RHS of \eqref{int_K-LcalE} vanishes and we can write that 
\begin{align}
(\mathcal{L} - \mathcal{K})^{(2)}_{t, \boldsymbol{\sigma}, \ba} & =
    \left(\mathcal{U}^{(2)}_{s, t, \boldsymbol{\sigma}} \circ (\mathcal{L} - \mathcal{K})^{(2)}_{s, \boldsymbol{\sigma}}\right)_{\ba} + \int_{s}^t \left(\mathcal{U}^{(2)}_{u, t, \boldsymbol{\sigma}} \circ \mathcal{E}^{(2)}_{u, \boldsymbol{\sigma}}\right)_{\ba} \dd u   \nonumber\\
    &+ \int_{s}^t \left(\mathcal{U}^{(2)}_{u, t, \boldsymbol{\sigma}} \circ \cW^{(2)}_{u, \boldsymbol{\sigma}}\right)_{\ba} \dd u + \int_{s}^t \left(\mathcal{U}^{(2)}_{u, t, \boldsymbol{\sigma}} \circ \dd \cB^{(2)}_{u, \boldsymbol{\sigma}}\right)_{\ba} . \label{int_K-L_ST_n2}
\end{align}
By the induction hypothesis \eqref{Eq:Gdecay+IND} for $({ \cal L-\cal K})_{s,\bsig,\ba}^{(2)}$, using the bounds on the evolution kernel in \Cref{lem:sum_Ndecay,TailtoTail}, we can bound the first term on the RHS of \eqref{int_K-L_ST_n2} by 
\begin{equation} \label{res_deccalE_0}
\left(\mathcal{U}^{(2)}_{s, t, \boldsymbol{\sigma}} \circ (\mathcal{L} - \mathcal{K})^{(2)}_{s, \boldsymbol{\sigma}}\right)_{\ba}\Big/ 
  {\cal T}_{t, D}(|[a_1]-[a_2]|) 
  \prec (\ell_t^d/\ell_s^d)^2\cdot {\bf 1}(|[a_1]-[a_2]|\le \ell_t^*) + 1.
\end{equation}
Here, the second term arises from applying \eqref{neiwuj}, while the first term results from applying \Cref{lem:sum_Ndecay}, where the factor $\left(\eta_u / \eta_t\right)^2$ transforms into $ (\ell_t^d/\ell_s^d)^2$ due to the difference in prefactors between $ {\cal T}_{s, D}$ and  ${\cal T}_{t, D}$. 

For the rest of the three terms on the RHS of \eqref{int_K-L_ST_n2}, we will use tacitly the following simple fact: for any fixed $\fn\ge 2$, \smash{$\ba\in (\Zn)^\fn$}, and function \smash{$f:(\Zn)^\fn\to \C$} of order $\OO(W^C)$ for a constant $C>0$, there is 
$$ \left| \left({\cal U}^{(\fn)}_{u,t, \boldsymbol{\sigma}} \circ f' \right)_\ba\right| \le W^{-D'}$$
for any large constant $D'>D$, where $f'$ is defined as $f'(\mathbf b)=f(\mathbf b)\mathbf 1(\|\mathbf b-\mathbf a\|_\infty\ge \ell_t^*)$. This follows from the exponential decay of \smash{${\cal U}^{(\fn)}_{u,t, \boldsymbol{\sigma}}$} when $\|\mathbf b-\mathbf a\|_\infty = \max_i |[a_i]-[b_i]|\ge \ell_t^*$, by the expression \eqref{eq:decompU} and the estimate \eqref{eq:decayXi}. 
Using \Cref{lem:sum_Ndecay} and the estimate \eqref{res_deccalE_lk}, we can bound the second term on the RHS of \eqref{int_K-L_ST_n2} as 
\begin{align}
 \left({\cal U}^{(2)}_{u,t, \boldsymbol{\sigma}}\circ  {\cal E}^{(2)}_{u, \boldsymbol{\sigma}}\right)_\ba
 &\prec (\eta_u/\eta_t)^2 \max_{\| \textbf{b}-\ba\|_\infty\le \ell_t^*}{\cal E}^{(2)}_{u, \boldsymbol{\sigma}, \mathbf{b}}+W^{-D'}\nonumber\\
 &\prec \frac{\eta_u}{\eta_t^2}  \left(W^d\ell_u^d \eta_u\right)^{-1}\left({\cal J}^{*}_{u,D} \right)^{2} \max_{\| \textbf{b}-\ba\|_\infty\le \ell_t^*} 
  \cal T_{t,D}(|[b_1]-[b_2]|)+W^{-D'}\nonumber\\
  &\prec \frac{\eta_u}{\eta_t^2}  \left(W^d\ell_u^d \eta_u\right)^{-1}\left({\cal J}^{*}_{u,D} \right)^{2}  
  \cal T_{t,D}(|[a_1]-[a_2]|),\label{res_deccalE_1}
\end{align}
where in the third step, we used \eqref{Tell*} under the condition $\| \textbf{b}-\ba\|_\infty\le \ell_t^*$, and chose $D'$ sufficiently large depending on $D$ (recall that $\cal T_{t,D}(|[a_1]-[a_2]|)\ge W^{-D}$ by definition \eqref{def_WTuD}). 
Similarly, combining \Cref{lem:sum_Ndecay} with the estimates \eqref{res_deccalE_wG} and \eqref{res_deccalE_dif} yields that 
\begin{align}  \label{res_deccalE_2}
 &    \Big(\mathcal{U}^{(2)}_{u,t, \boldsymbol{\sigma}} \circ  \cW^{(2)}_{u, \boldsymbol{\sigma}}\Big)_{\mathbf{a}} 
    \Big/ 
{\cal T}_{t, D}(|[a_1]-[a_2]|)\nonumber\\
 &\qquad \prec   \frac{\eta_u}{\eta_t^2} \left[\left(\frac{\ell_u^d}{\ell_s^d}\right)^2 {\bf1}(|[a_1]-[a_2]|\le 3\ell_t^*) + \frac{({\cal J}^{*}_{u,D})^{3 }}{\left(W^d \ell_u^d \eta_u\right)^{1/3}}   \right],\\
&  \left(\left(
   \mathcal{U}^{(2)}_{u, t,  \boldsymbol{\sigma}}
   \otimes 
   \mathcal{U}^{(2)}_{u, t,  \overline{\boldsymbol{\sigma}}}
   \right) \;\circ  \;
   \left( \mathcal{B} \otimes  \mathcal{B} \right)^{(4)}
   _{u, \,\boldsymbol{\sigma}  }
   \right)_{{\textbf a}, {\textbf a}}
    \Big/ {\cal T}_{t, D}^2(|[a_1]-[a_2]|) \nonumber\\
 &\qquad \prec \frac{\eta_u^3}{\eta_t^4}\left[\left(\frac{\ell^d_u}{\ell^d_s}\right)^5{\bf1}(|[a_1]-[a_2]|\le 6\ell_t^*) +\frac{({\cal J}^{*}_{u,D})^{3 }}{\left(W^d \ell_u^d \eta_u\right)^{1/3}}  \right]. \label{res_deccalE_3}
\end{align}

By the monotonically increasing property of $\ell_t$, $\eta_t^{-1}$, and $\cal T_{t,D}$ with respect to $t$, it is clear that the bounds in \eqref{res_deccalE_0}--\eqref{res_deccalE_3} hold uniformly in $t'\in [u,t]$ if we replace the evolution operator \smash{$\cal U_{u,t,\bsig}^{(2)}$} and the function ${\cal T}_{t, D}(|[a_1]-[a_2]|)$ with \smash{$\cal U_{u,t',\bsig}^{(2)}$} and ${\cal T}_{t', D}(|[a_1]-[a_2]|)$, respectively, while keeping all other $t$-dependent quantities (such as $\eta_t$, $\ell_t$, and $\ell_t^*$) unchanged. In particular, if we define the stopping time $\tau = T \wedge t$, where $T$ is defined in \eqref{eq:def_TTT} for an arbitrarily small constant $\e>0$, then the bounds \eqref{res_deccalE_0}--\eqref{res_deccalE_3} remain valid with \smash{$\cal U_{u,t,\bsig}^{(2)}$} and ${\cal T}_{t, D}(|[a_1]-[a_2]|)$ replaced by \smash{$\cal U_{u,\tau,\bsig}^{(2)}$} and ${\cal T}_{\tau, D}(|[a_1]-[a_2]|)$, respectively. 

First, using \eqref{res_deccalE_3}, we can bound the quadratic variation of the martingale term in \eqref{int_K-L_ST} as:
\begin{align*}
& \int_{s}^\tau 
   \left(\left(
   \mathcal{U}^{(2)}_{u, \tau,  \boldsymbol{\sigma}}
   \otimes 
   \mathcal{U}^{(2)}_{u, \tau,  \overline{\boldsymbol{\sigma}}}
   \right) \;\circ  \;
   \left( \cB \otimes  \cB \right)^{(4)}
   _{u, \boldsymbol{\sigma}  }
   \right)_{\ba, \ba}\dd u    
  \nonumber \\
\prec~ &  {\cal T}_{\tau, D}^2(|[a_1]-[a_2]|)\cdot  \int_s^\tau   \frac{\eta_u^3}{\eta_t^4}\left[\left(\frac{\ell^d_u}{\ell^d_s}\right)^5{\bf1}(|[a_1]-[a_2]|\le 6\ell_t^*) +\frac{({\cal J}^{*}_{u,D})^{3 }}{\left(W^d \ell_u^d \eta_u\right)^{1/3}}  \right]\dd u   \nonumber \\
\lesssim ~ &  {\cal T}_{\tau, D}^2(|[a_1]-[a_2]|)\cdot   \big [ (\eta_s/\eta_t)^{5}{\bf 1}(|[a_1]-[a_2]|\le 6\ell_t^*)+1 \big ],
 \end{align*}
 where we also used the condition \eqref{con_st_ind} in the second step. Combining this bound with \Cref{lem:DIfREP}, we get 
 \begin{align}\label{51}
& \int_{s}^\tau \left(\mathcal{U}^{(2)}_{u, \tau, \boldsymbol{\sigma}} \circ \dd \cB^{(2)}_{u, \boldsymbol{\sigma}}\right)_{\ba}  \prec  {\cal T}_{\tau, D}(|[a_1]-[a_2]|) \big [ (\eta_s/\eta_t)^{5/2}{\bf 1}(|[a_1]-[a_2]|\le 6\ell_t^*)+1 \big ].
 \end{align}
Next, using \eqref{res_deccalE_1} and \eqref{res_deccalE_2}, we can bound the third and fourth terms on the RHS of \eqref{int_K-L_ST} as  
\begin{align}
    ~&\int_{s}^\tau \left(\mathcal{U}^{(2)}_{u, \tau, \boldsymbol{\sigma}} \circ \mathcal{E}^{(2)}_{u, \boldsymbol{\sigma}}\right)_{\ba} \dd u  + \int_{s}^\tau \left(\mathcal{U}^{(2)}_{u, \tau, \boldsymbol{\sigma}} \circ \cW^{(2)}_{u, \boldsymbol{\sigma}}\right)_{\ba} \dd u \nonumber\\
 \prec~& \cal T_{\tau,D}(|[a_1]-[a_2]|)\cdot \int_s^\tau \frac{\eta_u}{\eta_t^2} \left[\left(\frac{\ell_u^d}{\ell_s^d}\right)^2 {\bf1}(|[a_1]-[a_2]|\le 3\ell_t^*) + \frac{({\cal J}^{*}_{u,D})^{3 }}{\left(W^d \ell_u^d \eta_u\right)^{1/3}}   \right]\dd u  
  \nonumber\\
 \prec~ &\cal T_{\tau,D}(|[a_1]-[a_2]|)\cdot\big [ (\eta_s/\eta_t)^{2} {\bf 1}(|a_1-a_2|\le 3\ell_t^*)+1 \big ]\label{eq:bddE+W},
\end{align}
where we again used the condition \eqref{con_st_ind} in the second step. Plugging the estimates \eqref{res_deccalE_0}, \eqref{51}, and \eqref{eq:bddE+W} into \eqref{int_K-L_ST}, we obtain that 
\be\label{eq:est2KuptoT}
(\mathcal{L} - \mathcal{K})^{(2)}_{\tau, \boldsymbol{\sigma}, \ba}/ \cal T_{\tau,D}(|[a_1]-[a_2]|) \prec (\eta_s/\eta_t)^{5/2} {\bf 1}(|[a_1]-[a_2]|\le 6\ell_t^*) + 1.
\ee

By the induction hypothesis \eqref{Eq:Gdecay+IND}, we know that $\max_{0\le \ell \le n}   {\cal J} _{s,D} (\ell) \prec 1$, i.e., $T\ge s$ with high probability. Combining this with the estimate \eqref{eq:est2KuptoT} and applying a standard continuity argument, we can show that $T\ge t$ with high probability. More precisely, we partition $(t-s)$ into an $N^{-C}$-net by defining $t_k=s+(k-1)\cdot (t-s)/N^C$ for a large enough constant $C>0$. Suppose a stopping time $\tau\in [s, t]$ falls within an interval $[t_{k},t_{k+1})$ for some $0\le k\equiv k_\tau \le N^C$. By \eqref{eq:est2KuptoT}, we can choose \smash{${\cal J}^*_{\tau,D}= W^{\e/2}\left(\eta_s / \eta_t\right)^{5/2}$}. Then, using a standard perturbation argument, we can easily establish that
$$\max_{0\le \ell \le n}   {\cal J} _{t_{k+2},D} (\ell) \le  W^{\e/2}{\cal J}^*_{\tau,D}=W^\e\left(\eta_s / \eta_t\right)^{5/2}\ \implies \ T\ge t_{k+2} \ge \tau+(t-s)/N^C,$$ 
provided that $C$ is chosen sufficiently large. In other words, by losing a small probability (of order $\OO(N^{-D'})$ for any large constant $D'>C$), we can extend $\tau$ slightly to the right. Iterating this process inductively, we conclude that $T\ge t$ with high probability. Together with \eqref{eq:est2KuptoT}, it concludes \eqref{53}.

\subsection{Proof of \Cref{lem:STOeq_NQ}}\label{sec:pf_STOeq_NQ}
Using \eqref{def:XiL} and \eqref{def:XIL-K}, we can bound the quantities on the RHS of \eqref{int_K-LcalE} as follows using the fast decay property shown in \Cref{lem_decayLoop}: for any $\lenk \in \qqq{3,\fn}$ and $\ba\in (\Zn)^\fn$, we have 
\begin{align}\label{CalEbwXi1}
\Big[\OK^{(\lenk)} (\mathcal{L} - \mathcal{K})\Big]^{(\fn)}_{u, \boldsymbol{\sigma},\ba} &\prec 
 W^d\ell_u^d\cdot (W^d\ell_u^d\eta_u)^{-\lenk+1}\cdot  \Xi^{(\cal L-\cal K)}_{u,n-l_{\mathcal{K}}+2}(W^d\ell_u^d\eta_u)^{-(n-l_{\mathcal{K}}+2)} \nonumber\\
&\le \eta_u^{-1}(W^d\ell_u^d\eta_u)^{-\fn }\cdot  \max_{k\in\qqq{2,\fn-1}}\Xi^{(\cal L-\cal K)}_{u,k} ,
\end{align}
where we used the definition \eqref{DefKsimLK} and the bound \eqref{eq:bcal_k} on primitive loops, and the factor $\ell_u^d$ comes from the summation over $[a]$, whose range is restricted by the fast decay property; by definition \eqref{def_ELKLK}, we have 
\begin{align}
    \mathcal{E}_{u, \boldsymbol{\sigma},\fa}^{(\fn)}&\prec W^d\ell_u^d \sum_{k=2}^{\fn} \Xi^{(\cal L-\cal K)}_{u,k} (W^d\ell_u^d\eta_u)^{-k}
 \cdot \Xi^{(\cal L-\cal K)}_{u,\fn-k+2}(W^d\ell_u^d\eta_u)^{-(\fn-k+2)} \nonumber\\
  & \lesssim \eta_u^{-1}(W^d\ell_u^d\eta_u)^{-\fn }\cdot \max_{k\in \qqq{2, \fn} }\left(\Xi^{(\cal L-\cal K)}_{u,k} \Xi^{(\cal L-\cal K)}_{u,\fn-k+2}\right)\cdot(W^d\ell_u^d\eta_u)^{-1}  ;\label{CalEbwXi2}
  \end{align}  
by definition \eqref{def_EwtG}, we have that  
  \begin{align}
  \cW_{u, \boldsymbol{\sigma},\ba}^{(\fn)}&\prec W^d\ell_u^d \cdot  (W^d\ell_u^d\eta_u)^{-1}
 \cdot \left[\Xi^{\cal L }_{u,\fn+1} (W\ell_u\eta_u)^{-\fn}\right]\prec 
  \eta_u^{-1} (W^d\ell_u^d\eta_u)^{-\fn }\cdot  \Xi^{(\cal L)}_{u,\fn+1} , \label{CalEbwXi3} 
    \end{align}  
where in the first step we used \eqref{GavLGEX} and \eqref{lk2safyas} to bound $ \langle \Gc_u(\sigma_k) E_{[a]} \rangle$ by $(W^d\ell_u^d\eta_u)^{-1}$; 
for any \smash{$\ba,\ba'\in (\Zn)^n$}, by \Cref{def:CALE}, we have that   
    \begin{align}
(\mathcal{B} \otimes \mathcal{B})^{(2\fn)}_{u, \boldsymbol{\sigma},\ba,\ba'}\prec W^d\ell_u^d\cdot \Xi^{(\cal L)}_{u, {2\fn+2}} (W^d\ell_u^d\eta_u)^{-2\fn-1 }  \prec \eta_u^{-1}(W^d\ell_u^d\eta_u)^{-2\fn }\cdot \Xi^{(\cal L)}_{u, {2\fn+2}}   .\label{CalEbwXi}
\end{align}
We emphasize that the above estimates \eqref{CalEbwXi1}--\eqref{CalEbwXi} do not depend on the assumption \eqref{NALsigm}.


We now analyze equation \eqref{int_K-LcalE} with the above estimates. 
Under the assumption \eqref{NALsigm}, using \eqref{sum_res_2_NAL}, the above bounds \eqref{CalEbwXi1}--\eqref{CalEbwXi3}, and the estimate \eqref{Eq:L-KGt+IND} on \smash{$(\mathcal{L} - \mathcal{K})^{(\fn)}_{s, \boldsymbol{\sigma}, \ba}$}, we obtain that 
\begin{align}\label{sahwNQ}
(W^d\ell_t^d\eta_t)^{\fn}\cdot      (\mathcal{L} - \mathcal{K})_{t, \boldsymbol{\sigma}, \ba} &\prec 
1 +   \int_{s}^t  \eta_u^{-1} \max_{k\in \qqq{2,\fn-1}}\Xi^{({\cal L-\cal K})}_{u,k}\dd u  +  \int_{s}^t \eta_u^{-1} \max_{k: k\in\qqq{2,\fn} } \frac{\Xi^{({\cal L-\cal K})}_{u,k} \Xi^{({\cal L-\cal K})}_{u,\fn-k+2}}{W^d\ell_u^d\eta_u} \dd u 
    \nonumber \\
    &+ \int_{s}^t \eta_u^{-1} \Xi^{(\cal L)}_{u,\fn+1}\dd u 
   + (W^d\ell_t^d\eta_t)^\fn \int_{s}^t \left(\mathcal{U}^{(\fn)}_{u, t, \boldsymbol{\sigma}} \circ  \dd  \mathcal{B}^{(\fn)}_{u, \boldsymbol{\sigma}}\right)_{\ba}.
\end{align}
By Lemma \ref{lem:DIfREP}, using \eqref{CalEbwXi} and \eqref{sum_res_2_NAL}, we obtain that  
\begin{align} \label{sahwNQ2}
(W^d\ell_t^d\eta_t)^\fn\int_{s}^t \left(\mathcal{U}^{(\fn)}_{u, t, \boldsymbol{\sigma}} \circ  \dd  \mathcal{B}^{(\fn)}_{u, \boldsymbol{\sigma}}\right)_{\ba} & \prec (W^d\ell_t^d\eta_t)^\fn \left\{\int_{s}^t 
   \left(\left(
   \mathcal{U}^{(\fn)}_{u,t,  \boldsymbol{\sigma}}
   \otimes 
   \mathcal{U}^{(\fn)}_{u,t,  \overline{\boldsymbol{\sigma}}}
   \right) \;\circ  \;
   \left( \cB \otimes  \cB \right)^{(2\fn)}
   _{u, \boldsymbol{\sigma}  }
   \right)_{\ba, \ba}\dd u\right\}^{1/2} \nonumber\\
   &\prec \left\{\int_{s}^t 
   \eta_u^{-1}\Xi^{(\cal L)}_{u, {2\fn+2}}\dd u\right\}^{1/2} .
  \end{align} 
Plugging it into \eqref{sahwNQ} and performing the integral over $u$, we conclude \eqref{am;asoiuw}.

It is clear from the above argument that the same estimate \eqref{am;asoiuw} indeed holds for each $\Xi^{({\cal L-\cal K})}_{u,\fn}$ with $u\in[s,t]$. Then, a standard $N^{-C}$-net argument allows us to extend it uniformly to all $u\in[s,t]$ and hence conclude \eqref{am;asoiuw} for \smash{$\sup_{u\in[s,t]}\wh\Xi^{({\cal L-\cal K})}_{u,\fn}$}.

\subsection{Proof of \Cref{lem:STOeq_Qt_weak}} \label{sec:pf_STOeq_Qt_weak}
We first claim that when $\bsig$ satisfies \eqref{NALsig_diff}, the following estimate holds uniformly in $u\in[s,t]$:
\begin{align}\label{jywiiwsoks}
 \left[\cal P\circ (\mathcal{L} - \mathcal{K})^{(\fn)}_{u, \boldsymbol{\sigma}}\right]_{[a_1]}  \prec \ell_u^{-d} (W^d\eta_u)^{-\fn} \Xi^{(\mathcal{L}-\mathcal{K})}_{u, \fn-1}.
\end{align}
To see this estimate, we apply \Cref{lem_WI_K} at the vertex $[a_\fn]$ and write the LHS of \eqref{jywiiwsoks} as 
$$ \frac{1}{2\ii W^d \eta_u}\left[\cal P\circ (\mathcal{L} - \mathcal{K})^{(\fn-1)}_{u, \wh\bsig^{(+,\fn)}}-\cal P\circ (\mathcal{L} - \mathcal{K})^{(\fn-1)}_{u, \wh\bsig^{(-,\fn)}}\right]_{[a_1]}.$$
By \eqref{def:XIL-K}, the two $(\fn-1)$-$G$ loops are controlled by $$(\mathcal{L} - \mathcal{K})^{(\fn-1)}_{u, \wh\bsig^{(\pm,\fn)},\hat \ba^{(n)}} \prec (W^d\ell_u^d\eta_u)^{-(\fn-1)}\Xi^{(\mathcal{L}-\mathcal{K})}_{u, \fn-1}.$$ Moreover, due to the fast decay property of the $(\cal L-\cal K)$-loops, the partial sums over the remaining $(\fn-2)$ vertices lead to an additional \smash{$\ell_u^{d(\fn-2)}$} factor up to a negligible error $\OO(W^{-D})$. This leads to \eqref{jywiiwsoks}.
On the other hand, using \eqref{prop:ThfadC}, we can bound \smash{$\dthn^{(\fn)}_{u, \ba}$ by $(\ell_u^d)^{-(\fn-1)}$}. Together with \eqref{jywiiwsoks}, it implies that 
\begin{align*}
\left[\cal P\circ (\mathcal{L} - \mathcal{K})^{(\fn)}_{u, \boldsymbol{\sigma}}\right]_{[a_1]} \dthn^{(\fn)}_{u, \ba} \prec (W^d\ell_u^{d}\eta_u)^{-\fn}\Xi^{(\mathcal{L}-\mathcal{K})}_{u, \fn-1}
\end{align*}
uniformly in $u\in [s,t]$, which leads to the following bound at $u=t$:
\be\label{eq:Ward_typeP}
(W^d\ell_t^{d}\eta_t)^{\fn} \left[\cal P\circ (\mathcal{L} - \mathcal{K})^{(\fn)}_{t, \boldsymbol{\sigma}}\right]_{[a_1]} \dthn^{(\fn)}_{t, \ba} \prec \Xi^{(\mathcal{L}-\mathcal{K})}_{t, \fn-1}.
\ee

It remains to control ${\cal Q}_t \circ (\mathcal{L} - \mathcal{K})^{(\fn)}_{t, \boldsymbol{\sigma}, \ba}$. For this purpose, we need the following claim on the $(\infty\to \infty)$-norm of the sum zero operator, which follows easily from \Cref{Def:QtPt} and the estimate \eqref{prop:ThfadC}. 

\begin{claim}\label{lem_+Q} 
Let ${\cal A}: (\Zn)^{\fn}\to \mathbb C$ be an $\fn$-dimensional tensor for a fixed $\fn\in \N$ with $\fn\ge 2$. If $\cal A$ satisfies the $(t, \e, D)$-decay property, then we have that 
\begin{align}\label{normQA}
\left\|{\cal Q}_t\circ \cal A\right\|_{\infty} \le W^{C_\fn \e} \|\cal A\|_\infty +W^{-D+C_\fn}
\end{align}
for a constant $C_\fn$ that does not depend on $\e$ or $D$. Furthermore, if $\|\cal A\|_\infty\le W^C$ for a constant $C>0$, then $\cal A_\fa-({\cal Q}_t\circ \cal A)_{\ba}$ satisfies the $(t, \e', D')$-decay property for any constants $\e',D'>0$.  
\end{claim}
\begin{proof}
The $(t, \e', D')$-decay property of $\cal A_\fa-({\cal Q}_t \circ \cal A)_{\ba}$ follows easily from the exponential decay of $\Theta^{(+,-)}_{t}$ given by \eqref{prop:ThfadC}. To show \eqref{normQA}, it suffices to prove that   
\be\label{eq:boundPA}
 \max_{\ba}\left|\left({\cal P} \circ {\cal A}\right)_{[a_1]}\dthn^{(\fn)}_{t, \ba}\right| \le W^{C_\fn \e}\|\cal A\|_\infty + W^{-D+C_\fn}. 
\ee
Using the estimate \eqref{prop:ThfadC} and the $(t,\e,D)$-decay property of $\cal A$, we get that 
\begin{align*}
   \max_{\ba}\left|\left({\cal P} \circ {\cal A}\right)_{[a_1]}\dthn^{(\fn)}_{t, \ba}\right| &\prec  \frac{|1-t|^{\fn-1}}{(\ell_t^d|1-t|)^{\fn-1}}\cdot (W^\e\ell_t^d)^{(\fn-1)}\|\cal A\|_\infty + W^{-D}n^d \\
   &= W^{\e (\fn-1)}\|\cal A\|_\infty + W^{-D}n^d, 
\end{align*}
which concludes \eqref{eq:boundPA}.
\end{proof}

We derive from equation \eqref{eq_L-Keee} that 
\begin{align} 
\dd{\cal Q}_t \circ & (\mathcal{L} - \mathcal{K})^{(\fn)}_{t, \boldsymbol{\sigma}, \ba} 
    = {\cal Q}_t \circ \left[\mathcal{O}_{\cK}^{(2)} (\mathcal{L} - \mathcal{K})\right]^{(\fn)}_{t, \boldsymbol{\sigma}, \ba} 
    + \sum_{l_\mathcal{K}=3}^\fn {\cal Q}_t \circ \Big[\mathcal{O}_{\cK}^{(\lenk)} (\mathcal{L} - \mathcal{K})\Big]^{(\fn)}_{t, \boldsymbol{\sigma}, \ba} + {\cal Q}_t \circ \mathcal{E}^{(\fn)}_{t, \boldsymbol{\sigma}, \ba}\dd t \nonumber
    \\
    & 
    + {\cal Q}_t \circ \mathcal{W}^{(\fn)}_{t, \boldsymbol{\sigma}, \ba} \dd t
    + {\cal Q}_t \circ \dd\mathcal{B}^{(\fn)}_{t, \boldsymbol{\sigma}, \ba} 
    - \left[{\cal P} \circ \left(\mathcal{L} - \mathcal{K}\right)^{(\fn)}_{t,\bsig}\right]_{[a_1]} \partial_t \dthn_{t,\ba}^{(\fn)} \dd t.\label{zjuii1}
\end{align} 
Recalling \Cref{DefTHUST}, we can rewrite the first term on the RHS as 
\begin{align}\label{zjuii2}
{\cal Q}_t \circ \Big[\mathcal{O}_{\cK}^{(2)}(\mathcal{L} - \mathcal{K})\Big]_{t, \boldsymbol{\sigma},\ba}^{(\fn)}
& = \thn^{(\fn)}_{t, \boldsymbol{\sigma}} \circ \left[{\cal Q}_t \circ (\mathcal{L} - \mathcal{K})_{t, \boldsymbol{\sigma}}^{(\fn)}\right]_{\ba} 
+ \left[{\cal Q}_t , \thn^{(\fn)}_{t, \boldsymbol{\sigma}} \right] \circ (\mathcal{L} - \mathcal{K})_{t, \boldsymbol{\sigma}, \ba}^{(\fn)},
\end{align}
where  $[{\cal Q}_t , \thn^{(\fn)}_{t, \boldsymbol{\sigma}} ] = {\cal Q}_t \circ \thn^{(\fn)}_{t, \boldsymbol{\sigma}} - \thn^{(\fn)}_{t, \boldsymbol{\sigma}}\circ \cal Q_t$ denotes the commutator
between ${\cal Q}_t $ and $\thn^{(\fn)}_{t, \boldsymbol{\sigma}}$. Since \({\cal P} \circ {\cal Q}_t = 0\), we notice that the first 5 terms on the RHS of \eqref{zjuii1} satisfy the sum zero property. 
Using 
$${\cal P}  \circ \dthn^{(\fn)}_{t,  \ba}= \sum_{\ba'}\dthn^{(\fn)}_{t,  \ba}\equiv 1,\quad \text{where}\quad \ba'=([a_2],\cdots, [a_\fn]),$$ 
we see that the last term on the RHS of \eqref{zjuii1} also satisfies the sum zero property:
\be\label{eq:sumzero666}
 {\cal P}  \circ \left\{\left[{\cal P} \circ \left(\mathcal{L} - \mathcal{K}\right)^{(\fn)}_{t,\bsig}\right]_{[a_1]} \partial_t \dthn_{t,\ba}^{(\fn)}\right\}
=\left[{\cal P} \circ \left(\mathcal{L} - \mathcal{K}\right)^{(\fn)}_{t,\bsig}\right]_{[a_1]} {\cal P} \circ \left(\partial_t \dthn_{t,\ba}^{(\fn)}\right)=0.
\ee
Next, due to \eqref{eq:sum0PA}, the first term on the RHS of \eqref{zjuii2} also satisfies the sum-zero property. 
Finally, since the LHS of \eqref{zjuii2} has the sum zero property, the second term on the RHS of \eqref{zjuii2} also satisfies that: 
\begin{align}\label{pqthlk}
   {\cal P}\circ \big[{\cal Q}_t , \thn^{(\fn)}_{t,\boldsymbol{\sigma}} \big]\circ (\mathcal{L} - \mathcal{K}) _{t, \boldsymbol{\sigma}, \ba}=0 .
\end{align} 

With Duhamel's principle, we can derive from \eqref{zjuii1} and \eqref{zjuii2} the following counterpart of \eqref{int_K-LcalE}:
 \begin{align}\label{int_K-L+Q}
 & {\cal Q}_t\circ  (\mathcal{L} - \mathcal{K})^{(\fn)}_{t, \boldsymbol{\sigma}, \ba}   = 
\left(\mathcal{U}^{(\fn)}_{s, t, \boldsymbol{\sigma}} \circ  {\cal Q}_s\circ (\mathcal{L} - \mathcal{K})^{(\fn)}_{s, \boldsymbol{\sigma}}\right)_{\ba}  + \sum_{l_\mathcal{K} = 3}^\fn \int_{s}^t \left(\mathcal{U}^{(\fn)}_{u, t, \boldsymbol{\sigma}} \circ  {\cal Q}_u\circ \Big[\mathcal{O}_{\cK}^{(\lenk)} (\mathcal{L} - \mathcal{K})\Big]^{(\fn)}_{u, \boldsymbol{\sigma}}\right)_{\ba} \dd u \nonumber \\
&+ \int_{s}^t \left(\mathcal{U}^{(\fn)}_{u, t, \boldsymbol{\sigma}} \circ  {\cal Q}_u\circ \mathcal{E}^{(\fn)}_{u, \boldsymbol{\sigma}}\right)_{\ba} \dd u + \int_{s}^t \left(\mathcal{U}^{(\fn)}_{u, t, \boldsymbol{\sigma}} \circ  {\cal Q}_u\circ \mathcal{W}^{(\fn)}_{u, \boldsymbol{\sigma}}\right)_{\ba} \dd u + \int_{s}^t \left(\mathcal{U}^{(\fn)}_{u, t, \boldsymbol{\sigma}} \circ  {\cal Q}_u\circ \dd\mathcal{B}^{(\fn)}_{u, \boldsymbol{\sigma}}\right)_{\ba} 
    \nonumber \\
&+ \int_{s}^t \left(\mathcal{U}^{(\fn)}_{u, t, \boldsymbol{\sigma}} 
\circ \left( \left[{\cal Q}_u , \thn^{(\fn)}_{u,\boldsymbol{\sigma}} \right]\circ (\mathcal{L} - \mathcal{K})^{(\fn)}_{u, \boldsymbol{\sigma} }\right) \right)_{\ba} \dd u
   - \int_{s}^t \left(\mathcal{U}^{(\fn)}_{u, t, \boldsymbol{\sigma}} 
     \circ \left\{\left[ {\cal P} \circ
      \left(\mathcal{L} - \mathcal{K}\right)^{(\fn)}_{u, \boldsymbol{\sigma}}\right] \partial_u\dthn_{u}^{(\fn)} \right\}\right)_{\ba} \dd u.
\end{align}
We can control the terms on the RHS of \eqref{int_K-L+Q} by using the improved estimate \eqref{sum_res_2} due to the sum zero and fast decay properties of the terms on the RHS of \eqref{zjuii1}. With \Cref{lem_+Q}, the induction hypothesis \eqref{Eq:L-KGt+IND} at time $s$, the estimates established in \eqref{CalEbwXi1}--\eqref{CalEbwXi},  \Cref{lem:DIfREP} for the martingale term, and the evolution kernel estimate \eqref{sum_res_2}, we can control the first 5 terms on the RHS of \eqref{int_K-L+Q} in a manner similar to that in \Cref{lem:STOeq_NQ}. It remains to address the last two terms on the RHS of \eqref{int_K-L+Q}.

First, with the definition of \smash{$\dthn_u^{(\fn)}$} in \eqref{eq:sumzero_op} and the estimate \eqref{prop:ThfadC}, we can check directly that 
\be\label{eq:derv_Theta} \|\dthn_{u}^{(\fn)}\|_\infty \prec (\ell_u^d)^{-(\fn-1)},\quad 
\|\partial_u\dthn_{u}^{(\fn)}\|_\infty \prec \eta_u^{-1}(\ell_u^d)^{-(\fn-1)}. \ee
Combining the second bound with \eqref{jywiiwsoks}, we obtain that 
\be\label{A5}  
\left\|\left[ {\cal P} \circ \left(\mathcal{L} - \mathcal{K}\right)^{(\fn)}_{u, \boldsymbol{\sigma}}\right] \partial_u\dthn_{u}^{(\fn)}\right\|_\infty \prec \eta_u^{-1}(W^d\ell_u^d\eta_u)^{-\fn}\cdot \Xi^{(\mathcal{L}-\mathcal{K})}_{u, \fn-1}.
\ee
Second, using the definition of $\cal Q_t$ in \eqref{eq:sumzero_op} and the definition of $\thn^{(\fn)}_{u,\bsig}$ in \eqref{def:op_thn}, we can bound that
\begin{align}\label{A4}  
\br{{\cal Q}_u , \thn^{(\fn)}_{u,\boldsymbol{\sigma}} }\circ (\mathcal{L} - \mathcal{K})^{(\fn)}_{u, \boldsymbol{\sigma}, \ba}
   &= \thn^{(\fn)}_{u, \boldsymbol{\sigma}} \circ \left[ \left({\cal P} \circ \left( \mathcal{L} - \mathcal{K} \right)^{(\fn)}_{u,\bsig}\right) \cdot \dthn^{(\fn)}_u \right]_{\ba} - \left[\left( {\cal P} \circ \thn^{(\fn)}_{u, \boldsymbol{\sigma}} \circ (\mathcal{L} - \mathcal{K})^{(\fn)}_{u,\bsig}\right)  \cdot \dthn^{(\fn)}_u \right]_{\ba} \nonumber
\\ 
& \prec \eta_u^{-1}\left\|\dthn^{(\fn)}_{u}\right\|_\infty \cdot   \left\| {\cal P} \circ \left(  \mathcal{L} - \mathcal{K} \right)^{(\fn)}_{u,\boldsymbol{\sigma}} \right\|_{\infty} \prec \eta_u^{-1}(W^d\ell_u^d\eta_u)^{-\fn}\cdot \Xi^{(\mathcal{L}-\mathcal{K})}_{u, \fn-1},
\end{align}
where in the second step, we used the simple fact that $ \|\thn_{u,\bsig}^{(\fn)}\|_{\infty\to\infty} \lesssim \eta_u^{-1}$ for any $\bsig\in \{+,-\}^\fn$ due to \eqref{eq:WardM2}, and in the third step, we used \eqref{jywiiwsoks} and the first bound in \eqref{eq:derv_Theta}. 
Finally, applying the evolution kernel estimate \eqref{sum_res_2} and the estimates \eqref{A5} and \eqref{A4}, and performing the integral over $u$, we can control the last two terms on the RHS of \eqref{int_K-L+Q} by $\p{{\ell_t^{d-1}}/{\ell_s^{d-1}}} \cdot (W^d\ell_t^d\eta_t)^{-\fn}\sup_{u\in[s,t]}\Xi^{({\cal L-\cal K})}_{u,\fn-1}$. This concludes the proof of \Cref{lem:STOeq_Qt_weak}.


\subsection{Proof of \Cref{lem:iterations}}\label{sec:pf_iterations}
By \eqref{Eq:L-KGt+IND}, we know that \smash{$\wh\Xi_{s,\fn}^{(\mathcal{L}-\mathcal{K})} \prec 1$} initially. We claim that under the condition \eqref{eq:iteration_induc}, the following bound holds uniformly in $u\in[s,t]$ for every fixed $k\ge 1$: 
 \begin{align} \label{xiu2n+2psi}
\sup_{v\in[s,u]}\wh\Xi^{({\cal L})}_{v, 2\fn+2} \prec \Psi_u(\fn,k;[s,u])^2=:\Xi^{(\cal L)}_{v, {2\fn+2}} \ .
\end{align}
Given this bound, we obtain from \eqref{am;asoi333} that
\begin{align} 
\sup_{v\in[s,u]}\wh\Xi_{v,\fn }^{(\mathcal{L}-\mathcal{K})}  &\prec \Psi_u(\fn,k;[s,u])  + 
  \sup_{v\in [s,u]}\left( \max_{r\in \qqq{2,\fn-1} }\Xi^{({\cal L-\cal K})}_{v,r}+\max_{r\in \qqq{2, \fn} }\frac{\Xi^{({\cal L-\cal K})}_{v,r}\Xi^{({\cal L-\cal K})}_{v,\fn-r+2}}{W^d\ell_v^d\eta_v}+\Xi^{({\cal L})}_{v,\fn+1}\right)\label{sadui_w0}
 \end{align}
uniformly in $u\in[s,t]$. By the induction hypothesis \eqref{eq:iteration_induc}, we can choose our control parameters as 
\be \nonumber
\Xi^{({\cal L-\cal K})}_{v,r}= \Psi_u(r,k;[s,u]),\quad  \forall r\in \qqq{2,\fn-1},\ \  \text{and} \ \ \Xi^{({\cal L-\cal K})}_{v,\fn}=\sup_{v\in [s,u]}\Psi_v(\fn,k-1;[s,u])\ee
for each $v\in[s,u]$, where we also used that $\Psi_v(r,k;[s,u])\equiv \Psi_v(r,k;[s,u])$ when $k\ge 1$. When $\fn=2$, by \eqref{Eq:Gdecay_w}, we can choose a slightly better \smash{$\Xi^{({\cal L-\cal K})}_{v,2}$}:
\be \nonumber
\Xi^{({\cal L-\cal K})}_{v,2}= \Psi_u(r,k;[s,u])\wedge \p{{\eta_s}/{\eta_u}}^4.\ee
By \eqref{rela_XILXILK},  \eqref{eq:iteration_induc}, and that $\ell_u^d\eta_u\lesssim \ell_v^d\eta_v$ for any $v\in[s,u]$, we can choose {$\Xi^{({\cal L})}_{v,\fn+1}$} as follows: 
$$\sup_{v\in [s,u]}\wh \Xi^{({\cal L})}_{v,\fn+1}\prec 1+(W^d\ell_u^d\eta_u)^{-1}\sup_{v\in [s,u]} \Psi_v(\fn+1,k-1;[s,u]) \prec \Psi_u(\fn,k;[s,u])=:\Xi^{({\cal L})}_{v,\fn+1}.$$
Plugging these parameters into \eqref{sadui_w0} and using that $\Psi_{u}(r,k;[s,u])$ is increasing in $r$, we obtain that   
 \begin{align*}
\sup_{v\in[s,u]}\wh\Xi_{v,\fn }^{(\mathcal{L}-\mathcal{K})}\prec &~\Psi_u(\fn,k;[s,u]) + \sup_{v\in [s,u]} \p{\frac{\eta_s}{\eta_u}}^4\frac{\Psi_{v}(\fn,k-1;[s,u])}{W^d\ell_u^d\eta_u } \\
&~+ \max_{r\in \qqq{3, \fn-1}} \frac{\Psi_{u}(r,k;[s,u])\Psi_{u}(\fn-r+2,k;[s,u])}{W^d\ell_u^d\eta_u } \prec \Psi_u(\fn,k;[s,u]), 
 \end{align*}
where we again used the condition \eqref{con_st_ind} in the second step. 
This concludes \eqref{eq:iteration_improve}.

It remains to prove \eqref{xiu2n+2psi} under the induction hypothesis \eqref{eq:iteration_induc}. For simplicity, given $t\in [0,1]$, $\bsig=(\sigma_1,  \ldots, \sigma_\fn)\in \{+,-\}^\fn,$ and \smash{$\ba = ([a_1], \ldots, [a_{\fn-1}])\in (\Zn)^{\fn-1}$}, we introduce the following notation of a $G$-chain of length $\fn$ as
\begin{equation}\label{defC=GEG}
    {\cal C}^{(\fn)}_{t, \boldsymbol{\sigma}, \ba} = \prod_{i=1}^{\fn-1}\left(G_t(\sigma_i) E_{[a_i]}\right)\cdot  G_t(\sigma_\fn) .
\end{equation}
Given an arbitrary $(2\fn+2)$-$G$ chain $\cL^{(2\fn+2)}_{v,\bsig,\ba}$ with 
$$\bsig=(\sig,\sig_1,\ldots, \sig_\fn, \sig',\sig_1',\ldots, \sig_\fn'),\quad \ba=([b],[a_1],\ldots,[a_\fn],[b'],[a_1'],\ldots,[a_\fn']),$$ 
we can write it is 
\begin{align*}
\cL^{(2\fn+2)}_{v,\bsig,\ba}=W^{-2d}\sum_{x\in[b],x'\in [b']} \left(\cal C^{(\fn+1)}_{v,\bsig_1,\ba_1}\right)_{xx'} \left(\cal C^{(\fn+1)}_{v,\bsig_2,\ba_2}\right)_{x'x} ,  
\end{align*}
where $\bsig_1=(\sig_1,\ldots, \sig_\fn, \sig')$, $\bsig_2=(\sig_1',\ldots, \sig_\fn',\sig)$, $\ba_1=([a_1],\ldots,[a_\fn])$, and $\ba_2=([a_1'],\ldots,[a_\fn'])$. Applying the Cauchy-Schwarz inequality, we obtain that 
\begin{equation}\label{suauwiioo1}
\max_{\ba\in(\Zn)^{2\fn+2}}\max_{\bsig\in \{+,-\}^{2\fn+2}}\left|\cL^{(2\fn+2)}_{v,\bsig,\ba}\right|
\le 
\max_{\ba\in(\Zn)^{\fn}}\max_{\boldsymbol{\sigma} \in\{+,-\}^{\fn+1}} W^{-2d}\sum_{x\in[b],x'\in[b']}
\left|\left({\cal C}^{(\fn+1)} _{v,\boldsymbol{\sigma},\ba}\right)_{xx'} 
\right|^2 .
\end{equation}
Now, we split the $(\fn+1)$-$G$ chain further into an $l_1$-$G$ chain and an $l_2$-$G$ chain, where $l_1=\lfloor(\fn+1)/{2}\rfloor$ and $l_2=\fn+1-l_1$. Applying the Cauchy-Schwarz inequality again, we get that  \begin{align}\label{suauwiioo2}
    W^{-2d}\sum_{x\in[b],x'\in[b']}
\left|\left({\cal C}^{(\fn+1)} _{v,\boldsymbol{\sigma},\ba}\right)_{xx'} 
\right|^2 &\le W^{-2d}\sum_{x\in[b],x'\in[b']} \left(\cal C^{(2l_1)}_{v,\bsig_1',\ba_1'}\right)_{xx} \left(\cal C^{(2l_2)}_{v,\bsig_2',\ba_2'}\right)_{x'x'}  =\cL^{(2l_1)}_{v,\bsig_1',\wt\ba_1}\cL^{(2l_2)}_{v,\bsig_2',\wt\ba_2},
\end{align}
where $\bsig_1'$, $\ba_1'$, $\bsig_2'$, $\ba_2'$ are defined by  
\begin{align*}
\bsig_1'=(\sig_1,\ldots,\sig_{l_1},-\sig_{l_1},\ldots,-\sig_1),\quad  &\ba_1'=([a_1],\ldots, [a_{l_1-1}],[a_{l_1}],[a_{l_1-1}],\ldots, [a_1]),\\
\bsig_2'=(-\sig_{\fn+1},\ldots,-\sig_{l_1+1},\sig_{l_1+1},\ldots,\sig_{\fn+1}),\quad &\ba_2'=([a_{\fn}],\ldots, [a_{l_1+1}],[a_{l_1}],[a_{l_1+1}],\ldots, [a_{\fn}]),
\end{align*}
$\wt\ba_1$ is obtained by adding $[b]$ to $\ba_1'$, and $\wt\ba_2$ is obtained by adding $[b']$ to $\ba_2'$. Plugging \eqref{suauwiioo2} into \eqref{suauwiioo1} and using that $l_1+l_2=\fn+1$, we obtain that 
\begin{align}\label{auskoppw2.0}
\wh\Xi^{({\cal L})}_{v, 2n+2}&= \left(W^d\ell_v^d\eta_v\right)^{2\fn+1}\max_{\ba\in(\Zn)^{2\fn+2}}\max_{\bsig\in \{+,-\}^{2\fn+2}}\left|\cL^{(2\fn+2)}_{v,\bsig,\ba}\right|  \prec \left(W^d\ell_v^d\eta_v\right) \Xi^{({\cal L})}_{v, 2l_1}\Xi^{({\cal L})}_{v, 2l_2}.
\end{align}
Now, by the induction hypothesis \eqref{eq:iteration_induc} (since $\max\{2l_1,2l_2\}\le \fn+2$) and the relation \eqref{rela_XILXILK}, we have
\begin{align}\label{auskoppw2.00}
\wh\Xi_{v,2l_i}^{(\cal L)} \prec 1+\left(W^d\ell_v^d\eta_v\right)^{-1}\Psi_v(k-1,2l_i;[s,u]) \prec 1+ \e_v^{(i)}(k),\quad \forall i\in\{1,2\},
\end{align}
where $\e_v^{(i)}(k)$ is defined as 
$$\e_v^{(i)}(k):=(\ell_u^d/\ell_s^d)^{2l_i-1}\times
    \begin{cases}
        1 , &  \ \text{if}\ k=1\\
       \left(W^d\ell_v^d\eta_v\right)^{-1} (W^d\ell_s^d\eta_s)^{1-(k-1)/4}, &   \ \text{if}\  k\ge 2 
    \end{cases}.$$
Plugging \eqref{auskoppw2.00} into \eqref{auskoppw2.0} gives that 
\begin{align}\label{auskoppw2}
\sup_{v\in[s,u]}\wh\Xi^{({\cal L})}_{v, 2\fn+2}&\prec  \sup_{v\in[s,u]} \left[\left(W^d\ell_v^d\eta_v\right)\prod_{i=1}^2\left(1+\e_v^{(i)}(k)\right)\right] \\
&\lesssim \left((W^d\ell_s^d\eta_s)^{1/2}+(\ell_u^d/\ell_s^d)^{\fn-1}(W^d\ell_s^d\eta_s)^{1-k/4}\right)^2
      =\Psi_u(\fn,k;[s,u])^2 \nonumber
\end{align}
for $k\ge 1$, where we also used the condition \eqref{con_st_ind} in the second step. This concludes \eqref{xiu2n+2psi}, and hence completes the proof of \Cref{lem:iterations}.

\subsection{Step 6 of the proof of \Cref{lem:main_ind}}\label{sec:pf_step6}

In this step, we have the initial estimate \eqref{Eq:Gtlp_exp+IND} at time $s$, sharp local law \eqref{Gt_bound_flow}, and sharp $G$-loop estimates \eqref{Eq:LGxb}, \eqref{Eq:L-KGt-flow}, and \eqref{Eq:Gdecay_flow}. We will use them to establish \eqref{Eq:Gtlp_exp_flow}. 
First, we have established a sharp averaged local law in \eqref{eq:res_ELK_n=1}. Next, we prove an improved bound on its expectation. 

\begin{lemma}\label{lem:improve_exp_aver}
In the setting of \Cref{lem:main_ind}, suppose the estimates \eqref{Gt_bound_flow}, \eqref{Eq:LGxb}, and \eqref{eq:res_ELK_n=1} hold uniformly in $u\in[s,t]$. Then, we have 
\begin{align}\label{res_ELK_n=1} 
  \max_{[a]}  \left|\mathbb E\langle (G_u-M) E_{[a]}\rangle\right|
  \prec (W^d\ell_u^d\eta_u)^{-2},\quad \forall  s\le u \le t .
  \end{align} 
\end{lemma}
\begin{proof}
Plugging \eqref{G-M} into $\mathbb{E} \langle (G_u - M) E_{[a]} \rangle$ gives (recall the notation in \eqref{Eq:defwtG})
\begin{align}
    \mathbb{E} \langle \Gc_u E_{[a]} \rangle &= - \mathbb{E} \langle M(um+V_u) G_u E_{[a]} \rangle 
    = uW^d \sum_{[b]} \mathbb{E}\langle G_u E_{[a]}ME_{[b]} \rangle \langle \Gc_u E_{[b]} \rangle  \nonumber\\
    &= u \sum_{[b]} M^{(+,+)}_{[a][b]}\E\langle \Gc_u E_{[b]} \rangle + uW^d \sum_{[b]} \mathbb{E}\langle \Gc_u E_{[a]}ME_{[b]} \rangle \langle \Gc_u E_{[b]} \rangle, \label{eq:G-MG-M}
\end{align}
where in the second step we applied Gaussian integration by parts to the entries of $V_u$ (whose variance matrix is $uS=u(I_n\otimes \mathbf E)$), and in the third step we recall the notation \eqref{eq:Msig}. Then, we apply \eqref{GavLGEX} to \smash{$W^d\langle \Gc_u E_{[a]}ME_{[b]} \rangle$} and get that 
\be \label{eq:res_ELK_n=1.0} 
W^d\langle \Gc_u E_{[a]}ME_{[b]} \rangle \prec(CW^{-\e_A})^{|[a]-[b]|} (W^d\ell_u^d\eta_u)^{-1},\ee
where we also used that the $([a],[b])$-th block of $M$ has $L^2\to L^2$ 
norm of order {$(C\heta)^{|[a]-[b]|} 
\le (CW^{-\e_A})^{|[a]-[b]|}$} by \eqref{Mbound_AO} and \eqref{eq:cond_A12}. With \eqref{eq:res_ELK_n=1} and \eqref{eq:res_ELK_n=1.0}, we can bound the second term on the RHS of \eqref{eq:G-MG-M} by  
\be\label{eq:second_uGM}
uW^d \sum_{[b]} \mathbb{E}\langle \Gc_u E_{[a]}ME_{[b]} \rangle \langle \Gc_u E_{[b]} \rangle\prec (W^d\ell_u^d\eta_u)^{-2}. 
\ee
Hence, we can rewrite equation \eqref{eq:G-MG-M} as 
\begin{align*}
\sum_{[b]} \left(1-uM^{(+,+)}\right)_{[a][b]} \mathbb{E} \langle \Gc_u E_{[b]} \rangle = \OO_\prec \left((W^d\ell_u^d\eta_u)^{-2}\right).
\end{align*}
Solving this equation and using $\|\Theta_{u}^{(+,+)}\|_{\infty\to \infty}\lesssim 1$ by \eqref{prop:ThfadC_short}, we conclude \eqref{res_ELK_n=1}.
\end{proof}

\begin{proof}[\bf Step 6: Proof of \eqref{Eq:Gtlp_exp_flow}]
Taking the expectation of both sides of equation \eqref{int_K-LcalE} when $\fn=2$, we get that  
\begin{align}\label{Eexpint_K-L}
 \mathbb E (\mathcal{L} - \mathcal{K})^{(2)}_{t, \boldsymbol{\sigma}, \ba}  =
 &~\left(\mathcal{U}^{(2)}_{s, t, \boldsymbol{\sigma}} \circ \E (\mathcal{L} - \mathcal{K})^{(2)}_{s, \boldsymbol{\sigma}}\right)_{\ba} + \int_{s}^t \left(\mathcal{U}^{(2)}_{u, t, \boldsymbol{\sigma}} \circ \E\mathcal{E}^{(2)}_{u, \boldsymbol{\sigma}}\right)_{\ba} \dd u  + \int_{s}^t \left(\mathcal{U}^{(2)}_{u, t, \boldsymbol{\sigma}} \circ \E\cW^{(2)}_{u, \boldsymbol{\sigma}}\right)_{\ba} \dd u .
\end{align} 
At time $s$, by \eqref{Eq:Gtlp_exp+IND}, we have 
\begin{align}\label{jdasfuao01}
  \mathbb E\;  (\mathcal{L} - \mathcal{K})^{(2)}_{s, \boldsymbol{\sigma},\ba}\prec (W^d\ell_s^d\eta_s)^{-3}.
\end{align}
For $\cal E^{(2)}_{u,\bsig,\ba}$ defined in \eqref{def_ELKLK}, using the 2-$G$ loop estimate in \eqref{Eq:Gdecay_flow}, we  bound it by 
\begin{align}\label{jdasfuao02}
 \cal E^{(2)}_{u,\bsig,\ba} \prec W^d\ell_u^d (W^d\ell_u^d\eta_u)^{-4}= \eta_u^{-1} (W^d\ell_u^d\eta_u)^{-3}. 
\end{align}
To control $\E\cW^{(2)}_{u,\bsig,\ba}$, we use the expression \eqref{eq:calW} and write 
\begin{align}
\E\mathcal{W}^{(2)}_{t, \boldsymbol{\sigma}, \ba} &= W^d \sum_{[a]} \left(\E\langle \Gc_u E_{[a]}\rangle {\cal K}^{(3)}_{u, \bsig_3,\ba_3} +\E\langle \Gc_u E_{[a]}\rangle ({\cal L}-\cal K)^{(3)}_{u, \bsig_3,\ba_3}\right)+ c.c. \nonumber\\
&\prec W^d\ell_u^d (W^d\ell_u^d\eta_u)^{-4}= \eta_u^{-1} (W^d\ell_u^d\eta_u)^{-3}, \label{jdasfuao03}
\end{align}
where $\bsig_3:=(+,+,-)$ and $\ba_3:=([a],[a_1],[a_2])$, and in the second step, we used \eqref{eq:res_ELK_n=1}, \eqref{res_ELK_n=1}, the primitive loop bound \eqref{eq:bcal_k}, and the $3$-$G$ loop estimate \eqref{Eq:L-KGt-flow}.

By \Cref{lem_decayLoop}, we observe that all tensors $\E (\mathcal{L} - \mathcal{K})^{(2)}_{s, \boldsymbol{\sigma}}$, $\E\cal E^{(2)}_{u,\bsig}$, and  $\E\cal W^{(2)}_{u,\bsig}$ on the RHS of \eqref{Eexpint_K-L} have the fast decay property. 
If $d=1$, using the bound \eqref{sum_res_1} and the estimates   \eqref{jdasfuao01}--\eqref{jdasfuao03}, we can estimate the RHS of  \eqref{Eexpint_K-L} as 
\begin{align*}
\mathbb E (\mathcal{L} - \mathcal{K})^{(2)}_{t, \boldsymbol{\sigma}, \ba} \prec 
 \frac{\ell_t}{\ell_s}\left(\frac{\ell_s\eta_s}{\ell_t\eta_t}\right)^2 (W\ell_s\eta_s)^{-3}+
 \int_{s}^t \frac{\ell_t}{\ell_u}\left(\frac{\ell_u\eta_u}{\ell_t\eta_t}\right)^2 \cdot  \eta_u^{-1}(W\ell_u\eta_u)^{-3}
 \dd u 
 \prec  (W\ell_t\eta_t)^{-3},
 \end{align*}
where in the second step we used that $\eta_t\ell_t^{2}\lesssim \eta_u\ell_u^{2}$ for  $u\le t$. 
If $\sig_1 = \sig_2$, then using the bound \eqref{sum_res_2_NAL} and the estimates   \eqref{jdasfuao01}--\eqref{jdasfuao03}, we can also get the above bound for $d=2$. To conclude \eqref{Eq:Gtlp_exp_flow}, it remains to deal with the $\sig_1\ne \sig_2$ case when $d=2$. Without loss of generality, we assume $\sig_1=-\sig_2=+$ throughout the following proof. We then adopt a similar idea as in \Cref{sec_sumzero}.  

First, using Ward's identity in \Cref{lem_WI_K} and \eqref{res_ELK_n=1}, we obtain that
\be\label{eq:ELK2}
\br{\cal P\circ \mathbb E (\mathcal{L} - \mathcal{K})^{(2)}_{t, \boldsymbol{\sigma}}}_{[a_1]} = {\im \E \qq{(G_u -M)E_{[a_1]}}}/(W^d\eta_t)  \prec (W^d\ell_t^d\eta_t)^{-2}(W^d\eta_t)^{-1},
\ee
which, together with \eqref{prop:ThfadC}, implies that
\be\label{eq:EPL-K}
\br{\cal P\circ \mathbb E (\mathcal{L} - \mathcal{K})^{(2)}_{t, \boldsymbol{\sigma}}}_{[a_1]} \cdot (1-t)\Theta_{t,[a_1][a_2]}^{(+,-)}\prec (W^d\ell_t^d\eta_t)^{-3}.
\ee
For $\cal Q_t\circ \mathbb E (\mathcal{L} - \mathcal{K})^{(2)}_{t, \boldsymbol{\sigma}, \ba}$, taking the expectation of equation \eqref{int_K-L+Q} with $\fn=2$ yields
\begin{align}\label{int_K-L+QE}
& {\cal Q}_t\circ  \E(\mathcal{L} - \mathcal{K})^{(2)}_{t, \boldsymbol{\sigma}, \ba}   = 
\left(\mathcal{U}^{(2)}_{s, t, \boldsymbol{\sigma}} \circ  {\cal Q}_s\circ \E(\mathcal{L} - \mathcal{K})^{(2)}_{s, \boldsymbol{\sigma}}\right)_{\ba} + \int_{s}^t \left(\mathcal{U}^{(2)}_{u, t, \boldsymbol{\sigma}} \circ  {\cal Q}_u\circ \E \mathcal{E}^{(2)}_{u, \boldsymbol{\sigma}}\right)_{\ba} \dd u \nonumber\\
& + \int_{s}^t \left(\mathcal{U}^{(2)}_{u, t, \boldsymbol{\sigma}} \circ  {\cal Q}_u\circ \E\mathcal{W}^{(2)}_{u, \boldsymbol{\sigma}}\right)_{\ba} \dd u  + \int_{s}^t 
\left(\mathcal{U}^{(2)}_{u, t, \boldsymbol{\sigma}} 
\circ \left( \left[{\cal Q}_u , \thn^{(2)}_{u,\boldsymbol{\sigma}} \right]\circ \E(\mathcal{L} - \mathcal{K})^{(2)}_{u, \boldsymbol{\sigma} }\right) \right)_{\ba} \dd u \nonumber\\
& - \int_{s}^t \left(\mathcal{U}^{(2)}_{u, t, \boldsymbol{\sigma}} \circ \left\{\left[ {\cal P} \circ \E\left(\mathcal{L} - \mathcal{K}\right)^{(2)}_{u, \boldsymbol{\sigma}}\right] \partial_u\dthn_{u}^{(2)} \right\}\right)_{\ba} \dd u.
\end{align}
Note that all tensors on the RHS satisfy the sum zero property \eqref{sumAzero} and the symmetry \eqref{eq:A_zero_sym} (which is due to the translation and parity symmetries of our model on the block level as explained below \Cref{def: BM}). Then, for the first three terms on the RHS, using \Cref{lem_+Q} and the estimates \eqref{sum_res_2_sym} and   \eqref{jdasfuao01}--\eqref{jdasfuao03}, we can bound them by 
\begin{align}\label{eq:EKL_added}
 \left(\frac{\ell_s^d\eta_s}{\ell_t^d\eta_t}\right)^3 (W^d\ell_s^d\eta_s)^{-3}+
 \int_{s}^t  \left(\frac{\ell_u^d\eta_u}{\ell_t^d\eta_t}\right)^3 \cdot  \eta_u^{-1}(W^d\ell_u^d\eta_u)^{-3}
 \dd u 
 \prec  (W^d\ell_t^d\eta_t)^{-3}.
 \end{align} 
It remains to control the last two terms on the RHS of \eqref{int_K-L+QE}. 
First, combining the second bound in \eqref{eq:derv_Theta} (for $\fn=2$) with the estimate \eqref{eq:ELK2}, we get that
\be\label{eq:boundELKQ1}
\left\|\left[ {\cal P} \circ \E\left(\mathcal{L} - \mathcal{K}\right)^{(2)}_{u, \boldsymbol{\sigma}}\right] \partial_u\dthn_{u}^{(2)}\right\|_\infty \prec \eta_u^{-1} (W^d\ell_u^d\eta_u)^{-3}. 
\ee
Second, using a similar estimate as in \eqref{A4} (for $\fn=2$), we can bound that
\begin{align}
\big[{\cal Q}_u , \thn^{(\fn)}_{u,\boldsymbol{\sigma}} \big]\circ \E(\mathcal{L} - \mathcal{K})^{(2)}_{u, \boldsymbol{\sigma}, \ba}
\prec \eta_u^{-1}\left\|\dthn^{(2)}_{u}\right\|_\infty \cdot   \left\| {\cal P} \circ \E\left(  \mathcal{L} - \mathcal{K} \right)^{(2)}_{u,\boldsymbol{\sigma}} \right\|_{\infty}\prec \eta_u^{-1} (W^d\ell_u^d\eta_u)^{-3},\label{eq:boundcommutator}
\end{align} 
where we used the first bound in \eqref{eq:derv_Theta} and \eqref{eq:ELK2} in the second step. 
Finally, using the bound \eqref{sum_res_2_sym} and the estimates \eqref{eq:boundELKQ1} and \eqref{eq:boundcommutator},  
we can also get the above bound \eqref{eq:EKL_added} for the last two terms on the RHS of \eqref{int_K-L+QE}.
This concludes the proof for the $d=2$ case with $\sig_1\ne \sig_2$.
 \end{proof}

\subsection{Proof of \Cref{lem:STOeq_Qt} }\label{sec:add_d=2}

Due to \Cref{lem:STOeq_NQ,lem:STOeq_Qt_weak}, we only need to prove \Cref{lem:STOeq_Qt} for the $d=2$ case and for loops with $\bsig$ satisfying \eqref{NALsig_diff}. Recalling \eqref{eq:Ward_typeP}, to show \Cref{lem:STOeq_Qt}, we still need to control the terms on the RHS of \eqref{int_K-L+Q} by the RHS of \eqref{am;asoi222}.  
We first bound the martingale term. Through a direct calculation, we derive that its quadratic variation $[\cdot]_t$ takes the form 
\begin{align} \label{UEUEdif+Q}
  &\left[\int_{s}^t \left(\mathcal{U}^{(\fn)}_{u, t, \boldsymbol{\sigma}} \circ  {\cal Q}_u\circ \dd\mathcal{B}^{(\fn)}_{u, \boldsymbol{\sigma}}\right)_{\ba}\right]_t
 =  \int_s^t \bigg\{\left(\mathcal{U}^{(\fn)}_{u, t,  \boldsymbol{\sigma}}\otimes \mathcal{U}^{(\fn)}_{u, t,  \overline{\boldsymbol{\sigma}}}\right)\circ \sum_{i,j\in \ZL} S_{ij} \left|\cal Q_u\left( \partial_{ij}\cL^{(\fn)}_{u,\bsig}\right) \right|^2\bigg\}_{\ba,\ba}\dd u  \nonumber\\
 &\qquad \lesssim \int_{s}^t
 \left( \left(
   \mathcal{U}^{(\fn)}_{u, t,  \boldsymbol{\sigma}}
   \otimes 
   \mathcal{U}^{(\fn)}_{u, t,  \overline{\boldsymbol{\sigma}}}
   \right)\circ 
\left( \mathcal{Q}_u\otimes \mathcal{Q}_u\right) \circ \left( \mathcal{B}\otimes  \mathcal{B} \right)_{u,  \boldsymbol{\sigma}}^{(2\fn)}\right)_{\ba, \ba}\dd u,
\end{align}
where $(\mathcal{B}\otimes  \mathcal{B})_{u,  \boldsymbol{\sigma}}^{(2\fn)}$ and $\mathcal{U}^{(\fn)}_{u, t,  \boldsymbol{\sigma}}\otimes \mathcal{U}^{(\fn)}_{u, t,  \overline{\boldsymbol{\sigma}}}$ are defined in \Cref{def:CALE} and \Cref{lem:DIfREP}, respectively, and $\mathcal{Q} _u\otimes \mathcal{Q}_u$ is defined as follows: given a $(2\fn)$-dimensional tensor $\cal C$,
\begin{align}
\left(\left( \mathcal{Q} _u\otimes \mathcal{Q}_u \right)\circ {\cal C}\right)_{\ba, \textbf{b}} 
&:=  
{\cal C }_{\ba, \textbf{b} }-
\delta_{[a_1'][a_1]}\sum_{[a_2'],\ldots, [a_\fn']}{\cal C }_{ \ba' ,  \textbf{b}}\cdot \dthn^{(\fn)}_{u, \ba}
-\delta_{[b_1'] [b_1]}
\sum_{[b_2'],\ldots, [b_\fn']}{\cal C }_{\ba, \textbf{b}'}\cdot \dthn^{(\fn)}_{u, \textbf{b} }
\nonumber \\  
&+\delta_{[a_1'][a_1]}\delta_{[b_1'] [b_1]}
\sum_{[a_2'],\ldots, [a_\fn']}\sum_{[b_2'],\ldots, [b_\fn']} {\cal C }_{\ba', \textbf{b}'}\cdot \left(\dthn^{(\fn)}_{u, \ba}  \dthn^{(\fn)}_{u, \textbf{b}}\right)  ,\label{eq:QQA} 
\end{align}
where we denote $\ba=([a_1],\ldots,[a_\fn])$, $\ba'=([a_1'],\ldots,[a_\fn'])$, and $\mathbf b=([b_1],\ldots,[b_\fn])$, $\mathbf b'=([b_1'],\ldots,[b_\fn'])$. By definition, we can verify that the tensor in \eqref{eq:QQA}, denoted by $\cal A$, satisfies the \emph{double sum zero property}:
\be\label{eq:doublesumzero}
\sum_{[a_2],\ldots,[a_\fn]} \cal A_{\ba,\mathbf b}=0,  \quad \text{and}\quad \sum_{[a_2],\ldots,[a_\fn]} \cal A_{\mathbf b, \ba}=0, \quad \forall [a_1]\in \Zn, \ \mathbf b\in (\Zn)^{\fn}. 
\ee
Furthermore, similar to \Cref{lem_+Q}, we can check that if $\cal A$ satisfies the $(u, \e, D)$-decay property, then 
\begin{align}\label{norm_doubleQA}
\left\|\left({\cal Q}_u\otimes {\cal Q}_u\right) \circ \cal A\right\|_{\infty} \le W^{C_\fn \e} \|\cal A\|_\infty +W^{-D+C_\fn}
\end{align}
for a constant $C_\fn$ that does not depend on $\e$ or $D$. We now show that the bound \eqref{sum_res_2} can be improved if $\cal A$ satisfies both the fast decay and double sum zero properties.

\begin{lemma} 
Let ${\cal A}: (\Zn)^{2\fn}\to \mathbb C$ be a $(2\fn)$-dimensional tensor for a fixed $\fn\ge 2$. Suppose it satisfies the double sum zero property \eqref{eq:doublesumzero} and the $(s,\e,D)$-decay property for some constants $\e,D>0$. Fix any $0\le s \le t < 1$ such that $(1-t)/(1-s)\ge W^{-d}$. Then, the following bound holds for every $\bsig\in \{+,-\}^{2\fn}$:
\begin{align}\label{sum_res_2_doublezero}
    \left\|{\cal U}^{(2\fn)}_{s,t,\boldsymbol{\sigma}} \circ {\cal A}\right\|_\infty \le W^{C_\fn\e} \left(\frac{\ell_s^d\eta_s}{\ell_t^d\eta_t}\right)^{2\fn}  \|{\cal A}\|_{\infty} 
   +W^{-D+C_\fn},
\end{align} 
where $C_\fn>0$ is a constant that does not depend on $\e$ or $D$.
\end{lemma} 
\begin{proof}
If there exists a $k\in\qqq{2\fn}$ such that $\sig_k=\sig_{k+1}$, then \eqref{sum_res_2_doublezero} follows from \eqref{sum_res_2_NAL}. 
Otherwise, the proof of \eqref{sum_res_2_doublezero} is similar to that of \eqref{sum_res_2}, where we need to prove a similar bound as in \eqref{sum_res_2_red}:
$$\sum_{\mathbf a',\mathbf b'\in (\Zn)^\fn} \prod_{i=1}^\fn \Xi^{(i)}_{[a_i][a_i']}  \cdot \prod_{i=1}^\fn \Xi^{(\fn+i)}_{[b_i][b_i']} \cdot \mathcal{A}_{\ba',\mathbf b'} \le W^{C_\fn\e} \left(\frac{\ell_s^d\eta_s}{\ell_t^d\eta_t}\right)^{2\fn} \|{\cal A}\|_{\infty} 
   +W^{-D+C_\fn}.$$
The only modification is that we choose the following decompositions of $\Xi^{(i)}$'s:  
\be\nonumber \Xi^{(i)}_{[a_i][a_i']}=\Xi^{(i)}_{[a_i][a_1']} + \wh \Xi^{(i)}_{[a_i];[a_1'][a_i']},\quad \Xi^{(\fn+i)}_{[b_i][b_i']}=\Xi^{(i)}_{[b_i][b_1']} + \wh \Xi^{(i)}_{[b_i];[b_1'][b_i']},\quad 
\forall i\in\qqq{\fn}. 
\ee
Then, we can define $f(A)$ in a similar way as \eqref{eq:decompXii}. Due to the double sum zero property of $\cal A$, we notice that $f(A)=0$ if $|A^c|\le 1$. When $|A^c|\ge 2$, the derivation below \eqref{eq:decompXii} leads to \eqref{sum_res_2_doublezero} as explained in \eqref{eq:fA_A>2}. 
\end{proof}

Recall that $(\mathcal{B}\otimes  \mathcal{B})_{u,  \boldsymbol{\sigma}}^{(2\fn)}$ is bounded as in \eqref{CalEbwXi}. Then, using \eqref{norm_doubleQA} and \eqref{sum_res_2_doublezero}, we can bound that 
\begin{align*}
 \left\|  \left(
   \mathcal{U}^{(\fn)}_{u, t,  \boldsymbol{\sigma}}
   \otimes 
   \mathcal{U}^{(\fn)}_{u, t,  \overline{\boldsymbol{\sigma}}}
   \right)\circ 
\left( \mathcal{Q}_u\otimes \mathcal{Q}_u\right) \circ \left( \mathcal{B}\otimes  \mathcal{B} \right)_{u,  \boldsymbol{\sigma}}^{(2\fn)}\right\|_\infty \prec  \eta_u^{-1} \Xi^{(\cal L)}_{u, {2\fn+2}}\cdot   (W^d\ell_t^d\eta_t)^{-2\fn }.
\end{align*}
Plugging it into \eqref{UEUEdif+Q} and performing the integral over $u$,  we  obtain that  
$$\int_{s}^t \left(\mathcal{U}^{(\fn)}_{u, t, \boldsymbol{\sigma}} \circ  {\cal Q}_u\circ \dd\mathcal{B}^{(\fn)}_{u, \boldsymbol{\sigma}}\right)_{\ba} \prec (W^d\ell_t^d\eta_t)^{-\fn} \sup_{u\in[s,t]} \left(\Xi^{(\cal L)}_{u, {2\fn+2}}\right)^{1/2}
$$
by using the Burkholder-Davis-Gundy inequality.   

We still need to control the remaining 6 terms on the RHS of \eqref{int_K-L+Q}. It suffices to prove that:  
\begin{align}\label{int_K-L+Q_pf}
\left(\mathcal{U}^{(\fn)}_{s, t, \boldsymbol{\sigma}} \circ  {\cal Q}_s\circ \cal A_0(s)\right)_{\ba}  + \sum_{i = 1}^5 \int_{s}^t \left(\mathcal{U}^{(\fn)}_{u, t, \boldsymbol{\sigma}} \circ  {\cal Q}_u\circ \cal A_{i}(u)\right)_{\ba} \dd u \prec {\left(W^d\ell_t^d \eta_t\right)^{-\fn}}{\Phi},
\end{align}
where we abbreviate that 
\begin{align*}
& \cal A_0(s):=(\mathcal{L} - \mathcal{K})^{(\fn)}_{s, \boldsymbol{\sigma}},\ \ \cal A_1(u):= \sum_{\lenk=3}^\fn \Big[\mathcal{O}_{\cK}^{(\lenk)} (\mathcal{L} - \mathcal{K})\Big]^{(\fn)}_{u, \boldsymbol{\sigma}},\ \ \cal A_2(u):=\mathcal{E}^{(\fn)}_{u, \boldsymbol{\sigma}},\ \ \cal A_3(u):=\mathcal{W}^{(\fn)}_{u, \boldsymbol{\sigma}},  \\
& \cal A_4(u):=\left[{\cal Q}_u , \thn^{(\fn)}_{u,\boldsymbol{\sigma}} \right]\circ (\mathcal{L} - \mathcal{K})^{(\fn)}_{u, \boldsymbol{\sigma} },\ \ \cal A_5(u):=-\left[ {\cal P} \circ
      \left(\mathcal{L} - \mathcal{K}\right)^{(\fn)}_{u, \boldsymbol{\sigma}}\right] \partial_u\dthn_{u}^{(\fn)},\\
&\Phi:=\sup_{u\in [s,t]}
  \left(\max_{k\in \qqq{1,\fn-1} }\Xi^{({\cal L-\cal K})}_{u,k}+(W^d\ell_u^d\eta_u)^{-1}\max_{k\in \qqq{2, \fn} }\left(\Xi^{({\cal L-\cal K})}_{u,k}\Xi^{({\cal L-\cal K})}_{u,\fn-k+2}\right) +\Xi^{({\cal L})}_{u,\fn+1}\right).
\end{align*}
Here, we also used that $\cal Q_u\circ \cal A_i(u)=\cal A_i(u)$, $i\in\{4,5\}$, since $\cal A_4$ and $\cal A_5$ satisfy the sum zero property as shown by \eqref{eq:sumzero666} and \eqref{pqthlk}. 
We take the $\cal A_0$ term as an example to explain heuristically how to get an improved estimate in 2D to eliminate the $\ell_t/\ell_s$ factor in \eqref{am;asoi222_weak}. Note $\cal Q_s\circ \cal A_0$ satisfies the sum zero and fast decay properties. Moreover, similar to the argument below \eqref{int_K-L+QE}, we know its expectation satisfies the symmetry \eqref{eq:A_zero_sym} by the translation invariance and parity symmetry of our model on the block level (as explained below \Cref{def: BM}). Hence, applying the bound \eqref{sum_res_2_sym} instead of \eqref{sum_res_2} leads to the desired improvement for
\smash{$\mathcal{U}^{(\fn)}_{s, t, \boldsymbol{\sigma}} \circ  {\cal Q}_s\circ \E\cal A_0(s)$}.
It remains to deal with the fluctuation term:
\be\label{eq:CLT-Ust}\left(\mathcal{U}^{(\fn)}_{s, t, \boldsymbol{\sigma}} \circ  {\cal Q}_s\circ \Q\left(\cal A_0(s)\right)\right)_{\ba},\quad \text{with}\quad \Q:=1-\E. \ee
We know the kernel \smash{$\mathcal{U}^{(\fn)}_{s, t, \boldsymbol{\sigma}}$} will propagate the  estimate of ${\cal Q}_s\circ \Q(\cal A_0(s))$ on shorter scales of order $\ell_s$ to the larger scale of order $\ell_t$. 
On the other hand, intuitively, $(\cal A_0(s))_{\mathbf b}$ and $(\cal A_0(s))_{\mathbf b'}$ should be asymptotically independent when $\|\mathbf b-\mathbf b'\|_\infty$ is much larger than the typical scale of decay $\ell_s$ for $\cal A_0(s)$. Then, the term \eqref{eq:CLT-Ust} roughly involves a superposition of $\ell_t^2/\ell_s^2$ many asymptotically independent centered random variables, which leads to an improvement of order $\ell_s/\ell_t$ by CLT.   

Now, we present a rigorous proof of the above reasoning. We first bound the expectation of the LHS of \eqref{int_K-L+Q_pf}. By the induction hypothesis \eqref{Eq:L-KGt+IND}, the estimates \eqref{CalEbwXi1}, \eqref{CalEbwXi2}, \eqref{CalEbwXi3}, \eqref{A5}, and \eqref{A4}, we know
\be\label{A0-3} \cal A_0(s) \prec (W^d\ell_s^d\eta_s)^{-\fn},\quad \max_{i=1}^5 |\cal A_i(u)| \prec \eta_u^{-1}(W^d\ell_u^d\eta_u)^{-\fn}\cdot \Phi.\ee
With the above estimates in \eqref{A0-3}, using \Cref{lem_+Q} and the improved bound \eqref{sum_res_2_sym}, we obtain that 
\begin{align*}
\left(\mathcal{U}^{(\fn)}_{s, t, \boldsymbol{\sigma}} \circ  {\cal Q}_s\circ \E\cal A_0(s)\right)_{\ba}  + \sum_{i = 1}^5 \int_{s}^t \left(\mathcal{U}^{(\fn)}_{u, t, \boldsymbol{\sigma}} \circ  {\cal Q}_u\circ \E\cal A_{i}(u)\right)_{\ba} \dd u \prec \frac{\Phi}{(W^d\ell_t^d \eta_t)^{\fn}}.
\end{align*}
To conclude \eqref{int_K-L+Q_pf}, it remains to show that 
\begin{align}\label{int_K-L+Q_pfexp}
\left(\mathcal{U}^{(\fn)}_{s, t, \boldsymbol{\sigma}} \circ  {\cal Q}_s\circ \Q\cal A_0(s)\right)_{\ba}  + \sum_{i = 1}^5 \int_{s}^t \left(\mathcal{U}^{(\fn)}_{u, t, \boldsymbol{\sigma}} \circ  {\cal Q}_u\circ \Q\cal A_{i}(u)\right)_{\ba} \dd u \prec \frac{\Phi}{(W^d\ell_t^d \eta_t)^{\fn}}.
\end{align}
For this purpose, we define the rescaled random variables 
$$Y_0(s):=(W^d\ell_s^d\eta_s)^{\fn} \cal Q_s(\cal A_0(s)),\ \ \text{and}\ \  Y_i(u):=\eta_u(W^d\ell_u^d\eta_u)^{\fn} \cal Q_u(\cal A_i(u))/\Phi,\quad \forall i\in\qqq{5}. $$
By \eqref{A0-3}--\eqref{A4}, we have that $Y_0(s)\prec 1$ and $Y_i(u)\prec 1$ uniformly in $u\in[s,t]$ for any $i\in\qqq{5}$. To show \eqref{int_K-L+Q_pfexp}, it suffices to establish the following bounds for each $u\in[s,t]$ and $  i \in\qqq{5}$:
\begin{align}\label{int_K-L+Y_pfexp}
\left(\mathcal{U}^{(\fn)}_{s, t, \boldsymbol{\sigma}} \circ \Q (Y_0(s))\right)_{\ba} \prec \left(\frac{\ell_s^d\eta_s}{\ell_t^d\eta_t}\right)^{\fn} ,\quad   \left(\mathcal{U}^{(\fn)}_{u, t, \boldsymbol{\sigma}} \circ  \Q (Y_{i}(u))\right)_{\ba}  \prec \left(\frac{\ell_u^d\eta_u}{\ell_t^d\eta_t}\right)^{\fn}.
\end{align}

Without loss of generality, in the following proof of \eqref{int_K-L+Y_pfexp}, we will always take the starting time for the evolution operator as $s$; it is obvious that the same argument applies to \smash{$\mathcal{U}^{(\fn)}_{u, t, \boldsymbol{\sigma}}$}. We pick an $\fn$-dimensional tensor $Y(s)\in \{Y_i(s):i\in \qqq{0,5}\}$, which satisfies the sum zero and $(s,\e,D)$-decay properties for any constants $\e,D>0$. Then, we repeat the argument in the proof of \Cref{lem:sum_decay}. First, in the decomposition \eqref{eq:decomp_U2} (with $\cal A$ replaced by $Y$), it suffices to consider the case $A=\emptyset$, since for 
$A\ne\emptyset$, the desired bound \eqref{int_K-L+Y_pfexp} already follows from \eqref{sum_res_1_red0}. Next, for the decomposition \eqref{eq:decompXii} (with $\cal A$ replaced by $Y$), the leading term $f(A)$ with $A^c=\emptyset$ vanishes due to the sum zero property of $Y$, while the case $|A^c|\ge 2$ can also be controlled as explained around \eqref{eq:fA_A>2}. We are again left with the main challenging case $|A^c|=1$. Without loss of generality, suppose $A^c=\{2\}$. Then, we need to prove that: \be\label{eq:main_challenge}
\sum_{\mathbf b\in (\Zn)^\fn}\prod_{i\ne 2}  \Xi^{(i)}_{[a_i][b_1]} \cdot   \wh \Xi^{(2)}_{[a_2];[b_1][b_2]}  \cdot \Q\left(Y_{\mathbf{b}}\right) \prec \left(\frac{\ell_s^d\eta_s}{\ell_t^d\eta_t}\right)^{\fn},\quad \forall \ba\in (\Zn)^\fn. 
\ee
Using the bounds in \eqref{eq:Xibb} and the fast decay property of $\Xi^{(i)}$, $\wh\Xi^{(2)}$, and $Y$, for any constant $\e,D>0$ we can write the LHS of \eqref{eq:main_challenge} as
\be\label{eq:main_challenge2}
\left(\frac{\eta_s}{\ell_t^d \eta_t}\right)^{\fn}\sum_{\mathbf b} c^\star_{\mathbf b}\frac{|[b_2]-[b_1]|}{\langle [a_2]-[b_1]\rangle + |[b_2]-[b_1]|} 
 \Q\left(Y_{\mathbf{b}}\right) + \OO(W^{-D})\ee
with high probability. Here, 
$c^\star_{\mathbf b}$ is a deterministic tensor satisfying $\|c^\star\|_\infty \prec 1$ and that $c^\star_{\mathbf b}=0$ if $\max_{i=2}^\fn|[b_i]-[b_1]|> W^\e\ell_s$ or $\max_{i=1}^\fn|[a_i]-[b_1]|> W^\e\ell_t$, i.e., 
\be\label{eq:restrict_c} c^\star_{\mathbf b}= c^\star_{\mathbf b}\cdot \mathbf 1\left( \max_{i=2}^\fn|[b_i]-[b_1]|\le W^\e\ell_s\right)\cdot \mathbf 1\left(\max_{i=1}^\fn|[a_i]-[b_1]|\le W^\e\ell_t\right).\ee

We now control the first term in \eqref{eq:main_challenge}. By Markov's inequality, to show \eqref{eq:main_challenge}, it suffices to  establish the following bound on its $(2p)$-th moment for any fixed $p\in\N$ (recall that $Y_{\mathbf b}\prec 1$): 
\be\label{eq:main_challenge3}
\E\Big|\sum_{\mathbf b} \mu^\star_{\mathbf b}  \Q\left(Y_{\mathbf{b}}\right) \Big|^{2p} \prec \left(W^{d\e}\ell^d_s\right)^{2p\fn } ,\quad \mu^\star_{\mathbf b}:=c^\star_{\mathbf b}\frac{|[b_2]-[b_1]|}{\langle [a_2]-[b_1]\rangle + |[b_2]-[b_1]|}. 
\ee
We write the LHS of \eqref{eq:main_challenge3} as 
\be\label{eq:2p_product}\sum_{\mathbf b^{(1)},\ldots, \mathbf b^{(2p)}} \prod_{k=1}^{2p}\mu^\star_{\mathbf b^{(k)}}\cdot \E\bigg\{\prod_{k=1}^{p}   \left(Y_{\mathbf{b}^{(k)}}-\E Y_{\mathbf{b}^{(k)}}\right)\cdot \prod_{k=p+1}^{2p} \left(\overline Y_{\mathbf{b}^{(k)}}-\E \overline Y_{\mathbf{b}^{(k)}}\right)\bigg\}.\ee
First, suppose the following pairing condition holds: 
\be\label{eq:pairingcond}
\forall k\in \qqq{2p},\ \text{ there exists } \ l\ne k \ \text{ such that } \ |b^{(k)}_1-b^{(l)}_1|\le W^\e\ell_s.\ee  
Then, we can divide $\qqq{2p}$ into a disjoint union of $r$ subsets $\qqq{2p}=\sqcup_{i=1}^r E_i,$ constructed such that for each $k\in E_i$, any $l\in \qqq{2p}$ satisfying $|b^{(k)}_1-b^{(l)}_1|\le W^\e\ell_s$ also belongs to $E_i$ (note these subsets $E_i$ are unique up to relabeling of indices).  
According to the pairing condition, we have $r\le p$. Then, we can bound \eqref{eq:2p_product} subject to the condition \eqref{eq:pairingcond} as 
\begin{align}
    \label{eq:2p_product_pair}
\sum^*_{\mathbf b^{(1)},\ldots, \mathbf b^{(2p)}} \prod_{i=1}^{r}\left(\prod_{k\in E_i} \mu^\star_{\mathbf b^{(k)}}\right) \prec (W^{d\e} \ell_s^d)^{2p(\fn-1)}\sum^*_{[b^{(1)}_1],\ldots, [b^{(2p)}_1]} \prod_{i=1}^{r}\left(\prod_{k\in E_i} \frac{W^{\e}\ell_s}{\qq{[a_2]-[b_1^{(k)}]}+ W^{\e}\ell_s}\right) ,
\end{align}
where $\sum^*$ refers to the summation over the region of $\mathbf b^{(1)},\ldots, \mathbf b^{(2p)}$ such that the pairing condition \eqref{eq:pairingcond} holds, and we have used $\|c^\star\|_\infty \prec 1$ and the constraint imposed by \eqref{eq:restrict_c}. To control the RHS of \eqref{eq:2p_product_pair}, we further partition the summation $\sum^*$ into $\OO(1)$ distinct regions, denoted by $\sum^{*,(k_1,\cdots, k_r)}$, where each index $k_i\in E_i$ is chosen to attain the minimal distance within its respective group: 
$$\left|[a_2]-[b_1^{(k_i)}]\right|= \min_{k\in E_i}\left|[a_2]-[b_1^{(k)}]\right|,\quad i\in\qqq{r}.$$ 
We can bound the summation over each region as follows (recall that $d=2$):   
\begin{align*}
\sum^{*,(k_1,\cdots, k_r)}_{[b^{(1)}_1],\ldots, [b^{(2p)}_1]} \prod_{i=1}^{r}\left(\prod_{k\in E_i} \frac{W^{\e}\ell_s}{\qq{[a_2]-[b_1^{(k)}]}+W^{\e}\ell_s}\right)  \prec &~\left(W^{d\e}\ell_s^d\right)^{2p-r} \sum_{[b^{(k_1)}_1],\ldots, [b^{(k_r)}_1]} \prod_{i=1}^{r}\left[\frac{W^\e\ell_s}{\qq{[a_2]-[b_1^{(k_i)}]} + W^\e\ell_s}\right]^{2} \\
\prec &~\left(W^{d\e}\ell_s^d\right)^{2p-r} (W^\e\ell_s)^{2r} = \left(W^{d\e}\ell_s^d\right)^{2p}.
\end{align*}
This gives \eqref{eq:main_challenge3} under the pairing condition \eqref{eq:pairingcond}. 

It remains to control \eqref{eq:2p_product} when \eqref{eq:pairingcond} does not hold. We will prove the following result: if there exists an isolated vertex $\mathbf b^{(i)}$ such that 
\be\label{eq:isolated}\min_{j:j\ne i}\left|[b_1^{(i)}]-[b_1^{(j)}]\right|\ge W^\e\ell_s,\ee 
then the following bound holds for any large constant $D>0$:
\be\label{eq:bound_isolated}
\E\bigg\{\prod_{k=1}^{p}   \left(Y_{\mathbf{b}^{(k)}}-\E Y_{\mathbf{b}^{(k)}}\right)\cdot \prod_{k=p+1}^{2p} \left(\overline Y_{\mathbf{b}^{(k)}}-\E \overline Y_{\mathbf{b}^{(k)}}\right)\bigg\} \le W^{-D}. 
\ee
Without loss of generality, suppose $i=1$ and abbreviate that $V\equiv V_s$ (recall \Cref{Def:stoch_flow}). We regard $Y$ as a function of the entries of $V$. Letting $V'$ be an \emph{independent copy} of $V$, we can write the LHS of \eqref{eq:bound_isolated} as  
\be\label{eq:EfF} \E \left[    \left(Y_{\mathbf{b}^{(1)}}(V)- Y_{\mathbf{b}^{(1)}}(V')\right)\cdot F(V)\right],\ee 
where we abbreviate that 
$$ F(V):=\prod_{k=2}^{p}   \left(Y_{\mathbf{b}^{(k)}}(V)-\E Y_{\mathbf{b}^{(k)}}\right)\cdot \prod_{k=p+1}^{2p} \left(\overline Y_{\mathbf{b}^{(k)}}(V)-\E \overline Y_{\mathbf{b}^{(k)}}\right).$$
Fix an arbitrary ordering $1,2,\ldots, N$ of the vertices in $\ZL$, and consider an arbitrary bijective ordering map on the index set of $V$: 
$$\phi:\left\{(i,j):1 \le i \le j\le N \right\}\to \{1,2,\ldots, \gamma_{\max}:=N(N+1)/2\}.$$
For any $0\le \gamma \le \gamma_{\max}$, we define the Hermitian matrix $V^{\gamma}= (V^{\gamma}_{ij}:1\le i\le j\le N)$ such that $V_{ij}^{\gamma}=\overline V_{ji}^\gamma = V_{ij}$ if $\phi(i,j)\leq \gamma$, and $V_{ij}^{\gamma}=\overline V_{ji}^\gamma = V_{ij}'$ otherwise. Under this definition, we have $V^0=V'$ and $V^{\gamma_{\max}}=V$. Then, we can write \eqref{eq:EfF} as a telescoping sum: 
\be\label{eq:EfF2} 
\sum_{\gamma=1}^{\gamma_{\max}}\E \left[    \left(Y_{\mathbf{b}^{(1)}}(V^{\gamma})- Y_{\mathbf{b}^{(1)}}(V^{\gamma-1})\right)\cdot F(V)\right].\ee
If $\gamma=\phi(i,j)$, then the matrices $V^\gamma$ and $V^{\gamma-1}$ differ only in the $(i,j)$ and $(j,i)$-th entries. In particular, in the summation, we only need to consider $\gamma$ that corresponds to indices $(i,j)$ belonging to the diagonal blocks; otherwise, we trivially have $V^{\gamma}=V^{\gamma-1}$. Now, for such a choice of $\gamma =\phi(i,j)$, 
we view $Y_{\mathbf{b}^{(1)}}(V^{\gamma})$ and $Y_{\mathbf{b}^{(1)}}(V^{\gamma-1})$ as functions $f(V_{ij},V_{ji})$ and $f( V'_{ij},V'_{ji})$, respectively. Then, performing Taylor expansions of 
$g_1(\theta):=f(\theta V_{ij},\theta V_{ji})$ and $g_2(\theta):=f(\theta V'_{ij},\theta V'_{ji})$ around $\theta=0$, we obtain that for any fixed $K\in \N$, 
\begin{align}\label{eq:f-f}
f(V_{ij},V_{ji})-f(V_{ij}',V_{ji}')=\sum_{k = 1}^{K} \frac{1}{k!}\left[ g_1^{(k)}(0)-g_2^{(k)}(0)\right] + \frac{1}{(K+1)!}\left[g_1^{(K+1)}(\xi_1)-g_2^{(K+1)}(\xi_2)\right],
\end{align}
where $\xi_1,\xi_2\in[0,1]$ are random variables depending on $f$,  $V^\gamma$, $V^{\gamma-1}$, and $K$. Let $\partial_{ij}$ and $\partial_{ji}$ denote the partial derivative of $f$ with respect to the first and second arguments, respectively. We can write \eqref{eq:f-f} as 
\begin{align*}
f(V_{ij},V_{ji})-& f(V_{ij}',V_{ji}')=
\sum_{\al,\beta:1\le \al+\beta \le K}\frac{1}{\al!\beta!}\cal I_{\al,\beta} + \OO(\Err_{K+1}),
\end{align*}
where $\cal I_{\al,\beta}$ and $\Err_{K+1}$ are defined as
\begin{align}
\cal I_{\al,\beta}&:=\partial_{ij}^\al\partial_{ji}^\beta f(0,0)\cdot \left[(V_{ij})^\al (V_{ji})^\beta - (V_{ij}')^\al (V_{ji}')^\beta  \right],\label{eq:Ialbeta}\\
\Err_{K+1}&:=\left(|V_{ij}|^{K+1}+ |V_{ij}'|^{K+1}\right)\cdot \max_{\al,\beta:\al+\beta = K+1} \left( \left|\partial_{ij}^\al\partial_{ji}^\beta f(\xi_1V_{ij},\xi_1V_{ji})\right|+ \left|\partial_{ij}^\al\partial_{ji}^\beta f(\xi_2V'_{ij},\xi_2V'_{ji})\right|\right).\label{eq:ErrK}
\end{align}
To control these terms, we claim the following estimates. 
\begin{claim}\label{claim_locallaw}
Fix any $u\in [s,t]$. In the above setting, let $\xi$ be a random variable satisfying $|\xi|\prec W^{-d/2}$ and define $W_u\equiv W^\gamma_u(\xi)$ to be the random matrix obtained by replacing the $(i,j)$-th and $(j,i)$-th entries of $V_u$ with $\xi$ and $\overline \xi$, respectively. Suppose the estimates \eqref{Gt_bound_flow} and \eqref{Eq:Gdecay_w} holds for $G_u$. Then, the resolvent \be\label{eq:Rugamma}R_u\equiv R_u^\gamma(\xi):=(\lambda\Psi+W_u-z_u)^{-1}\ee satisfies the following local law:
\begin{equation}\label{eq:locallaw_Wu}
    \|R_u - M \|\prec (W^d\ell_u^d\eta_u)^{-1/2},
\end{equation}
and the following decay estimate: if $|[a]-[b]|\ge W^{\e'}\ell_u$ for a constant $\e'>0$, then for any constant $D'>0$,
\begin{equation}\label{eq:locallaw_Wu_decay}
    \max_{x\in[a],y\in[b]}|(R_u-M)_{xy}|\prec W^{-D'}.
\end{equation}
\end{claim}
\begin{proof}
This claim is a simple consequence of the resolvent expansion formula. More precisely, define the following $N\times N$ matrix with only two nonzero entries at $(i,j)$ and $(j,i)$-th locations:
$$(\Delta^{\gamma}_u)_{kl}= ((V_u)_{ij}-\xi) \mathbf 1(k=i, l=j) + \overline {((V_u)_{ij}-\xi)} \mathbf 1(k=j, l=i) . $$
Then, we can write that $W_u=V_u-\Delta^{\gamma}_u$, and we can expand $R_u$ in terms of $G_u$ as follows for each $K\in \N$:
\begin{equation}
R_u=(I-G_u \Delta^{\gamma}_u)^{-1}G_u=G_u +G_u \Delta^{\gamma}_u G_u +(G_u \Delta^{\gamma}_u)^2 G_u +\ldots+ (G_u \Delta^{\gamma}_u)^K G_u + (G_u \Delta^{\gamma}_u)^{K+1}R_u. \label{RESOLVENTEXPANSION}
\end{equation}
By assumptions and \Cref{lem_GbEXP}, we know that $G_u-M$ satisfies the two bounds in \eqref{eq:locallaw_Wu} and \eqref{eq:locallaw_Wu_decay}. For the remaining terms in \eqref{RESOLVENTEXPANSION}, using the local law \eqref{Gt_bound_flow} for $G_u$, the fact $|\xi|+|(V_u)_{ij}|\prec W^{-d/2}$, and the trivial bound $\|R_u\|\le \eta_u^{-1}$, it is easy to see that for any $x,y\in \Zn$,
\begin{align}\label{eq:Gk} 
    \left((G_u \Delta^{\gamma}_u)^k G_u\right)_{xy}&\prec W^{-dk/2}\le (W^d\ell_u^d\eta_u)^{-1/2},\quad  \ 1\le k\le K,\\
    \left((G_u \Delta^{\gamma}_u)^{K+1} R_u\right)_{xy}&\prec W^{-d(K+1)/2}\eta_u^{-1} \le W^{-D'},
\end{align}
as long as we take $K$ sufficiently large depending on $D'$. These conclude \eqref{eq:locallaw_Wu}. To show \eqref{eq:locallaw_Wu_decay}, let $x_i$ and $x_j$ represent the vertices in $\ZL$ indexed by $i$ and $j$. Then, $((G_u \Delta^{\gamma}_u)^k G_u)_{xy}$ can be expanded into a sum of $2^k$ many terms, each of which contains two resolvent entries of the form $(G_u)_{xw_1}(G_u)_{w_2y}$ with $w_1,w_2\in \{x_i,x_j\}$. Since $x_i$ and $x_j$ belong to the same block, we must have $|[a]-[w_1]|\vee |[b]-[w_2]|\ge W^{\e'}\ell_u/2$. Then, by the fast decay property of $G_u$, \smash{$(G_u)_{xw_1}(G_u)_{w_2y}$} can be bounded by $W^{-D'}$ for any large constant $D'>0$, which concludes \eqref{eq:locallaw_Wu_decay}.
\end{proof}

Finally, we use \Cref{claim_locallaw} to control \eqref{eq:Ialbeta} and \eqref{eq:ErrK}. With the local law \eqref{eq:locallaw_Wu}, we obtain that
$$\left|\partial_{ij}^\al\partial_{ji}^\beta f(\xi_1V_{ij},\xi_1V_{ji})\right|+ \left|\partial_{ij}^\al\partial_{ji}^\beta f(\xi_2V'_{ij},\xi_2V'_{ji})\right| \prec N^{C_\fn}$$
for a constant $C_\fn$ that does not depend on $K$. Then, using $|V_{ij}|^{K+1}+ |V_{ij}'|^{K+1}\prec W^{-dK/2}$, for any large constant $D'>0$, we can ensure that 
\be\label{eq:ErrK+1} \Err_{K+1} \prec W^{-D'}\ee
by taking $K$ sufficiently large depending on $C_\fn$ and $D'$. Next, we consider 
$$\E \left[\cal I_{\al,\beta}\cdot F(V)\right]= \E \left\{ \partial_{ij}^\al\partial_{ji}^\beta f(0,0)\cdot \left[(V_{ij})^\al (V_{ji})^\beta - \E\left((V_{ij}')^\al (V_{ji}')^\beta\right) \right]\cdot F(V)\right\}$$
with $1\le \al+\beta \le K$.  
We apply Gaussian integration by parts repeatedly with respect to $V_{ij}$ and $V_{ji}$. For example, we can first apply Gaussian integration by parts with respect to $V_{ij}$ for $\al$ many times and get
\begin{align}\label{eq:GIBP}
\E \left[\cal I_{\al,\beta}\cdot F(V)\right] = \E \left\{ \partial_{ij}^\al\partial_{ji}^\beta f(0,0)\cdot \left[(S_{ij})^{\al}  \partial_{V_{ji}}^\al \left[(V_{ji})^\beta \cdot F(V)\right] - \E\left((V_{ij}')^\al (V_{ji}')^\beta\right) \cdot F(V) \right]\right\}, 
\end{align}
where $S_{ij}$ denotes the variance of $V_{ij}$ (at time $s$, we have $S_{ij}=sW^{-d}$). We expand the first term as a sum of $\OO(1)$ many terms:
$$ (S_{ij})^{\al}  \partial_{V_{ji}}^{\al} \left[(V_{ji})^\beta \cdot F(V)\right]=(S_{ij})^{\al}\sum_{k=0}^\al c_{\beta,k} (V_{ji})^{\beta-k}\partial_{V_{ji}}^{\al-k}F(V),$$
where $c_{\beta,k}=\OO(1)$ are deterministic coefficients. We consider the following three cases.
\begin{itemize}
    \item[(1)] If $k=\al=\beta$, then by Wick's theorem for Gaussian random variables, we have $c_{\al,\al}(S_{ij})^{\al}=\E((V_{ij}')^\al (V_{ji}')^\beta)$, which cancels the second term in \eqref{eq:GIBP}. 
    
    \item[(2)] If $k<\al$, then $\partial_{V_{ji}}^{\al-k}F(V)$ contains at least one $G$ edge between one of the blocks \smash{$[b_i^{(j)}]$}, $i\in\qqq{2,2p}$ and $j\in\qqq{1,\fn}$, and the block, denoted by $[a]$, containing the vertices indexed by $i$ and $j$. On the other hand, \smash{$\partial_{ij}^\al\partial_{ji}^\beta f(0,0)$} also contains at least one $R_{u}^\gamma(\xi=0)$ edge (recall \eqref{eq:Rugamma}) between one of the blocks {$[b_1^{(j)}]$}, $j\in\qqq{1,\fn}$, and the block $[a]$. Hence, using the condition \eqref{eq:isolated} and the fast decay property in \eqref{eq:locallaw_Wu_decay}, we obtain that 
$$\partial_{ij}^\al\partial_{ji}^\beta f(0,0)\cdot (S_{ij})^{\al}\cdot  c_{\beta,k} (V_{ji})^{\beta-k}\partial_{V_{ji}}^{\al-k}F(V)\prec W^{-D'},$$
for any large constant $D'>0$. 

\item[(3)] If $k=\al <\beta$, then we apply another Gaussian integration by parts with respect to $V_{ji}$ and get 
$$\partial_{ij}^\al\partial_{ji}^\beta f(0,0)\cdot(S_{ij})^{\al+1}\cdot c_{\beta,k} (V_{ji})^{\beta-k-1}\partial_{V_{ij}}F(V).$$
With the same argument as in Case (2), we see that this term is also of order $\OO(W^{-D'})$ for any large constant $D'>0$. 
\end{itemize}
In sum, we have derived that 
\be\label{eq:Errkkk}
\E \left[\cal I_{\al,\beta}\cdot F(V)\right]  \prec W^{-D'}.
\ee
Combining \eqref{eq:Errkkk} with \eqref{eq:ErrK+1}, and choosing $D'>0$ sufficiently large, we obtain \eqref{eq:bound_isolated}, which in turn establishes \eqref{eq:main_challenge3}. This leads to \eqref{eq:main_challenge}, which implies \eqref{int_K-L+Y_pfexp} and hence concludes the proof of \Cref{lem:STOeq_Qt}.

\section{Proof of \Cref{lem:propM,lem_propTH}}\label{appd:deter}

\subsection{Proof of \Cref{lem:propM}}
Property (1) follows directly from the definition of $M^{(\sigma_1,\sigma_2)}$.
Property (2) is due to the translation invariance and parity symmetry of $M$ at the block level, which, in turn, arise from the corresponding translation invariance and parity symmetry of $\Psi$ at the block level. 
Using the definition of $M$ in \eqref{def_G0} and Ward's identity in \Cref{lem-Ward}, we obtain that for any $x\in \ZL$ and $z=E+\ii \eta$, 
\be\label{L2M} 
\sum_{y}|M_{xy}(z)|^2  = \sum_{y}|M_{yx}(z)|^2 =\frac{\im M_{xx}(z)}{\eta + \im m(z)}.   
\ee
Together with equation \eqref{eq:averm}, it implies the identity \eqref{eq:WardM1} when $\sigma_1\ne \sigma_2$. 
Combining \eqref{eq:WardM1} with the Cauchy-Schwarz inequality yields \eqref{eq:WardM2}.

To show \eqref{eq:M-msc} for $M$ defined in \eqref{def_G0}, we use that for $p\in\{2,\infty\}$,
$$ \|M-m_{sc}I_N\|_{p\to p} \le \|M+(z+m)^{-1}I_N\|_{p\to p} + |(z+m_{sc})^{-1} -(z+m)^{-1}| \lesssim \|\Psi\|_{p\to p}+|m-m_{sc}|.$$
By \eqref{eq:expandm}, we have $m=m_{sc}+\OO(\lambda^2)$. Plugging it into the above equation and using the condition \eqref{eq:cond_A12}, we derive \eqref{eq:M-msc} immediately.  
For the bound \eqref{Mbound_AO}, 
with the expansion \eqref{eq:expandM}, and using the condition \eqref{eq:cond_A12} and the fact that $\Psi|_{[x][y]}=0$ for $|[x]-[y]|>1$, we can readily obtain \eqref{Mbound_AO}. Finally, for \eqref{Mbound_AO2}, we notice that by \eqref{eq:cond_A12}, there exists a constant $C>0$ such that
$$ \left\|\lambda^k\prod_{i=1}^{k}A_{s_i}\right\|_{\HS}^2 \le C\lambda^2(\|A_0\|_{\HS}^2\vee\|A_1\|_{\HS}^2\vee\|A_2\|_{\HS}^2)\cdot (C\heta)^{2k-2},\ \ \ \forall k\ge 1,\ \ (s_1,\ldots, s_k)\in \{0,1,2\}^k.$$
With this bound, \eqref{Mbound_AO2} can be derived easily from \eqref{eq:expandM} and the assumption $\|A_i\|_{\HS}^2\asymp W^{d}$ for $i\in\{1,2\}$. 

\subsection{Proof of \Cref{lem_propTH}}
The property (1) follows properties (1) and (2) of \Cref{lem:propM}. 
The estimate \eqref{prop:ThfadC_short} can be readily derived from the following expansion for $\sig_1=\sig_2=\sig$:
\begin{align}
\left(1-tM^{(\sigma,\sigma)}\right)^{-1}_{0[x]} &=\left[\left(1-tM^{(\sigma,\sigma)}_{00}I\right)- t \wt M^{(\sigma,\sigma)}\right]^{-1}_{0[x]} = \sum_{k=0}^{\infty}\left(1-tM^{(\sigma,\sigma)}_{00}\right)^{-(k+1)}
 \left(t \wt M^{(\sigma,\sigma)}\right)^{k}_{0[x]},\label{eq:expMLn}
\end{align}
where $\wt M^{(\sigma,\sigma)}:=M^{(\sigma,\sigma)}- M^{(\sigma,\sigma)}_{00}I$ is the matrix obtained by setting the diagonal entries of $M^{(\sigma,\sigma)}$ to 0. By \eqref{eq:M-msc} and \eqref{Mbound_AO2}, we have that
\be\label{eq:op1-tM}\left|1-tM^{(\sigma,\sigma)}_{00}\right|=\left|1-tm_{sc}(\sigma)^2\right|+\oo(1)\gtrsim 1,\quad \left\|\wt M^{(\sigma,\sigma)}\right\|\lesssim \lambda^2.\ee 
Thus, there exists a constant $C>0$ such that  
\begin{align}\label{eq:Tyalor1-M}
\left|\left(1-tM^{(\sigma,\sigma)}\right)^{-1}_{0[x]}\right| \le C \sum_{k=0}^{\infty}
 \left|\left(C t \wt M^{(\sigma,\sigma)}\right)^{k}_{0[x]}\right| \lesssim \mathbf 1_{[x]=0} + \sum_{k=1}^{K}
 \left|\left(C t \wt M^{(\sigma,\sigma)}\right)^{k}_{0[x]}\right|+\OO(W^{-D}),
\end{align}
where $K\in \N$ is a large constant depending on $D$. By \eqref{Mbound_AO2} and \eqref{eq:op1-tM}, it is easy to see that for $1\le k\le K$,  
$$\left|\left(C t \wt M^{(\sigma,\sigma)}\right)^{k}_{0[x]}\right| \lesssim \lambda^{2k}\heta^{2(|[x]|-k)_+},$$
where we recall the notations in \eqref{Japanesebracket2}. 
Plugging it into \eqref{eq:Tyalor1-M}  concludes \eqref{prop:ThfadC_short}.

For the proofs of \eqref{prop:ThfadC}, \eqref{prop:BD1}, and \eqref{prop:BD2}, assume that $\sig_1=-\sig_2=+$ without loss of generality. By \eqref{eq:WardM1},  $M^{(+,-)}$ is a doubly stochastic matrix, which can be regarded as the probability transition matrix $P$ for a random walk on \smash{$\Zn$}. Furthermore, due to the translation invariance of $M^{(+,-)}$, its eigenvectors are given by the plane waves:
$$v_{\mathbf p}:=n^{-d/2}\left(e^{\ii \bp\cdot [x]}:[x]\in  \Zn\right), \quad \bp\in  \T_n^d:= \left({2\pi }/{n}  \right)^d\Zn.$$
Denoting the corresponding eigenvalues of $(1-M^{(+,-)})$ by $e_{\mathbf p}$, it is easy to check that $e_0=0$ and $e_{\mathbf p}$ are non-negative eigenvalues of order $\lambda^2|\mathbf p|^2$. Then, we can express \smash{$\Theta_{t}^{(+,-)}$} through a Fourier series:
\be\label{eq:Fourier_Theta} \Theta_{t}^{(+,-)} ([0],[x])=\frac1{n^d}  \sum_{\mathbf p\in \T_n^d}  \frac{e^{\ii \bp\cdot [x]}}{(1-t)+te_{\mathbf p}}. \ee
With it, we also express the LHS of \eqref{prop:BD1} and \eqref{prop:BD2} by 
\begin{align}\label{eq:Fourier_Theta1}
&\Theta^{(+,-)}_{t}(0, [x])-\Theta^{(+,-)}_{t}(0, [y]) = \frac1{n^d}  \sum_{\mathbf p\in \T_n^d}  \frac{1-e^{\ii\bp\cdot ([y]-[x])}}{(1-t)+te_{\mathbf p}}e^{\ii \bp\cdot [x]},\\
&\Theta^{(+,-)}_{t} (0,[x]+[y]) + \Theta^{(+,-)}_{t} (0,[x]-[y])-  2\Theta^{(+,-)}_{t} (0,[x]) =\frac{2}{n^d}  \sum_{\mathbf p\in \T_n^d}  \frac{\cos(\bp\cdot[y])-1}{(1-t)+te_{\mathbf p}}e^{\ii \bp\cdot [x]}.\label{eq:Fourier_Theta2}
\end{align}
By analyzing these Fourier series using a summation by parts argument, as in Appendix E of \cite{RBSO}, we can derive the estimates \eqref{prop:BD1} and \eqref{prop:BD2}. Here, we omit the detailed proof and instead provide a heuristic explanation for why these bounds hold. 
In \eqref{eq:Fourier_Theta1}, $1-e^{\ii\bp\cdot ([y]-[x])}$ contributes a factor of $|\bp\cdot ([y]-[x])|$. Using the fact that $e_{\mathbf p}\gtrsim \lambda^2|\mathbf p|^2$ for $|\bp|=\oo(1)$, we can bound the RHS of \eqref{eq:Fourier_Theta1} by  
$$\frac{|[y]-[x]|}{n^d}  \sum_{\mathbf p\in \T_n^d} \left( \frac{|\bp|}{|1-t|}\wedge \frac{1}{\lambda^2|\bp|}\right) \prec \frac{|[y]-[x]|}{\lambda^2+|1-t|}$$
in 1D, which establishes \eqref{prop:BD1} for $d=1$. For $d=2$, we first apply a summation by parts argument, which introduces a loss of one $|\bp|$ factor but gains an additional factor of $\qq{[x]}^{-1}$. This allows us to bound \eqref{eq:Fourier_Theta1} by 
$$ \frac{|[y]-[x]|}{\qq{[x]}}\frac{1}{n^d}  \sum_{\mathbf p\in \T_n^d} \frac{1}{|1-t|+t\lambda^2|\bp|^2}  \prec \frac{1}{\lambda^2+|1-t|} \frac{|[y]-[x]|}{\qq{[x]}}.$$
By symmetry, a similar bound holds with $\qq{[x]}$ in the denominator replaced $\qq{[y]}$. This concludes \eqref{prop:BD1} for $d=2$.
Similarly, in \eqref{eq:Fourier_Theta2}, $\cos(\bp\cdot [y])-1$ contributes a factor of $|\bp\cdot [y]|^2$. In $d=1$, applying a summation by parts argument to \eqref{eq:Fourier_Theta2} yields the bound 
$$ \frac{|[y]|^2}{\qq{[x]}}\frac{1}{n^d}  \sum_{\mathbf p\in \T_n^d} \frac{|\bp|}{|1-t|+t\lambda^2|\bp|^2}  \prec \frac{1}{\lambda^2+|1-t|} \frac{|[y]|^2}{\qq{[x]}}$$
in 1D. For $d=2$, applying summation by parts twice to \eqref{eq:Fourier_Theta2} results in a loss of a $|\bp|^2$ factor but gains a factor of $\qq{[x]}^{-2}$. Then, we can bound \eqref{eq:Fourier_Theta2} by 
$$ \frac{|[y]|^2}{\qq{[x]}^2}\frac{1}{n^d}  \sum_{\mathbf p\in \T_n^d} \frac{1}{|1-t|+t\lambda^2|\bp|^2}  \prec \frac{1}{\lambda^2+|1-t|} \frac{|[y]|^2}{\qq{[x]}^2}$$
in 2D. These bounds establish \eqref{prop:BD2}.

Finally, we prove the bound \eqref{prop:ThfadC}. To simplify notations, in the following proof, we denote the lattice \smash{$\Zn$} and its vertices $[x]$ by $\Z_n^d$ and $x$ instead. Recall that $P=M^{(+,-)}$ is the probability transition matrix for a random walk on $\Z_n^d$. Given $K\in \N$, we define $P_K$ as a cutoff of $P$: $$P_K(x,y)=P(0,x)\mathbf 1_{|x-y|\le K}.$$
We then normalize $P_K$ by its row sum (denoted by $b$) to obtain another doubly stochastic matrix \smash{$\wt P$}. Due to \eqref{Mbound_AO2}, given any large constant $D'>0$, we can find a constant $K\in \N$ large enough such that 
\be\label{eq:P-P'}|b-1|\le W^{-D'}\quad \text{and}\quad \|P-\wt P\| \le W^{-D'}.\ee
Using these bounds and the fact that $1-t\ge N^{-1}$, we readily obtain the estimate  
\be\label{eq:P-P'2} \big\|(1-tP)^{-1} - (1-t\wt P)^{-1}\big\| \lesssim N^2W^{-D'}\le W^{-D}\ee
as long as we choose $D'$ large enough. Now, to show \eqref{prop:ThfadC}, it suffices to control the following random walk representation of \smash{$(1- t\wt P)^{-1}$}: 
\begin{align}\label{series_RW}
\left(1 - t\wt P\right)^{-1}(0,x)=\sum_{k=0}^{\infty}\; t^k \wt P^k(0,x)=\sum_{k=0}^{\infty}\; t^k \P(\mathbf S_k=x),
\end{align} 
where $\mathbf S_k$ denotes $X_1+\cdots+X_k$ with $X_i$ being i.i.d.~random variables distributed according to $\wt P(0,\cdot).$
By \eqref{Mbound_AO2} and \eqref{eq:P-P'}, we have $\wt P(0,x)\le C\lambda^2(C\heta)^{2|x|-2}$ for $1\le |x|\le K$, and
\be\label{eq:propP}
\lambda^2 \lesssim \lambda_{d}(\Sigma)\le  \lambda_{1}(\Sigma)=\|\Sigma\| \lesssim \lambda^2, \ee
where $\Sigma$ denotes the covariance matrix for the random walk, i.e., $\Sigma:=\E(X_1X_1^\top)$, and $\lambda_{1}(\Sigma)$ and $\lambda_{d}(\Sigma)$ denote the largest and smallest eigenvalues of $\Sigma$, respectively. 

First, by applying the Bernstein inequality to $\mathbf S_k\cdot \wh x$ with $\wh x$ denoting the unit vector $x/\|x\|_2$, we can derive the following large deviation estimate: there exists a constant $c_1>0$ such that for each $x\in \Z_n^d$,  
\be\label{eq:Pkx}
\P(\mathbf S_k=x)\le 2\exp\left(-c_1\left(\frac{|x|^2}{\lambda^2 k}\wedge |x|\right)\right).
\ee
Furthermore, when $\lambda^2k\gg 1$, we can establish the following local CLT-type bounds on $P^k(0,x)$: there exist constants $c_2,c_3>0$ such that for any constants $\e,D>0$, if $W^\e \le \lambda^{2}k \le c_2 n^{2} $ and $|x|\ll (\lambda^2 k)^{2/3}$, then
\be\label{eq:Pkx1}
\P(\mathbf S_k=x)\prec \left(\lambda^2 k\right)^{-d/2} \exp\left(-c_3|x|^2/(\lambda^2 k)\right) + W^{-D} ; 
\ee
for $\lambda^{2}k \ge c_2 n^{2}$, we have
\be\label{eq:Pkx2}
\P(\mathbf S_k=x)\prec n^{-d}. 
\ee
The estimates \eqref{eq:Pkx1} and \eqref{eq:Pkx2} can be proved via the method of characteristic function and inverse discrete Fourier transform, following the proof of Theorem 1.5 in \cite{RW_Torus}; a similar argument is also presented in Lemma 30 of \cite{Band1D_III}. 
With the above estimates, we can control the series in \eqref{series_RW} as:
\begin{align*}
&(1 - t\wt P)^{-1}(0,x) \prec \sum_{\lambda^2 k\le |x| } t^k e^{-c_1|x|} + \sum_{  |x| \le  \lambda^2k \le W^\e\vee (|x|^{3/2} \log W)} t^k e^{-c_1|x|^2/(\lambda^2k)} +\frac{1}{n^d}\sum_{\lambda^2 k \ge c_2n^2}t^{k}\\
&\quad + 
\sum_{W^\e\vee (|x|^{3/2} \log W)\le \lambda^2 k \le c_2n^2}t^{k}\frac{e^{-c_3|x|^2/(\lambda^2 k)}}{(\lambda^2k)^{d/2}}   + W^{-D}=:I_1 + I_2 + I_3 + I_4+ W^{-D}. 
\end{align*}
Using that $\ell_t\ge 1$ and $|1-t|{\ell}_t^d \lesssim \lambda^2+|1-t|$, we can bound the term $I_1$ by 
\begin{align*}
 I_1\lesssim \left(\frac{1}{1-t}\wedge \frac{|x|}{\lambda^2}\right)e^{-c_1|x|} \lesssim \frac{e^{-c |[x]|/ {\ell}_t}}{|1-t|{\ell}_t^d}  \, ,
\end{align*}
where we also used that $|x|e^{-c_1|x|}\lesssim e^{-c_1|x|/2}$. For term $I_2$, if $|x|\ge W^{2\e/3}$, we can bound it by 
\begin{align*}
 I_2\lesssim \sum_{  |x| \le  \lambda^2k \le |x|^{3/2} \log W}t^k e^{-c_1|x|^{1/2}/\log W}\le W^{-D}.  
\end{align*}
If $|x|\le W^{2\e/3}$, using that $t^{k}\lesssim e^{-c_0(1-t)}$ for an absolute constant $c_0>0$, we can bound $I_2$ by 
\begin{align*}
 I_2\lesssim \sum_{  |x| \le  \lambda^2k \le W^\e \log W}e^{-\frac{c_0}{2}k|1-t|} e^{-\frac{c_0}{2}k|1-t|-\frac{c_1|x|^{2}}{\lambda^2 k}}\prec  \left(\frac{1}{1-t}\wedge \frac{W^\e}{\lambda^2}\right)\exp\left(\frac{-\sqrt{c_0c_1}|x|}{\lambda|1-t|^{-1/2}}\right)  \prec \frac{W^\e e^{-c|x|/\ell_t}}{{|1-t|{\ell}_t^d}}
\end{align*}
for the constant $c=\sqrt{c_0c_1}$. Here, in the second step, we used that 
\be\label{eq:AM-GM}\frac{c_0}{2}k|1-t|+\frac{c_1|x|^2}{\lambda^2k} \ge \sqrt{c_0c_1}\frac{|x|}{\lambda|1-t|^{-1/2}},\ee
and in the third step we used that for any constant $c>0$, 
\be\label{eq:expxl}
\exp\left(-\frac{c|x|}{\lambda|1-t|^{-1/2}}\right)\lesssim \exp\left(-\frac{c|x|}{\ell_t}\right),
\ee
since we have $\lambda|1-t|^{-1/2}\le \ell_t$ if $\lambda|1-t|^{-1/2}\le n$ or $e^{-c|x|/\ell_t}\gtrsim e^{-c|x|/(\lambda|1-t|^{-1/2})}$ if $\lambda|1-t|^{-1/2}\ge n=\ell_t$. The term $I_3$ can be bounded easily by 
$$I_3\lesssim \frac{e^{-c_0c_2 n^2/\lambda^2}}{n^d|1-t|} \lesssim \frac{e^{-c|x|/\ell_t}}{{|1-t|{\ell}_t^d}}.$$
Finally, the term $I_4$ can be bounded in a similar way as $I_2$: 
\begin{align*}
    I_4&\lesssim \sum_k \frac{e^{- c_3|x|^2/(\lambda^2 k)}}{(\lambda^2k)^{d/2}} \mathbf 1\left(\frac{W^\e\vee (|x|^{3/2} \log W)}{\lambda^{2}}\le k \le \frac{|x|}{\lambda|1-t|^{1/2}}\right) \\
   &+ \sum_k \frac{e^{- c_0k|1-t|}}{(\lambda^2k)^{d/2}}\mathbf 1\left(\frac{W^\e\vee (|x|^{3/2} \log W)}{\lambda^{2}} \vee \frac{|x|}{\lambda|1-t|^{1/2}}\le k \le \frac{c_2n^2}{\lambda^2}\right):=I_{41}+I_{42}.
\end{align*}
Note that for the term $I_{41}$ to be nonzero, we have to assume that $|1-t|\le \lambda^2$. This gives $\lambda^d |1-t|^{1-d/2} \gtrsim \lambda^d |1-t|^{1-d/2} + |1-t|\gtrsim |1-t|\ell_t^d$, and we can bound $I_{41}$ by 
\begin{align*}
I_{41} &\prec \frac{1}{\lambda^{d}}\left(\frac{|x|}{\lambda|1-t|^{1/2}}\right)^{1-d/2}\exp\left(-\frac{c_3 |x|}{\lambda|1-t|^{-1/2}} \right)\\
&\lesssim \frac{1}{\lambda^{d}|1-t|^{1-d/2}}\exp\left(-\frac{c_3 |x|}{2\lambda|1-t|^{-1/2}} \right)\lesssim\frac{e^{-{c |x|}/{\ell_t}}}{|1-t|\ell_t^d},    
\end{align*}
where we used \eqref{eq:expxl} in the third step. For the term $I_{42}$, if $|1-t|\le \lambda^2$, we can bound it as follows: for any constant $D>0$,
\begin{align*}
    I_{42} &\le \sum_k \frac{e^{- c_0k|1-t|}}{(\lambda^2k)^{d/2}}\mathbf 1\left(\frac{|x|}{\lambda|1-t|^{1/2}}\le k \le \frac{W^\e}{|1-t|}\right) + \sum_k \frac{e^{- c_0k|1-t|}}{(\lambda^2k)^{d/2}}\mathbf 1\left( k \ge \frac{W^\e}{|1-t|}\right) \\
    &\prec \frac{1}{ \lambda^d}\left( \frac{W^\e}{|1-t|}\right)^{1-d/2}\exp\left(-\frac{c_0 |x|}{\lambda|1-t|^{-1/2}} \right) + W^{-D} \lesssim \frac{W^\e e^{-{c |x|}/{\ell_t}}}{|1-t|\ell_t^d} + W^{-D}.
\end{align*}
If $|1-t|\ge \lambda^2$, we can trivially bound $I_{42}$ by  
$$I_{42}\le \sum_k \frac{e^{- c_0k|1-t|}}{(\lambda^2k)^{d/2}}\mathbf 1\left(\frac{W^\e}{\lambda^{2}} \le k \le \frac{c_2n^2}{\lambda^2}\right) \le W^{-D}.$$
Combing the above estimates for $I_i$, $i=1,2,3,4$, we conclude that 
$$(1 - t\wt P)^{-1}(0,x) \prec \frac{W^\e e^{-c|x|/\ell_t}}{{|1-t|{\ell}_t^d}} + W^{-D}.$$
Together with \eqref{eq:P-P'2}, this concludes \eqref{prop:ThfadC} since $\e$ is arbitrary.  

Finally, it is evident that all the above arguments also apply to the $\Theta$-propagator in \eqref{eq:Theta_WO}, defined for the Wegner orbital model, with $m_{sc}(\sigma_1)m_{sc}(\sigma_2)S^{\LK}$ replacing the role of $M^{(\sig_1,\sig_2)}$.

\end{document}